\documentclass[10pt]{article}
\usepackage[margin=1in]{geometry}

\usepackage{amsmath}
\usepackage{amssymb}
\usepackage{amsthm}
\usepackage{mathrsfs}
\usepackage{bm}
\usepackage{enumerate}
\usepackage{color}
\usepackage[colorlinks,linkcolor=black,citecolor=black,urlcolor=black]{hyperref}
\usepackage[sort, compress]{cite}

\def\R{\mathbb{R}}
\def\Z{\mathbb{Z}}
\def\P{\mathbb{P}}
\def\E{\mathbb{E}}

\def\cA{\mathcal{A}}
\def\cB{\mathcal{B}}

\def\N{\mathcal{N}}
\def\cD{\mathcal{D}}
\def\cF{\mathcal{F}}
\def\cG{\mathcal{G}}
\def\cM{\mathcal{M}}
\def\cP{\mathcal{P}}
\def\cS{\mathcal{S}}
\def\cT{\mathcal{T}}

\def\b{\mathbf{b}}
\def\e{\mathbf{e}}
\def\bE{\mathbf{E}}

\def\u{\mathbf{u}}
\def\v{\mathbf{v}}
\def\r{\mathbf{r}}

\def\x{\mathbf{x}}
\def\y{\mathbf{y}}

\def\H{\mathbf{H}}
\def\I{\mathbf{I}}
\def\J{\mathbf{J}}
\def\M{\mathbf{M}}
\def\U{\mathbf{U}}
\def\V{\mathbf{V}}
\def\W{\mathbf{W}}
\def\X{\mathbf{X}}

\def\bC{\mathbf{C}}

\def\bR{\mathbf{R}}
\def\sP{\mathsf{P}}

\def\btheta{\boldsymbol{\theta}}
\def\bbeta{\boldsymbol{\eta}}

\def\boldeta{\boldsymbol{\eta}}

\def\bOmega{\boldsymbol{\Omega}}
\def\eps{\varepsilon}
\def\beps{\boldsymbol{\varepsilon}}
\def\d{\mathrm{d}}

\def\event{\mathcal{E}}

\def\L{\mathrm{L}}
\def\1{\mathbf{1}}

\newcommand{\floor}[1]{\left\lfloor #1 \right\rfloor}
\newcommand{\ceil}[1]{\left\lceil #1 \right\rceil}
\newcommand{\pnorm}[2]{\lVert #1\rVert_{#2}}
\newcommand{\bigpnorm}[2]{\big\lVert#1\big\rVert_{#2}}
\newcommand{\biggpnorm}[2]{\bigg\lVert#1\bigg\rVert_{#2}}

\DeclareMathOperator{\diag}{Diag}
\DeclareMathOperator{\Tr}{Tr}

\DeclareMathOperator{\op}{op}

\def\op{\mathrm{op}}

\def\deps{\partial_\eps|_{\eps = 0}}
\def\Prob{\mathbb{P}}
\def\der{\d}

\newtheorem{theorem}{Theorem}[section]
\newtheorem{lemma}[theorem]{Lemma}
\newtheorem{proposition}[theorem]{Proposition}
\newtheorem{corollary}[theorem]{Corollary}
\theoremstyle{definition}

\newtheorem{assumption}[theorem]{Assumption}
\newtheorem{remark}[theorem]{Remark}

\title{Dynamical mean-field analysis of adaptive Langevin diffusions:
Propagation-of-chaos and convergence of the linear response}

\author{Zhou Fan\thanks{Department of Statistics and Data Science, Yale University}, Justin Ko\thanks{Department of Statistics and Actuarial Science, University of Waterloo}, Bruno Loureiro\thanks{Departement d'Informatique,  École Normale Supérieure, PSL \& CNRS}, Yue M.\ Lu\thanks{Departments of Electrical Engineering and Applied Mathematics, Harvard University}, Yandi Shen\thanks{Department of Statistics and Data Science, Carnegie Mellon University}}

\date{}

\begin{document}

\maketitle

\begin{abstract}
Motivated by an application to empirical Bayes learning in
high-dimensional regression, we study a class of Langevin diffusions in a system
with random disorder, where the drift coefficient is driven by a parameter that
continuously adapts to the empirical distribution of the realized process up to
the current time. The resulting dynamics take the form of a stochastic
interacting particle system having both a McKean-Vlasov type interaction and a pairwise interaction defined by the random disorder.
We prove a propagation-of-chaos result, showing that in the large
system limit over dimension-independent time horizons, the empirical
distribution of sample paths of the Langevin process converges to a
deterministic limit law that is described by dynamical mean-field theory. This
law is characterized by a system of dynamical fixed-point equations for the
limit of the drift parameter and for the correlation and response kernels of
the limiting dynamics. Using a dynamical cavity argument, we verify that these
correlation and response kernels arise as the asymptotic limits of the averaged
correlation and linear response functions of single coordinates of the system.
These results enable an asymptotic analysis of an empirical Bayes
Langevin dynamics procedure for learning an unknown prior parameter in a
linear regression model, which we develop in a companion paper.
\end{abstract}

\setcounter{tocdepth}{2}
\tableofcontents

\section{Introduction}

Let $\btheta=(\theta_1,\ldots,\theta_d) \in \R^d$
be a system of $d$ interacting particles, evolving
according to a stochastic dynamics
\begin{equation}\label{eq:dynamicsnull}
\d\btheta^t=\Big[{-}\beta\X^\top \X\btheta^t+\big(s(\theta_j^t,\widehat\alpha^t)\big)_{j=1}^d\Big]\d t+\sqrt{2}\,\d \b^t,\qquad
\frac{\d}{\d t}\widehat\alpha^t=\cG\Big(\widehat\alpha^t,\frac{1}{d}\sum_{j=1}^d
\delta_{\theta_j^t}\Big).
\end{equation}
Here $\X \in \R^{n \times d}$ is a matrix of random disorder,
and $s(\,\cdot\,,\widehat\alpha^t):\R \to \R$ in the drift coefficient 
is a nonlinear function
driven by a stochastic time-dependent parameter $\widehat\alpha^t \in \R^K$ that
adapts to the past history $\{\btheta^s\}_{s \in [0,t]}$. (We defer formal
definitions and conditions for the functions $s(\cdot)$ and $\cG(\cdot)$ to
Section \ref{sec:result}.) We will study the pathwise convergence of the
empirical measure $\frac{1}{d}\sum_{j=1}^d \delta_{\theta_j^t}$ and of the
parameter $\widehat\alpha^t$ to deterministic limits as $n,d \to \infty$ at a fixed rate, in
this model (\ref{eq:dynamicsnull}) and in a closely related statistical model.

In the setting of $\beta=0$, i.e.\ with no random disorder,
the dynamics (\ref{eq:dynamicsnull}) take a
pathwise-exchangeable form as studied classically by
\cite{kac1956foundations,mckean1966class},
where the evolution $\d\theta_j^t$ of each $j^\text{th}$ particle depends on the remaining particles only via the empirical law
$\frac{1}{d}\sum_{j=1}^d \delta_{\{\theta_j^s\}_{s \in [0,t]}}$. The
convergence of this law in the asymptotic limit $d \to \infty$, together with a resulting asymptotic
decoupling of low-dimensional marginals of $\{\btheta^s\}_{s \in [0,t]}$,
is commonly referred to as propagation-of-chaos. We refer to the classical
monographs \cite{gartner1988mckean,sznitman1991topics} for a detailed treatment
of such models, and to \cite{chaintron2022propagation,chaintron2021propagation}
and \cite{lacker2023hierarchies,lacker2023sharp} for modern surveys and examples
of recent quantitative convergence results.

The study of propagation-of-chaos for dynamics with random disorder ($\beta \neq 0$) 
has also a separate and rich
development in the literature, using techniques of dynamical mean-field theory (DMFT).
DMFT was initially developed to study Langevin dynamics in the
soft Sherrington-Kirkpatrick (SK) model
\cite{sompolinsky1981dynamic,sompolinsky1982relaxational} and related
spherical p-spin models in spin glass theory \cite{crisanti1993spherical, cugliandolo1993analytical, cugliandolo1994out},
and relied on deep but non-rigorous techniques of the dynamical cavity method
\cite{mezard1987spin,agoritsas2018out} and generating functional methods
\cite{de1978dynamics,kurchan1992supersymmetry,cugliandolo2002dynamics} of
statistical physics. In recent years, DMFT has been applied to shed insight
into the learning dynamics in an increasingly wide range of 
statistical and machine learning models, including matrix and tensor PCA
\cite{mannelli2019passed,sarao2019afraid,sarao2020marvels, sarao2021analytical, liang2023high},
phase retrieval and generalized linear models
\cite{agoritsas2018out,mignacco2021stochasticity,sarao2020complex, han2024gradient}, Gaussian
mixture classification \cite{mignacco2021dynamical, mignacco2022effective}, and deep neural networks
\cite{bordelon2022self,dandi2024benefits,bordelon2024infinite,bordelon2024dynamical,montanari2025dynamical}.

Pioneering work of
\cite{benarous1995large,grunwald1996sanov,benarous1997symmetric,guionnet1997averaged}
established the first mathematical formalizations of DMFT in variants of the SK
model, in the forms of large deviations principles for the empirical
distributions of sample paths. Mathematical results for spherical models were
subsequently obtained in \cite{benarous2001aging,benarous2006cugliandolo}, and
universality of such results with respect to the law of the disorder in
\cite{dembo2021diffusions,dembo2021universality}.
Recently, \cite{celentano2021high} developed a different and innovative new approach to formalizing DMFT via time discretization and
reduction to Approximate Message Passing schemes \cite{bolthausen2014iterative,bayati2011dynamics,javanmard2013state}, and applied this
to derive a DMFT limit for gradient flow dynamics
in statistical multi-index models. A related strategy via iterative
Gaussian conditioning was developed in \cite{gerbelot2024rigorous}, which
extended the results of \cite{celentano2021high} to a class of discrete-time
Langevin and stochastic gradient dynamics. Non-asymptotic analyses of the
entrywise behavior of such dynamics were obtained in \cite{han2024entrywise}.

In this work, we will prove a DMFT approximation for the dynamics
(\ref{eq:dynamicsnull}), which has both the above elements of a
pathwise-exchangeable interaction driven by the empirical law, as well as a
pairwise interaction driven by random disorder. Our motivation is the study of
a variant of Langevin dynamics for posterior sampling in a statistical linear
model
\[\y=\X\btheta^*+\beps, \qquad \theta_j^* \overset{iid}{\sim}
g(\,\cdot\,,\alpha^*), \qquad \eps_i \overset{iid}{\sim} \N(0,\sigma^2),\]
where the regression coefficients of interest $\theta_1^*,\ldots,\theta_d^*$
are distributed according to a prior law $g(\,\cdot\,,\alpha^*)$ that has
an unknown parameter $\alpha^* \in \R^K$.
To implement empirical Bayes learning of $\alpha^*$
\cite{spence2022flexible,mukherjee2023mean,kim2024flexible},
a Langevin dynamics method was introduced in
\cite{fan2023gradient}\footnote{\cite{fan2023gradient} proposed a
nonparametric variant of this method, and we simplify our discussion here to a
parametric formulation} which,
from an initial estimate or guess $\widehat\alpha^0$, evolves a prior parameter
estimate
\begin{equation}\label{eq:dynamicsposterior1}
\frac{\d}{\d t}\widehat\alpha^t=\frac{1}{d}\sum_{j=1}^d \nabla_\alpha \log
g(\theta_j^t,\widehat\alpha^t)
\end{equation}
based on the coordinates of a coupled Langevin diffusion
\begin{equation}\label{eq:dynamicsposterior2}
\d\btheta^t=\nabla_{\btheta}\Big({-}\frac{1}{2\sigma^2}
\|\y-\X\btheta^t\|_2^2+\sum_{j=1}^d \log g(\theta_j^t,\widehat\alpha^t)\Big)\d
t+\sqrt{2}\,\d\b^t
\end{equation}
that samples from the posterior law
$\sP(\btheta \mid \X,\y)$. Such dynamics comprise a minor extension of
(\ref{eq:dynamicsnull}) (which motivates our choice to study a disorder matrix
having the covariance form $\X^\top \X$), and we state in
(\ref{eq:langevin_sde}--\ref{eq:gflow}) of Section \ref{sec:result}
an extended general dynamics that encompasses this application. We defer a
more detailed discussion and analysis of this specific empirical Bayes procedure 
to our companion paper \cite{paper2}, focusing in this work
on the formalization of the limiting dynamics in a general context.

We summarize the main contributions of our paper as follows:
\begin{enumerate}
\item Adapting and building upon the methods of \cite{celentano2021high},
we prove a DMFT limit for the dynamics (\ref{eq:dynamicsnull}) (with a natural
extension to the dynamics (\ref{eq:langevin_sde}--\ref{eq:gflow}) to follow).
This will take the form of almost-sure convergence to certain deterministic
limits, as $n,d \to \infty$, 
\[\{\widehat \alpha^t\}_{t
\in [0,T]} \to \{\alpha^t\}_{t \in [0,T]},
\quad \frac{1}{d}\sum_{j=1}^d \delta_{\{\theta_j^t\}_{t \in [0,T]}}
\overset{W_2}{\to} \sP(\{\theta^t\}_{t \in [0,T]}),
\quad \frac{1}{n}\sum_{i=1}^n \delta_{\{\eta_i^t\}_{t \in [0,T]}}
\overset{W_2}{\to} \sP(\{\eta^t\}_{t \in [0,T]})\]
for the sample path of $\widehat\alpha^t$ and for the empirical laws of sample
paths of $\btheta^t$ and $\bbeta^t=\X\btheta^t$.

Each limit $\sP(\{\theta^t\}_{t \in [0,T]})$ and $\sP(\{\eta^t\}_{t \in
[0,T]})$ represents the law of a univariate stochastic process, which is
driven by the above limit $\{\alpha^t\}_{t \in [0,T]}$ for the evolving drift
parameter, an additional Gaussian process representing the mean field,
and an integrated response.
These Gaussian processes and integrated responses are
described by correlation and response kernels $C_\theta,C_\eta,R_\theta,R_\eta$,
where
\[\{\alpha^t\},C_\theta,C_\eta,R_\theta,R_\eta\]
are defined self-consistently from the laws
$\sP(\{\theta^t\}_{t \in [0,T]})$ and $\sP(\{\eta^t\}_{t \in [0,T]})$ via a
system of dynamical fixed-point equations.
We establish that this dynamical fixed point is unique in a certain
domain of functions with exponential growth.

\item We show that the dynamics (\ref{eq:dynamicsnull})
admit a well-defined notion of a linear response function $R_{AB}(t,s)$
for a class of observables $A,B:\R^d \to \R$, where $R_{AB}(t,s)$
represents a linear response of $A(\btheta^t)$ to a perturbation of the
drift coefficient at a previous time $s$ by $\nabla B(\btheta^s)$.

We then verify that the above DMFT correlation and response kernels
$C_\theta,C_\eta,R_\theta,R_\eta$ arise as the mean-field limits
of averages of the correlation and linear response
functions for the ``single-particle'' coordinate observables
$A(\btheta),B(\btheta)=\theta_j$ and $A(\btheta),B(\btheta)=[\X\btheta-\y]_i$
of the high-dimensional system.
\end{enumerate}

Our methods and analyses in the first contribution above follow the approach
of \cite{celentano2021high}. We incorporate into the dynamical fixed-point
system a deterministic limit $\{\alpha^t\}_{t \in [0,T]}$ for the trajectory of
the stochastic drift parameter $\{\widehat\alpha^t\}_{t \in [0,T]}$, extend the
analyses to encompass processes with more irregular sample paths resulting from the
additional Brownian diffusion term $\d \b^t$ in the dynamics, and simplify the
approach in \cite{celentano2021high} for embedding a discrete-time DMFT system
into a continuous-time limit.

Our second contribution above is, to our knowledge, novel in the mathematical
literature on DMFT (although anticipated by statistical physics derivations
of the DMFT equations). To understand $R_\theta,R_\eta$ as asymptotic limits
of averaged single-particle linear response functions, we formalize a
dynamical cavity calculation that analyzes the response of a single
coordinate $\theta_j^t$ to a perturbation of $\theta_j^s$ at a preceding time
$s$, via a DMFT approximation of the cavity system with this coordinate left
out. This result will be important to our companion
work \cite{paper2}, allowing us to transfer a fluctuation-dissipation theorem
\cite{kubo1966fluctuation} from the
high-dimensional dynamics to the DMFT correlation and response
kernels $C_\theta,C_\eta,R_\theta,R_\eta$. This will then allow us to carry
out an analysis of the long-time behavior of the DMFT equations in an approximately
time-translation-invariant setting, and to show convergence of the prior
parameter estimate $\widehat\alpha^t$ in the above empirical Bayes dynamics to
a stationary point of a replica-symmetric limit for the model free
energy.

\subsection*{Acknowledgments}
This research was initiated during the ``Huddle on Learning and Inference from
Structured Data'' at ICTP Trieste in 2023. We'd like to thank the huddle
organiers Jean Barbier, Manuel S\'aenz, Subhabrata Sen, and Pragya Sur for
their hospitality and many helpful discussions. We are also grateful to Pierfrancesco Urbani, Francesca Mignacco, Emanuele Troiani, and Andrea Montanari for useful discussions. Z. Fan was supported in part by NSF DMS-2142476 and a Sloan Research Fellowship. J. Ko was supported in part by the Natural Sciences and Engineering Research Council of Canada (NSERC), the Canada Research Chairs programme, and the Ontario Research Fund [RGPIN-2020-04597, DGECR-2020-00199]. B. Loureiro was supported by the French government, managed by the National Research Agency (ANR), under the France 2030 program with the reference ``ANR-23-IACL-0008'' and the Choose France - CNRS AI Rising Talents program. Y. M. Lu was supported in part by the Harvard FAS Dean's Competitive Fund for Promising Scholarship and by a Harvard College Professorship.

\subsection*{Notational conventions}
Constants $C,C',c,c'>0$ are independent of the dimensions $n,d$ unless otherwise specified, and
may depend on the time horizon $T$, dimension $K$ of the drift parameter, and
scalar parameter $\beta \in \R$.

In a separable and complete normed vector space $(\cM,\|\cdot\|)$, for
any $p \geq 1$, $\cP_p(\cM)$ is the space of probability
distributions $\sP$ on $(\cM,\|\cdot\|)$ such that $\E_{\xi \sim
\sP}\|\xi\|^p<\infty$, $W_p(\cdot)$ is the Wasserstein-$p$ metric on
$\cP_p(\cM)$, and $\sP_n \overset{W_p}{\to} \sP$ denotes $W_p(\sP_n,\sP) \to 0$
as $n \to \infty$. For a
random variable $\xi$ in $\cM$, we will use $\sP(\xi)$ to denote its law. 
For a vector $\x \in \cM^n$,
$\widehat\sP(\x)=\frac{1}{n}\sum_{i=1}^n \delta_{x_i} \in \cP_p(\cM)$ (for any $p \geq 1$) denotes
the empirical distribution of the coordinates $x_1,\ldots,x_n$ of $\x$.

On a Euclidean space $\R^d$, $\|\cdot\|$ without subscript is, by convention,
the Euclidean (i.e.\ $\ell_2$) norm. 
$C([0,T],\R)$ is the space of continuous
functions $f:[0,T] \to \R$ equipped with the norm of uniform convergence
$\|f\|_\infty=\sup_{t \in [0,T]} |f(t)|$.
$\Z_+=\{0,1,2,\ldots\}$ denotes the nonnegative integers, and $\R_+=[0,\infty)$
denotes the nonnegative reals. For a function $f:\R^d \to \R$, $\nabla
f(\x) \in \R^d$ and $\nabla^2 f(\x) \in \R^{d \times d}$ are its gradient and
Hessian at $\x$. For $f:\R^d \to \R^m$, $\d f(\x) \in \R^{d \times m}$ is its
derivative at $\x$. $\Tr M$, $\|M\|_\op$, and $\|M\|_F$ are the matrix
trace, Euclidean operator norm, and Frobenius norm.

\section{Model and main results}\label{sec:result}

\subsection{Model and dynamics}

Let $\X \in \R^{n \times d}$ be a random matrix with independent entries,
and $\y=\X\btheta^*+\beps \in \R^n$ the observations of a linear model
with regression design $\X$, regression coefficients
$\btheta^* \in \R^d$, and noise $\beps \in \R^n$. 

Let $s:\R \times \R^K \to \R$ be a Lipschitz drift function,
$\cG:\R^K \times \cP_2(\R) \to \R^K$ a Lipschitz
gradient map (where $\cP_2(\R)$ is the space of probability measures on
$\R$ with finite second moment), $\{\b^t\}_{t \geq 0}$ a
standard Brownian motion on $\R^d$, and $\beta \in \R$ a scalar parameter.
We will study the dynamics
\begin{align}
\d\btheta^t&=\Big[{-}\beta\,\X^\top(\X\btheta^t-\y)
+ \big(s(\theta^t_j, \widehat\alpha^t)\big)_{j=1}^d\Big]\d t+\sqrt{2}\,\d\b^t
\label{eq:langevin_sde}\\
\d\widehat\alpha^t&=\cG\Big(\widehat\alpha^t,\frac{1}{d}\sum_{j=1}^d
\delta_{\theta_j^t}\Big)\d t\label{eq:gflow}
\end{align}
with initial conditions
\begin{equation}\label{eq:langevin_init}
(\btheta^0,\widehat \alpha^0) \in \R^d \times \R^K.
\end{equation}
This encompasses the general dynamics (\ref{eq:dynamicsnull}) and the
application (\ref{eq:dynamicsposterior1}--\ref{eq:dynamicsposterior2})
under a unified model:
Specializing to $\btheta^*=0$ and $\beps=0$ (hence $\y=0$) recovers
(\ref{eq:dynamicsnull}), while specializing to $\beta=\sigma^{-2}$,
$s(\theta,\alpha)=\partial_\theta \log g(\theta,\alpha)$,
and $\cG(\alpha,\sP)=\E_{\theta \sim \sP}[\nabla_\alpha \log g(\theta,\alpha)]$
recovers (\ref{eq:dynamicsposterior1}--\ref{eq:dynamicsposterior2}). 
We will refer to these general dynamics (\ref{eq:langevin_sde}--\ref{eq:gflow})
as an adaptive Langevin diffusion.

We impose the following assumptions on the components of the above
model and dynamics throughout this work.

\begin{assumption}[Model and initial conditions]\label{assump:model}
\phantom{}
\begin{enumerate}[(a)]
\item (Asymptotic scaling)
$\lim_{n,d \to \infty} \frac{n}{d}=\delta \in (0,\infty)$.
\item (Random design)
$\X=(x_{ij}) \in \R^{n \times d}$ has independent entries satisfying
$\E x_{ij}=0$, $\E x_{ij}^2=\frac{1}{d}$, and $\|\sqrt{d}x_{ij}\|_{\psi_2} \leq
C$ for a constant $C>0$ where $\|\cdot\|_{\psi_2}$ is the sub-gaussian norm.
\item (Linear model and initial conditions) $\btheta^0,\btheta^*,\beps$ are independent of $\X$, and
$\y=\X\btheta^*+\beps$. For some probability distributions
$\sP(\theta^*,\theta^0)$ and $\sP(\eps)$
having finite moment generating functions in a neighborhood of 0,
and for each fixed $p \geq 1$, the entries of $\btheta^0,\btheta^*,\beps$ satisfy the
Wasserstein-$p$ convergence almost surely as $n,d \to \infty$,
\begin{equation}\label{eq:thetaepsdistribution}
\frac{1}{d}\sum_{j=1}^d \delta_{(\theta^\ast_j,\theta_j^0)} \overset{W_p}{\to}
\sP(\theta^*,\theta^0),
\qquad \frac{1}{n}\sum_{i=1}^n \delta_{\eps_i} \overset{W_p}{\to} \sP(\eps).
\end{equation}
For a deterministic parameter $\alpha^0 \in \R^K$, almost surely 
$\lim_{n,d \to \infty} \widehat \alpha^0=\alpha^0$.
\end{enumerate}
\end{assumption}

\begin{assumption}[Drift function]\label{assump:prior}
$s:\R \times \R^K \to \R$ is twice
continuously-differentiable, and
for some constant $C>0$ and all
$(\theta,\alpha) \in \R \times \R^K$,
\begin{equation}\label{eq:s_smooth}
|s(\theta,\alpha)| \leq C(1+|\theta|+\|\alpha\|_2),
\quad \|\nabla_{(\theta,\alpha)} s(\theta,\alpha)\|_2,
\|\nabla_{(\theta,\alpha)}^2 s(\theta,\alpha)\|_\op \leq C.
\end{equation}
\end{assumption}

\begin{assumption}[Gradient map]\label{assump:gradient}
Let $\widehat\sP(\btheta)=d^{-1}\sum_{j=1}^d \delta_{\theta_j}$ denote
the empirical distribution of coordinates of $\btheta \in \R^d$,
and let $\cG_k:\R^K \to \R$
be the $k^\text{th}$ component of $\cG$. 
\begin{enumerate}[(a)]
\item For some constant $C>0$ and all $(\alpha,\sP) \in \R^K \times \cP_2(\R)$,
\begin{equation}\label{eq:G_smooth1}
\|\cG(\alpha,\sP)\|_2 \leq C(1+ \|\alpha\|_2+\E_{\sP}[\theta^2]^{1/2}),
\quad \|\cG(\alpha,\sP)-\cG(\alpha',\sP')\|_2 \leq
C\big(\|\alpha-\alpha'\|_2+W_2(\sP,\sP')\big).
\end{equation}
\item For each $k=1,\ldots,K$,
$(\btheta,\alpha) \mapsto \cG_k(\alpha,\widehat\sP(\btheta))$ is twice
continuously-differentiable, and
for some constant $C>0$ and
all $(\btheta,\alpha) \in \R^d \times \R^K$,
\begin{equation}\label{eq:G_smooth2}
\pnorm{\nabla_{\alpha} \cG_k(\alpha,\widehat\sP(\btheta))}{2}
\leq C,
\qquad \sqrt{d}\,\pnorm{\nabla_{\btheta} \cG_k(\alpha,\widehat\sP(\btheta))}{2} \leq
C,
\end{equation}
\begin{equation}\label{eq:G_smooth3}
\max\Big(d\,\pnorm{\nabla_{\btheta}^2 \cG_k(\alpha,\widehat\sP(\btheta))}{\op},\,
\sqrt{d}\,\pnorm{\nabla_{\btheta}\nabla_{\alpha}
\cG_k(\alpha,\widehat\sP(\btheta))}{\op},\,\pnorm{\nabla^2_{\alpha}
\cG_k(\alpha,\widehat\sP(\btheta))}{\op}\Big) \leq C.
\end{equation}
\end{enumerate}
\end{assumption}

Viewing (\ref{eq:langevin_sde}--\ref{eq:gflow}) as a joint
diffusion of $(\btheta^t,\widehat\alpha^t)$ on $\R^{d+K}$, we remark that
(\ref{eq:s_smooth}) and (\ref{eq:G_smooth2}) imply that the
drift function of this joint diffusion is Lipschitz with respect to the
Euclidean norm. Then there exists a unique
solution $\{\btheta^t,\widehat \alpha^t\}_{t \geq 0}$ 
to (\ref{eq:langevin_sde}--\ref{eq:gflow}) with initial condition
(\ref{eq:langevin_init}) that is
adapted to the filtration $\cF_t:=\cF(\{\b^s\}_{s \in
[0,t]},\btheta^0,\widehat\alpha^0,\X,\btheta^*,\beps)$,
which will be the process of interest in our main results.

\subsection{Existence and uniqueness of the DMFT fixed point}

In this section we define the DMFT limit for the preceding dynamics 
(\ref{eq:langevin_sde}--\ref{eq:gflow}). Let
$\delta=\lim_{n,d \to \infty} \frac{n}{d}$ and $\alpha^0=\lim_{n,d \to \infty}
\widehat \alpha^0$ be as in Assumption \ref{assump:model}. Let
\[(\theta^*,\theta^0) \sim \sP(\theta^*,\theta^0)\]
denote scalar variables with the distribution
(\ref{eq:thetaepsdistribution}). Let
\[\{b^t\}_{t \geq 0}, \qquad \{u^t\}_{t \geq 0}\]
be univariate mean-zero Gaussian processes independent of each other
and of $(\theta^*,\theta^0)$, where $\{b^t\}_{t \geq 0}$ is a standard Brownian
motion and $\{u^t\}_{t \geq 0}$ has a correlation kernel $C_\eta(\cdot)$ on
$[0,\infty)$, defined
self-consistently below. Let
\[\{\alpha^t\}_{t \geq 0}\]
be a deterministic 
continuous process on $\R^K$, also defined self-consistently below.
We consider univariate processes $\{\theta^t\}_{t \geq 0}$ and
$\{\frac{\partial \theta^t}{\partial u^s}\}_{t \geq s \geq 0}$ adapted to the
filtration
\[\cF_t^\theta:=\cF(\{b^s\}_{s \in [0,t]},\{u^s\}_{s \in [0,t]},\theta^*,\theta^0)\]
(in the sense that $\theta^t$ and $\frac{\partial
\theta^t}{\partial u^s}$ for all $s \in [0,t]$ are $\cF_t^\theta$-measurable),
defined by the stochastic differential equations
\begin{align}
\d\theta^t&=\Big[{-}\delta\beta(\theta^t-\theta^*) +
s(\theta^t,\alpha^t)
+\int_0^t R_\eta(t, s)(\theta^{s} - \theta^\ast)\d s + u^t\Big]\d t +
\sqrt{2}\,\d b^t
\text{ with } \theta^t\big|_{t=0}=\theta^0,\label{def:dmft_langevin_cont_theta}\\
\d\Big(\frac{\partial\theta^t}{\partial u^s}\Big) &=
\bigg[
{-}\bigg(\delta\beta-\partial_\theta s(\theta^t,\alpha^t)\bigg)\frac{\partial \theta^{t}}{\partial u^s} + \int_s^{t}
R_\eta(t,s')\frac{\partial \theta^{s'}}{\partial u^s}\d s'\bigg]\d t
\text{ with } \frac{\partial\theta^t}{\partial u^s}\bigg|_{t=s}=1
\label{def:response_theta}.
\end{align}
We clarify that $\frac{\partial \theta^t}{\partial u^s}$
is a notation for a univariate process
on $t \in [s,\infty)$, defined via (\ref{def:response_theta})
for each $s \geq 0$.

Consider also
\[\eps \sim \sP(\eps), \qquad (w^*,\{w^t\}_{t \geq 0}),\]
where $\eps$ is a scalar variable with the distribution
(\ref{eq:thetaepsdistribution}), and
$(w^*,\{w^t\}_{t \geq 0})$ is a univariate mean-zero Gaussian process indexed by
$\{*\} \cup [0,\infty)$, independent of $\eps$ and
with a correlation kernel $C_\theta(\cdot)$ on $\{*\} \cup [0,\infty)$ also
defined self-consistently below. We consider univariate processes
$\{\eta^t\}_{t \geq 0}$ and
$\{\frac{\partial \eta^t}{\partial w^s}\}_{t \geq s \geq 0}$ 
adapted to the filtration
\[\cF_t^\eta:=\cF(\{w^s\}_{s \leq t},w^*,\eps),\]
defined by the integral equations
\begin{align}
\eta^t &= {-}\beta\int_0^t R_\theta(t,s)\big(\eta^s + w^* -
\eps\big)\d s - w^t,\label{def:dmft_langevin_cont_eta}\\
\frac{\partial \eta^t}{\partial w^s} &= \beta\Big[{-}\int_{s}^{t}
R_\theta(t,s')\frac{\partial \eta^{s'}}{\partial w^s}\d s' +
R_\theta(t,s)\Big]\label{def:response_g}.
\end{align}
Again $\frac{\partial \eta^t}{\partial u^s}$ is a notation for a univariate
process on $t \in [s,\infty)$, defined by (\ref{def:response_g}) for each
$s \geq 0$.

The centered Gaussian processes $\{u^t\}_{t \geq 0}$ and
$(w^*,\{w^t\}_{t \geq 0})$ above have correlation kernels
\begin{equation}\label{def:dmft_covuw}
\E[u^t u^s]=C_\eta(t,s),
\qquad \E[w^t w^s]=C_\theta(t,s),
\qquad \E[w^t w^*]=C_\theta(t,*),
\qquad \E[(w^*)^2]=C_\theta(*,*).
\end{equation}
Denoting by $\sP(\theta^t)$ the law of $\theta^t$ solving
(\ref{def:dmft_langevin_cont_theta}),
the above deterministic process $\{\alpha^t\}_{t \geq 0}$ and correlation/response kernels
$C_\eta,C_\theta,R_\eta,R_\theta$ are defined for all $t \geq s \geq 0$
self-consistently by
\begin{equation}
\frac{\d}{\d t} \alpha^t = \cG(\alpha^t,\sP(\theta^t))
\text{ with } \alpha^t\big|_{t=0}=\alpha^0,\label{def:dmft_langevin_alpha}
\end{equation}
\begin{equation}\label{def:CRfixedpoint}
\begin{gathered}
C_\theta(t,s)=\E[\theta^t \theta^s],
\quad C_\theta(t,*)=\E[\theta^t \theta^*],
\quad C_\theta(*,*)
=\E[(\theta^*)^2],\\
C_\eta(t,s)=\delta\beta^2\,\E[(\eta^t + w^*-\eps)(\eta^s + w^*-\eps)],\\
R_\theta(t,s)=\E\Big[\frac{\partial \theta^t}{\partial u^s}\Big],
\quad R_\eta(t,s)=\delta\beta\,\E\Big[\frac{\partial \eta^t}{\partial w^s}\Big].
\end{gathered}
\end{equation}
We note that the above process
$\{\frac{\partial \eta^t}{\partial w^s}\}_{t \geq s \geq 0}$ defined by
(\ref{def:response_g})
is in fact deterministic, but we keep the expectation defining
$R_\eta(t,s)$ for symmetry of notation.

The equations (\ref{def:dmft_langevin_alpha}--\ref{def:CRfixedpoint}) should
be understood as fixed-point equations for
$\alpha,C_\theta,C_\eta,R_\theta,R_\eta$,
where the laws of the processes $\{\theta^t,\frac{\partial
\theta^t}{\partial u^s},u^t\}_{t \geq s \geq 0}$
and $\{\eta^t,\frac{\partial \eta^t}{\partial w^s},w^t\}_{t \geq s \geq 0}$ 
defining (\ref{def:dmft_langevin_alpha}--\ref{def:CRfixedpoint}) 
are in turn defined by $\alpha,C_\theta,C_\eta,R_\theta,R_\eta$
via (\ref{def:dmft_langevin_cont_theta}--\ref{def:dmft_covuw}).
For each fixed time horizon $T>0$, let $\cS(T)$ be a space of functions
\[(\alpha,C_\theta,C_\eta,R_\theta,R_\eta) \equiv
\{\alpha^t,C_\theta(t,s),C_\theta(t,*),C_\theta(*,*),C_\eta(t,s),R_\theta(t,s),
R_\eta(t,s)\}_{0 \leq s \leq t \leq T}\]
having at most exponential growth, and $\cS(T)^\text{cont}$ a subset of
continuous such functions, whose precise
definitions we defer to Section \ref{sec:functionspace} to follow.
The following result establishes existence and uniqueness of a fixed point to
(\ref{def:dmft_langevin_alpha}--\ref{def:CRfixedpoint}) in this space
$\cS(T)^\text{cont}$.

\begin{theorem}\label{thm:dmftsolexists}
Under Assumptions \ref{assump:model}, \ref{assump:prior}, and
\ref{assump:gradient}, for any fixed $T>0$: 
\begin{enumerate}[(a)]
\item For any 
$(\alpha,C_\theta,C_\eta,R_\theta,R_\eta) \in \cS(T)$
and any realization of the mean-zero Gaussian processes
$\{u^t\}_{t \geq 0}$ and $(w^*,\{w^t\}_{t \geq 0})$ satisfying
(\ref{def:dmft_covuw}) (independent of $(\theta^*,\theta^0,\{b^t\}_{t \geq 0})$
and $\eps$ respectively), there exist unique
solutions to (\ref{def:dmft_langevin_cont_theta}--\ref{def:response_theta})
and (\ref{def:dmft_langevin_cont_eta}--\ref{def:response_g}) adapted to 
$\{\cF_t^\theta\}_{t \in [0,T]}$ and $\{\cF_t^\eta\}_{t \in [0,T]}$
for times
$0 \leq s \leq t \leq T$.
\item There exists a unique fixed point
$(\alpha,C_\theta,C_\eta,R_\theta,R_\eta) \in \cS(T)$ satisfying
(\ref{def:dmft_langevin_alpha}--\ref{def:CRfixedpoint}) for the solution of
part (a). This fixed point belongs to $\cS(T)^\text{cont}$, and in
particular $\{\alpha^t\}_{t \geq 0}$ is a deterministic
continuous process on $\R^K$.
\end{enumerate}
\end{theorem}

The proof of Theorem \ref{thm:dmftsolexists} is given in Section
\ref{sec:existence}.
We will call the components of Theorem \ref{thm:dmftsolexists}(a--b) the unique
solution of the DMFT system
(\ref{def:dmft_langevin_cont_theta}--\ref{def:CRfixedpoint}).

\subsection{The dynamical mean-field approximation}

The following is the first main result of our work,
showing that the preceding solution to the DMFT system
describes the limit of $\{\widehat \alpha^t\}_{t \in [0,T]}$
and empirical distributions of coordinates of $\{\btheta^t\}_{t \in [0,T]}$
and $\{\bbeta^t\}_{t \in [0,T]} \equiv \{\X\btheta^t\}_{t \in [0,T]}$ solving
(\ref{eq:langevin_sde}--\ref{eq:gflow}), for fixed time horizons $T>0$
in the limit $n,d \to \infty$.

\begin{theorem}\label{thm:dmft_approx}
Suppose Assumptions \ref{assump:model}, \ref{assump:prior}, and
\ref{assump:gradient} hold. Denote
\[\boldeta^t=\X\btheta^t, \qquad \boldeta^*=\X\btheta^*\]
let $\theta^*$, $\eps$, $\eta^*={-}w^*$, and $\{\theta^t,\eta^t,\alpha^t\}_{t \in
[0,T]}$ be the components of the unique solution to the DMFT system
(\ref{def:dmft_langevin_cont_theta}--\ref{def:CRfixedpoint}) given by
Theorem \ref{thm:dmftsolexists}, and let $\sP(\cdot)$ denote the law of these components. Then for each fixed $T>0$,
almost surely as $n,d \to \infty$,
\begin{enumerate}[(a)]
\item $(\widehat \alpha^t)_{t \in [0,T]} \to (\alpha^t)_{t \in [0,T]}$
in $C([0,T],\R^K)$.
\item In the sense of Wasserstein-2 convergence over $\R \times C([0,T],\R)$ and
$\R \times \R \times C([0,T],\R)$,
\[\frac{1}{d}\sum_{j=1}^d \delta_{\theta_j^*,\{\theta_j^t\}_{t \in [0,T]}}
\overset{W_2}{\to} \sP(\theta^*,\{\theta^t\}_{t \in [0,T]}),
\qquad \frac{1}{n}\sum_{i=1}^n \delta_{\eta_i^*,\eps_i,\{\eta_i^t\}_{t \in [0,T]}}
\overset{W_2}{\to} \sP(\eta^*,\eps,\{\eta^t\}_{t \in [0,T]}).\]
\end{enumerate}
\end{theorem}
The proof of Theorem \ref{thm:dmft_approx} is given in Section
\ref{sec:dmft_approx}.
For ease of interpretation, we record here two corollaries of this
result. The first clarifies an implication of the above Wasserstein-2
convergence in terms of the convergence of pseudo-Lipschitz test functions of
finite-dimensional marginals of the processes.

\begin{corollary}\label{cor:dmft_approx}
In the setting of Theorem \ref{thm:dmft_approx},
for any fixed $m \geq 1$ and times
$t_1,\ldots,t_m \in [0,T]$, and for any pseudo-Lipschitz test functions
$f_\theta:\R^{m+1} \to \R$ and $f_\eta:\R^{m+2} \to \R$
(i.e.\ satisfying $|f(x)-f(y)| \leq C\|x-y\|_2(1+\|x\|_2+\|y\|_2)$),
almost surely as $n,d \to \infty$,
\begin{equation}\label{eq:dmftapproximplication}
\begin{aligned}
\frac{1}{d}\sum_{j=1}^d f_\theta(\theta_j^*,\theta_j^{t_1},\ldots,\theta_j^{t_m})
&\to \E f_\theta(\theta^*,\theta^{t_1},\ldots,\theta^{t_m})\\
\frac{1}{n}\sum_{i=1}^n f_\eta(\eta_i^*,\eps_i,\eta_i^{t_1},\ldots,\eta_i^{t_m})
&\to \E f_\eta(\eta^*,\eps,\eta^{t_1},\ldots,\eta^{t_m})
\end{aligned}
\end{equation}
where the expectations on the right side are under the joint laws of
the solution to the DMFT system.
\end{corollary}
\begin{proof}
Any pseudo-Lipschitz function $(\theta^*,\theta^{t_1},\ldots,\theta^{t_m})
\mapsto f_\theta(\theta^*,\theta^{t_1},\ldots,\theta^{t_m})$ is also a
pseudo-Lipschitz function of the full sample path
$(\theta^*,\{\theta^t\}_{t \in [0,T]}) \in \R \times C([0,T],\R)$. Thus the
first statement of (\ref{eq:dmftapproximplication}) follows from Theorem
\ref{thm:dmft_approx}(b) and the characterization of Wasserstein-$p$ convergence
in \cite[Definition 6.8 and Theorem 6.9]{villani2008optimal}, and the second
statement follows similarly.
\end{proof}

The second corollary asserts an asymptotic decoupling of the finite-dimensional
marginal distributions of $(\btheta^*,\{\btheta^t\}_{t \in [0,T]})$ in a
coordinate-exchangeable setting, which is the usual notion of
propagation-of-chaos for interacting particle systems.

\begin{corollary}
In the setting of Theorem \ref{thm:dmft_approx}, suppose in addition that
$(\btheta^*,\btheta^0) \in \R^{d \times
2}$ and $\X \in \R^{n \times d}$ are both invariant in law under permutations
of the coordinates $\{1,\ldots,d\}$.

Fix any $J \geq 1$, and let
$\sP(\theta_{1:J}^*,\{\theta_{1:J}^t\}_{t \in [0,T]})$ denote the joint law of
sample paths
$(\theta_j^*,\{\theta^t_j\}_{t \in [0,T]}) \in \R \times C([0,T],\R)$ for
$j=1,\ldots,J$. Let $\sP(\theta^*,\{\theta^t\}_{t \in [0,T]})^{\otimes J}$
denote the $J$-fold product of the limit law in Theorem
\ref{thm:dmft_approx}(b). Then as $n,d \to \infty$, in the sense of weak
convergence,
\[\sP(\theta_{1:J}^*,\{\theta_{1:J}^t\}_{t \in [0,T]})
\to \sP(\theta^*,\{\theta^t\}_{t \in [0,T]})^{\otimes J}.\]
\end{corollary}
\begin{proof}
Under the stated assumptions and the definition of the process
(\ref{eq:langevin_sde}--\ref{eq:gflow}),
the law of $(\btheta^*,\{\btheta^t\}_{t \in
[0,T]}) \in (\R \times C([0,T],\R))^d$ remains invariant under permutations of
the coordinates $\{1,\ldots,d\}$. Then the stated result is equivalent to
convergence of the empirical law $\frac{1}{d}\sum_{j=1}^d
\delta_{\theta_j^*,\{\theta_j^t\}_{t \in [0,T]}}$ to
$\sP(\theta^*,\{\theta^t\}_{t \in [0,T]})$ weakly in probability
(c.f.\ \cite[Proposition 2.2]{sznitman1991topics}), and this is implied by
Theorem \ref{thm:dmft_approx}(b).
\end{proof}

We clarify that $\sP(\theta_{1:J}^*,\{\theta_{1:J}^t\}_{t \in [0,T]})$ in this
statement refers to the law over all randomness including that of
$\btheta^*,\btheta^0$ and the disorder $\X$. It would be interesting to also study propagation-of-chaos phenomena conditional on parts of this
randomness, and we leave such investigations to future work.

\subsection{Interpretation of the DMFT correlation and response}


Fixing $\X,\btheta^*,\beps$ and $\y=\X\btheta^*+\beps$,
define the coordinate observables
\begin{equation}\label{eq:coordfuncs}
e_j(\btheta)=\theta_j,
\qquad x_i(\btheta)=\sqrt{\delta}\beta([\X\btheta]_i-y_i).
\end{equation}
Fixing also the initial conditions $\x=(\btheta^0,\widehat\alpha^0)$,
for each pair $A,B \in \{e_1,\ldots,e_d,x_1,\ldots,x_n\}$, define
\[\{R_{AB}^\x(t,s)\}_{0 \leq s \leq t}\]
as a response function for the joint dynamics
(\ref{eq:langevin_sde}--\ref{eq:gflow}) that satisfies the following condition: 
For any continuous bounded function
$h:[0,\infty) \to \R$ and any $\eps>0$, consider the perturbed dynamics
\begin{align*}
\d\btheta^{t,\eps}&=\Big[{-}\beta\X^\top(\X\btheta^{t,\eps}-\y)
+\eps h(t)\nabla_{\btheta} B(\btheta^{t,\eps},\widehat \alpha^{t,\eps})
+\big(s(\theta_j^{t,\eps},\widehat\alpha^{t,\eps})\big)_{j=1}^d
\Big]\d t+\sqrt{2}\,\d\b^t\\
\d\widehat\alpha^{t,\eps}&=\cG\Big(\widehat\alpha^{t,\eps},
\frac{1}{d}\sum_{j=1}^d \delta_{\theta_j^{t,\eps}}\Big)\d t
\end{align*}
with the same initial condition $(\btheta^{0,\eps},\widehat
\alpha^{0,\eps})=\x$. Denote the expectation conditional on $\X,\btheta^*,\beps$
and $\x=(\btheta^0,\widehat\alpha^0)$ as $\langle
f(\{\btheta^t,\widehat\alpha^t\}_{t \geq 0}) \rangle_\x$. Then for any $t>0$,
\begin{equation}\label{eq:responsedef}
\lim_{\eps \to 0} \frac{1}{\eps}\Big(\langle A(\btheta^{t,\eps},\widehat
\alpha^{t,\eps}) \rangle_\x
-\langle A(\btheta^t,\widehat \alpha^t) \rangle_\x
\Big)=\int_0^t R_{AB}^\x(t,s)h(s)\,ds.
\end{equation}
Thus $R_{AB}^\x(t,s)$ may be understood as the linear response of the observable
$A(\btheta)$ at time $t$ to a perturbation of the Langevin
potential by $B(\btheta)$ at a preceding time $s$. Existence of such a response function for smooth bounded observables in uniformly elliptic and hypoelliptic diffusions has been shown in \cite{dembo2010markovian,chen2020mathematical}.
We verify in Proposition \ref{prop:responsegeneral} that the arguments of \cite{chen2020mathematical} may be extended to show also the existence of a response function $R_{AB}^\x(t,s)$ satisfying (\ref{eq:responsedef}) in our adaptive Langevin diffusion,
for a class of unbounded and
Lipschitz observables including all
$A,B \in \{e_1,\ldots,e_d,x_1,\ldots,x_n\}$.

Let $\{\btheta^t,\widehat \alpha^t\}_{t \geq 0}$ be the
solution to (\ref{eq:langevin_sde}--\ref{eq:gflow}) with the given
initial condition $\x=(\btheta^0,\widehat\alpha^0)$ of
Assumption \ref{assump:model},
and define the corresponding correlation and response matrices
\begin{equation}\label{eq:matrixcorrresponse}
\begin{gathered}
\bC_\theta(t,s)=\Big(\big\langle e_j(\btheta^t)e_k(\btheta^s) \big\rangle_\x
\Big)_{j,k=1}^d, \quad
\bC_\theta(t,*)=\Big(\big\langle e_j(\btheta^t)e_k(\btheta^*) \big\rangle_\x
\Big)_{j,k=1}^d, \quad
\bR_\theta(t,s)=\Big(R_{e_je_k}^\x(t,s)\Big)_{j,k=1}^d,\\
\bC_\eta(t,s)=\Big(\big\langle x_j(\btheta^t)x_k(\btheta^s) \big\rangle_\x \Big)_{j,k=1}^n,
\quad
\bR_\eta(t,s)=\Big(R_{x_jx_k}^\x(t,s)\Big)_{j,k=1}^n
\end{gathered}
\end{equation}
for the above coordinate observables
$e_1,\ldots,e_d,x_1,\ldots,x_n$.
The following is the second main result of our work, showing that
the correlation and response kernels $C_\theta,C_\eta,R_\theta,R_\eta$ defining
the DMFT limit in Theorem \ref{thm:dmft_approx} are the almost-sure limits of
the normalized traces of these matrices, i.e.\ the correlation and
self-responses of the observables $e_j$ and $x_i$ averaged
across coordinates $j=1,\ldots,d$ and $i=1,\ldots,n$.

\begin{theorem}\label{thm:dmftresponse}
Suppose Assumptions \ref{assump:model}, \ref{assump:prior}, and
\ref{assump:gradient} hold, and $(\theta,\alpha) \mapsto \nabla_{(\theta,\alpha)}^2 s(\theta,\alpha)$ and $(\btheta,\alpha) \mapsto \nabla_{(\btheta,\alpha)}^2 \cG_k(\alpha,\widehat \sP(\btheta))$ are uniformly H\"older-continuous for each $k=1,\ldots,K$.
Let $C_\theta,C_\eta,R_\theta,R_\eta$ be the correlation and response
kernels of the solution to the DMFT system
(\ref{def:dmft_langevin_cont_theta}--\ref{def:CRfixedpoint}) given by
Theorem \ref{thm:dmftsolexists}.
Then for any fixed $t \geq s \geq 0$, almost surely as $n,d \to \infty$,
\[\begin{gathered}
d^{-1}\Tr \bC_\theta(t,s) \to C_\theta(t,s),
\qquad d^{-1}\Tr \bC_\theta(t,*) \to C_\theta(t,*),
\qquad n^{-1}\Tr \bC_\eta(t,s) \to C_\eta(t,s), \\
d^{-1}\Tr \bR_\theta(t,s) \to R_\theta(t,s),
\qquad n^{-1}\Tr \bR_\eta(t,s) \to R_\eta(t,s).
\end{gathered}\]
\end{theorem}

The proof of Theorem \ref{thm:dmftresponse} is provided in Section
\ref{sec:dmftresponse}. We note that
the convergence of $d^{-1}\Tr \bC_\theta$ and $n^{-1}\Tr \bC_\eta$ is an
immediate consequence of Corollary \ref{cor:dmft_approx}.
The additional content of this theorem is
the convergence of $d^{-1}\Tr \bR_\theta$ and
$n^{-1}\Tr \bR_\eta$, which relies on an inductive analysis of dynamics at a
single particle level using a dynamical cavity argument.

\begin{remark}
By an argument similar to our proof of Theorem \ref{thm:dmftresponse}, one
may show that the DMFT response kernels $R_\theta(t,s)$ and $R_\eta(t,s)$ also
represent the limits of $d^{-1}\Tr \bR_\theta(t,s)$ and $n^{-1}\Tr
\bR_\eta(t,s)$ defined for a non-adaptive version of the dynamics
\[\d\tilde\btheta^t=\Big[{-}\beta\X^\top(\X\tilde\btheta^t-\y)
+\big(s(\tilde\theta^t_j,\alpha^t)\big)_{j=1}^d\Big]\d t+\sqrt{2}\,\d\b^t\]
which replaces the adaptively-evolving
drift parameter $\{\widehat\alpha^t\}_{t \geq 0}$ by its
deterministic DMFT limit $\{\alpha^t\}_{t \geq 0}$.
The response matrices $\bR_\theta,\bR_\eta$ for this non-adaptive dynamics
$\{\tilde \btheta^t\}_{t \geq 0}$
are different from those for the adaptive dynamics
(\ref{eq:langevin_sde}--\ref{eq:gflow}), in that a perturbation
in the adaptive system affects $\{\widehat\alpha^t\}_{t \geq s}$
whereas it does not change $\{\alpha^t\}_{t \geq s}$ in the non-adaptive
system. However, our result implies that the almost-sure limits of
$d^{-1}\Tr \bR_\theta$ and $n^{-1}\Tr\bR_\eta$ coincide for these two
dynamics, i.e.\ the propagation of the effect of the perturbation through
$\{\widehat\alpha^t\}$ is negligible in the large-$(n,d)$ limit.
\end{remark}

The remainder of this paper will prove the preceding results of Theorems
\ref{thm:dmftsolexists}, \ref{thm:dmft_approx}, and \ref{thm:dmftresponse}.

\section{Existence and uniqueness of the DMFT fixed point}\label{sec:existence}

In this section we prove Theorem \ref{thm:dmftsolexists}. We assume throughout Assumptions \ref{assump:model}, \ref{assump:prior}, and \ref{assump:gradient}.
Section \ref{sec:functionspace} defines
the spaces $\cS(T)$ and $\cS(T)^\text{cont}$ and proves Theorem \ref{thm:dmftsolexists}(a) on existence
and uniqueness of the processes
(\ref{def:dmft_langevin_cont_theta}--\ref{def:response_g}).
Section \ref{sec:contractivemapping} then proves
Theorem \ref{thm:dmftsolexists}(b) on existence and uniqueness of the dynamical
fixed point via a contractive mapping argument similar to that of \cite{celentano2021high}.

\subsection{The function spaces $\cS(T)$ and $\cS(T)^\text{cont}$}\label{sec:functionspace}

Let $\tau_*^2=\E{\theta^*}^2$ and $\sigma^2=\E \eps^2$, and let $C_0>0$
denote a constant larger than the constants $C>0$ of (\ref{eq:s_smooth})
and (\ref{eq:G_smooth1}).
Consider the following system of equations for functions
$\Phi_\alpha,\Phi_{C_\theta},\Phi_{C_\eta},\Phi_{R_\theta},\Phi_{R_\eta}$
on $[0,\infty)$:
\begin{align}
\frac{\d}{\d t} \Phi_{\alpha}(t)
&=4.1C_0(1+\Phi_{C_\theta}(t))+3C_0\Phi_\alpha(t)
\text{ with } \Phi_{\alpha}(0)=\|\alpha^0\|^2,\label{eq:aux_alpha}\\
\frac{\d}{\d t}\Phi_{C_\theta}(t)&=(6\delta^2\beta^2+18C_0^2+1.1)\Phi_{C_\theta}(t)
+6\int_0^t (t-s+1)^2 \Phi_{R_\eta}^2(t-s)\Phi_{C_\theta}(s)\d s\notag\\
&\hspace{0.2in}+6\Big(\delta^2\beta^2 \tau_*^2+3C_0^2+3C_0^2\Phi_\alpha(t)
+\int_0^t (t-s+1)^2\Phi_{R_\eta}^2(t-s)\d s \cdot \tau_*^2
+\Phi_{C_\eta}(t)\Big)+2\notag\\
&\text{ with }
\Phi_{C_\theta}(0)=\E(\theta^0)^2,\label{eq:aux_Ctheta}\\
\Phi_{C_\eta}(t)&= 2\delta\beta^2\Big[\frac{1}{\delta}\int_0^t (t-s+1)^2 \cdot
\Phi^2_{R_\theta}(t-s) \Phi_{C_\eta}(s)\d s + 2\Phi_{C_\theta}(t) + 2\tau_*^2 +
\sigma^2\Big],\label{eq:aux_Ceta}\\
\frac{\d}{\d t}\Phi_{R_\theta}(t)&=(\delta|\beta|+C_0)\Phi_{R_\theta}(t) + \int_0^t
\Phi_{R_\eta}(t-s)\Phi_{R_\theta}(s)\d s \text{ with }
\Phi_{R_\theta}(0)=1,\label{eq:aux_Rtheta}\\
\Phi_{R_\eta}(t)&=|\beta|\Big(\int_0^t \Phi_{R_\theta}(t-s)\Phi_{R_\eta}(s)\d s
+\delta|\beta|\Phi_{R_\theta}(t)\Big).\label{eq:aux_Reta}
\end{align}

\begin{lemma}\label{eq:lemma:auxfunctions}
The system (\ref{eq:aux_alpha}--\ref{eq:aux_Reta}) has a unique continuous
solution. Defining
\[E(\lambda)=\Big\{\text{functions } f:[0,\infty) \to [0,\infty) \text{ such that } \int_0^\infty e^{-\lambda s} f(s)\d
s<\infty\Big\},\]
for any sufficiently large constant $\lambda>0$, this solution satisfies
$\Phi_\alpha,\Phi_{C_\theta},\Phi_{C_\eta},\Phi_{R_\theta},\Phi_{R_\eta}
\in E(\lambda)$.
\end{lemma}
\begin{proof}
Let $\Phi_\eta=(\Phi_{C_\eta},\Phi_{R_\eta})$,
$\Phi_\theta=(\Phi_\alpha,\Phi_{C_\theta},\Phi_{R_\theta})$, and
$\Phi=(\Phi_\eta,\Phi_\theta)$. For any two continuous solutions $\Phi$ and
$\tilde \Phi$, there exists some $M>0$ such that all components of both solutions are uniformly
bounded over $[0,T]$ by $M$. The above equations then imply
\[\|\Phi_\theta(t)-\tilde \Phi_\theta(t)\|
\leq \int_0^t C\|\Phi(s)-\tilde \Phi(s)\|\d s,
\quad \|\Phi_\eta(t)-\tilde \Phi_\eta(t)\|
\leq C\Big(\int_0^t \|\Phi(s)-\tilde \Phi(s)\|\d s
+\|\Phi_\theta(t)-\tilde \Phi_\theta(t)\|\Big)\]
for a constant $C>0$ depending on $M,T$. Applying Gronwall's lemma to the
second inequality shows
\begin{equation}\label{eq:gronwalltmp}
\sup_{s\in [0,t]} \|\Phi_\eta(s)-\tilde \Phi_\eta(s)\|
\leq C'\sup_{s \in [0,t]} \|\Phi_\theta(s)-\tilde \Phi_\theta(s)\|.
\end{equation}
Then applying this in the first inequality gives
\[\|\Phi_\theta(t)-\tilde \Phi_\theta(t)\|
\leq \int_0^t C''\sup_{r \in [0,s]} \|\Phi_\theta(r)-\tilde \Phi_\theta(r)\|\d s,\]
so Gronwall's lemma applied again shows $\Phi_\theta(t)=\tilde \Phi_\theta(t)$
for all $t \in [0,T]$. Then by (\ref{eq:gronwalltmp}), also $\Phi(t)=\tilde \Phi(t)$
for all $t \in [0,T]$, so any continuous solution to
(\ref{eq:aux_alpha}--\ref{eq:aux_Reta}) is unique.

It remains to show existence of a continuous solution with all components in
$E(\lambda)$.
Consider (\ref{eq:aux_Rtheta}--\ref{eq:aux_Reta}) as a mapping from
$\Phi_{R_\theta},\Phi_{R_\eta}$ on the right side to
$\tilde \Phi_{R_\theta},\tilde \Phi_{R_\eta}$ on the left side, i.e.
\begin{align*}
\tilde \Phi_{R_\theta}(t)&=1+\int_0^t \Big((\delta|\beta|+C_0)\Phi_{R_\theta}(t')
+\int_0^{t'}\Phi_{R_\eta}(t'-s)\Phi_{R_\theta}(s)\d s\Big)\d{t'},\\
\tilde \Phi_{R_\eta}(t)&=|\beta|\Big(\int_0^t \Phi_{R_\theta}(t-s)\Phi_{R_\eta}(s)\d s
+\delta|\beta|\Phi_{R_\theta}(t)\Big).
\end{align*}
If $\Phi_{R_\theta},\Phi_{R_\eta}$ are non-negative, then clearly so are
$\tilde \Phi_{R_\theta},\tilde \Phi_{R_\eta}$. Furthermore,
if $\Phi_{R_\theta},\Phi_{R_\eta} \in E(\lambda)$, then writing
$L_\theta(\lambda)=\int_0^\infty \Phi_{R_\theta}(s)e^{-\lambda s}\d s$ for
the Laplace transform of $\Phi_{R_\theta}$ and similarly writing $L_\eta,\tilde
L_\theta,\tilde L_\eta$ for those of $\Phi_{R_\eta},\tilde
\Phi_{R_\theta},\tilde \Phi_{R_\eta}$, taking Laplace transforms of
the above gives
\begin{equation}\label{eq:Laplacetransforms}
\begin{aligned}
\lambda \tilde L_\theta(\lambda)-1&=(\delta |\beta|+C_0)L_\theta(\lambda)
+L_\eta(\lambda)L_\theta(\lambda),\\
\tilde L_\eta(\lambda)&=|\beta| L_\theta(\lambda)L_\eta(\lambda)+\delta \beta^2
L_\theta(\lambda).
\end{aligned}
\end{equation}
This implies in particular that
$\tilde L_\theta(\lambda),\tilde L_\eta(\lambda)<\infty$, i.e.\
$\tilde \Phi_{R_\theta},\tilde \Phi_{R_\eta} \in E(\lambda)$.
For $\iota>0$, define further
\[E(\lambda,\iota)=\{(\Phi_{R_\theta},\Phi_{R_\eta}) \in E(\lambda) \times
E(\lambda):L_\theta(\lambda) \leq \iota,
\,L_\eta(\lambda) \leq (\delta\beta^2+1)\iota\}.\]
If $(\Phi_{R_\theta},\Phi_{R_\eta}) \in E(\lambda,\iota)$, then for $\iota>0$ sufficiently
small and $\lambda>0$ sufficiently large, this
implies $\tilde L_\theta(\lambda) \leq \lambda^{-1}(1+(\delta|\beta|+C_0)\iota
+(\delta \beta^2+1)\iota^2) \leq \iota$ and
$\tilde L_\eta(\lambda) \leq |\beta|(\delta\beta^2+1)\iota^2
+\delta \beta^2 \iota \leq (\delta\beta^2+1)\iota$, so
$(\tilde \Phi_{R_\theta},\tilde \Phi_{R_\eta}) \in E(\lambda,\iota)$.
For two pairs of inputs $(\Phi_{R_\theta}^1,\Phi_{R_\eta}^1),
(\Phi_{R_\theta}^2,\Phi_{R_\eta}^2) \in E(\lambda,\iota)$,
note also from the above that
\begin{align*}
|(\tilde \Phi_{R_\theta}^1-\tilde \Phi_{R_\theta}^2)(t)|
&\leq \int_0^t
\Big[(\delta|\beta|+C_0)|(\Phi_{R_\theta}^1-\Phi_{R_\theta}^2)(t')|\\
&\hspace{0.5in}+\int_0^{t'}
\Big(\Phi_{R_\eta}^1(t'-s)|(\Phi_{R_\theta}^1-\Phi_{R_\theta}^2)(s)|
+|(\Phi_{R_\eta}^1-\Phi_{R_\eta}^2)(t'-s)|\Phi_{R_\theta}^2(s)\Big)\d s\Big]\d
t',\\
|(\tilde \Phi_{R_\eta}^1-\tilde \Phi_{R_\eta}^2)(t)|
&\leq |\beta|\int_0^t \Big(\Phi_{R_\theta}^1(t-s)|(\Phi_{R_\eta}^1-\Phi_{R_\eta}^2)(s)|
+|(\Phi_{R_\theta}^1-\Phi_{R_\theta}^2)(t-s)|\Phi_{R_\eta}^2(s)\Big)\d s\\
&\hspace{0.5in}+\delta \beta^2|(\Phi_{R_\theta}^1-\Phi_{R_\theta}^2)(t)|.
\end{align*}
Denoting by $L_\theta^1,L_\theta^2,L_\theta^\Delta$ the Laplace
transforms of
$\Phi_{R_\theta}^1,\Phi_{R_\theta}^2,|\Phi_{R_\theta}^1-\Phi_{R_\theta}^2|$, and
defining similarly $L_\eta^1,L_\eta^2,L_\eta^\Delta$ for $\Phi_{R_\eta}$
and $\tilde L_\theta^1,\tilde L_\theta^2,\tilde L_\theta^\Delta,
\tilde L_\eta^1,\tilde L_\eta^2,\tilde L_\eta^\Delta$ for the outputs
$\tilde \Phi_{R_\theta},\tilde \Phi_{R_\eta}$, taking Laplace transforms gives
\begin{align*}
\tilde L_\theta^\Delta
&\leq \frac{\delta|\beta|+C_0}{\lambda}L_\theta^\Delta
+\frac{L_\eta^1}{\lambda}L_\theta^\Delta
+\frac{L_\theta^2(\delta\beta^2+1)}{\lambda}\frac{L_\eta^\Delta}{\delta\beta^2+1}\\
\frac{\tilde L_\eta^\Delta}{\delta\beta^2+1}
&\leq |\beta|L_\theta^1\frac{L_\eta^\Delta}{\delta\beta^2+1}
+\frac{|\beta|L_\eta^2}{\delta\beta^2+1}L_\theta^\Delta
+\frac{\delta\beta^2}{\delta\beta^2+1}L_\theta^\Delta.
\end{align*}
Thus, defining a metric on $E(\lambda,\iota)$ (with the usual identification of
functions that are equal Lebesgue-a.e.) given by
\begin{align*}
d((\Phi_{R_\theta}^1,\Phi_{R_\eta}^1),(\Phi_{R_\theta}^2,\Phi_{R_\eta}^2))
&=L_\theta^\Delta(\lambda)+(\delta\beta^2+1)^{-1}L_\eta^\Delta(\lambda)\\
&=\int_0^\infty |(\Phi_{R_\theta}^1-\Phi_{R_\theta}^2)(s)|e^{-\lambda s}\d s
+(\delta\beta^2+1)^{-1}\int_0^\infty
|(\Phi_{R_\eta}^1-\Phi_{R_\eta}^2)(s)|e^{-\lambda s}\d s,
\end{align*}
and applying the bounds $L_\theta(\lambda) \leq \iota$ and $L_\eta(\lambda)
\leq (\delta\beta^2+1)\iota$ on $E(\lambda,\iota)$, this shows that the mapping
$(\Phi_{R_\theta},\Phi_{R_\eta}) \mapsto (\tilde \Phi_{R_\theta},\tilde
\Phi_{R_\eta})$ is Lipschitz on $E(\lambda,\iota)$ 
with respect to $d(\cdot)$, with Lipschitz constant at most
\[\frac{\delta|\beta|+C_0}{\lambda}+\frac{2(\delta\beta^2+1)\iota}{\lambda}
+2|\beta|\iota+\frac{\delta\beta^2}{\delta\beta^2+1}.\]
This is less than 1
for $\iota>0$ sufficiently small and $\lambda>0$ sufficiently large, so
$(\Phi_{R_\theta},\Phi_{R_\eta}) \mapsto (\tilde \Phi_{R_\theta},\tilde
\Phi_{R_\eta})$ is a contraction on $(E(\lambda,\iota),d(\cdot))$. Since
$(E(\lambda,\iota),d(\cdot))$ is a closed subset of a weighted $L^1$-space, it
is complete, so by the Banach fixed-point theorem, there
exists a fixed point $(\Phi_{R_\theta},\Phi_{R_\eta})
\in E(\lambda,\iota) \subset E(\lambda) \times E(\lambda)$ of this mapping.
The functions $\Phi_{R_\theta},\Phi_{R_\eta}$ of this fixed point are
continuous since $\tilde \Phi_{R_\theta},\tilde \Phi_{R_\eta}$ of this mapping
are continuous. This establishes that there is a unique continuous solution to
(\ref{eq:aux_Rtheta}--\ref{eq:aux_Reta}) in $E(\lambda) \times E(\lambda)$.

Given this solution to (\ref{eq:aux_Rtheta}--\ref{eq:aux_Reta}), consider now
(\ref{eq:aux_alpha}--\ref{eq:aux_Ceta}) as a mapping from
$(\Phi_\alpha,\Phi_{C_\theta},\Phi_{C_\eta})$ on the right side to
$(\tilde \Phi_\alpha,\tilde \Phi_{C_\theta},\tilde \Phi_{C_\eta})$ on the left side.
Now let $L_\alpha(\lambda),L_\theta(\lambda),L_\eta(\lambda)$ denote the
Laplace transforms of $(\Phi_\alpha,\Phi_{C_\theta},\Phi_{C_\eta})$, and define
also the Laplace transforms
$K_\eta(\lambda)=\int_0^t (s+1)^2\Phi_{R_\eta}^2(s)e^{-\lambda s}\d s$
and $K_\theta(\lambda)=\int_0^t (s+1)^2\Phi_{R_\theta}^2(s)e^{-\lambda s}\d s$.
Choosing $\lambda$ large enough so that
$K_\eta(\lambda),K_\theta(\lambda)<\infty$,
if $\Phi_\alpha,\Phi_{C_\theta},\Phi_{C_\eta} \in E(\lambda)$, then
taking Laplace transforms of (\ref{eq:aux_alpha}--\ref{eq:aux_Ceta}) gives
\begin{align*}
\lambda \tilde L_\alpha(\lambda)-\|\alpha^0\|^2
&=\frac{4.1C_0}{\lambda}+4.1C_0L_\theta(\lambda)+3C_0L_\alpha(\lambda)\\
\lambda \tilde L_\theta(\lambda)-\E(\theta^0)^2
&=C_1L_\theta(\lambda)+6K_\eta(\lambda)L_\theta(\lambda)
+18C_0^2L_\alpha(\lambda)+\frac{6\tau_*^2}{\lambda}K_\eta(\lambda)
+6L_\eta(\lambda)+\frac{C_2}{\lambda}\\
\tilde L_\eta(\lambda)&=2\beta^2 K_\theta(\lambda)L_\eta(\lambda)
+4\delta\beta^2L_\theta(\lambda)+\frac{C_3}{\lambda}
\end{align*}
for some constants $C_1,C_2,C_3$ depending only on
$\delta,\beta,C_0,\sigma^2,\tau_*^2$. For small $\iota>0$, suppose further that
$(\Phi_\alpha,\Phi_{C_\theta},\Phi_{C_\eta}) \in E(\lambda,\iota)$ where
\[E(\lambda,\iota)=\{(\Phi_\alpha,\Phi_{C_\theta},\Phi_{C_\eta}) \in E(\lambda)
\times E(\lambda) \times E(\lambda):
L_\alpha(\lambda) \leq \iota,\,L_\theta(\lambda) \leq \iota,\,
L_\eta(\lambda) \leq (4\delta\beta^2+1)\iota\}.\]
Then, using that $\lim_{\lambda \to \infty} K_\theta(\lambda)=0$
and $\lim_{\lambda \to \infty} K_\eta(\lambda)=0$, for sufficiently large
$\lambda>0$ and small $\iota>0$, the above Laplace transform
equations imply $(\tilde \Phi_{\alpha},\tilde \Phi_{C_\theta},\tilde \Phi_{C_\eta})
\in E(\lambda,\iota)$. Furthermore, defining the metric
\begin{align*}
&d((\Phi_\alpha^1,\Phi_{C_\theta}^1,\Phi_{C_\eta}^1),
(\Phi_\alpha^2,\Phi_{C_\theta}^2,\Phi_{C_\eta}^2))
=L_\alpha^\Delta(\lambda)+L_\theta^\Delta(\lambda)+(4\delta\beta^2+1)L_\eta^\Delta(\lambda)\\
&=\int_0^\infty |(\Phi_\alpha^1-\Phi_\alpha^2)(s)|e^{-\lambda s}\d s
+\int_0^\infty |(\Phi_{C_\theta}^1-\Phi_{C_\theta}^2)(s)|e^{-\lambda s}\d s
+(4\delta\beta^2+1)\int_0^\infty |(\Phi_{C_\eta}^1-\Phi_{C_\eta}^2)(s)|e^{-\lambda s}\d s
\end{align*}
it may be verified as above that the mapping
$(\Phi_\alpha,\Phi_{C_\theta},\Phi_{C_\eta}) \mapsto
(\tilde \Phi_\alpha,\tilde \Phi_{C_\theta},\tilde \Phi_{C_\eta})$ is Lipschitz
on $(E(\lambda,\iota),d(\cdot))$, with Lipschitz constant at most
\[\frac{7.1C_0}{\lambda}+\frac{C_1+6K_\eta(\lambda)+18C_0^2+6(4\delta\beta^2+1)}{\lambda}+2\beta^2K_\theta(\lambda)+\frac{4\delta\beta^2}{4\delta\beta^2+1}.\]
For sufficiently large $\lambda>0$, this is again less than 1, so there exists a
unique fixed point $(\Phi_\alpha,\Phi_{C_\theta},\Phi_{C_\eta}) \in
E(\lambda,\iota) \subset E(\lambda) \times E(\lambda) \times E(\lambda)$ and
hence a unique continuous solution to (\ref{eq:aux_alpha}--\ref{eq:aux_Ceta}).
\end{proof}

Let $(\Phi_\alpha,\Phi_{C_\theta},\Phi_{C_\eta},\Phi_{R_\theta},\Phi_{R_\eta})$
be the above solution to (\ref{eq:aux_alpha}--\ref{eq:aux_Reta}).
For any $T>0$ and
finite set $D=\{d_1,\ldots,d_m\} \subset (0,T)$, we call
$[0,d_1),[d_1,d_2),\ldots,[d_m,T]$ the maximal intervals of $[0,T] \setminus D$.
Fixing $T>0$ and denoting
\[(\alpha,C_\theta,C_\eta,R_\theta,R_\eta) \equiv
\{\alpha^t,C_\theta(t,s),C_\theta(t,*),C_\theta(*,*),C_\eta(t,s),R_\theta(t,s),
R_\eta(t,s)\}_{0 \leq s \leq t \leq T},\]
we define the space $\cS \equiv \cS(T)$ in Theorem \ref{thm:dmftsolexists} as
\begin{equation}\label{eq:solutionspace}
\cS=\{(\alpha,C_\theta,C_\eta,R_\theta,R_\eta):
(R_\eta,C_\eta,\alpha) \in \cS_\eta,\,
(R_\theta,C_\theta,\alpha) \in \cS_\theta\}.
\end{equation}
Here $\cS_\eta \equiv \cS_\eta(T)$ is the collection of $(R_\eta,C_\eta,\alpha)$ such that,
for some (possibly empty) discontinuity set $D \subset (0,T)$ of at most
finite cardinality:
\begin{itemize}
\item $C_\eta$ is a positive-semidefinite covariance kernel on $[0,T]$
(identifying $C_\eta(s,t)=C_\eta(t,s)$) satisfying
\begin{align}\label{eq:Ceta_cond_1}
C_\eta(t,t) \leq \Phi_{C_\eta}(t) \text{ for all } 0 \leq t \leq T.
\end{align}
Furthermore, $C_\eta(t,s)$ is uniformly continuous over $s,t \in I$ 
for each maximal interval $I$ of $[0,T] \setminus D$, and satisfies
\begin{align}\label{eq:Ceta_cond_2}
\big|C_\eta(t,t) - 2C_\eta(t,s) + C_\eta(s,s)\big|
&\leq 3\beta^2\bigg[\Big(T^3\sup_{r \in [0,T]} \Phi_{R_\theta}'(r)^2
+T\sup_{r \in [0,T]} \Phi_{R_\theta}(r)^2\Big)
\sup_{r \in [0,T]} \Phi_{C_\eta}(r)\notag\\
&\hspace{0.5in}+\delta\Big(2T\sup_{r \in [0,T]} \Phi_{C_\theta}'(r)+4\Big)\bigg]
\cdot |t-s| \text{ for all } s,t \in I.
\end{align}
\item $R_\eta(t,s)$ satisfies
\begin{align}\label{eq:Reta_cond_1}
|R_\eta(t,s)| \leq \Phi_{R_\eta}(t-s) \text{ for all } 0 \leq s \leq t \leq T.
\end{align} 
Furthermore, $R_\eta(t,s)$ is uniformly continuous over $s \in I'$ and $t \in I$
for any two (possibly equal) maximal intervals $I,I'$ of $[0,T] \setminus D$.
\item $\alpha^t$ satisfies
\begin{equation}\label{eq:alpha_cond}
\|\alpha^t\|^2 \leq \Phi_\alpha(t) \text{ for all } 0 \leq t \leq T.
\end{equation}
and is uniformly continuous on each 
maximal interval $I$ of $[0,T] \setminus D$.
\end{itemize}
Similarly $\cS_\theta \equiv \cS_\theta(T)$ is the set of $(R_\theta,C_\theta,\alpha)$ such that
\begin{itemize}
\item $C_\theta$ is a positive-semidefinite covariance kernel on $\{*\} \cup [0,T]$
(identifying $C_\theta(s,t)=C_\theta(t,s)$ and $C_\theta(t,*)=C_\theta(*,t)$) satisfying
\begin{align}\label{eq:Ctheta_cond_1}
C_\theta(t,t) &\leq \Phi_{C_\theta}(t) \text{ for all } 0 \leq t \leq T.
\end{align}
Furthermore, $C_\theta(t,s)$ is uniformly continuous over $s,t \in I$ 
for each maximal interval $I$ of $[0,T] \setminus D$ and satisfies
\begin{align}\label{eq:Ctheta_cond_2}
\big|C_\theta(t,t)-2C_\theta(t,s)+C_\theta(s,s)\big|
\leq \Big(2T\sup_{r \in [0,T]} \Phi_{C_\theta}'(r)+4\Big)|t-s|
\text{ for all } s,t \in I,
\end{align}
and $C_\theta(t,*)$ is uniformly continuous over $t \in I$.
\item $R_\theta(t,s)$ satisfies
\begin{align}\label{eq:Rtheta_cond}
|R_\theta(t,s)| \leq \Phi_{R_\theta}(t-s) \text{ for all } 0 \leq s \leq t \leq
T.
\end{align}
Furthermore, $R_\theta(t,s)$ is uniformly continuous over $s \in I'$ and $t \in
I$ for any two (possibly equal) maximal intervals $I,I'$ of $[0,T] \setminus D$,
and satisfies
\begin{align}\label{eq:Rtheta_cont}
\big|R_\theta(t',s)-R_\theta(t,s)\big|
\leq \Big(\sup_{r \in [0,T]} \Phi_{R_\theta}'(r)\Big)|t'-t|
\text{ for each fixed } s \in [0,T] \text{ and all } t,t' \in [s,T] \cap I.
\end{align}
\item $\alpha^t$ satisfies (\ref{eq:alpha_cond}) and is
uniformly continuous on each maximal interval $I$ of $[0,T] \setminus D$.
\end{itemize}
We define
\[\cS^\text{cont}(T) \equiv \cS^\text{cont} \subset \cS,
\qquad \cS_\eta^\text{cont}(T)\equiv \cS_\eta^\text{cont} \subset \cS_\eta,
\qquad \cS_\theta^\text{cont}(T)\equiv \cS_\theta^\text{cont} \subset \cS_\theta\]
as the subsets of the above spaces where $D=\emptyset$,
i.e.\ the above continuity conditions hold on all of $[0,T]$.

\begin{remark}
By (\ref{eq:Ceta_cond_2}), letting $\{u^t\}_{t \in [0,T]}$ be a
mean-zero Gaussian process with covariance $C_\eta$, for any maximal interval
$I$ of $[0,T] \setminus D$, any $s,t \in I$, and some constant $C>0$,
\[\E(u^t-u^s)^4=3[\E(u^t-u^s)^2]^2 \leq C|t-s|^2.\]
Then Kolmogorov's continuity theorem (\cite[Theorem
2.9]{legall2016brownian}) implies that there exists a modification of
$\{u^t\}_{t \in [0,T]}$ that is
uniformly H\"older continuous on each such maximal interval $I$, and similarly
for $\{w^t\}_{t \in [0,T]}$ with covariance $C_\theta$ satisfying
(\ref{eq:Ctheta_cond_2}). We will always take $\{u^t\}$ and
$\{w^t\}$ to be the versions of these processes
that satisfy this H\"older continuity.
\end{remark}

Let us now establish existence and uniqueness of the solutions to
(\ref{def:dmft_langevin_cont_theta}--\ref{def:response_g}) given
$(\alpha,C_\theta,C_\eta,R_\theta,R_\eta) \in \cS$.

\begin{lemma}\label{lem:sub_theta_solution}
Fix any $T>0$, any $(R_\eta,C_\eta,\alpha) \in
\cS_\eta$, and any realizations of $\theta^0,\theta^*,\{b^t\}_{t \leq
T}$ and $\{u^t\}_{t \leq T}$. Then there exist
unique $\cF_t^\theta$-adapted processes $\{\theta^t\}_{t \leq T}$ and
$\{\frac{\partial \theta^t}{\partial u^s}\}_{s \leq t \leq T}$
solving (\ref{def:dmft_langevin_cont_theta}--\ref{def:response_theta}).
\end{lemma}
\begin{proof}
Consider the drift function
\[v(t,\{\theta^s\}_{s \leq t})
={-}\delta\beta(\theta^t-\theta^*)+s(\theta^t,\alpha^t)
+\int_0^t R_\eta(t,s)(\theta^s-\theta^*)\d s+u^t.\]
Conditioning on $\theta^0,\theta^*$ and $\{u^t\}$ and writing 0 for the process
$\theta^t \equiv 0$, we have (with probability 1 over
$\theta^0,\theta^*$ and $\{u^t\}$)
\begin{align*}
\sup_{t \in [0,T]} |v(t,0)|
\leq \delta|\beta\theta^*|+\sup_{t \in [0,T]}
|s(0,\alpha^t)|+\int_0^T \Phi_{R_\eta}(t)\d t \cdot
|\theta^*|+\sup_{t \in [0,T]} |u^t|<\infty.
\end{align*}
Furthermore, for all $t \in [0,T]$,
\begin{align*}
|v(t,\{\theta^s\}_{s \leq t})-v(t,\{\tilde \theta^s\}_{s \leq t})|
&\leq \Big(\delta|\beta|
+\sup_{(\theta,\alpha) \in \R \times \R^K}
|\partial_\theta s(\theta,\alpha)|+\int_0^T \Phi_{R_\eta}(s)\d s\Big)
\sup_{s \in [0,t]}|\theta^s-\tilde \theta^s|,
\end{align*}
showing under Assumption \ref{assump:prior} that
$\{\theta^s\}_{s \leq t} \mapsto v(t,\{\theta^s\}_{s \leq t})$ is
Lipschitz in the norm of uniform convergence, uniformly over $t \in [0,T]$.
Then existence and uniqueness of a solution
$\{\theta^t\}_{t \leq T}$ with $\theta^t|_{t=0}=\theta^0$
adapted to the filtration of $\{b^t\}_{t \geq 0}$
is classical, see e.g.\ \cite[Theorem 11.2]{rogers2000diffusions}.
This solution is a measurable function of $\theta^0$, $\theta^*$, and
$\{u^t\}$, and hence is also $\cF_t^\theta$-adapted.

Conditioning now on $\{\theta^t\}$, for any fixed $s \in [0,T]$, consider
\[v(t,\{x^{s'}\}_{s' \in [s,t]})
={-}\Big(\delta \beta-\partial_\theta s(\theta^t,\alpha^t)\Big)x^t
+\int_s^t R_\eta(t,s')x^{s'}\d s'.\]
This satisfies $v(t,0) \equiv 0$ for all $t \in [s,T]$ and
\begin{align*}
|v(t,\{x^{s'}\}_{s' \in [s,t]})-v(t,\{\tilde x^{s'}\}_{s' \in [s,t]})|
&\leq \Big(\delta|\beta|
+\sup_{(\theta,\alpha) \in \R \times \R^K}
|\partial_\theta s(\theta,\alpha)|+\int_0^T \Phi_{R_\eta}(t)\d t\Big)
\sup_{s' \in [s,t]}|x^{s'}-\tilde x^{s'}|,
\end{align*}
so $\{x^{s'}\}_{s' \in [s,t]} \mapsto v(t,\{x^{s'}\}_{s' \in [s,t]})$ is
also Lipschitz in the norm of uniform convergence,
uniformly over $t \in [s,T]$. Then again for each $s \in [0,T]$,
there exists a unique solution 
$\{\frac{\partial \theta^t}{\partial u^s}\}_{t \in [s,T]}$
with $\frac{\partial \theta^t}{\partial u^s}|_{t=s}=1$, which is adapted to the
filtration $\cF_t \equiv \cF(\{\theta^r\}_{r \in [s,t]})$ and hence also to
$\cF_t^\theta$, showing the lemma.
\end{proof}

\begin{lemma}\label{lem:sub_eta_solution}
Fix any $T>0$, any
$(R_\theta,C_\theta,\alpha) \in \cS_\theta$, and any realizations of $\eps$ and
$(w^*,\{w^t\}_{t \leq T})$. Then there exist
unique $\cF_t^\eta$-adapted processes $\{\eta^t\}_{t \leq T}$ and
$\{\frac{\partial \eta^t}{\partial w^s}\}_{s \leq t \leq T}$
solving (\ref{def:dmft_langevin_cont_eta}--\ref{def:response_g}).
\end{lemma}
\begin{proof}
Conditional on $\eps$ and $(w^*,\{w^t\})$, the equations
(\ref{def:dmft_langevin_cont_eta}--\ref{def:response_g}) are linear
Volterra integral equations for which the kernel $(s,t) \mapsto R_\theta(t,s)$
is continuous on each maximal interval $I$ of $[0,T] \setminus D$. Then, for
each maximal interval $I=[a,b)$, given the values of $\{\eta^t\}$ for
$t \in [0,a]$, existence and uniqueness of $\{\eta^t\}_{t \in [a,b)}$ is
classical and follows from e.g.\ \cite[Theorem 2.1.2]{brunner2004collocation}.
Applying this successively to each maximal interval $I$ shows existence
and uniqueness of $\{\eta^t\}$ over $t \in [0,T]$. A similar argument shows, for each fixed $s \in [0,T]$, the existence and uniqueness of
$\{\frac{\partial \eta^t}{\partial w^s}\}$ over $t \in [s,T]$. Here
$\frac{\partial \eta^t}{\partial w^s}$ is deterministic by its definition,
while $\eta^t$ is a measurable function of $\eps,w^*,\{w^s\}_{s \leq t}$ and
hence is adapted to $\cF_t^\eta$.
\end{proof}

\begin{proof}[Proof of Theorem \ref{thm:dmftsolexists}(a)]
This follows from Lemmas \ref{lem:sub_theta_solution}
and \ref{lem:sub_eta_solution}.
\end{proof}

\subsection{Contractive mapping}\label{sec:contractivemapping}

We fix $T>0$. For any $(R_\eta,C_\eta,\alpha) \in \cS_\eta$, define a map
$\cT_{\eta\rightarrow \theta}: (R_\eta,C_\eta,\alpha) \rightarrow
(R_\theta,C_\theta,\tilde \alpha)$ by 
\[\begin{gathered}
R_\theta(t,s)=\E\Big[\frac{\partial \theta^t}{\partial u^s}\Big], \quad
C_\theta(t,s)=\E[\theta^t\theta^s], \quad C_\theta(t,*)=\E[\theta^t\theta^\ast], \quad C_\theta(*,*)=\E[(\theta^\ast)^2],\\
\frac{\d}{\d t}\tilde \alpha^t=\cG(\tilde \alpha^t,\sP(\theta^t))
\text{ with } \tilde \alpha^t|_{t=0}=\alpha^0
\end{gathered}\]
where $\{\theta^t\}_{t\in[0,T]}$ and
$\{\frac{\partial \theta^t}{\partial u^s}\}_{0\leq s \leq t\leq T}$ are the
unique solutions to
(\ref{def:dmft_langevin_cont_theta}--\ref{def:response_theta}) given
$(R_\eta,C_\eta,\alpha)$ and $\theta^0,\theta^*,\{u^t\}$, guaranteed by Lemma \ref{lem:sub_theta_solution},
and $\sP(\theta^t)$ is the law of $\theta^t$.
Similarly, for any $(R_\theta,C_\theta,\tilde\alpha) \in \cS_\theta$, define a map
$\cT_{\theta\rightarrow\eta}: (R_\theta,C_\theta,\tilde\alpha) \rightarrow
(R_\eta,C_\eta,\alpha)$ 
by
\begin{align*}
R_\eta(t,s)=\delta\beta\,\E\Big[\frac{\partial \eta^t}{\partial w^s}\Big], \quad C_\eta(t,s)=\delta\beta^2\E[(\eta^t + w^\ast -
\eps)(\eta^s + w^\ast - \eps)],
\quad \alpha^t=\tilde \alpha^t
\end{align*}
where $\{\eta^t\}_{t\in[0,T]}$ and $\{\frac{\partial \eta^t}{\partial w^s}\}_{0\leq s < t \leq T}$ are the unique solutions to
(\ref{def:dmft_langevin_cont_eta}--\ref{def:response_g}) given
$(R_\theta,C_\theta,\tilde\alpha)$ and $\eps,w^*,\{w^t\}$,
guaranteed by Lemma \ref{lem:sub_eta_solution}.
Finally, define the composite maps
\begin{align}\label{eq:composite_map}
\cT_{\eta\rightarrow\eta}=\cT_{\theta\rightarrow\eta} \circ
\cT_{\eta\rightarrow\theta}, \qquad
\cT_{\theta\rightarrow\theta}=\cT_{\eta\rightarrow\theta} \circ
\cT_{\theta\rightarrow\eta}.
\end{align}
The rest of this subsection is divided into two parts:
\begin{itemize}
\item (\textbf{Part 1}) We show in Lemma \ref{lem:Trange_theta_2_eta} (resp. Lemma \ref{lem:Trange_eta_2_theta}) that $\cT_{\theta \rightarrow \eta}$ maps $\cS_\theta$ into $\cS_\eta$ (resp. $\cT_{\eta \rightarrow \theta}$ maps $\cS_\eta$ into $\cS_\theta$). 
\item (\textbf{Part 2}) We equip $\cS_\eta$ and $\cS_\theta$ with certain
metrics and derive the moduli-of-continuity of the maps $\cT_{\eta \rightarrow
\theta}$ and $\cT_{\theta \rightarrow \eta}$ in Lemmas
\ref{lem:modulus_eta_2_theta} and \ref{lem:modulus_theta_2_eta}, thereby
concluding that $\cT_{\eta\rightarrow\eta}$
and $\cT_{\theta\rightarrow\theta}$ in
(\ref{eq:composite_map}) are contractions under these metrics.
\end{itemize}

\begin{lemma}\label{lem:Trange_theta_2_eta}
$\cT_{\theta \rightarrow \eta}$ maps $\cS_\theta$ into $\cS_\eta$, and
$\cS_\theta^\text{cont}$ into $\cS_\eta^\text{cont}$.
\end{lemma}
\begin{proof}
(\textbf{Condition for $C_\eta$}) Define $\xi^t = \eta^t + w^\ast - \eps$, so
that
\begin{align*}
\xi^t = -\beta \int_0^t R_\theta(t,s)\xi^s \d s - w^t + w^\ast - \eps
\end{align*}
and $C_\eta(t,s)=\delta\beta^2\E[\xi^t\xi^s]$. Then by Cauchy-Schwarz,
\begin{equation}\label{eq:cauchyschwarztrick}
\begin{aligned}
C_\eta(t,t) &\leq 2\delta\beta^2\,\E\Big[\beta^2\Big(\int_0^t R_\theta(t,s)\xi^s\d s\Big)^2 +
(w^t - w^\ast + \eps)^2\Big]\\
&\leq 2\delta\beta^2\,\Big[\beta^2\int_0^t (t-s+1)^2 \cdot R_\theta(t,s)^2 \E(\xi^s)^2\d s
\cdot \int_0^t (t-s+1)^{-2}\d s + 2 C_\theta(t,t) + 2\tau_*^2 + \sigma^2\Big]\\
&\leq 2\delta\beta^2\Big[\frac{1}{\delta}\int_0^t (t-s+1)^2 \cdot \Phi^2_{R_\theta}(t-s) C_\eta(s,s)\d
s + 2 \Phi_{C_\theta}(t) + 2\tau_*^2 + \sigma^2\Big].
\end{aligned}
\end{equation}
Recalling the equation for $\Phi_{C_\eta}(\cdot)$ in (\ref{eq:aux_Ceta}),
\begin{align*}
\Phi_{C_\eta}(t) = 2\delta\beta^2\Big[\frac{1}{\delta}\int_0^t (t-s+1)^2 \cdot
\Phi^2_{R_\theta}(t-s) \Phi_{C_\eta}(s)\d s + 2\Phi_{C_\theta}(t) + 2\tau_*^2 + \sigma^2\Big],
\end{align*}
Gronwall's inequality implies that $C_\eta(t,t) \leq \Phi_{C_\eta}(t)$,
showing (\ref{eq:Ceta_cond_1}).

We now check (\ref{eq:Ceta_cond_2}) on each maximal interval $I$ of $[0,T]
\setminus D$ where $D \subset (0,T)$ is the discontinuity set of $\cS_\theta$. Note that
\begin{align*}
&C_\eta(t,t) - 2 C_\eta(t,s) + C_\eta(s,s)
=\delta\beta^2\E[(\xi^t - \xi^s)^2]\\
&=\delta\beta^2\E\Big[\Big({-}\beta\int_0^t R_\theta(t,r)\xi^r\d r + \beta\int_0^s
R_\theta(s,r)\xi^r\d r - w^t + w^s\Big)^2\Big]\\
&\leq 3\delta\beta^2\Big[\beta^2\E \Big(\int_0^s \big(R_\theta(t,r) -
R_\theta(s,r)\big) \xi^r \d r\Big)^2 + \beta^2\E\Big(\int_s^t R_\theta(t,r)\xi^r\d r\Big)^2 + \E(w^t - w^s)^2\Big].
\end{align*}
Using $\delta\beta^2\E(\xi^t)^2=C_\eta(t,t)\leq \Phi_{C_\eta}(t)$ established above, together with
the continuity conditions (\ref{eq:Rtheta_cont}) for $R_\theta$ and
(\ref{eq:Ctheta_cond_2}) for $C_\theta$, for any $s,t \in I$
it holds that
\begin{align*}
\E\Big(\int_0^s \big(R_\theta(t,r) - R_\theta(s,r)\big) \xi^r \d r\Big)^2 &\leq
s \int_0^s (R_\theta(t,r)-R_\theta(s,r))^2\E(\xi^r)^2\d r\\
&\leq \frac{T^2}{\delta\beta^2}
\Big(\sup_{r \in [0,T]} \Phi_{R_\theta}'(r)\Big)^2
\cdot \sup_{r \in [0,T]} \Phi_{C_\eta}(r) \cdot |t-s|^2,\\
\E\Big(\int_s^t R_\theta(t,r)\xi^r\d r\Big)^2 &\leq
(t-s)\int_s^t R_\theta(t,r)^2\E(\xi^r)^2\d r\\
&\leq \frac{1}{\delta\beta^2}\sup_{r \in [0,T]} \Phi_{R_\theta}(r)^2
\cdot \sup_{r \in [0,T]} \Phi_{C_\eta}(r) \cdot |t-s|^2,\\
\E(w^t - w^s)^2&=C_\theta(t,t) - 2C_\theta(t,s) + C_\theta(s,s) \leq
(2T\sup_{r \in [0,T]} \Phi_{C_\theta}'(r)+4) \cdot |t-s|.
\end{align*}
Combining these bounds shows (\ref{eq:Ceta_cond_2}) over $s,t \in I$.
Applying $|\E[\xi^s\xi^t-\xi^{s'}\xi^{t'}]|^2
\leq 2\E(\xi^s-\xi^{s'})^2\E(\xi^{t'})^2
+2\E(\xi^{s'})^2\E(\xi^t-\xi^{t'})^2$, this shows also that $C_\eta(t,s)$ is
uniformly continuous over $s,t \in I$. If $(C_\theta,R_\theta,\tilde\alpha) \in
\cS_\theta^{\text{cont}}$, then $D=\emptyset$ so this maximal interval is $I=[0,T]$.\\

\noindent (\textbf{Condition for $R_\eta$}) By definition,
$R_\eta(t,s) = \beta[{-}\int_s^t R_\theta(t,s')R_\eta(s',s)\d s' +
\delta\beta R_\theta(t,s)]$, hence
\begin{align*}
|R_\eta(t,s)| &\leq |\beta|\Big(\int_s^t |R_\theta(t,s')||R_\eta(s',s)|\d s' +
\delta|\beta R_\theta(t,s)|\Big)\\
&\leq |\beta|\Big(\int_0^{t-s} \Phi_{R_\theta}(t-s-s')|R_\eta(s+s',s)|\d s' +
\delta|\beta|\Phi_{R_\theta}(t-s)\Big).
\end{align*}
Recalling the equation for $\Phi_{R_\eta}$ in (\ref{eq:aux_Reta}),
\begin{align*}
\Phi_{R_\eta}(t-s) =|\beta|\Big(\int_0^{t-s}
\Phi_{R_\theta}(t-s-s')\Phi_{R_\eta}(s')\d s'
+\delta|\beta|\Phi_{R_\theta}(t-s)\Big),
\end{align*}
this implies for all $t \in  [s,T]$ that
$|R_\eta(t,s)|\leq \Phi_{R_\eta}(t-s)$, verifying
(\ref{eq:Reta_cond_1}). To show uniform continuity on each pair of maximal
intervals $I,I'$ defining $\cS_\theta$,
observe first that for any $s,s' \in I'$ and $\tau \geq 0$
for which $s+\tau,s'+\tau \in I$,
\begin{align*}
&|R_\eta(s'+\tau,s')-R_\eta(s+\tau,s)|\\
&\leq |\beta| \cdot \left|\int_s^{s+\tau} R_\theta(s+\tau,r)R_\eta(r,s)\d r
-\int_{s'}^{s'+\tau} R_\theta(s'+\tau,r)R_\eta(r,s')\d r\right|
+\delta\beta^2|R_\theta(s'+\tau,s')-R_\theta(s+\tau,s)|\\
&=|\beta| \int_0^\tau \Big|R_\theta(s+\tau,s+r)R_\eta(s+r,s)
-R_\theta(s'+\tau,s'+r)R_\eta(s'+r,s')\Big|\d r
+\delta\beta^2|R_\theta(s'+\tau,s')-R_\theta(s+\tau,s)|\\
&\leq |\beta|\int_0^\tau |R_\theta(s+\tau,s+r)| \cdot |R_\eta(s'+r,s')-R_\eta(s+r,s)| \d r\\
&\hspace{0.3in}+|\beta|\int_0^\tau |R_\theta(s+\tau,s+r)-R_\theta(s'+\tau,s'+r)| \cdot |R_\eta(s'+r,s')| \d r
+\delta\beta^2|R_\theta(s'+\tau,s')-R_\theta(s+\tau,s)|.
\end{align*}
Denoting by $o_{|s-s'|}(1)$ an error that converges to 0 uniformly in $\tau$
as $|s-s'| \to 0$, observe that the last term above is $o_{|s-s'|}(1)$ 
by the uniform continuity of $R_\theta$ on $I' \times I$.
For the second term, writing the range of integration as $[0,\tau]=A \cup B$
where $r \in A$ are the values for which $s+r,s'+r$ belong to a single maximal interval of $[0,T] \setminus D$ and $r \in B$ are the values for which $s+r,s'+r$ belong to two different maximal intervals, the integral over $r \in A$ is $o_{|s-s'|}(1)$ 
again by the continuity of $R_\theta$, while the integral over $r \in B$ is also $o_{|s-s'|}(1)$ by the boundedness of $R_\theta,R_\eta$ and the bound $|B| \leq C|s-s'|$ for the total length of $B$.
Putting this together,
\[|R_\eta(s'+\tau,s')-R_\eta(s+\tau,s)|
\leq C\int_0^\tau |R_\eta(s'+r,s')-R_\eta(s+r,s)| \d r
+o_{|s-s'|}(1).\]
Since $R_\eta(s',s')=R_\eta(s,s)$,
the above and Gronwall's inequality imply that 
\begin{equation}\label{eq:Retaconts}
|R_\eta(s'+\tau,s')-R_\eta(s+\tau,s)|=o_{|s-s'|}(1)
\end{equation}
uniformly in $\tau$. Now for any $s \in I'$ and $\tau' \geq
\tau \geq 0$ for which $s+\tau,s+\tau'\in I$,
\begin{align*}
|R_\eta(s+\tau',s)-R_\eta(s+\tau,s)|
&\leq |\beta|\int_s^{s+\tau} |R_\theta(s+\tau',r)-R_\theta(s+\tau,r)|
|R_\eta(r,s)|\d r\\
&\hspace{0.2in}+\int_{s+\tau}^{s+\tau'} |R_\theta(s+\tau',r)| \cdot
|R_\eta(r,s)|\d r+\delta |\beta| |R_\theta(s+\tau',s)-R_\theta(s+\tau,s)|,
\end{align*}
so the continuity of $R_\theta$
and boundedness of $R_\theta,R_\eta$ again imply that
\begin{equation}\label{eq:Retaconttau}
|R_\eta(s+\tau',s)-R_\eta(s+\tau,s)|=o_{|\tau-\tau'|}(1)
\end{equation}
uniformly in $s$.
The statements (\ref{eq:Retaconts}) and (\ref{eq:Retaconttau}) show
that $(s,\tau) \mapsto R_\eta(s+\tau,s)$ is uniformly
continuous over $\{(s,\tau):s \in I',\tau \geq 0,s+\tau \in I\}$, implying uniform
continuity of $(s,t) \mapsto R_\eta(t,s)$ over $(s,t) \in I' \times I$.
Again if $(C_\theta,R_\theta,\tilde\alpha) \in
\cS_\theta^{\text{cont}}$, then this continuity holds over all of $I=[0,T]$.\\

\noindent (\textbf{Condition for $\alpha$}) By definition, the mapping
$\tilde \alpha \mapsto \alpha$ under $\cT_{\theta \to \eta}$ is the identity,
so the required conditions for $\alpha$ hold by those assumed for $\tilde
\alpha$.
\end{proof}

\begin{lemma}\label{lem:Trange_eta_2_theta}
$\cT_{\eta\rightarrow \theta}$ maps $\cS_\eta$ into $\cS_\theta^\text{cont}$.
\end{lemma}
\begin{proof}
(\textbf{Condition for $C_\theta$}) To verify (\ref{eq:Ctheta_cond_1}), denote
\begin{align*}
v^t={-}\delta\beta(\theta^t-\theta^*) + s(\theta^t,\alpha^t)+\int_0^t
R_\eta(t, s)(\theta^{s}-\theta^*)\d s + u^t. 
\end{align*}
Applying Ito's formula to $(\theta^t)^2$ yields
\begin{equation}\label{eq:dCtt}
\frac{\d C_\theta(t,t)}{\d t}
=\frac{\d \E (\theta^t)^2}{\d t} = \E[2\theta^t v^t] + 2 \leq
1.1 \cdot \E(\theta^t)^2+\E(v^t)^2 + 2.
\end{equation}
(The bound holds with 1 in place of 1.1, and we enlarge this to 1.1 to
accommodate a later discretized version of this computation.)
Using an argument similar to (\ref{eq:cauchyschwarztrick}),
and letting $C_0>0$ be the constant defining
(\ref{eq:aux_alpha}--\ref{eq:aux_Reta}) which upper bounds
$C>0$ in (\ref{eq:s_smooth}) of Assumption \ref{assump:prior},
we may bound $\E(v^t)^2$ as
\begin{align}
\E(v^t)^2 &\leq 6\bigg[\delta^2\beta^2 \E(\theta^t)^2
+\delta^2\beta^2 \tau_*^2+\E[s(\theta^t,\alpha^t)^2]
+\E\Big(\int_0^t R_\eta(t,s) \theta^s \d s\Big)^2
+\E\Big(\int_0^t R_\eta(t,s) \theta^* \d s\Big)^2
+\E(u^t)^2\bigg]\notag\\
&\leq (6\delta^2\beta^2+18C_0^2)\E(\theta^t)^2
+6\int_0^t (t-s+1)^2 \Phi_{R_\eta}^2(t-s)\E(\theta^s)^2\d s\notag\\
&\hspace{1in}+6\Big(\delta^2\beta^2 \tau_*^2+3C_0^2+3C_0^2\Phi_\alpha(t)
+\int_0^t (t-s+1)^2\Phi_{R_\eta}^2(t-s)\d s \cdot \tau_*^2
+\Phi_{C_\eta}(t)\Big).\label{eq:vtbound}
\end{align}
Applying this to (\ref{eq:dCtt}) and
comparing with the equation for $\Phi_{C_\theta}$ from (\ref{eq:aux_Ctheta}),
\begin{align}
\frac{\d}{\d t}\Phi_{C_\theta}(t)&=(6\delta^2\beta^2+18C_0^2+1.1)\Phi_{C_\theta}(t)
+6\int_0^t (t-s+1)^2 \Phi_{R_\eta}^2(t-s)\Phi_{C_\theta}(s)\d s\label{eq:dPhiCtheta}\\
&\hspace{0.5in}+6\Big(\delta^2\beta^2 \tau_*^2+3C_0^2+3C_0^2\Phi_\alpha(t)
+\int_0^t (t-s+1)^2\Phi_{R_\eta}^2(t-s)\d s \cdot \tau_*^2
+\Phi_{C_\eta}(t)\Big)+2,\notag
\end{align}
we see that since $C_\theta(0,0)=\Phi_{C_\theta}(0)$, we have
$C_\theta(t,t) \leq \Phi_{C_\theta}(t)$.

Next we prove (\ref{eq:Ctheta_cond_2}) for all $0 \leq s \leq t \leq T$.
We have $\theta^t-\theta^s=\int_s^t v^r\,\d r + \sqrt{2}(b^t - b^s)$.
Then it holds that
\begin{align*}
C_\theta(t,t) - 2C_\theta(t,s) + C_\theta(s,s)
&= \E[(\theta^t - \theta^s)^2]
\leq 2|t-s| \int_s^t \E(v^r)^2\d r+4\E(b^t-b^s)^2\\
&\leq 2|t-s|^2 \sup_{r \in [0,T]} |\Phi_{C_\theta}'(r)|
+4|t-s| \leq \Big(2T\sup_{r \in [0,T]} \Phi_{C_\theta}'(r)+4\Big)|t-s|,
\end{align*}
where the second inequality compares (\ref{eq:vtbound}) to the definition of $\Phi_{C_\theta}'(t)$ in
(\ref{eq:dPhiCtheta}).
This verifies (\ref{eq:Ctheta_cond_2}).
As in the preceding argument for $C_\eta$, this condition
(\ref{eq:Ctheta_cond_2}) and Cauchy-Schwarz implies that $C_\theta(t,s)$
is uniformly continuous over all $0 \leq s \leq t \leq T$, and also that $C_\theta(t,*)$ is uniformly continuous over all $t \in [0,T]$.\\

\noindent(\textbf{Condition for $R_\theta$}) Let $\bar{R}_\theta(t,s)=\E|\frac{\partial \theta^t}{\partial u^s}|$ so that $|R_\theta(t,s)|\leq \bar{R}_\theta(t,s)$ by definition. Note that
\begin{align}\label{eq:theta_response_dt}
\notag\frac{\d}{\d t}\Big|\frac{\partial \theta^t}{\partial u^s}\Big| \leq
\Big|\frac{\d}{\d t}\frac{\partial \theta^t}{\partial u^s}\Big| &\leq
\big(\delta|\beta|+|\partial_\theta s(\theta^t,\alpha^t)|\big)\Big|\frac{\partial \theta^t}{\partial u^s}\Big| + \int_s^t |R_\eta(t,s')|\Big|\frac{\partial \theta^{s'}}{\partial u^s}\Big|\d s'\\
&\leq (\delta|\beta|+C_0)\Big|\frac{\partial \theta^t}{\partial u^s}\Big| + \int_s^t \Phi_{R_\eta}(t-s')\Big|\frac{\partial \theta^{s'}}{\partial u^s}\Big|\d s',
\end{align}
where $C_0>0$ is the constant defining (\ref{eq:aux_alpha}--\ref{eq:aux_Reta})
which upper bounds $C>0$ in (\ref{eq:s_smooth})
of Assumption \ref{assump:prior}. Taking expectation on both sides yields
\begin{align*}
\frac{\d}{\d t}\bar{R}_\theta(t,s)
\leq (\delta|\beta|+C_0)\bar{R}_\theta(t,s) + \int_0^{t-s}
\Phi_{R_\eta}(t-s-s')\bar{R}_{\theta}(s+s',s)\d s'.
\end{align*}
Recall the equation for $\Phi_{R_\theta}$ in (\ref{eq:aux_Rtheta}),
\begin{align*}
\frac{\d}{\d t}\Phi_{R_\theta}(t-s)=(\delta|\beta|+C_0)\Phi_{R_\theta}(t-s) +
\int_0^{t-s}
\Phi_{R_\eta}(t-s-s')\Phi_{R_\theta}(s')\d s'.
\end{align*}
Since $\bar{R}_\theta(s,s)=1=\Phi_{R_\theta}(0)$, this implies for all $t \in [s,T]$ that
$|R_\theta(t,s)|\leq \bar{R}_\theta(t,s) \leq \Phi_{R_\theta}(t-s)$,
verifying (\ref{eq:Rtheta_cond}) for all $0 \leq s \leq t \leq T$.

To show (\ref{eq:Rtheta_cont}) for all $0 \leq s \leq t \leq t' \leq T$,
observe that we have
\begin{align*}
|R_\theta(t',s)-R_\theta(t,s)|&=\bigg|\int_t^{t'} \E\Big[{-}\Big(\delta\beta-\partial_\theta
s(\theta^r,\alpha^r)\Big)\frac{\partial \theta^r}{\partial u^s}\Big]\d
r+\int_t^{t'}\Big(\int_s^r R_\eta(r,r')R_\theta(r',s)\d r'\Big)\d r\bigg|\\
&\leq \int_t^{t'} \Big((\delta|\beta|+C_0) \bar{R}_\theta(r,s)\d r
+\int_s^r \Phi_{R_\eta}(r-r')\bar R_\theta(r',s)\d r'\Big)\d r
\leq |t'-t| \cdot \sup_{r \in [0,T]} \Phi_{R_\theta}'(r),
\end{align*}
verifying (\ref{eq:Rtheta_cont}). In particular, this shows
continuity of $\tau \mapsto R_\theta(s+\tau,s)$ uniformly over $s \in [0,T]$ and $\tau \in [0,T-s]$. For
continuity in $s$, observe that
\begin{align*}
&\frac{\d}{\d \tau}\Big|\frac{\partial \theta^{s+\tau}}{\partial u^s}
-\frac{\partial \theta^{s'+\tau}}{\partial u^{s'}}\Big|\\
&\leq (\delta|\beta|+C_0)\Big|\frac{\partial \theta^{s+\tau}}{\partial u^s}
-\frac{\partial \theta^{s'+\tau}}{\partial u^{s'}}\Big|
+\int_0^\tau \Big|R_\eta(s+\tau,s+r)\frac{\partial \theta^{s+r}}{\partial
u^s}-R_\eta(s'+\tau,s'+r)\frac{\partial \theta^{s'+\tau}}{\partial
u^{s'}}\Big|\d r.
\end{align*}
We may again divide the range of integration of the second term as $[0,\tau]=A \cup B$ where $s+r,s'+r$ belong to the same maximal interval defining $\cS_\eta$ for $r \in A$, and to two different maximal intervals for $r \in B$. Then
taking expectations on both sides above and applying boundedness of $\bar R_\theta,R_\eta$, 
continuity of $R_\eta$ to bound the integral over $r \in A$, and $|B| \leq C|s-s'|$ to bound the integral over $r \in B$, this shows
\begin{align*}
&\frac{\d}{\d \tau}\E\Big|\frac{\partial \theta^{s+\tau}}{\partial u^s}
-\frac{\partial \theta^{s'+\tau}}{\partial u^{s'}}\Big|
\leq C\bigg(\E\Big|\frac{\partial \theta^{s+\tau}}{\partial u^s}
-\frac{\partial \theta^{s'+\tau}}{\partial u^{s'}}\Big|
+\int_0^\tau 
\E\Big|\frac{\partial \theta^{s+r}}{\partial
u^s}-\frac{\partial \theta^{s'+\tau}}{\partial u^{s'}}\Big|\d r\bigg)
+o_{|s-s'|}(1)
\end{align*}
where $o_{|s-s'|}(1)$ converges to 0 uniformly in $\tau$ as $|s-s'| \to 0$.
Then, since $\E|\frac{\partial \theta^s}{\partial u^s}
-\frac{\partial \theta^{s'}}{\partial u^{s'}}|=0$, a Gronwall argument implies
$\E\Big|\frac{\partial \theta^{s+\tau}}{\partial u^s}
-\frac{\partial \theta^{s'+\tau}}{\partial u^{s'}}\Big|=o_{|s-s'|}(1)$,
so also $s \mapsto R_\theta(s+\tau,s)$ is continuous uniformly over $s \in [0,T]$ and $\tau \in [0,T-s]$.
Thus $(s,t) \mapsto R_\theta(t,s)$ is uniformly continuous over all
$0 \leq s \leq t \leq T$.\\

\noindent(\textbf{Condition for $\tilde \alpha$}) By definition, we have
\[\frac{\d}{\d t}\tilde \alpha^t=\cG(\tilde \alpha^t,\sP(\theta^t))\]
with $\tilde\alpha^0=\alpha^0$. The condition (\ref{eq:G_smooth1}) and
boundedness of $C_\theta$ shown above imply that
$\alpha \mapsto \cG(\alpha,\sP(\theta^t))$ is Lipschitz uniformly
over $t \in [0,T]$, so there exists a unique 
solution $\{\tilde\alpha^t\}_{t \in [0,T]}$ of this equation, which is uniformly continuous on
$[0,T]$. Letting $C_0>0$ be the constant defining
(\ref{eq:aux_alpha}--\ref{eq:aux_Reta}) which upper bounds
(\ref{eq:G_smooth1}) of
Assumption \ref{assump:gradient}, and applying the above bound
$\E(\theta^t)^2=C_\theta(t,t) \leq \Phi_{C_\theta}(t)$, this solution satisfies
\[\frac{\d}{\d t}\|\tilde \alpha^t\|^2
\leq 2\|\tilde\alpha^t\| \cdot \|\cG(\tilde \alpha^t,\sP(\theta^t))\|
\leq 2C_0(1+\sqrt{\Phi_{C_\theta}(t)}+\|\tilde \alpha^t\|)\|\tilde\alpha^t\|
\leq 4.1C_0(1+\Phi_{C_\theta}(t))+3C_0\|\tilde\alpha^t\|^2\]
(where we again relax a constant 4 to 4.1).
Recalling the equation for $\Phi_\alpha$ in (\ref{eq:aux_alpha}),
\[\frac{\d}{\d t}
\Phi_{\alpha}(t)=4.1C_0(1+\Phi_{C_\theta}(t))+3C_0\Phi_\alpha(t),\]
since $C_0>0$ and $\|\tilde \alpha^0\|^2=\Phi_\alpha(0)$,
this shows $\|\tilde \alpha^t\|^2 \leq \Phi_\alpha(t)$.
\end{proof}

Next we equip the spaces $\cS_\eta$ and $\cS_\theta$ with metrics. Fixing a
large constant $\lambda>0$, define
\begin{equation}\label{eq:metrics}
\begin{aligned}
d(\alpha_1,\alpha_2)&=\sup_{t \in [0,T]} e^{-\lambda
t}\|\alpha_1^t-\alpha_2^t\|\\
d(C_\theta^1,C_\theta^2)
&=\inf_{(w_1^\ast,\{w_1^t\})\sim C_\theta^1, (
w_2^\ast,\{w_2^t\}) \sim C_\theta^2} \Big[\sqrt{\E(w_1^*-w_2^*)^2}+\sup_{t\in
[0,T]} e^{-\lambda t} \sqrt{\E (w_1^t - w_2^t)^2}\Big]\\
d(C_\eta^1,C_\eta^2)
&=\inf_{\{u_1^t\}\sim C_\eta^1, \{u_2^t\} \sim C_\eta^2}
\sup_{t\in [0,T]} e^{-\lambda t} \sqrt{\E (u_1^t - u_2^t)^2}\\
d(R_\theta^1,R_\theta^2)&=\sup_{0\leq s \leq t \leq T} e^{-\lambda t}
\big|R_\theta^1(t,s) - R_\theta^2(t,s)\big|\\
d(R_\eta^1,R_\eta^2)&=\sup_{0\leq s \leq t \leq T} e^{-\lambda t}
\big|R_\eta^1(t,s) - R_\eta^2(t,s)\big|.\\
\end{aligned}
\end{equation}
In the definitions of $d(C_\theta^1,C_\theta^2)$ and
$d(C_\eta^1,C_\eta^2)$ above,
the infima are taken over all couplings of mean-zero
Gaussian processes with covariances
$(C_\theta^1,C_\theta^2)$ and $(C_\eta^1,C_\eta^2)$.
Writing $X^i=(R_\eta^i,C_\eta^i,\alpha_i) \in \cS_\eta$ and
$Y^i=(R_\theta^i,C_\theta^i,\tilde\alpha_i) \in \cS_\theta$
for $i=1,2$, let
\begin{align}
d(X^1,X^2)&=d(R_\eta^1,R_\eta^2)+d(C_\eta^1,C_\eta^2)
+d(\alpha_1,\alpha_2),\label{eq:dist_Seta}\\
d(Y^1,Y^2)&=d(R_\theta^1,R_\theta^2)+d(C_\theta^1,C_\theta^2)
+d(\tilde\alpha_1,\tilde\alpha_2).\label{eq:dist_Stheta}
\end{align}

\begin{lemma}[Modulus of $\cT_{\eta\rightarrow\theta}$]\label{lem:modulus_eta_2_theta}
Let $X^i=(R_\eta^i,C_\eta^i,\alpha_i) \in \cS_\eta$ and
$Y^i=\cT_{\eta\rightarrow\theta}(X^i)=(R_\theta^i,C_\theta^i,\tilde\alpha_i)
\in \cS_\theta$ for $i=1,2$. Then for any $\eps>0$, there exists a constant
$\lambda=\lambda(\eps)>0$ sufficiently large defining the metrics
(\ref{eq:metrics}) such that
\begin{align*}
d(Y^1,Y^2) \leq \eps \cdot d(X^1,X^2).
\end{align*}
\end{lemma}
\begin{proof}
We write $C,C'>0$ for constants that may depend on $T$, but not on $\lambda$,
and changing from instance to instance.\\

\noindent \textbf{Bound of $d(C_\theta^1,C_\theta^2)$.}
Let $\{u_1^t\}_{t\in[0,T]}$ and $\{u_2^t\}_{t\in[0,T]}$ be an optimal coupling
in the definition of $d(C^1_\eta,C^2_\eta)$, i.e., 
\begin{align}\label{eq:opt_u_couple}
\sup_{t\in[0,T]} e^{-\lambda t}\sqrt{\E[(u_1^t - u_2^t)^2]} = d(C_\eta^1, C_\eta^2).
\end{align}
Let $\{\theta_i^t\}$ be the solution to
(\ref{def:dmft_langevin_cont_theta}) driven by $\{u_i^t,\alpha_i^t,R_\eta^i\}$
for $i=1,2$, with a common Brownian motion $\{b^t\}$ and initialization
$\theta^0$, i.e.
\begin{align}\label{eq:theta_modulus_def}
\theta_i^t=\theta^0 + \int_0^t \Big({-}\delta\beta(\theta^s_i - \theta^\ast) +
s(\theta^s_i,\alpha^s_i)+\int_0^s R^i_\eta(s,s')(\theta_i^{s'}-\theta^*)\d s'
+u^s_i\Big)\d s + \sqrt{2}b^t.
\end{align}
By definition, we have $\E[\theta_1^t\theta_1^s] = \E[\theta_2^t\theta_2^s] = C_\theta(t,s)$. Moreover,
\[\E(\theta_1^t-\theta_2^t)^2 \leq 5[(I)+(II)+(III)+(IV)+(V)]\]
where we set
\begin{align*}
(I)&=\E \Big(\int_0^t \delta\beta|\theta^s_1 - \theta^s_2|\d s\Big)^2\\
(II)&=\E \Big(\int_0^t |s(\theta^s_1,\alpha^s_1)-s(\theta^s_2,\alpha^s_2)|\d s
\Big)^2\\
(III)&=\E\Big(\int_0^t |\theta^{s'}_1-\theta^{s'}_2|\Big(\int_{s'}^t
|R^1_\eta(s,s')|\d s\Big)\d s'\Big)^2\\
(IV)&=\E\Big(\int_0^t |\theta_2^{s'}-\theta^*|
\Big(\int_{s'}^t |R^1_\eta(s,s')-R_\eta^2(s,s')|\d s\Big)
\d s'\Big)^2\\
(V)&=\E \Big(\int_0^t |u^s_1 - u^s_2| \d s\Big)^2.
\end{align*}
Term $(I)$ satisfies
\begin{align*}
(I) \leq C\int_0^t \E(\theta^s_1 - \theta^s_2)^2\d s &= C\int_0^t e^{2\lambda s}e^{-2\lambda s}\E(\theta^s_1 - \theta^s_2)^2\d s\\
&\leq C\sup_{t\in[0,T]} e^{-2\lambda t}\E(\theta^t_1 - \theta^t_2)^2 \int_0^t e^{2\lambda s}\d s
\leq \frac{C'}{\lambda}e^{2\lambda t} \sup_{t\in[0,T]} e^{-2\lambda
t}\E(\theta^t_1 - \theta^t_2)^2.
\end{align*}
To bound $(II)$, applying the Lipschitz properties of $s(\cdot)$ in
Assumption \ref{assump:prior} and a similar argument,
\begin{align*}
(II) \leq C\int_0^t \Big(\E(\theta_1^s-\theta_2^s)^2
+\|\alpha_1^s-\alpha_2^s\|^2\Big)\d s
&\leq \frac{C'}{\lambda}e^{2\lambda t}
\sup_{t \in [0,T]} e^{-2\lambda t}\Big(\E(\theta_1^t-\theta_2^t)^2
+\|\alpha_1^t-\alpha_2^t\|^2\Big)\\
&\leq \frac{C'}{\lambda}e^{2\lambda t}
\Big(\sup_{t \in [0,T]} e^{-2\lambda t}\E(\theta_1^t-\theta_2^t)^2
+d(\alpha_1,\alpha_2)^2\Big).
\end{align*}
For $(III)$, using the condition $|R^1_\eta(t,s)| \leq \Phi_{R_\eta}(t-s)
\leq C$, we have
\begin{align*}
(III) \leq C\int_0^t \E(\theta_1^s-\theta_2^s)^2\d s
\leq \frac{C'}{\lambda}e^{2\lambda t}\sup_{t\in[0,T]} e^{-2\lambda t}\E(\theta^t_1 - \theta^t_2)^2.
\end{align*}
For $(IV)$, using $\E(\theta^s_2-\theta^*)^2 \leq 2C_\theta(s,s)+2\tau_*^2
\leq 2\Phi_{C_\theta}(s)+2\tau_*^2 \leq C$, we have
\begin{align*}
(IV) \leq 
C\int_0^t \int_{s'}^t \big(R_\eta^1(s,s')-R_\eta^2(s,s')\big)^2\d s\,\d s'
&\leq C\int_0^t \int_{s'}^t e^{2\lambda s} \d s\,\d s'
\cdot \sup_{0 \leq s \leq t \leq T}
e^{-2\lambda t}\big(R_\eta^1(t,s)-R_\eta^2(t,s)\big)^2\\
&\leq \frac{C'}{\lambda}e^{2\lambda t} d(R_\eta^1,R_\eta^2)^2.
\end{align*}
Lastly for $(V)$, using (\ref{eq:opt_u_couple}), we have
\begin{align*}
(V) \leq C\int_0^t \E(u^s_1 - u^s_2)^2 \d s \leq \frac{C'}{\lambda}e^{2\lambda
t} \cdot d(C_\eta^1, C_\eta^2)^2.
\end{align*}
Combining these bounds, for a constant $C>0$ independent of $\lambda$,
\[\sup_{t\in[0,T]} e^{-2\lambda t}\E(\theta^t_1 - \theta^t_2)^2
\leq \frac{C}{\lambda}\Big(
\sup_{t\in[0,T]} e^{-2\lambda t}\E(\theta^t_1 - \theta^t_2)^2
+d(X^1,X^2)^2\Big).\]
Thus for any $\eps>0$, choosing $\lambda=\lambda(\eps)$ large enough and
rearranging gives
\[\sup_{t\in[0,T]} e^{-2\lambda t}\E(\theta^t_1 - \theta^t_2)^2
\leq \eps^2 d(X^1,X^2)^2.\]
Finally, let $(w^*,\{w_1^t\},\{w_2^t\})$
be a centered Gaussian process with second moments matching
$(\theta^*,\{\theta_1^t\},\{\theta_2^t\})$. Then
$(w^*,\{w_1^t\})$ and $(w^*,\{w_2^t\})$ realizes a coupling defining the metric
$d(C_\theta^1,C_\theta^2)$ in (\ref{eq:metrics}), so
\begin{equation}\label{eq:dCthetabound}
d(C_\theta^1,C_\theta^2)
\leq
\sup_{t\in[0,T]} e^{-\lambda t}\sqrt{\E(w^t_1-w^t_2)^2}=\sup_{t\in[0,T]} e^{-\lambda t}\sqrt{\E(\theta^t_1 - \theta^t_2)^2} \leq \eps \cdot d(X^1,X^2).
\end{equation}

\noindent \textbf{Bound of $d(R_\theta^1, R_\theta^2)$.} Defining
the processes $\frac{\partial \theta_i^t}{\partial u^s}$ for $i=1,2$ from the
above coupling of $\{\theta_1^t\}$ and $\{\theta_2^t\}$, 
by definition we have
\begin{align*}
\frac{\partial \theta_i^t}{\partial u^s} = 1 - \int_s^t \Big(\delta\beta-
\partial_\theta s(\theta_i^{s'},\alpha^{s'}_i)\Big)\frac{\partial
\theta_i^{s'}}{\partial u^s}\d s' + \int_s^t \Big(\int_{s}^{s'} R^i_\eta(s',s'')\frac{\partial \theta_i^{s''}}{\partial u^s}\d s''\Big)\d s',
\end{align*}
Then
\[\E\Big|\frac{\partial \theta_1^t}{\partial u^s} - \frac{\partial
\theta_2^t}{\partial u^s}\Big|
\leq 4[(I)+(II)+(III)+(IV)]\]
where
\begin{align*}
(I) &= \int_s^t \E\Big[\Big|\partial_\theta
s(\theta_1^{s'},\alpha^{s'}_1)-\partial_\theta s(\theta_2^{s'},\alpha^{s'}_2)\Big|\Big|\frac{\partial \theta_1^{s'}}{\partial u^s}\Big|\Big]\d s',\\
(II) &= \int_s^t \E\Big[\Big(\delta|\beta|+|\partial_\theta
s(\theta_2^{s'},\alpha^{s'}_2)|\Big)\Big|\frac{\partial \theta_1^{s'}}{\partial u^s} - \frac{\partial \theta_2^{s'}}{\partial u^s}\Big|\Big]\d s',\\
(III) &= \int_s^t \int_s^{s'} \E\Big[\Big|R^1_\eta(s',s'') - R^2_\eta(s',s'')\Big|\Big|\frac{\partial \theta_1^{s''}}{\partial u^s}\Big|\Big]\d s''\d s',\\
(IV) &= \int_s^t \int_s^{s'} \E\Big[|R^2_\eta(s',s'')|\Big|\frac{\partial \theta_1^{s''}}{\partial u^s} - \frac{\partial \theta_2^{s''}}{\partial u^s}\Big|\Big]\d s''\d s'.
\end{align*}

For $(I)$, note that (\ref{eq:theta_response_dt}) implies
$|\frac{\partial \theta_1^{s'}}{\partial u^s}|\leq C$ for a constant $C>0$
with probability 1. Then,
using the Lipschitz continuity of $\partial_\theta s(\cdot)$ in
Assumption \ref{assump:prior}, we have
\begin{align*}
(I) \leq C\int_s^t \Big(\E\big|\theta_1^{s'}-\theta_2^{s'}\big|
+\|\alpha_1^{s'}-\alpha_2^{s'}\|\Big)\d s'
&\leq C\int_s^t e^{\lambda s'} \d s'
\sup_{s' \in [0,T]} e^{-\lambda s'}
\Big(\E\big|\theta_1^{s'}-\theta_2^{s'}\big|
+\|\alpha_1^{s'}-\alpha_2^{s'}\|\Big)\\
&\leq \frac{C'}{\lambda} e^{\lambda t} \sup_{s' \in [0,T]} e^{-\lambda s'}
\Big(\sqrt{\E\big(\theta_1^{s'}-\theta_2^{s'}\big)^2}
+\|\alpha_1^{s'}-\alpha_2^{s'}\|\Big)\\
&\leq \frac{C'}{\lambda} e^{\lambda t} \Big(\eps \cdot d(X^1,X^2)
+d(\alpha_1,\alpha_2)\Big),
\end{align*}
the last step using (\ref{eq:dCthetabound}) already shown. For $(II)$, applying
the boundedness of $\partial_\theta s(\cdot)$ in
Assumption \ref{assump:prior}, we have 
\begin{align*}
(II) \leq C \int_s^t e^{\lambda s'} e^{-\lambda s'}\E\Big[\Big|\frac{\partial
\theta_1^{s'}}{\partial u^s} - \frac{\partial \theta_2^{s'}}{\partial
u^s}\Big|\Big]\d s' &\leq \frac{C'}{\lambda}e^{\lambda t} \cdot
\sup_{0\leq s\leq t\leq T}e^{-\lambda t}\E\Big|\frac{\partial \theta_1^{t}}{\partial u^s} - \frac{\partial \theta_2^{t}}{\partial u^s}\Big|.
\end{align*}
For $(III)$, applying again (\ref{eq:theta_response_dt}) to bound
$|\frac{\partial \theta_1^{s'}}{\partial u^s}|\leq C$,
\[(III) \leq \frac{C}{\lambda}e^{\lambda t} \cdot d(R_\eta^1, R_\eta^2).\]
For $(IV)$, applying $|R^2_\eta(t,s)| \leq \Phi_{R_\eta}(t-s) \leq C$,
\begin{align*}
(IV) \leq \frac{C}{\lambda}e^{\lambda t} \sup_{0\leq s\leq t\leq T}e^{-\lambda
t}\E\Big|\frac{\partial \theta_1^{t}}{\partial u^s} - \frac{\partial
\theta_2^{t}}{\partial u^s}\Big|.
\end{align*}
Combining these bounds,
\[\sup_{0 \leq s \leq t \leq T} e^{-\lambda
t}\E\Big|\frac{\partial \theta_1^{t}}{\partial u^s} - \frac{\partial
\theta_2^{t}}{\partial u^s}\Big|
\leq \frac{C}{\lambda}\Big(
\sup_{0 \leq s \leq t \leq T} e^{-\lambda
t}\E\Big|\frac{\partial \theta_1^{t}}{\partial u^s} - \frac{\partial
\theta_2^{t}}{\partial u^s}\Big|+d(X^1,X^2)\Big),\]
so rearranging and choosing $\lambda=\lambda(\eps)$ large enough gives
\begin{equation}\label{eq:dRthetabound}
d(R^1_\theta,R_\theta^2) \leq \sup_{0\leq s\leq t\leq T} e^{-\lambda t}
\E\Big|\frac{\partial \theta_1^t}{\partial u^s}-\frac{\partial
\theta_2^t}{\partial u^s}\Big| \leq \eps \cdot d(X^1,X^2).
\end{equation}

\noindent \textbf{Bound of $d(\tilde \alpha_1,\tilde \alpha_2)$.} By definition,
\[\tilde \alpha_i^t=\alpha^0+\int_0^t \cG(\tilde \alpha_i^s,\sP(\theta_i^s))\d
s\]
for $i=1,2$. Letting $\{\theta_1^t\}$ and $\{\theta_2^t\}$ be coupled as above and applying Assumption \ref{assump:gradient},
\[\|\tilde \alpha_1^t-\tilde \alpha_2^t\|
\leq C\int_0^t \Big(\|\tilde \alpha_1^s-\tilde \alpha_2^s\|
+W_2(\sP(\theta_1^s),\sP(\theta_2^s))\Big)\d s
\leq C\int_0^t \Big(\|\tilde \alpha_1^s-\tilde \alpha_2^s\|
+\sqrt{\E(\theta_1^s-\theta_2^s)^2}\Big)\d s.\]
Then
\[\|\tilde \alpha_1^t-\tilde \alpha_2^t\|
\leq C\int_0^t e^{\lambda s}\d s
\sup_{s \in [0,T]} e^{-\lambda s}
\Big(\|\tilde \alpha_1^s-\tilde \alpha_2^s\|
+\sqrt{\E(\theta_1^s-\theta_2^s)^2}\Big)
\leq \frac{C'}{\lambda}e^{\lambda t}\Big(d(\tilde \alpha_1,\tilde
\alpha_2)+\eps \cdot
d(X^1,X^2)\Big).\]
Choosing $\lambda=\lambda(\eps)$ large enough and rearranging shows
\begin{equation}\label{eq:dtildealphabound}
d(\tilde \alpha_1,\tilde \alpha_2)
=\sup_{t \in [0,T]} e^{-\lambda t} \|\tilde \alpha_1^t-\tilde \alpha_2^t\|
\leq \eps \cdot d(X^1,X^2).
\end{equation}
The lemma follows from (\ref{eq:dCthetabound}), (\ref{eq:dRthetabound}),
and (\ref{eq:dtildealphabound}).
\end{proof}

\begin{lemma}[Modulus of $\cT_{\theta\rightarrow\eta}$]\label{lem:modulus_theta_2_eta}
Let $Y^i=(R_\theta^i,C_\theta^i,\tilde\alpha_i) \in \cS_\theta$ and
$X^i=\cT_{\theta\rightarrow\eta}(Y^i)=(C_\eta^i,R_\eta^i,\alpha_i) \in \cS_\eta$
for $i=1,2$. Then there exists a constant $C>0$ such that for any sufficiently
large $\lambda>0$ defining the metrics (\ref{eq:metrics}),
\begin{align*}
d(X^1,X^2) \leq C\cdot d(Y^1,Y^2).
\end{align*}
\end{lemma}
\begin{proof}
The proof is similar to that of Lemma \ref{lem:modulus_eta_2_theta} so we will
omit some details. Again let $C,C',C''>0$ denote constants depending on $T$ but
not on $\lambda$.\\

\noindent \textbf{Bound of $d(C_\eta^1, C_\eta^2)$.} 
Let $(w_1^*,\{w^t_1\})$ and $(w_2^*,\{w^t_2\})$ be an optimal coupling for
which
\begin{align*}
\sqrt{\E[(w^\ast_1 - w^\ast_2)^2]} + \sup_{t\in [0,T]} e^{-\lambda t}\sqrt{\E[(w^t_1 - w^t_2)^2]} = d(C^1_\theta,C^2_\theta).
\end{align*}
For $i=1,2$, let
\begin{align*}
\eta^t_i&=-\beta\int_0^{t} R^i_\theta(t,s)\big(\eta_i^s+w^\ast_i-\eps\big)\d s-w^t_i
\end{align*}
be the corresponding coupled solutions to (\ref{def:dmft_langevin_cont_eta}). 
We write $\xi^t_i=\eta^t_i+w^\ast_i-\eps$, so that
\begin{align*}
\xi^t_i = -\beta\int_0^t R^i_\theta(t,s)\xi_i^s \d s - w^t_i + w^\ast_i - \eps
\end{align*}
and $C_\eta^i(t,s)=\delta\beta^2\E[\xi_i^t\xi_i^s]$. Then
\begin{align*}
\E(\xi^t_1 - \xi^t_2)^2 &\leq C\Big[\int_0^t
(R^1_\theta(t,s)-R^2_\theta(t,s))^2\E(\xi_1^s)^2 \d s+
\int_0^t R^2_\theta(t,s)^2\E(\xi_1^s - \xi^s_2)^2\d s
+\E(w_1^t-w_1^*-w_2^t+w_2^*)^2\Big]\\
&\leq C'\Big[\int_0^t e^{2\lambda s} \cdot e^{-2\lambda s}
\Big((R^1_\theta(t,s)-R^2_\theta(t,s))^2
+\E(\xi_1^s - \xi^s_2)^2\Big)\d s+\E(w_1^t-w_1^*-w_2^t+w_2^*)^2\Big]\\
&\leq \frac{C''}{\lambda}e^{2\lambda t}
\Big(\sup_{s \in [0,T]} e^{-2\lambda s} \E(\xi_1^s - \xi^s_2)^2
+d(R_\theta^1,R_\theta^2)^2\Big)+C''e^{2\lambda t}d(C_\theta^1,C_\theta^2)^2.
\end{align*}
Choosing $\lambda>2C''$ and rearranging yields, for a constant $C>0$,
\begin{align*}
\sup_{t\in[0,T]} e^{-2\lambda t}\E(\xi^t_1 - \xi^t_2)^2 \leq C \cdot
d(Y^1,Y^2)^2.
\end{align*}
Then letting $(\{u^t_1\},\{u^t_2\})$ be a centered Gaussian process with
second moments $\E[u^t_1u^s_2]=\delta\beta^2\E[\xi^t_1\xi^s_2]$,
this realizes a coupling defining $d(C_\eta^1,C_\eta^2)$, so
\begin{align*}
d(C_\eta^1,C_\eta^2)
\leq \sup_{t\in[0,T]} e^{-\lambda
t}\sqrt{\E[(u^t_1-u^t_2)^2]}=\sqrt{\delta\beta^2} \cdot \sup_{t\in[0,T]}
e^{-\lambda t}\sqrt{\E(\xi^t_1 - \xi^t_2)^2} \leq C'\cdot d(Y^1,Y^2).
\end{align*}

\noindent \textbf{Bound of $d(R_\eta^1,R_\eta^2)$.}
Defining the (deterministic) process $\frac{\partial \eta_i^t}{\partial w^s}$ 
driven by $R_\theta^i$ for $i=1,2$, we have
\begin{align*}
R_\eta^i(t,s) = -\beta\int_s^t R_\theta^i(t,s')R_\eta^i(s',s)\d s'+\delta\beta^2R_\theta^i(t,s),
\end{align*}
hence
\begin{align*}
|R^1_\eta(t,s) - R^2_\eta(t,s)| &\leq |\beta|\int_s^t |R_\theta^1(t,s') -
R_\theta^2(t,s')||R_\eta^1(s',s)|\d s' + |\beta|\int_s^t |R_\theta^2(t,s')||R_\eta^1(s',s) - R_\eta^2(s',s)|\d s'\\
&\hspace{1in}+ \delta\beta^2|R^1_\theta(t,s) - R^2_\theta(t,s)|\\
&\leq C\int_s^t e^{\lambda s'}e^{-\lambda s'}|R_\eta^1(s',s) -
R_\eta^2(s',s)|\d s'+Ce^{\lambda t} d(R_\theta^1,R_\theta^2)\\
&\leq \frac{C'}{\lambda}e^{\lambda t}
\Big(\sup_{0 \leq s \leq t \leq T} e^{-\lambda
t}|R_\eta^1(t,s)-R_\eta^2(t,s)|\Big)+Ce^{\lambda t} d(R_\theta^1,R_\theta^2).
\end{align*}
Choosing $\lambda>2C'$ and rearranging yields
\begin{align*}
d(R_\eta^1,R_\eta^2)
=\sup_{0 \leq s \leq t \leq T} e^{-\lambda t}|R_\eta^1(t,s)-R_\eta^2(t,s)|
\leq C \cdot d(R_\theta^1,R_\theta^2)
\leq C \cdot d(Y^1,Y^2).
\end{align*}

We note that $\alpha_i=\tilde \alpha_i$ for $i=1,2$ by definition, so also
$d(\alpha_1,\alpha_2)=d(\tilde \alpha_1,\tilde \alpha_2) \leq d(Y^1,Y^2)$.
Combining these bounds shows the lemma.
\end{proof}

\begin{proof}[Proof of Theorem \ref{thm:dmftsolexists}(b)]
Combining Lemmas \ref{lem:modulus_eta_2_theta} and
\ref{lem:modulus_theta_2_eta}, for sufficiently large $\lambda>0$,
the composition map $\cT_{\eta \to \eta}$ is a contraction on
$\cS_\eta^\text{cont}$ with respect to the metric $d(X^1,X^2)$
and similarly $\cT_{\theta \to \theta}$ is a contraction on
$\cS_\theta^\text{cont}$. We note that for any sequence
$\{C_\theta^k\}$ of correlation functions in $\cS^\text{cont}$, as $k \to
\infty$,
\[d(C_\theta^k,C_\theta) \to 0 \text{ implies }
\sup_{s,t \in [0,T]} |C_\theta^k(s,t)-C_\theta(s,t)| \to 0\]
by definition of the metric and Cauchy-Schwarz, while
\[\sup_{s,t \in [0,T]} |C_\theta^k(s,t)-C_\theta(s,t)| \to 0
\text{ implies } d(C_\theta^k,C_\theta) \to 0\]
by e.g.\ the construction of a coupling in \cite[Lemma D.1]{paper2}.
The same holds for $C_\eta$, so
each metric in (\ref{eq:metrics}) induces a topology equivalent to
that of uniform convergence over
continuous functions on the appropriate space
$[0,T]$, $\{s,t:0 \leq s \leq t \leq T\}$, or
$\{*\} \cup \{s,t:0 \leq s \leq t \leq T\}$.
Furthermore, each condition
defining $\cS_\eta^\text{cont},\cS_\theta^\text{cont}$ is closed with respect
to this topology. Thus $d(X_\eta^1,X_\eta^2)$ and
$d(Y_\theta^1,Y_\theta^2)$ are complete metrics on
$\cS_\eta^\text{cont},\cS_\theta^\text{cont}$,
so the Banach fixed-point theorem guarantees $\cT_{\eta
\to \eta}$ and $\cT_{\theta \to \theta}$ have
unique fixed points $X=(R_\eta,C_\eta,\alpha) \in
\cS_\eta^\text{cont}$ and $Y=(R_\theta,C_\theta,\alpha) \in
\cS_\theta^\text{cont}$, for which also $\cT_{\eta \to \theta}(X)=Y$.
These fixed points remain unique in $\cS_\eta$ and $\cS_\theta$, because
Lemmas \ref{lem:Trange_theta_2_eta} and \ref{lem:Trange_eta_2_theta} imply
that the images of $\cT_{\eta \to \eta},\cT_{\theta \to \theta}$ on
$\cS_\eta,\cS_\theta$ are contained in
$\cS_\eta^\text{cont},\cS_\theta^\text{cont}$.
Then the tuple $(\alpha,C_\theta,C_\eta,R_\theta,R_\eta) \in \cS^\text{cont}$
is the unique fixed point in $\cS$ solving the dynamical fixed point
equations (\ref{def:dmft_langevin_alpha}--\ref{def:CRfixedpoint}).
\end{proof}

\section{The dynamical mean-field approximation}\label{sec:dmft_approx}

In this section we prove Theorem \ref{thm:dmft_approx}. We assume throughout Assumptions \ref{assump:model}, \ref{assump:prior}, and \ref{assump:gradient}.
The proof consists of three steps:
\begin{itemize}
\item (\textbf{Step 1}) We prove in Section \ref{subsec:discretedmft}
a discrete DMFT limit for a discretized version of the dynamics.
\item (\textbf{Step 2}) We show in Section \ref{subsec:discretizedmft} that, as the discretization step size goes to
zero, the discrete DMFT equations converge in an appropriate
sense to (\ref{def:dmft_langevin_cont_theta}--\ref{def:CRfixedpoint}).
\item (\textbf{Step 3}) We show in Section \ref{subsec:discretizelangevin}
that, as the discretization step size goes to
zero, the discretized dynamics converges in an appropriate sense to
(\ref{eq:langevin_sde}--\ref{eq:gflow}).
\end{itemize}
This argument follows closely the approach of \cite{celentano2021high},
although we will use in Steps 2 and 3 a different and somewhat simpler
piecewise-constant embedding of the discretized DMFT process and discretized
Langevin dynamics into continuous time.

\subsection{Step 1: DMFT approximation of discrete
dynamics}\label{subsec:discretedmft}

Fix a step size $\gamma>0$. We first define a discretized version of the process
(\ref{eq:langevin_sde}--\ref{eq:gflow}), which we denote by
$\{\btheta_\gamma^t\}$ and $\{\widehat{\alpha}^t_\gamma\}$
for $t \in \Z_+=\{0,1,2,\ldots\}$:
\begin{align}
\btheta^{t+1}_\gamma&=\btheta^t_\gamma + \gamma\Big({-}\beta\X^\top
(\X\btheta^t_\gamma-\y)+s(\btheta^t_\gamma, \widehat{\alpha}^t_\gamma)\Big) +
\sqrt{2}(\b_\gamma^{t+1} - \b_\gamma^t)\label{eq:langevin_discrete_1}\\
\widehat{\alpha}^{t+1}_\gamma &= \widehat{\alpha}^t_\gamma + \gamma \cdot
\cG\Big(\widehat{\alpha}^t_\gamma, \frac{1}{d}\sum_{j=1}^d \delta_{\theta^t_{\gamma,j}}\Big)\label{eq:langevin_discrete_2}
\end{align}
with initialization
$(\btheta^0_\gamma,\widehat\alpha^0_\gamma)=(\btheta^0,\widehat{\alpha}^0)$, where
$\{\b_\gamma^t\}$ is a discrete
Gaussian process with $\b_\gamma^0=0$ and independent increments
$\b_\gamma^{t+1}-\b_\gamma^t \sim \N(0,\gamma\I)$.
Here and throughout the sequel, we write as shorthand
$s(\btheta,\widehat\alpha)=(s(\theta_j,\widehat\alpha))_{j=1}^d$. We set
\[\bbeta_\gamma^t=\X\btheta_\gamma^t,
\qquad \bbeta^*=\X\btheta^*.\]

We correspondingly define a discretized version of the DMFT system
(\ref{def:dmft_langevin_cont_theta}--\ref{def:CRfixedpoint}): Given
discrete-time correlation and response matrices
$\{C_\eta^\gamma(s,r)\}_{r \leq s \leq t}$, $\{R_\eta^\gamma(s,r)\}_{r<s \leq
t}$ and a deterministic process $\{\alpha_\gamma^s\}_{s \leq t}$ up to time $t$,
define (in the probability space of 
$(\theta^*,\theta^0) \sim \sP(\theta^*,\theta^0)$)
\begin{align}
\theta^{t+1}_\gamma&=\theta^t_\gamma+\gamma\Big({-}\delta\beta(\theta^t_\gamma-\theta^\ast)
+s(\theta^t_\gamma,\alpha^t_\gamma)+\sum_{s=0}^{t-1}
R_\eta^\gamma(t,s)(\theta^s_\gamma - \theta^\ast) +u^t_\gamma\Big) +
\sqrt{2}(b_\gamma^{t+1}-b_\gamma^t)\notag\\
&\hspace{1in} \text{ with } \theta_0^\gamma=\theta^0,\label{eq:dmft_theta_discrete1}\\
\frac{\partial\theta^{t+1}_\gamma}{\partial u^s_\gamma} &=
\begin{cases}
\gamma & \text{ for } s = t,\\
\frac{\partial \theta_\gamma^t}{\partial u_\gamma^s}
+\gamma\Big[\Big({-}\delta\beta + \partial_\theta
s(\theta^t_\gamma,\alpha^t_\gamma)\Big)\frac{\partial
\theta_\gamma^t}{\partial u_\gamma^s}+
\sum_{r=s+1}^{t-1} R^\gamma_\eta(t,r)\frac{\partial
\theta_\gamma^r}{\partial u_\gamma^s}\Big] & \text{ for } s < t.
\end{cases}\label{eq:response_1}
\end{align}
Here, $\{u_\gamma^s\}_{0 \leq s \leq t}$ and
$\{b_\gamma^s\}_{0 \leq s \leq t}$ are mean-zero Gaussian vectors independent
of each other and of $(\theta^*,\theta^0)$, where
$\{u_\gamma^s\}_{0 \leq s \leq t}$ has covariance
\begin{equation}\label{eq:dmft_covu_discrete}
\E[u_\gamma^s u_\gamma^r]=C_\eta^\gamma(s,r),
\end{equation}
and $\{b_\gamma^s\}_{0 \leq s \leq t}$ has independent increments
$b_\gamma^{s+1}-b_\gamma^s \sim \N(0,\gamma)$ with $b_\gamma^0=0$. We note that 
$\frac{\partial \theta_\gamma^{t+1}}{\partial u_\gamma^s}$ is the usual partial
derivative of $\theta_\gamma^{t+1}$ in $u_\gamma^s$,
whose form (\ref{eq:response_1}) is derived from
(\ref{eq:dmft_theta_discrete1}) via the chain rule. These processes then define
$\{C_\theta^\gamma(s,r)\}_{r \leq s \leq t+1}$,
$\{C_\theta^\gamma(s,*)\}_{s \leq t+1}$,
$\{R_\theta^\gamma(s,r)\}_{r<s \leq t+1}$,
and $\{\alpha_\gamma^s\}_{s \leq t+1}$ up to time $t+1$ via
\begin{equation}\label{eq:dmft_CRtheta_discrete}
\begin{gathered}
C^\gamma_\theta(s,r)=\E[\theta^s_\gamma \theta^r_\gamma], \quad
C_\theta^\gamma(s,\ast) = \E[\theta^s_\gamma\theta^\ast], \quad C_\theta^\gamma(*,*)=\E[(\theta^*)^2],\\
R^\gamma_\theta(s,r) = \E\Big[{\frac{\partial
\theta^s_\gamma}{\partial u^r_\gamma}}\Big],
\quad \alpha_\gamma^{t+1}=\alpha_\gamma^t+\gamma \cdot
\cG(\alpha_\gamma^t,\sP(\theta_\gamma^t))
\end{gathered}
\end{equation}
where $\sP(\theta_\gamma^t)$ is the law of $\theta_\gamma^t$.

Conversely, given $\{C_\theta^\gamma(s,r)\}_{r \leq s \leq t}$,
$\{C_\theta^\gamma(s,*)\}_{s \leq t}$, and
$\{R_\theta^\gamma(s,r)\}_{r<s \leq t}$ up to time $t$, define (in the
probability space of $\eps \sim \sP(\eps)$)
\begin{align}
\eta^t_\gamma &= -\beta\sum_{s=0}^{t-1}
R^\gamma_\theta(t,s)(\eta^s_\gamma+w^\ast_\gamma-\eps)-w^t_\gamma,\label{eq:dmft_eta_discrete1}\\
\frac{\partial \eta^t_\gamma}{\partial w^s_\gamma} &=
\beta\Big[{-}\sum_{r=s+1}^{t-1}R^\gamma_\theta(t,r) \frac{\partial
\eta^r_\gamma}{\partial w^s_\gamma} + R^\gamma_\theta(t,s)\Big]
\quad \text{ for } s < t \label{eq:response_2}.
\end{align}
Here, $(w_\gamma^*,\{w_\gamma^s\}_{0 \leq s \leq t})$ is a mean-zero Gaussian
vector with covariance
\begin{equation}\label{eq:dmft_covw_discrete}
\E[w^s_\gamma w^r_\gamma] = C^\gamma_\theta(s,r),\qquad \E[w^s_\gamma
w^{\ast}_\gamma] = C^\gamma_\theta(s,\ast),\qquad
\E[(w_\gamma^*)^2]=C_\theta^\gamma(*,*),
\end{equation}
and again $\frac{\partial \eta^t_\gamma}{\partial w^s_\gamma}$ is the usual
partial derivative computed from the chain rule. These define
$\{C_\eta^\gamma(s,r)\}_{r \leq s \leq t}$,
$\{R_\eta^\gamma(s,r)\}_{r<s \leq t}$ up to time $t$ via
\begin{equation}\label{eq:dmft_CReta_discrete}
C^\gamma_\eta(s,r)=\delta\beta^2\E[(\eta^s_\gamma + w^\ast_\gamma -
\eps)(\eta_\gamma^r + w^\ast_\gamma - \eps)],
\quad R^\gamma_\eta(s,r) = \delta\beta\,\Big(\frac{\partial
\eta^s_\gamma}{\partial w^r_\gamma}\Big),
\end{equation}
where we note that $\frac{\partial \eta^s_\gamma}{\partial w_\gamma^r}$ is
deterministic.
These definitions should be understood in the iterative sense
\begin{equation}\label{eq:iterativeDMFTconstruction}
\begin{gathered}
\{\theta_\gamma^s\}_{s \leq t},\{u^s_\gamma\}_{s<t},
\{\tfrac{\partial \theta_\gamma^s}{\partial u_\gamma^r}\}_{r<s \leq t}
\Rightarrow \{C_\theta^\gamma(s,r),
C_\theta^\gamma(s,*)\}_{r \leq s \leq t},\{R_\theta^\gamma(s,r)\}_{r<s \leq t},
\{\alpha^s\}_{s \leq t} \Rightarrow \\
w_\gamma^*,\{\eta_\gamma^s,w_\gamma^s\}_{s \leq t},
\{\tfrac{\partial \eta_\gamma^s}{\partial w_\gamma^r}\}_{r<s \leq t}
\Rightarrow \{C_\eta^\gamma(s,r)\}_{r \leq s \leq t},
\{R_\eta^\gamma(s,r)\}_{r<s \leq t} \Rightarrow\\
\{\theta_\gamma^s\}_{s \leq t+1},\{u_\gamma^s\}_{s<t+1},
\{\tfrac{\partial \theta_\gamma^s}{\partial \theta_\gamma^r}\}_{r<s \leq t+1}
\Rightarrow \ldots
\end{gathered}
\end{equation}
with initialization $\theta_\gamma^0=\theta^0$.

The goal of this section is to show the following discrete analogue of Theorem
\ref{thm:dmft_approx}.

\begin{lemma}\label{lem:finite_dim_converge}
For any fixed integer $T \geq 0$, almost surely as $n,d\rightarrow\infty$,
\begin{align}
\frac{1}{d}\sum_{j=1}^d \delta_{\big(\theta^\ast_j,\theta^0_{\gamma,j}, \ldots,
\theta^T_{\gamma,j}\big)}  &\overset{W_2}{\to}
\sP\big(\theta^\ast,\theta^0_\gamma,\ldots, \theta^T_\gamma\big)\label{eq:finite_dim_converge_knot_theta}\\
\qquad \frac{1}{n}\sum_{i=1}^n \delta_{\big(\eta^*_i, \eps_i,\eta^0_{\gamma,i}, \ldots,
\eta^0_{\gamma,i}\big)}
&\overset{W_2}{\to} \sP\big({-}w_\gamma^\ast,\eps,\eta^0_\gamma,\ldots,
\eta^T_\gamma\big)\label{eq:finite_dim_converge_knot_eta}\\
(\widehat{\alpha}^0_\gamma,\ldots,\widehat{\alpha}^T_\gamma) &\rightarrow
(\alpha^0_\gamma, \ldots, \alpha^T_\gamma)\label{eq:finite_dim_converge_knot_alpha}.
\end{align}
\end{lemma}

For convenience of the proof,
we define also an auxiliary response function
\begin{align}
R^\gamma_\eta(t,\ast) = \delta\beta\,\Big(\frac{\partial
\eta^t_\gamma}{\partial w^\ast_\gamma}\Big) \quad \text{ where
}\quad \frac{\partial\eta^t_\gamma}{\partial w_\gamma^\ast} &=
-\beta\sum_{s=0}^{t-1} R_\theta^\gamma(t,s)\Big(\frac{\partial
\eta^s_\gamma}{\partial w^\ast_\gamma}+1\Big)\label{eq:response_3},
\end{align}
initialized from $\frac{\partial \eta^0_\gamma}{\partial w_\gamma^\ast} = 0$.
Here $\frac{\partial \eta^t_\gamma}{\partial w_\gamma^*}$ 
is the usual partial derivative of $\eta^t_\gamma$ with respect to
$w_\gamma^\ast$, which is also deterministic.
We have the following basic fact relating the response functions (\ref{eq:response_2}) and (\ref{eq:response_3}).

\begin{lemma}\label{lem:eta_response_id}
For any $t \geq 1$, we have
$\frac{\partial\eta^t_\gamma}{\partial w^\ast_\gamma} =
-\sum_{s=0}^{t-1} \frac{\partial \eta^t_\gamma}{\partial w^s_\gamma}$,
and consequently $R_\eta^\gamma(t,\ast)=-\sum_{s=0}^{t-1}
R_\eta^\gamma(t,s)$.
\end{lemma}
\begin{proof}
Let us shorthand $r_\eta(t,s) = \frac{\partial
\eta^t_\gamma}{\partial w^s_\gamma}$ for $s<t$ and $r_\eta(t,\ast)
=\frac{\partial\eta^t_\gamma}{\partial w^\ast_\gamma}$.
We prove $r_\eta(t,*)={-}\sum_{s=0}^{t-1} r_\eta(t,s)$ by
induction, with the base case $t=1$ verified by the initial conditions
$r_\eta(1,\ast)=-\gamma\beta$ and $r_\eta(1,0)=\gamma\beta$.
Suppose the claim holds for some $t$, then
\begin{align*}
r_\eta(t+1,\ast) &= -\beta\sum_{s=0}^t R_\theta^\gamma(t+1,s)(r_\eta(s,\ast) + 1)\\
&= -\beta\sum_{s=0}^t R_\theta^\gamma(t+1,s)\Big({-}\sum_{r=0}^{s-1}
r_\eta(s,r)+1\Big)\\
&= \beta\Big[\sum_{r=0}^{t-1} \sum_{s=r+1}^t
R_\theta^\gamma(t+1,s)r_\eta(s,r) - \sum_{r=0}^t
R_\theta^\gamma(t+1,r)\Big]\\
&= \beta\sum_{r=0}^{t-1}\Big(\sum_{s=r+1}^t
R_\theta^\gamma(t+1,s)r_\eta(s,r) -
R_\theta^\gamma(t+1,r)\Big) - \beta R_\theta^\gamma(t+1,t)\\
&= {-}\sum_{r=0}^{t-1} r_\eta(t+1,r) - r_\eta(t+1,t) =  {-}\sum_{r=0}^t
r_\eta(t+1,r),
\end{align*}
as desired.
\end{proof}

\begin{proof}[Proof of Lemma \ref{lem:finite_dim_converge}]
\textbf{Step 1: Convergence of auxiliary dynamics.}
Consider the following non-adaptive auxiliary dynamics 
\begin{align}\label{def:langevin_auxiliary}
\tilde{\btheta}_\gamma^{t+1} = \tilde{\btheta}_\gamma^t -
\gamma\bigg(\beta\X^\top (\X\tilde{\btheta}^t_\gamma - \X\btheta^\ast -
\beps) - s(\tilde{\btheta}^t_\gamma,\alpha^t_\gamma)\bigg) +
\sqrt{2}(\b_\gamma^{t+1} - \b_\gamma^t).
\end{align}
This differs from $\{\btheta^t_\gamma\}$ in that
we replace $\{\widehat{\alpha}^t_\gamma\}$ by the deterministic process
$\{\alpha^t_\gamma\}$ of the discrete DMFT system.
Let $\tilde{\bbeta}^t_{\gamma}=\X\tilde{\btheta}^t_{\gamma}$. We
will first show
\begin{align}\label{eq:emp_converge_aux}
\frac{1}{d}\sum_{j=1}^d \delta_{\big(\theta^\ast_j,\tilde{\theta}^0_{\gamma,j}, \ldots,
\tilde{\theta}^T_{\gamma,j}\big)}  \overset{W_2}{\to}
\sP\big(\theta^\ast,\theta^0_\gamma,\ldots,
\theta^T_\gamma\big), \quad \frac{1}{n}\sum_{i=1}^n
\delta_{\big( \eta^*_i, \eps_i,\tilde{\eta}^0_{\gamma,i}, \ldots,
\tilde{\eta}^T_{\gamma,i}\big)} \overset{W_2}{\to}
\sP\big({-}w_\gamma^\ast, \eps,\eta^0_\gamma,\ldots, \eta^T_\gamma\big).
\end{align}

The proof is based on a reduction to an AMP algorithm: Let
$\beps \in \R^n$ be as in the above dynamics, define
\[\V=(\btheta^*,\btheta^0,\b^1-\b^0,\ldots,\b^T-\b^{T-1}) \in
\R^{d\times (T+2)}\]
and let
\[\eps \sim \sP(\eps),
\qquad V=(\theta^*,\theta^0,\rho^1,\ldots,\rho^T) \sim \sP(\theta^*,\theta^0) \otimes \N(0,\gamma I_T).\]
Assumption \ref{assump:model} ensures that $|\eps|$ and $\|v\|_2$ have finite moment
generating functions in a neighborhood of 0, and
for each fixed $p \geq 1$, almost surely as $n,d \to \infty$,
\begin{align}\label{eq:amp_empirical}
\frac{1}{n}\sum_{i=1}^n \delta_{\eps_i} \overset{W_p}{\to} \eps,
\quad
\frac{1}{d}\sum_{j=1}^d \delta_{V_j} \overset{W_p}{\to} V
\end{align}
where $V_j$ is the $j^\text{th}$ row of $\V$.
Fixing some $k \geq 1$, consider the AMP iterations
\begin{equation}\label{eq:amp}
\begin{aligned}
\W^i &= \X g_i(\U^1,\ldots,\U^i;\V)-\sum_{j=0}^{i-1}
f_j(\W^0,\ldots,\W^j;\beps) \zeta_{ij} \in \R^{n\times k},\\
\U^{i+1} &= \X^\top f_i(\W^0,\ldots,\W^i;\beps)-\sum_{j=0}^i
g_j(\U^1,\ldots, \U^j;\V)\xi_{ij} \in \R^{d\times k},
\end{aligned}
\end{equation}
initialized at $\W^0=\X g_0(\V)$, where the nonlinearities 
\begin{align*}
f_i &=  \big(f_{i,1},\ldots, f_{i,k}\big): \R^{k(i+1)} \times \R \rightarrow
\R^k,\qquad g_i = \big(g_{i,1},\ldots, g_{i,k}\big): \R^{ki} \times \R^{T+2} \rightarrow \R^k
\end{align*}
are Lipschitz-continuous and
applied row-wise, the Onsager coefficients are recursively defined as
\begin{align*}
\xi_{ij} &= \bigg(\delta\,\E\Big[\d_{W^j}f_i(W^0,\ldots,
W^i;\eps)\Big]\bigg)^\top \in \R^{k\times k}, \quad 0\leq j\leq i,\\
\zeta_{ij} &= \bigg(\E\Big[\d_{U^{j+1}}
g_i(U^1,\ldots,U^i;V)\Big]\bigg)^\top \in \R^{k\times k},\quad 0\leq j\leq i-1,
\end{align*}
and $\{W^j\}_{j \geq 0}$ and $\{U^j\}_{j \geq 1}$ are mean-zero
Gaussian processes in $\R^k$ independent of $\eps,V$ with covariance structure
\begin{align}\label{eq:amp_gp_cov}
\notag\E[W^iW^{j\top}] &= \E\Big[g_i(U^1,\ldots,U^i;V)
g_j(U^1,\ldots,U^j;V)^\top\Big] \in \R^{k\times k}, \quad i,j \geq 0,\\
\E[U^{i+1}U^{j+1\top}] &= \E\Big[\delta f_i(W^0,\ldots,W^{i};\eps)
f_j(W^0,\ldots,W^j;\eps)^\top\Big] \in \R^{k\times k},\quad i,j \geq 0. 
\end{align}
This is a standard form of an AMP algorithm, see e.g.\
\cite{javanmard2013state,wang2024universality}. The iterations
for $(\W^0,\ldots,\W^{T-1}) \in \R^{n \times kT}$ and
$(\U^1,\ldots,\U^T) \in \R^{d \times kT}$ admit a mapping to the form of
\cite[Eqs.\ (2.14) and (D.1--D.2)]{wang2024universality} with $kT$ vector
iterates. Then by the AMP state evolution (c.f.\
\cite[Theorem 2.21 and Remark 2.2]{wang2024universality}), under the conditions
of Assumption \ref{assump:model}, almost surely as $n,d\rightarrow\infty$,
\begin{equation}\label{eq:state_evolution}
\begin{gathered}
\frac{1}{d}\sum_{j=1}^d \delta_{U^1_j,\ldots,U^m_j,V_j} \overset{W_2}{\to}
\sP(U^1,\ldots,U^m,V),\\
\frac{1}{n}\sum_{i=1}^n \delta_{W^0_i,\ldots,W^m_i, \eps_i} \overset{W_2}{\to}
\sP(W^0,\ldots,W^m,\eps).
\end{gathered}
\end{equation}

We will now use the above state evolution to prove the desired conclusion
(\ref{eq:emp_converge_aux}). In the AMP algorithm
(\ref{eq:amp}), let $k = 2$. We show the existence of Lipschitz nonlinearities
$g_i=(g_{i,1},g_{i,2}):\R^{2i} \times \R^{T+2} \rightarrow \R^2$ and
$f_i=(f_{i,1}, f_{i,2}): \R^{2(i+1)} \times \R \rightarrow \R^2$ such that
\begin{align}
(\tilde{\btheta}^j_\gamma, \btheta^\ast) &=
g_j(\U^1,\ldots,\U^j;\V),\label{eq:nonlinearity_1}\\
\Big({-}(\beta/\delta)(\X\tilde{\btheta}^j_\gamma - \X\btheta^\ast - \beps),
0\Big) &= f_j(\W^0,\ldots,\W^j; \beps).\label{eq:nonlinearity_2}
\end{align}
The base case is $g_0(\V)=(\btheta^0,\btheta^*)$ and 
$f_0(\W^0;\beps)=({-}(\beta/\delta)(\W^0_1 - \W^0_2 - \beps),0)$
where $\W^0=(\W_1^0,\W_2^0)=(\X\btheta^0, \X\btheta^\ast)$.
Supposing inductively that
(\ref{eq:nonlinearity_1}--\ref{eq:nonlinearity_2}) hold for some Lipschitz
functions $g_0,f_0,\ldots,g_j,f_j$, we note that this implies
$(\xi_{j\ell})_{12}=(\xi_{j\ell})_{22}=0$ for all $\ell \leq j$.
Then writing $\U^j=(\U^j_1,\U^j_2) \in \R^{d\times 2}$, we have
\begin{align*}
\U^{j+1}_1 &= \X^\top f_{j,1}(\W^0,\ldots,\W^j;\beps)-\sum_{\ell=0}^j
\Big(g_{\ell,1}(\U^1,\ldots,\U^\ell;\V)(\xi_{j\ell})_{11} +
g_{\ell,2}(\U^1,\ldots,\U^\ell;\V)(\xi_{j\ell})_{21}\Big)\\
&={-}\frac{\beta}{\delta}\X^\top(\X\tilde{\btheta}^j_\gamma - \X\btheta^\ast - \beps) -
\sum_{\ell=0}^j g_{\ell,1}(\U^1,\ldots, \U^\ell;\V)(\xi_{j\ell})_{11} -
\sum_{\ell=0}^j (\xi_{j\ell})_{21} \cdot \btheta^\ast.
\end{align*}
So
\begin{align*}
\tilde{\btheta}^{j+1}_\gamma &= \tilde{\btheta}^j_\gamma -
\gamma\bigg(\beta\X^\top (\X\tilde{\btheta}^j_\gamma - \X\btheta^\ast -
\beps) - s(\tilde{\btheta}^j_\gamma,\alpha^j_\gamma)\bigg) +
\sqrt{2} (\b^{j+1} - \b^j)\\
&= \tilde{\btheta}^j_\gamma + \gamma\delta \Big(\U^{j+1}_1+\sum_{\ell=0}^j
g_{\ell,1}(\U^1,\ldots,\U^\ell;\V)(\xi_{j\ell})_{11}+\sum_{\ell=0}^j
(\xi_{j\ell})_{21} \cdot \btheta^\ast\Big) + \gamma
s(\tilde{\btheta}^t_\gamma,\alpha^j_\gamma) + \sqrt{2}
(\b^{j+1} - \b^j),
\end{align*}
and to satisfy (\ref{eq:nonlinearity_1}) we may define $g_{j+1}(\cdot)$ as
$g_{j+1,2}(U^1,\ldots,U^{j+1};V) = \theta^\ast$ and
\begin{align}\label{eq:g_planted}
\notag g_{j+1,1}(U^1,\ldots,U^{j+1};V) &=
g_{j,1}(U^1,\ldots,U^j;V)+\gamma\delta
\Big(U^{j+1}_1+\sum_{\ell=0}^j g_{\ell,1}(U^1,\ldots,
U^\ell;V)(\xi_{j\ell})_{11} + \sum_{\ell=0}^j (\xi_{j\ell})_{21} \cdot \theta^\ast\Big)\\
&\quad + \gamma s\big(g_{j,1}(U^1,\ldots,
U^j;V),\alpha^j_\gamma\big) + \sqrt{2} \rho_j,
\end{align}
where we recall $V=(\theta^\ast,\theta^0,\rho_1,\ldots,\rho_T)$. We note that
$\theta \mapsto s(\theta,\alpha_\gamma^j)$ is Lipschitz by Assumption
\ref{assump:prior}, so this function $g_{j+1}(\cdot)$ is also Lipschitz by the
induction hypothesis. Next, the condition $g_{j+1,2}(\cdot)=\theta^\ast$
implies
$(\zeta_{j+1,\ell})_{12} = (\zeta_{j+1,\ell})_{22} = 0$ for $\ell\leq j$, and
allows us to compute $\W^{j+1} = (\W^{j+1}_1, \W^{j+1}_2)$ as $\W^{j+1}_2 = \X\btheta^\ast$ and 
\begin{align*}
\W^{j+1}_1 &= \X g_{j+1,1}(\U^1,\ldots,\U^{j+1};\V)-\sum_{\ell=0}^j
\Big(f_{\ell,1}(\W^0,\ldots,\W^\ell;\beps)(\zeta_{j+1,\ell})_{11}+f_{\ell,2}(\W^0,\ldots,\W^\ell;\beps)(\zeta_{j+1,\ell})_{21}\Big)\\
&= \X\tilde{\btheta}^{j+1}_\gamma-\sum_{\ell=0}^{j}
f_{\ell,1}(\W^0,\ldots,\W^\ell;\beps) (\zeta_{j+1,\ell})_{11}.
\end{align*}
Hence with 
\begin{align*}
\X\tilde{\btheta}^{j+1}_\gamma - \X\btheta^\ast - \beps =
\W^{j+1}_1 - \W^{j+1}_2 +\sum_{\ell=0}^{j} f_{\ell,1}(\W^0,\ldots,\W^\ell;\beps) (\zeta_{j+1,\ell})_{11} - \beps,
\end{align*}
to satisfy (\ref{eq:nonlinearity_2}) we can define $f_{j+1,2} = 0$ and
\begin{align}\label{eq:f_planted}
f_{j+1,1}(W^0,\ldots,W^{j+1};\eps) = {-}\frac{\beta}{\delta}\Big(W^{j+1}_1 - W^{j+1}_2 +
\sum_{\ell=0}^{j} f_{\ell,1}(W^0,\ldots,W^\ell;\eps) (\zeta_{j+1,\ell})_{11} - \eps\Big).
\end{align}
This is also Lipschitz by the induction hypothesis, completing the induction.
So using (\ref{eq:nonlinearity_1}--\ref{eq:nonlinearity_2}), the state evolution
(\ref{eq:state_evolution}), and the fact that $X_n \overset{W_2}{\to} X$ 
implies $f(X_n) \overset{W_2}{\to} f(X)$ for Lipschitz $f$, we conclude that
\begin{align}\label{eq:AMPSE}
\frac{1}{d}\sum_{j=1}^d \delta_{\big(\theta^\ast_j,\tilde{\theta}^0_{\gamma,j}, \ldots,
\tilde{\theta}^T_{\gamma,j}\big)}  \overset{W_2}{\to}
\sP\big(\theta^\ast,\theta^{0},\ldots,
\theta^T\big),\quad\frac{1}{n}\sum_{i=1}^n
\delta_{\big(\eta^\ast_i,\eps_i,\tilde{\eta}^0_{\gamma,i}, \ldots,
\tilde{\eta}^T_{\gamma,i}\big)} \overset{W_2}{\to}
\sP\big(W^*, \eps,\eta^0,\ldots, \eta^T\big).
\end{align}
Here, the laws on the right side are defined by
setting $W^*=W_2^i$ for each $i \geq 1$, and
\begin{align*}
\theta^i = g_{i,1}(U^1,\ldots,U^i;V), \quad
\eta^i={-}\frac{\delta}{\beta}f_{i,1}(W^0,\ldots,W^i;\eps)+W^* + \eps,
\end{align*}
where $\{U^i\} = \{(U^i_1,U^i_2)\}$ and $\{W^i\}=\{(W^i_{1},W^i_2)\}$
are the Gaussian processes from AMP state evolution, independent of $\eps,V$
with covariance kernels given by (\ref{eq:amp_gp_cov}).

Let us now show that
\begin{equation}\label{eq:AMPjointlaws}
\sP(\theta^\ast,\theta^0,\ldots,\theta^T) =
\sP\big(\theta^\ast,\theta^0_\gamma,\ldots,
\theta^T_\gamma\big), \qquad
\sP(W^*,\eps,\eta^0,\ldots,\eta^T) = \sP\big({-}w_\gamma^\ast, \eps,\eta^0_\gamma,\ldots,
\eta^T_\gamma\big),
\end{equation}
where the laws on the right sides are the variables of the discrete DMFT
equations. This
will conclude the proof of (\ref{eq:emp_converge_aux}). To do so, let us
define from the AMP state evolution variables (\ref{eq:AMPSE}) the quantities
\begin{equation}\label{eq:identifications}
\begin{gathered}
u_\gamma^i=\delta U_1^{i+1}, \quad
w_\gamma^i={-}W_1^i, \quad w_\gamma^*={-}W^*,
\quad \theta_\gamma^i=\theta^i,
\quad \eta_\gamma^i=\eta^i,\\
\frac{\partial \theta^i_\gamma}{\partial
u^j_\gamma}=\frac{1}{\delta}\frac{\partial g_{i,1}}{\partial
U_1^{j+1}}(U^1,\ldots,U^i;V),
\quad \frac{\partial\eta^i_\gamma}{\partial w^j_\gamma}
=\frac{\delta}{\beta}\frac{\partial f_{i,1}}{\partial
W^j_1}(W^0,\ldots,W^i;\eps).
\end{gathered}
\end{equation}
Then it suffices to check that these quantities satisfy the discrete DMFT
equations (\ref{eq:dmft_theta_discrete1}--\ref{eq:dmft_CReta_discrete}),
by uniqueness of the iterative
construction (\ref{eq:iterativeDMFTconstruction}) of the solution to these
discrete DMFT equations. We first note that by (\ref{eq:amp_gp_cov}),
$\{u^j_\gamma\}$ and $(w_\gamma^*,\{w_\gamma^j\})$
thus defined are centered Gaussian processes with covariance
\begin{align*}
 \E[u_\gamma^i u_\gamma^j] &=
\delta^3 \E[f_{i,1}(W^0,\ldots,W^i;\eps) f_{j,1}(W^0,\ldots,W^j;\eps)] =
\delta \beta^2\E[(\eta^i-W^\ast-\eps)(\eta^j - W^\ast - \eps)], \\
\E[w_\gamma^i w_\gamma^j] &=
\E[g_{i,1}(U^1,\ldots,U^i;V)g_{j,1}(U^1,\ldots,U^j;V)] = \E[\theta^i\theta^j],\\
\E[w_\gamma^i w_\gamma^*] &=
\E[g_{i,1}(U^1,\ldots,U^i;V)g_{0,2}(V)] = \E[\theta^i\theta^*],\\
\E[(w_\gamma^*)^2]
&=\E[g_{0,2}(V)^2]
=\E[(\theta^*)^2],
\end{align*}
which verifies (\ref{eq:dmft_covu_discrete}) and (\ref{eq:dmft_covw_discrete})
in light of (\ref{eq:identifications}).

We next check the recursions (\ref{eq:response_1}) and
(\ref{eq:response_2}) for the response:
Recall that the AMP Onsager corrections are
\begin{equation}\label{eq:xi_zeta_simplify}
\begin{gathered}
(\zeta_{j,s})_{11} = \E\Big[\frac{\partial g_{j,1}}{\partial
U^{s+1}_1}(U^1,\ldots,U^j;V)\Big],\\
(\xi_{j,s})_{11} = \E\Big[\delta\,\frac{\partial f_{j,1}}{\partial
W^{s}_1}(W^0,\ldots,W^j;\eps)\Big],\quad(\xi_{j,s})_{21} =
\E\Big[\delta\,\frac{\partial f_{j,1}}{\partial W^{s}_2}(W^0,\ldots,W^j;\eps)\Big].
\end{gathered}
\end{equation}
By definition of $g_{j,1}$ in (\ref{eq:g_planted}), we have $\frac{\partial
g_{j,1}}{\partial U^{s+1}_1} = 0$ for $s \geq j$, $\frac{\partial
g_{j,1}}{\partial U_1^{s+1}} = \gamma\delta$ if $s = j-1$, and if $s \leq j-2$, 
\begin{align}\label{eq:g_recursion}
\notag\frac{\partial g_{j,1}}{\partial U^{s+1}_1} &= \frac{\partial
g_{j-1,1}}{\partial U^{s+1}_1} + \gamma\delta \sum_{\ell=s+1}^{j-1}
\frac{\partial g_{\ell,1}}{\partial U^{s+1}_1}(\xi_{j-1,\ell})_{11} + \gamma
\partial_\theta s(g_{j-1,1},\alpha^{t_{j-1}}_\gamma) \frac{\partial
g_{j-1,1}}{\partial U_1^{s+1}}\\
&=\Big(1+\gamma\delta(\xi_{j-1,j-1})_{11} + \gamma \partial_\theta s(g_{j-1,1},
\alpha^{t_{j-1}}_\gamma)\Big)\frac{\partial g_{j-1,1}}{\partial U_1^{s+1}}
+ \gamma \delta \sum_{\ell=s+1}^{j-2} (\xi_{j-1,\ell})_{11}\frac{\partial
g_{\ell,1}}{\partial U_1^{s+1}}
\end{align}
where both sides are evaluated at $(U^1,\ldots,U^j;V)$.
Similarly, by definition of $f_{j,1}(\cdot)$ in (\ref{eq:f_planted}), we have
$\frac{\partial f_{j,1}}{\partial W_1^s}
=\frac{\partial f_{j,1}}{\partial W_2^s}=0$ if $s > j$, $\frac{\partial
f_{j,1}}{\partial W_1^s}=-\frac{\partial f_{j,1}}{\partial W_2^s}
=-\beta/\delta$ if $s = j$, and if $s < j$,
\begin{align}
\frac{\partial f_{j,1}}{\partial W^s_1} =
-\frac{\beta}{\delta}\sum_{\ell=s}^{j-1} \frac{\partial f_{\ell,1}}{\partial
W^s_1} (\zeta_{j,\ell})_{11} = -\frac{\beta}{\delta}\Big(\sum_{\ell=s+1}^{j-1}
\frac{\partial f_{\ell,1}}{\partial W^s_1}
(\zeta_{j,\ell})_{11}-\frac{\beta}{\delta}(\zeta_{j,s})_{11}\Big),
\label{eq:f_recursion1}\\
\frac{\partial f_{j,1}}{\partial W^s_2} =
-\frac{\beta}{\delta}\sum_{\ell=s}^{j-1} \frac{\partial f_{\ell,1}}{\partial
W^s_2} (\zeta_{j,\ell})_{11}=-\frac{\beta}{\delta}\Big(\sum_{\ell=s+1}^{j-1}
\frac{\partial f_{\ell,1}}{\partial W^s_2}
(\zeta_{j,\ell})_{11}+\frac{\beta}{\delta}(\zeta_{j,s})_{11}\Big).
\label{eq:f_recursion2}
\end{align}
Here, these recursions imply that
$\{\frac{\partial f_{j,1}}{\partial W^s_1}\}$ and
$\{\frac{\partial f_{j,1}}{\partial W^s_2}\}$ are deterministic. Under the definitions
(\ref{eq:dmft_CRtheta_discrete}), (\ref{eq:dmft_CReta_discrete}),
and (\ref{eq:identifications}), we have
\begin{align}\label{eq:responseidentifications}
R_\theta^\gamma(j,s) = \frac{1}{\delta}\E\Big[\frac{\partial g_{j,1}}{\partial
U^{s+1}_1}(U^1,\ldots,U^j;V)\Big] = \frac{1}{\delta}(\zeta_{j,s})_{11},
\quad R_\eta^\gamma(j,s) = \delta^2\Big[\frac{\partial f_{j,1}}{\partial
W^s_1}\Big] = \delta(\xi_{j,s})_{11}.
\end{align}
Then by (\ref{eq:g_recursion}) and (\ref{eq:f_recursion1}), $\{\frac{\partial
g_{j,1}}{\partial U^{s+1}_1}\}$ and $\{\frac{\partial f_{j,1}}{\partial
W^s_1}\}$ satisfy the recursions
\begin{align*}
\frac{\partial g_{j,1}}{\partial U^{s+1}_1} &= \Big(1 - \gamma\delta\beta +
\gamma \partial_\theta
s(g_{j-1,1};\alpha^{j-1}_\gamma)\Big)\frac{\partial
g_{j-1,1}}{\partial U^{s+1}_1} + \gamma \sum_{\ell=s+1}^{j-2}
R_\eta^\gamma(j-1,\ell) \frac{\partial g_{\ell,1}}{\partial U^{s+1}_1},\\
\frac{\partial f_{j,1}}{\partial W^s_1} &= -\beta\Big(\sum_{\ell=s+1}^{j-1}
R_\theta^\gamma(j,\ell)\frac{\partial f_{\ell,1}}{\partial W^s_1} -
\frac{\beta}{\delta} R_\theta^\gamma(j,s)\Big),
\end{align*}
which verify (\ref{eq:response_1}) and (\ref{eq:response_2}) in view of
(\ref{eq:identifications}) and above boundary conditions 
$\frac{\partial g_{j,1}}{\partial U_1^j} = \gamma\delta$
and $\frac{\partial f_{j,1}}{\partial W_1^j}=-\beta/\delta$.

Finally we check the primary recursions
(\ref{eq:dmft_theta_discrete1}) and (\ref{eq:dmft_eta_discrete1}).
By (\ref{eq:identifications}),
definition of $g_{j+1,1}(\cdot)$ in (\ref{eq:g_planted}), and
$(\xi_{j,j})_{11}={-}(\xi_{j,j})_{21}=-\beta$,
we have
\begin{align}\label{eq:theta^t}
\notag\theta^{j+1} &= g_{j+1,1}(U^1,\ldots, U^{j+1};V)\\
\notag&= \theta^j + \gamma\delta\Big(U^{j+1}_1 + \sum_{\ell=0}^j \theta^\ell
(\xi_{j,\ell})_{11} + \sum_{\ell=0}^j (\xi_{j,\ell})_{21} \cdot \theta^\ast\Big)
+ \gamma s\big(\theta^j,\alpha^j_\gamma\big) + \sqrt{2}\rho_j\\
&= \theta^j - \gamma\delta\beta(\theta^j-\theta^\ast) + \gamma
s\big(\theta^j,\alpha^j_\gamma\big) + \gamma\delta\sum_{\ell=0}^{j-1}
\theta^\ell (\xi_{j,\ell})_{11} + \gamma\delta\sum_{\ell=0}^{j-1}
\theta^*(\xi_{j,\ell})_{21}+\gamma\delta U^{j+1}_1 + \sqrt{2}\rho_j.
\end{align} 
Similarly, by definition of $f_{j,1}(\cdot)$ in (\ref{eq:f_planted}) and
$W^i_2 \equiv W^\ast$,
\begin{align}\label{eq:eta^t}
\eta^j&={-}\frac{\delta}{\beta} f_{j,1}(W^0,\ldots,W^j;\eps) + W^* + \eps =
{-}\frac{\beta}{\delta}\sum_{\ell=0}^{j-1} (\zeta_{j,\ell})_{11}(\eta^\ell -
W^\ast -  \eps) + W^j_1.
\end{align}
Applying (\ref{eq:xi_zeta_simplify}), note that in (\ref{eq:theta^t}) we have
$A_j:=\sum_{\ell=0}^{j-1} (\xi_{j,\ell})_{21}
=\delta\sum_{\ell=0}^{j-1} \frac{\partial f_{j,1}}{\partial W_2^s}$.
Then by the recursion (\ref{eq:f_recursion2}) and first identification of
(\ref{eq:responseidentifications}), we have
\begin{align*}
A_j = -\beta\delta\sum_{\ell=0}^{j-1}\sum_{s=\ell}^{j-1} \frac{\partial
f_{s,1}}{\partial W^\ell_2} R_\theta^\gamma(j,s) = -\beta\sum_{s=0}^{j-1}
R_\theta^\gamma(j,s)
\delta\sum_{\ell=0}^{s} \frac{\partial f_{s,1}}{\partial W^\ell_2} =  -\beta
\sum_{s=0}^{j-1} R_\theta^\gamma(j,s) (A_s + \beta).
\end{align*}
This coincides with the recursion for $\{\beta \frac{\partial
\eta_\gamma^j}{\partial w_\gamma^\ast}\}$ in (\ref{eq:response_3}).
Hence $A_j=\beta\frac{\partial \eta^j_\gamma}{\partial w_\gamma^\ast} =
\frac{1}{\delta} R^\gamma_\eta(j,\ast) = {-}\frac{1}{\delta}\sum_{s=0}^{j-1}
R^\gamma_\eta(j,s)$, where the last step applies
Lemma \ref{lem:eta_response_id}. Applying this form of $A_j$
and (\ref{eq:responseidentifications}), we may write the equations
(\ref{eq:theta^t}--\ref{eq:eta^t}) as
\begin{align*}
\theta^{j+1} &= \theta^j - \gamma\delta\beta(\theta^j - \theta^\ast) + \gamma
s\big(\theta^j,\alpha^j_\gamma\big) + \gamma \sum_{\ell=0}^{j-1}
R^\gamma_\eta(j,\ell)(\theta^\ell  - \theta^\ast) + \gamma \delta
U^{j+1}_1 + \sqrt{2}\rho_j,\\
\eta^j &= -\beta\sum_{\ell=0}^{j-1} R^\gamma_\theta(j,\ell)(\eta^\ell - W^\ast -
\eps) + W^j_1, 
\end{align*}
which verifies (\ref{eq:dmft_theta_discrete1}) and (\ref{eq:dmft_eta_discrete1})
in view of (\ref{eq:identifications}).
This verifies that the definitions (\ref{eq:identifications}) indeed satisfy
(\ref{eq:dmft_theta_discrete1}--\ref{eq:dmft_CReta_discrete}),
concluding the proof of (\ref{eq:emp_converge_aux}).\\

\noindent \textbf{Step 2: Comparison with auxiliary dynamics.} Let us now prove
(\ref{eq:finite_dim_converge_knot_theta}--\ref{eq:finite_dim_converge_knot_alpha})
for the original dynamics with an adaptive drift parameter
$\widehat\alpha_\gamma^t$.
We will prove via induction that, almost surely as $n,d\rightarrow\infty$,
for each $t=0,\ldots,T$,
\begin{align}\label{eq:aux_diff}
\frac{1}{d}\pnorm{\btheta^t_\gamma - \tilde{\btheta}^t_\gamma}{}^2
\rightarrow 0, \qquad \widehat{\alpha}^t_\gamma \rightarrow
\alpha^t_\gamma. 
\end{align}
Since
\begin{align}
W_2\Big(\frac{1}{d}\sum_{j=1}^d \delta_{\big(\theta^\ast_j,\theta^0_{\gamma,j}, \ldots,
\theta^T_{\gamma,j}\big)}, \frac{1}{d}\sum_{j=1}^d
\delta_{\big(\theta^\ast_j,\tilde{\theta}^0_{\gamma,j}, \ldots,
\tilde{\theta}^T_{\gamma,j}\big)}\Big)^2 &\leq
\sum_{t=0}^T \frac{1}{d}\pnorm{\btheta^t_\gamma - \tilde{\btheta}^t_\gamma}{}^2
\label{eq:inductionimplication}\\
W_2\Big(\frac{1}{n}\sum_{i=1}^n \delta_{\big(\eta^\ast_i,\eps_i,\eta^0_{\gamma,i}, \ldots,
\eta^T_{\gamma,i}\big)}, \frac{1}{n}\sum_{i=1}^n
\delta_{\big(\eta^\ast_i,\eps_i,\tilde{\eta}^0_{\gamma,i}, \ldots,
\tilde{\eta}^T_{\gamma,i},\big)}\Big)^2 &\leq
\sum_{t=0}^T \frac{1}{n}\pnorm{\bbeta^t_\gamma - \tilde{\bbeta}^t_\gamma}{}^2
\leq \sum_{t=0}^T \frac{1}{n}\|\X\|_\op^2
\pnorm{\btheta^t_\gamma - \tilde{\btheta}^t_\gamma}{}^2\notag
\end{align}
and $\|\X\|_\op$ is almost surely bounded for all large $n,d$,
the above inductive claim together with (\ref{eq:emp_converge_aux}) implies
(\ref{eq:finite_dim_converge_knot_theta}--\ref{eq:finite_dim_converge_knot_alpha}). 

The base case of $t=0$ in (\ref{eq:aux_diff}) holds exactly. Suppose
(\ref{eq:aux_diff}) holds up to time $t$. For $t+1$, we see that
\begin{align*}
\frac{1}{d}\pnorm{\btheta^{t+1}_\gamma - \tilde{\btheta}^{t+1}_\gamma}{}^2 &\leq C\Big((1+\pnorm{\X}{\op}^4)\frac{1}{d}\pnorm{\btheta^t_\gamma - \tilde{\btheta}^t_\gamma}{}^2 + \frac{1}{d}\pnorm{s(\btheta^t_\gamma, \widehat{\alpha}^t_\gamma) - s(\tilde{\btheta}^t_\gamma, \alpha^t_\gamma)}{}^2\Big).
\end{align*} 
Applying boundedness of $\pnorm{\X}{\op}$ and Lipschitz continuity of $s(\cdot)$ in Assumption \ref{assump:prior}, we have by the induction hypothesis 
that $\frac{1}{d}\pnorm{\btheta^{t+1}_\gamma -
\tilde{\btheta}^{t+1}_\gamma}{}^2 \rightarrow 0$ almost surely. Next, we have by the Lipschitz continuity of $\cG(\cdot)$ in Assumption \ref{assump:gradient} that 
\begin{align*}
\pnorm{\widehat{\alpha}^{t+1}_\gamma - \alpha_\gamma^{t+1}}{}
&\leq \pnorm{\widehat{\alpha}^t_\gamma - \alpha_\gamma^t}{} +
\gamma \left\|\cG\Big(\widehat{\alpha}^t_\gamma, \frac{1}{d}\sum_{\ell=1}^d
\delta_{\theta^t_{\gamma,\ell}}\Big) - \cG(\alpha^t_\gamma,
\sP(\theta^t_\gamma))\right\|\\
&\leq C\pnorm{\widehat{\alpha}^t_\gamma - \alpha_\gamma^t}{} +
CW_2\Big(\frac{1}{d}\sum_{\ell=1}^d  \delta_{\theta^t_{\gamma,\ell}},
\sP(\theta^t_\gamma)\Big),
\end{align*}
which converges almost surely to 0 by the induction hypothesis and the above
implication (\ref{eq:inductionimplication}). This establishes the induction for (\ref{eq:aux_diff}) and hence completes the proof.
\end{proof}

\subsection{Step 2: Discretization error of DMFT
equation}\label{subsec:discretizedmft}

We now define a piecewise constant embedding of the components of the
discrete DMFT system
(\ref{eq:dmft_theta_discrete1}--\ref{eq:dmft_CReta_discrete}) into continuous
time, and show that this converges to the solution of the continuous DMFT system
established in Theorem \ref{thm:dmftsolexists}, in the limit $\gamma \to 0$.

For all times $t \in \R_+$, define
\begin{align}\label{eq:t_floor_ceil}
\floor{t} = \max\{i\gamma: i\gamma\leq t, i\in \Z_+\} \in \gamma\Z_+,
\quad \ceil{t} = \floor{t} + \gamma \in \gamma\Z_+,
\quad [t]=\floor{t}/\gamma \in \Z_+.
\end{align}
Fixing $T>0$, let $\cD_\eta^\gamma$ be the space of functions
$(\bar R_\eta^\gamma,\bar C_\eta^\gamma,\bar \alpha_\gamma)
\equiv \{(\bar R_\eta^\gamma(t,s),\bar C_\eta^\gamma(t,s),\bar \alpha_\gamma^t\}_{0 \leq s \leq t \leq T}$ that are piecewise constant
and right-continuous in $(s,t)$ with jumps at $\gamma\Z_+$, i.e.\
$\bar R_\eta^\gamma(t,s)=\bar R_\eta^\gamma(\floor{t},\floor{s})$ for all
$0 \leq s \leq t \leq T$ and similarly for $\bar C_\eta^\gamma,\bar
\alpha_\gamma$. Analogously, let $\cD_\theta^\gamma$ be the space of functions
$(\bar R_\theta^\gamma,\bar C_\theta^\gamma,\bar{\tilde\alpha}_\gamma)
\equiv \{\bar R_\theta^\gamma(t,s),\bar C_\theta^\gamma(t,s),\bar{\tilde\alpha}_\gamma^t\}_{0 \leq s \leq t \leq T}$ that are piecewise
constant and right-continuous with jumps at $\gamma\Z_+$.

We define a map $\cT^\gamma_{\eta \to \theta}:\cD_\eta^\gamma \to
\cD_\theta^\gamma$ as follows: Given $X^\gamma=(\bar R_\eta^\gamma,\bar
C_\eta^\gamma,\bar \alpha_\gamma) \in
\cD_\eta^\gamma$, let $\{\bar u_\gamma^t\}_{t \in [0,T]}$ be a mean-zero
Gaussian process with covariance $\bar C_\eta^\gamma$,
let $\{b^t\}_{t \geq 0}$ be a standard Brownian motion, and define
processes $\{\bar \theta_\gamma^t\}_{t \in [0,T]}$
and $\{\frac{\partial \bar \theta^t_\gamma}{\partial \bar u^s_\gamma}\}_{0 \leq
s \leq t \leq T}$ by
\begin{align}
\bar \theta^t_\gamma&=\theta^0+\int_0^{\floor{t}}\Big[{-}\delta\beta(\bar
\theta^s_{\gamma}-\theta^\ast) + s(\bar \theta^s_\gamma,\bar\alpha^s_\gamma) +
\int_0^{\floor{s}} \bar R^\gamma_\eta(s,r)(\bar \theta_\gamma^r-\theta^\ast)\d r  +
\bar{u}^s_\gamma\Big]\d s + \sqrt{2}b^{\floor{t}},\label{def:dmft_piecelinear_theta}\\
\frac{\partial \bar \theta^t_\gamma}{\partial \bar u^s_\gamma} &=
1+\1\{\ceil{s} \leq \floor{t}\}
\int_{\ceil{s}}^{\floor{t}}
\bigg[\Big({-}\delta\beta + \partial_\theta
s(\bar \theta_\gamma^r,\bar \alpha^r_\gamma)\Big)\frac{\partial
\bar \theta^r_\gamma}{\partial \bar u^s_\gamma} + \int_{\ceil{s}}^{\floor{r}}
\bar R_\eta^\gamma(r,r')\frac{\partial
\bar \theta^{r'}_\gamma}{\partial \bar u^s_\gamma}\d r'\bigg]\d r.
\label{def:derivative_process_piecelinear_1}
\end{align}
Define from these processes
\begin{equation}\label{def:CRtheta_piecelinear}
\begin{gathered}
\bar C_\theta^\gamma(t,s) = \E\Big[\bar \theta^t_\gamma \bar \theta^s_\gamma\Big],
\quad \bar C_\theta^\gamma(t,\ast)=\E\Big[\bar \theta^t_\gamma \theta^\ast\Big], 
\quad \bar C_\theta^\gamma(*,*)=\E[(\theta^*)^2],\\
\bar R^\gamma_\theta(t,s)=\E\Big[\frac{\partial\bar\theta^t_\gamma}{\partial
\bar u^s_\gamma}\Big], \quad
\bar{\tilde \alpha}_\gamma^t=\alpha^0+\int_0^{\floor{t}}
\cG(\bar{\tilde \alpha}_\gamma^s,\sP(\bar\theta_\gamma^s))\d s
\end{gathered}
\end{equation}
and set $\cT_{\eta \to \theta}^\gamma(X^\gamma)=(\bar R_\theta^\gamma,
\bar C_\theta^\gamma,\bar{\tilde \alpha}_\gamma)$. We also define a map
$\cT^\gamma_{\theta \to \eta}:\cD_\theta^\gamma \to
\cD_\eta^\gamma$ as follows: Given $Y^\gamma=(\bar R_\theta^\gamma,\bar
C_\theta^\gamma,\bar{\tilde \alpha}_\gamma) \in \cD_\eta^\gamma$,
let $(\bar w_\gamma^*,\{\bar w_\gamma^t\}_{t \in [0,T]})$ be a mean-zero
Gaussian process with covariance $\bar C_\theta^\gamma$, and define
\begin{align}
\bar \eta^t_\gamma &= {-}\beta\int_0^{\floor{t}} \bar R^\gamma_\theta(t,s)(\bar
\eta^s_\gamma + \bar w^\ast_\gamma - \eps)\d s - \bar w^t_\gamma,\label{def:dmft_piecelinear_eta}\\
\frac{\partial \bar \eta^t_\gamma}{\partial \bar w^s_\gamma} &= 
\beta \bar R_\theta(t,s)-\1\{\ceil{s} \leq \floor{t}\}
\beta \int_{\ceil{s}}^{\floor{t}} \bar R^\gamma_\theta(t,r)\frac{\partial
\bar \eta^r_\gamma}{\partial \bar w^s_\gamma}\d r.
\label{def:derivative_process_piecelinear_3}
\end{align}
Define from these processes
\begin{equation}\label{def:CReta_piecelinear}
\bar R^\gamma_\eta(t,s)=\delta\beta\,
\E\Big[\frac{\partial\bar\eta^t_\gamma}{\partial \bar w^s_\gamma}\Big], \quad
\bar C_\eta^\gamma(t,s)=\delta\beta^2
\E\Big[(\bar \eta^t_\gamma + \bar w^\ast_\gamma -
\eps)(\bar \eta^s_\gamma + \bar w^\ast_\gamma - \eps)\Big],
\quad \bar \alpha_\gamma^t=\bar{\tilde \alpha}_\gamma^t
\end{equation}
and set $\cT_{\theta \to \eta}^\gamma(Y^\gamma)=(\bar R_\eta^\gamma,\bar
C_\eta^\gamma,\bar \alpha_\gamma)$. These maps may be understood as discrete
approximations of $\cT_{\eta \to \theta}$ and $\cT_{\theta \to \eta}$
constructed in Section \ref{sec:contractivemapping}, with domains restricted to
the spaces $\cD_\theta^\gamma$ and $\cD_\eta^\gamma$ of piecewise constant
inputs.

Recall the spaces $\cS_\theta,\cS_\eta,\cS$ of Section \ref{sec:functionspace}.
The following lemma shows that the above maps
$\cT_{\eta \to \theta}^\gamma,\cT_{\theta \to \eta}^\gamma$ are well-defined,
that the unique fixed point of these maps is a
piecewise constant embedding of the discrete-time
DMFT system in Section \ref{subsec:discretedmft}, and that furthermore
this fixed point belongs to $\cS$.

\begin{lemma}\label{lem:discrete_dmft_rewrite}
\begin{enumerate}[(a)]
\item Given any $X^\gamma \in \cD_\eta^\gamma$ and realization of $\theta^*$, $\theta^0$, $\{b^t\}_{t \in [0,T]}$, and $\{\bar u_\gamma^t\}_{t \in [0,T]}$,
the processes (\ref{def:dmft_piecelinear_theta}--\ref{def:derivative_process_piecelinear_1}) have a unique solution,
and this solution is piecewise constant and right-continuous with jumps at
$\gamma \Z_+$. Consequently,
$\cT_{\eta \to \theta}^\gamma$ is a well-defined map from
$\cD_\eta^\gamma$ to $\cD_\theta^\gamma$.
\item Given any $Y^\gamma \in \cD_\theta^\gamma$ and realization of $\eps$ and
$(\bar w_\gamma^*,\{\bar w_\gamma^t\})$,
the processes (\ref{def:dmft_piecelinear_eta}--\ref{def:derivative_process_piecelinear_3}) have a unique solution,
and this solution is piecewise constant and right-continuous with jumps at
$\gamma \Z_+$. Consequently, $\cT_{\theta \to \eta}^\gamma$ is a well-defined
map from $\cD_\theta^\gamma$ to $\cD_\eta^\gamma$.
\item The map $\cT_{\eta \to \eta}^\gamma=\cT_{\eta \to \theta}^\gamma \circ
\cT_{\theta \to \eta}^\gamma$ has a unique fixed point in $\cD_\eta^\gamma$, and
the map $\cT_{\theta \to \theta}^\gamma=\cT_{\theta \to \eta}^\gamma \circ
\cT_{\eta \to \theta}^\gamma$ has a unique fixed point in $\cD_\theta^\gamma$.
These fixed points are given precisely by
\begin{equation}\label{eq:discretefixedpoint}
\begin{gathered}
\bar \alpha_\gamma^t=\bar{\tilde\alpha}_\gamma^t=\alpha_\gamma^{[t]},
\quad \bar C_\theta^\gamma(t,s)=C_\theta^\gamma([t],[s]), \quad
\bar C_\theta^\gamma(t,*)=C_\theta^\gamma([t],*),
\quad \bar C_\eta^\gamma(t,s)=C_\eta^\gamma([t],[s]),\\
\bar R_\theta^\gamma(t,s)=\begin{cases}
1 & \text{ if } [s]=[t] \\
\frac{1}{\gamma}\,R_\theta^\gamma([t],[s]) & \text{ if } [s]<[t],
\end{cases}
\quad \bar R_\eta^\gamma(t,s)=
\begin{cases}
\delta \beta^2 & \text{ if } [s]=[t] \\
\frac{1}{\gamma}\,R_\eta^\gamma([t],[s]) & \text{ if } [s]<[t]
\end{cases}
\end{gathered}
\end{equation}
for all $0 \leq s \leq t \leq T$, where
$(\alpha_\gamma,C_\theta^\gamma,C_\eta^\gamma,R_\theta^\gamma,R_\eta^\gamma)$
are the components of the discrete DMFT system defined iteratively in time
via (\ref{eq:dmft_CRtheta_discrete}) and (\ref{eq:dmft_CReta_discrete}).
\item For any $\gamma>0$ sufficiently small,
we have $\cT_{\eta \to \theta}^\gamma(\cD_\eta^\gamma \cap \cS_\eta)
\subseteq \cD_\theta^\gamma \cap \cS_\theta$,
$\cT_{\theta \to \eta}^\gamma(\cD_\theta^\gamma \cap \cS_\theta)
\subseteq \cD_\eta^\gamma \cap \cS_\eta$, and the fixed point
(\ref{eq:discretefixedpoint}) belongs to $\cS$.
\end{enumerate}
\end{lemma}
\begin{proof}
For (a), if $X^\gamma=(\bar R_\eta^\gamma,\bar C_\eta^\gamma,\bar \alpha_\gamma) \in \cD_\eta^\gamma$, then $\{\bar u_\gamma^t\}$ is also piecewise constant and
right-continuous by these properties of $\bar C_\eta^\gamma$. Then
an easy induction on $k$ shows that
(\ref{def:dmft_piecelinear_theta}) has a unique solution over
$t \in [0,k\gamma)$ for each integer $k \geq 1$, which is given by
$\bar\theta_\gamma^t=\bar\theta_\gamma^{\floor{t}}$. By definition,
(\ref{def:derivative_process_piecelinear_1}) is given by $\frac{\partial
\bar\theta_\gamma^t}{\partial \bar u_\gamma^s}=1$ for all $s \geq 0$ and
$t \in [s,\ceil{s})$. Then for each $s \geq 0$, an induction on $k$
shows also that (\ref{def:derivative_process_piecelinear_1}) has a unique
solution on $[s,\ceil{s}+k\gamma)$ for each integer $k \geq 1$,
which is given by $\frac{\partial \bar\theta_\gamma^t}{\partial \bar u_\gamma^s}
=\frac{\partial \bar\theta_\gamma^{\floor{t}}}{\partial \bar
u_\gamma^s}$, and furthermore this solution depends on $s$ only via $\floor{s}$,
i.e.\ $\frac{\partial
\bar\theta_\gamma^t}{\partial \bar u_\gamma^s}=\frac{\partial
\bar\theta_\gamma^{\floor{t}}}{\partial \bar u_\gamma^{\floor{s}}}$.
Thus the solutions of
(\ref{def:dmft_piecelinear_theta}--\ref{def:derivative_process_piecelinear_1})
are piecewise constant and right-continuous,
implying the same properties for $\bar R_\theta^\gamma,\bar
C_\theta^\gamma,\bar{\tilde \alpha}$ defined by (\ref{def:CRtheta_piecelinear}).
This shows (a).

Part (b) follows from analogous inductive arguments, using that if
$Y^\gamma=(\bar R_\theta^\gamma,\bar C_\theta^\gamma,\bar{\tilde \alpha}_\gamma)
\in \cD_\theta^\gamma$, then $\{\bar w_\gamma^t\}$ is also
piecewise constant and right-continuous, and hence so are $\{\bar\eta_\gamma^t\}$ and $\{\frac{\bar\eta_\gamma^t}{\bar w_\gamma^s}\}$.

Part (c) also follows by induction: Since any fixed point is piecewise constant,
it suffices to consider the values at $\gamma \Z_+$.
By (\ref{def:dmft_piecelinear_theta}),
$\bar \theta^0_\gamma=\theta^0$. Then by (\ref{def:CRtheta_piecelinear}),
\[\bar R_\theta^\gamma(0,0)=1,
\quad \bar C_\theta^\gamma(0,0)=\E[(\theta^0)^2]=C_\theta^\gamma(0,0),
\quad \bar C_\theta^\gamma(0,*)=\E[\theta^0\theta^*]=C_\theta^\gamma(0,*),
\quad \bar{\tilde \alpha}_\gamma^0=0.\]
Then by (\ref{def:dmft_piecelinear_eta}),
$\bar \eta_\gamma^0={-}\bar w_\gamma^0$,
so $(\bar\eta_\gamma^0,\bar w_\gamma^*)$ is equal in joint law
to the discrete DMFT variables $(\eta_\gamma^0,w_\gamma^*)$. Then
by (\ref{def:CReta_piecelinear}), any fixed point must satisfy
\[\bar R_\eta^\gamma(0,0)=\delta \beta^2,
\quad \bar C_\eta^\gamma(0,0)=\delta \beta^2
\E[(\eta_\gamma^0+w_\gamma^*-\eps)^2]=C_\eta(0,0),
\quad \bar \alpha_\gamma^0=0.\]

Suppose inductively that there is a unique fixed point over times
$s \leq t$ in $\{0,\gamma,\ldots,k\gamma\}$, and
consider now $t=(k+1)\gamma$. The equations
(\ref{def:dmft_piecelinear_theta}--\ref{def:derivative_process_piecelinear_1})
and the piecewise constant nature of all processes imply
\begin{align*}
\bar \theta_\gamma^{(k+1)\gamma}&=\bar \theta_\gamma^{k\gamma}
+\int_{k\gamma}^{(k+1)\gamma}\Big[{-}\delta\beta(\bar
\theta^s_{\gamma}-\theta^\ast) + s(\bar \theta^s_\gamma,\bar\alpha^s_\gamma) +
\int_0^{\floor{s}} \bar R^\gamma_\eta(s,r)(\bar \theta_\gamma^r-\theta^\ast)\d r  +
\bar{u}^s_\gamma\Big]\d s + \sqrt{2}(b^{(k+1)\gamma}-b^{k\gamma})\\
&=\bar
\theta_\gamma^{k\gamma}+\gamma\Big({-}\delta\beta(\bar\theta_\gamma^{k\gamma}-\theta^*)
+s(\bar\theta_\gamma^{k\gamma},\bar\alpha_\gamma^{k\gamma})
+\gamma \sum_{\ell=0}^{k-1} \bar{R}_\eta^\gamma(k\gamma,\ell\gamma)
(\bar\theta^{\ell\gamma}_\gamma-\theta^*)+\bar u^{k\gamma}_\gamma\Big)
+\sqrt{2}(b^{(k+1)\gamma}-b^{k\gamma}),
\end{align*}
$\frac{\partial \bar\theta_\gamma^{(k+1)\gamma}}{\partial \bar
u_\gamma^{(k+1)\gamma}}=1$, and for any $j \leq k$,
\begin{align*}
\frac{\partial \bar\theta_\gamma^{(k+1)\gamma}}{\partial \bar u_\gamma^{j\gamma}}
&=\frac{\partial \bar\theta_\gamma^{k\gamma}}{\partial \bar u_\gamma^{j\gamma}}
+\int_{k\gamma}^{(k+1)\gamma}
\bigg[\Big({-}\delta\beta + \partial_\theta
s(\bar \theta_\gamma^r,\bar \alpha^r_\gamma)\Big)\frac{\partial
\bar \theta^r_\gamma}{\partial \bar u^{j\gamma}_\gamma} +
\int_{(j+1)\gamma}^{k\gamma}
\bar R_\eta^\gamma(r,r')\frac{\partial
\bar \theta^{r'}_\gamma}{\partial \bar u^{j\gamma}_\gamma}\d r'\bigg]\d r\\
&=\frac{\partial \bar\theta_\gamma^{k\gamma}}{\partial \bar u_\gamma^{j\gamma}}
+\gamma\bigg[\Big({-}\delta\beta+\partial_\theta s(\bar
\theta_{\gamma}^{k\gamma},\bar\alpha_{\gamma}^{k\gamma})\Big)
\frac{\partial \bar\theta_\gamma^{k\gamma}}{\partial \bar u_\gamma^{j\gamma}}
+\gamma \sum_{\ell=j+1}^{k-1} \bar R_\eta^\gamma(k\gamma,\ell \gamma)
\frac{\partial \bar \theta_\gamma^{\ell \gamma}}{\partial \bar
u_\gamma^{j\gamma}}\bigg]
\end{align*}
Comparing these equations with
(\ref{eq:dmft_theta_discrete1}--\ref{eq:response_1}) and
applying $\gamma \bar R_\eta^\gamma(k\gamma,\ell\gamma)=R_\eta^\gamma(k,\ell)$,
$\bar\alpha_\gamma^{k\gamma}=\alpha_\gamma^k$, and the equality in law
$(\bar u_\gamma^0,\ldots,\bar u_\gamma^{k\gamma})\overset{L}{=}
(u_\gamma^0,\ldots,u_\gamma^k)$ by the induction hypothesis, this shows the
equality in law
\[\{\theta^*,\bar \theta_\gamma^{i\gamma},
\tfrac{\partial \bar\theta_\gamma^{i\gamma}}{\partial \bar
u_\gamma^{j\gamma}}\}_{i<j \leq k+1}
\overset{L}{=}\{\theta^*,\theta_\gamma^i,
\gamma^{-1}\tfrac{\partial \theta_\gamma^i}{\partial u_\gamma^j}\}_{i<j \leq k+1}.\]
Then (\ref{eq:discretefixedpoint}) holds for the components
$\bar R_\theta^\gamma,\bar C_\theta^\gamma,\bar{\tilde\alpha}_\gamma$ of any fixed
point up to times
$s \leq t$ in $\{0,\gamma,\ldots,(k+1)\gamma\}$.

Now the equations 
(\ref{def:dmft_piecelinear_eta}--\ref{def:derivative_process_piecelinear_3})
imply
\begin{align*}
\bar \eta_\gamma^{(k+1)\gamma}
&={-}\beta \int_0^{(k+1)\gamma} \bar R^\gamma_\theta((k+1)\gamma,s)(\bar
\eta^s_\gamma + \bar w^\ast_\gamma - \eps)\d s - \bar w^{(k+1)\gamma}_\gamma\\
&={-}\beta\gamma \sum_{j=0}^k \bar R^\gamma_\theta((k+1)\gamma,j\gamma)
\big(\bar \eta^{j\gamma}_\gamma+\bar w^*_\gamma-\eps\big)
- \bar w^{(k+1)\gamma}_\gamma,
\end{align*}
$\frac{\partial \bar \eta^{(k+1)\gamma}_\gamma}{\partial \bar
w^{(k+1)\gamma}_\gamma}=\delta\beta^2$, and for all $j \leq k$,
\begin{align*}
\frac{\partial \bar \eta^{(k+1)\gamma}_\gamma}{\partial \bar
w^{j\gamma}_\gamma}&=\beta \bar R_\theta((k+1)\gamma,j\gamma)-
\beta \int_{(j+1)\gamma}^{(k+1)\gamma} \bar R^\gamma_\theta((k+1)\gamma,r)
\frac{\partial \bar \eta^r_\gamma}{\partial \bar w^{j\gamma}_\gamma}\d r\\
&=\beta \bar R_\theta((k+1)\gamma,j\gamma)-
\beta\gamma \sum_{\ell=j+1}^k \bar R^\gamma_\theta((k+1)\gamma,\ell\gamma)
\frac{\partial \bar \eta^{\ell \gamma}_\gamma}{\partial \bar
w^{j\gamma}_\gamma}.
\end{align*}
Comparing these equations with
(\ref{eq:dmft_eta_discrete1}--\ref{eq:response_2}) and
applying again the induction hypothesis, this shows the equality in law
\[\{\bar \eta_\gamma^{i\gamma},
\tfrac{\partial \bar\eta_\gamma^{i\gamma}}{\partial \bar
w_\gamma^{j\gamma}}\}_{i<j \leq k+1}
\overset{L}{=}\{\eta_\gamma^i,
\gamma^{-1}\tfrac{\partial \eta_\gamma^i}{\partial w_\gamma^j}\}_{i<j\leq k+1}.\]
Then (\ref{eq:discretefixedpoint}) also holds for the components
$\bar R_\eta^\gamma,\bar C_\eta^\gamma,\bar\alpha_\gamma$ of any fixed point
up to times $s \leq t$ in $\{0,\gamma,\ldots,(k+1)\gamma\}$,
completing the induction. Thus any fixed points of $\cT_{\eta \to \eta}^\gamma$
and $\cT_{\theta \to \theta}^\gamma$ must satisfy
(\ref{eq:discretefixedpoint}) for all $0 \leq s \leq t \leq T$,
implying also that such fixed points are unique in $\cD_\eta^\gamma$ and
$\cD_\theta^\gamma$ by uniqueness of the iterative
construction (\ref{eq:iterativeDMFTconstruction}) of the solution to the
discrete DMFT equations. This shows (c).

For (d), we check that $\cT_{\eta \to \eta}^\gamma$ and $\cT_{\theta \to
\theta}^\gamma$ define contractive mappings on
$\cD_\eta^\gamma \cap \cS_\eta$ and $\cD_\theta^\gamma \cap \cS_\theta$.
Note that if $Y^\gamma=(\bar R_\theta^\gamma,\bar
C_\theta^\gamma,\bar{\tilde\alpha}_\gamma) \in \cD_\theta^\gamma \cap \cS_\theta$
and $X^\gamma=\cT_{\theta \to \eta}^\gamma(Y^\gamma)=(\bar R_\eta^\gamma,\bar
C_\eta^\gamma,\bar\alpha_\gamma)$, then
setting $\bar\xi_\gamma^t=\bar\eta_\gamma^t+\bar w^*_\gamma-\eps$, the
same arguments as in Lemma \ref{lem:Trange_theta_2_eta} show
\begin{align*}
\E(\bar \xi_\gamma^t)^2 &\leq 2\Big[\beta^2\int_0^{\floor{t}} (t-s+1)^2
\cdot
\Phi_{R_\theta}^2(t-s)\E(\bar \xi_\gamma^s)^2\d s+2\Phi_{C_\theta}(t)+2\tau_*^2+\sigma^2\Big],\\
|\bar R_\eta^\gamma(t,s)| &\leq |\beta| \Big(\1\{\ceil{s} \leq \floor{t}\}
\int_{\ceil{s}}^{\floor{t}} \Phi_{R_\theta}(t-s')|\bar R_\eta^\gamma(s',s)|\d s'
+\delta |\beta| \Phi_{R_\theta}(t-s)\Big).
\end{align*}
Upper-bounding these integrals $\int_0^{\floor{t}}$ and
$\int_{\ceil{s}}^{\floor{t}}$ by $\int_0^t$ and $\int_s^t$, we obtain as in
Lemma \ref{lem:Trange_theta_2_eta} that 
$\bar C_\eta^\gamma(t,t) \leq \Phi_{C_\gamma}(t)$ and
$\bar R_\eta^\gamma(t,s) \leq \Phi_{R_\gamma}(t-s)$.
All continuity conditions
defining $\cS_\eta$ are automatically satisfied since the components of
$X^\gamma$ are piecewise constant outside the knots $D=[0,T] \cap \gamma \Z_+$.
Thus $X^\gamma \in \cS_\eta$, i.e.\
$\cT_{\theta \to \eta}^\gamma$ maps $\cD_\theta^\gamma \cap \cS_\theta$ into
$\cD_\eta^\gamma \cap \cS_\eta$.

Conversely, suppose
$X^\gamma=(\bar R_\eta^\gamma,\bar C_\eta^\gamma,\bar\alpha) \in
\cD_\eta^\gamma \cap \cS_\eta$, and let
$Y^\gamma=\cT_{\theta \to \eta}^\gamma(\bar R_\theta^\gamma,\bar
C_\theta^\gamma,\bar{\tilde\alpha})$. A small extension of the argument
in Lemma \ref{lem:Trange_eta_2_theta} shows $Y^\gamma \in \cS_\theta$. Let us
explain this extension for $\bar C_\theta^\gamma$: Defining
\[\bar v_\gamma^t={-}\delta\beta(\bar\theta_\gamma^t-\theta^*)
+s(\bar\theta_\gamma^t,\bar\alpha_\gamma^t)+\int_0^{\floor{t}}
\bar R_\eta^\gamma(t,s)(\bar\theta_\gamma^s-\theta^*)\d s+\bar u_\gamma^t,\]
we have $\bar\theta_\gamma^{t+\gamma}=\bar\theta_\gamma^t+
\gamma\cdot\bar v_\gamma^t+\sqrt{2}(b^{t+\gamma}-b^t)$ for $t \in \gamma \Z_+$.
Then, analogous to the calculation using Ito's formula in
Lemma \ref{lem:Trange_eta_2_theta}, for sufficiently small $\gamma>0$,
\begin{align*}
\bar C_\theta^\gamma(t+\gamma,t+\gamma)-\bar C_\theta^\gamma(t,t)
&=2\gamma \E \bar\theta_\gamma^t \bar v_\gamma^t
+\gamma^2 \E(\bar v_\gamma^t)^2+2\gamma\\
&\leq \frac{\gamma}{1-\gamma}\E(\bar\theta_\gamma^t)^2
+[\gamma(1-\gamma)+\gamma^2]\E(\bar v_\gamma^t)^2+2\gamma\\
&\leq \gamma\Big(1.1\E(\bar \theta_\gamma^t)^2+\E(\bar v_\gamma^t)^2+2\Big)
=\int_t^{t+\gamma}
\Big(1.1\E(\bar \theta_\gamma^s)^2+\E(\bar v_\gamma^s)^2+2\Big)\d s,
\end{align*}
the last equality holding because $\bar v_\gamma$ and $\bar \theta_\gamma$ are
piecewise constant. Summing this inequality shows
\[\bar C_\theta^\gamma(t,t)-\bar C_\theta^\gamma(0,0)
\leq \int_0^t \Big(1.1\E(\bar \theta_\gamma^s)^2+\E(\bar v_\gamma^s)^2+2\Big)\d
s\]
for all $t \in [0,T]$. Then,
bounding $\E(\bar v_\gamma^s)^2$ as in (\ref{eq:vtbound}) of
Lemma \ref{lem:Trange_eta_2_theta}, this shows that
$\bar C_\theta^\gamma(t,t) \leq \Phi_{C_\theta}(t)$. Similar extensions of the
arguments in Lemma \ref{lem:Trange_eta_2_theta}
show that $\bar R_\theta^\gamma(t,s) \leq \Phi_{R_\theta}(t-s)$
and $\|\bar \alpha_\gamma^t\|^2 \leq \Phi_{\alpha}(t)$, so
$Y^\gamma \in \cS_\theta$ as claimed. Then
$\cT_{\eta \to \theta}^\gamma$ maps $\cD_\eta^\gamma \cap \cS_\eta$ into
$\cD_\theta^\gamma \cap \cS_\theta$.

The same argument as in Lemmas \ref{lem:modulus_eta_2_theta} and
\ref{lem:modulus_theta_2_eta} bound the moduli-of-continuity of
$\cT_{\eta \to \theta}^\gamma$ and
$\cT_{\theta \to \eta}^\gamma$ in the metrics $d(\cdot)$ of
Section \ref{sec:contractivemapping}, implying that
$\cT_{\eta \to \eta}^\gamma$ and $\cT_{\theta \to \theta}^\gamma$ are
contractive for sufficiently large $\lambda>0$ defining $d(\cdot)$.
These metrics induce the topologies of uniform
convergence on the spaces $\cD_\eta^\gamma \cap \cS_\eta$
and $\cD_\theta^\gamma \cap \cS_\theta$, which are equivalent to closed subsets of finite-dimensional vector spaces and hence also complete. Then $\cT_{\eta \to \eta}^\gamma$ and
$\cT_{\theta \to \theta}^\gamma$ have unique fixed points
in $\cD_\eta^\gamma \cap \cS_\eta$ and $\cD_\theta^\gamma \cap \cS_\theta$ 
by the Banach fixed-point theorem. These must coincide with the fixed point
(\ref{eq:discretefixedpoint}), by the uniqueness statement (without
restriction to $\cS_\theta$ and $\cS_\eta$) shown in part (c). Thus
this fixed point (\ref{eq:discretefixedpoint}) belongs to $\cS$.
\end{proof}

\begin{lemma}\label{lem:gamma_mod_map1}
There exists a constant $C>0$ (depending on $T$ but not on $\lambda,\gamma$)
such that for all large enough $\lambda>0$
defining the metrics (\ref{eq:metrics}) and all sufficiently small $\gamma>0$,
\begin{enumerate}[(a)]
\item For any $X^\gamma \in \cD_\eta^\gamma \cap \cS_\eta$,
\[d\big(\cT_{\eta\rightarrow\theta}(X^\gamma),\cT_{\eta\rightarrow\theta}^\gamma(X^\gamma)\big) \leq C\sqrt{\gamma}.\]
\item For any $Y^\gamma \in \cD_\theta^\gamma \cap \cS_\theta$,
\[d\big(\cT_{\theta\rightarrow\eta}(Y^\gamma),\cT_{\theta\rightarrow\eta}^\gamma(Y^\gamma)\big) \leq C\sqrt{\gamma}.\]
\end{enumerate}
\end{lemma}
\begin{proof}
We show part (a). Consider any
$X^\gamma=(\bar R_\eta^\gamma, \bar C_\eta^\gamma,\bar\alpha_\gamma)$, and
denote
$\cT_{\eta\rightarrow\theta}(X^\gamma)=(R_\theta,C_\theta,\tilde\alpha)$
and $\cT_{\eta\rightarrow\theta}^\gamma(X^\gamma)=(\bar R_\theta^\gamma,\bar
C_\theta^\gamma,\bar{\tilde\alpha}_\gamma)$. Throughout, $C,C'>0$ denote constants
depending on $T$ but not on $\lambda,\gamma$
and changing from instance to instance.\\

\noindent \textbf{Bound of $d(C_\theta,\bar C_\theta^\gamma)$.}
Given $X^\gamma$, let us couple the evolutions
(\ref{def:dmft_langevin_cont_theta}) and (\ref{def:dmft_piecelinear_theta})
by a common realization of $\{\bar u_\gamma^t\}$ with covariance $\bar
C_\eta^\gamma$ and a common Brownian motion. Then
by definition, we have
\begin{align*}
\theta^t &= \theta^0 + \int_0^t \Big[-\delta\beta(\theta^s-\theta^*) +
s(\theta^s,\bar \alpha_\gamma^s) + \int_0^s \bar R_\eta^\gamma(s,
s')(\theta^{s'} - \theta^\ast)\d s'  + \bar u_\gamma^s\Big] \d s + \sqrt{2}\,b^t,\\
\bar \theta^t_\gamma &= \theta^0 + \int_0^{\floor{t}}\Big[-\delta\beta(\bar
\theta^s_{\gamma} - \theta^\ast) + s(\bar \theta^s_\gamma, \bar \alpha^s_\gamma)
+\int_0^{\floor{s}} \bar R^\gamma_\eta(s, s')(\bar \theta_\gamma^{s'} -
\theta^\ast)\d s' + \bar u_\gamma^s\Big]\d s + \sqrt{2}\,b^{\floor{t}}.
\end{align*}
Then $\E(\theta^t - \bar \theta^t_\gamma)^2 \leq 6[(I)+(II)+(III)+(IV)+(V)+(VI)]$ where
\begin{align*}
(I)&=\E\Big(\int_0^{\floor{t}}
\delta\beta(\theta^s-\bar\theta^s_\gamma)\d s\Big)^2,\\
(II)&=\E\Big(\int_0^{\floor{t}} (s(\theta^s,\bar\alpha_\gamma^s)
-s(\bar \theta^s_\gamma, \bar\alpha_\gamma^s) \d s\Big)^2,\\
(III)&=\E\Big(\int_0^{\floor{t}} \int_0^{\floor{s}} \bar R_\eta^\gamma(s,s') (\theta^{s'}
- \bar \theta_\gamma^{s'})\d s'\,\d s\Big)^2,\\
(IV)&= \E\Big(\int_0^{\floor{t}} \int_{\floor{s}}^s \bar R_\eta^\gamma(s,s') (\theta^{s'}
- \theta^\ast)\d s'\,\d s\Big)^2,\\
(V)&= \E\Big(\int_{\floor{t}}^t \Big[-\delta\beta(\theta^s-\theta^*) +
s(\theta^s,\bar \alpha_\gamma^s) + \int_0^s \bar R_\eta^\gamma(s,
s')(\theta^{s'} - \theta^\ast)\d s'  + \bar u_\gamma^s\Big] \d s\Big)^2,\\
(VI)&=\E(\sqrt{2}\,b^t-\sqrt{2}\,b^{\floor{t}})^2.
\end{align*}
By the same arguments as in the proof of Lemma \ref{lem:modulus_eta_2_theta},
using the Lipschitz continuity of $s(\cdot)$ in Assumption \ref{assump:prior},
we may show
\[(I)+(II)+(III) \leq 
\frac{C}{\lambda}e^{2\lambda t}\sup_{s\in[0,T]} e^{-2\lambda s}
\E(\theta^s-\bar \theta^s_\gamma)^2.\]
Applying $s-\floor s \leq \gamma$, $t-\floor t \leq \gamma$,
and the bounds for $\bar R_\eta^\gamma,\bar C_\eta^\gamma$ implied by $X^\gamma
\in \cS_\eta$, we have
\[(IV)+(V)+(VI) \leq C\gamma.\]
Then
\[\sup_{t\in[0,T]} e^{-2\lambda t}
\E(\theta^t-\bar \theta^t_\gamma)^2
\leq \frac{C}{\lambda} \sup_{t\in[0,T]} e^{-2\lambda t}
\E(\theta^t-\bar \theta^t_\gamma)^2+C\gamma,\]
and choosing large enough $\lambda>0$ yields
\begin{align}\label{eq:theta_L2_bound}
\sup_{t\in[0,T]} e^{-\lambda t}\sqrt{\E(\theta^t-\bar \theta^{t}_\gamma)^2} \leq
C'\sqrt{\gamma}.
\end{align}
This implies as in the proof of Lemma \ref{lem:modulus_eta_2_theta}
that $d(C_\theta,\bar C_\theta^\gamma) \leq C'\sqrt{\gamma}$.\\

\noindent \textbf{Bound of $d(R_\theta,\bar R_\theta^\gamma)$.} Denote by
$r_\theta(t,s) = \frac{\partial \theta^t}{\partial u^s}$ and
$\bar r_\theta^\gamma(t,s) = \frac{\partial \bar \theta^t_\gamma}{\partial \bar
u^s_\gamma}$ the processes (\ref{def:response_theta}) and
(\ref{def:derivative_process_piecelinear_1}) defined from the above coupling of
$\{\theta^t\}$ and $\{\bar\theta_\gamma^t\}$. Then by definition, we have
\begin{align*}
r_{\theta}(t,s) &= 1 + \int_s^t \bigg[\Big({-}\delta\beta + \partial_\theta
s(\theta^{s'},\bar \alpha_\gamma^{s'})\Big)r_\theta(s',s)
+\int_s^{s'} \bar R_\eta^\gamma(s',s'')r_\theta(s'',s)\d s''\bigg]\d s',\\
\bar r^\gamma_{\theta}(t,s) &= 1 + \1\{\ceil{s} \leq \floor{t}\}
\int_{\ceil{s}}^{\floor{t}} \bigg[\Big({-}\delta\beta + \partial_\theta s(\bar
\theta^{s'}_\gamma,\bar \alpha_\gamma^{s'})\Big)\bar r_\theta^\gamma(s',s)
+\int_{\ceil{s}}^{\floor{s'}} \bar R_\eta^\gamma(s',s'')\bar
r^\gamma_\theta(s'',s)\d s''\bigg]\d s'.
\end{align*}
Hence $\E|r_{\theta}(t,s) - \bar r^\gamma_{\theta}(t,s)| \leq
(I)+(II)+(III)+(IV)$ where
\begin{align*}
(I)&=\E\bigg[\1\{\ceil{s} \leq \floor{t}\}\int_{\ceil{s}}^{\floor{t}}
\Big|\Big({-}\delta\beta+\partial_\theta s(\theta^{s'},\bar
\alpha_\gamma^{s'})\Big)r_\theta(s',s)
-\Big({-}\delta\beta+\partial_\theta s(\bar \theta^{s'}_\gamma,\bar
\alpha^{s'}_\gamma)\Big)\bar r_\theta^\gamma(s',s)\Big|\d s'\bigg],\\
(II)&=\E\bigg[\1\{\ceil{s} \leq \floor{t}\}\int_{\ceil{s}}^{\floor{t}}
\int_{\ceil{s}}^{\floor{s'}} \Big|\bar
R_\eta^\gamma(s',s'')(r_\theta(s'',s)-\bar r^\gamma_\theta(s'',s))\Big|\d
s''\,\d s'\bigg],\\
(III)&=\E\bigg[\1\{\ceil{s} \leq \floor{t}\}\int_{\ceil{s}}^{\floor{t}}
\int_{(s,\ceil{s}) \cup (\floor{s'},s')}
\Big|\bar R_\eta^\gamma(s',s'')r_\theta(s'',s)\Big|\d s''\,\d s'\bigg],\\
(IV)&=\int_{(s,\ceil{s}) \cup (\floor{t},t)}
\bigg|\Big({-}\delta\beta + \partial_\theta
s(\theta^{s'},\bar \alpha_\gamma^{s'})\Big)r_\theta(s',s)
+\int_s^{s'} \bar R_\eta^\gamma(s',s'')r_\theta(s'',s)\d s''\bigg|\d s'.
\end{align*}
By the same arguments as in the proof of Lemma \ref{lem:modulus_eta_2_theta},
using the above bound (\ref{eq:theta_L2_bound}) and Lipschitz continuity of
$\partial_\theta s(\cdot)$ in Assumption \ref{assump:prior}, we may show
\[(I)+(II) \leq \frac{C}{\lambda}e^{\lambda t}\Big(\sup_{0 \leq s \leq t \leq T}
e^{-\lambda t}\E|r_{\theta}(t,s) - \bar r^\gamma_{\theta}(t,s)|\Big)
+C\sqrt{\gamma}.\]
Applying $\ceil s-s \leq \gamma$, $s'-\floor{s'} \leq \gamma$, and $t-\floor t \leq \gamma$, we have
\[(III)+(IV) \leq C\gamma.\]
Then
\[\sup_{0 \leq s \leq t \leq T}
e^{-\lambda t}\E|r_{\theta}(t,s) - \bar r^\gamma_{\theta}(t,s)|
\leq \frac{C}{\lambda} \sup_{0 \leq s \leq t \leq T}
e^{-\lambda t}\E|r_{\theta}(t,s) - \bar r^\gamma_{\theta}(t,s)|
+C\sqrt{\gamma},\]
and choosing large enough $\lambda>0$ and rearranging gives
\begin{align*}
d(R_\theta, \bar R_\theta^\gamma) \leq \sup_{0\leq s\leq t \leq T} e^{-\lambda
t}\E|r_\theta(t,s) - \bar r^\gamma_\theta(t,s)| \leq C\sqrt{\gamma}.
\end{align*}

\noindent {\bf Bound of $d(\tilde \alpha,\bar{\tilde \alpha}_\gamma)$.} By definition,
\[\tilde \alpha^t=\alpha^0+\int_0^t \cG(\tilde \alpha^s,\sP(\theta^s))\d s,
\qquad \bar{\tilde \alpha}_\gamma^t=\alpha^0+\int_0^{\floor{t}} \cG(\bar{\tilde
\alpha}_\gamma^s,\sP(\bar \theta_\gamma^s))\d s,\]
so $\|\tilde \alpha^t-\bar{\tilde \alpha}_\gamma^t\| \leq (I)+(II)$ where
\[(I)=\int_0^{\floor{t}} \Big\|\cG(\tilde \alpha^s,\sP(\theta^s))-
\cG(\bar{\tilde \alpha}_\gamma^s,\sP(\bar \theta_\gamma^s))\Big\|\d s,
\qquad (II)=\int_{\floor{t}}^t \Big\|\cG(\tilde \alpha^s,\sP(\theta^s))\Big\|\d s.\]
By the same arguments as in the proof of Lemma \ref{lem:modulus_eta_2_theta},
using the above bound (\ref{eq:theta_L2_bound})
and the Lipschitz continuity of $\cG(\cdot)$ in Assumption
\ref{assump:gradient}, we have
\[(I) \leq \frac{C}{\lambda}e^{\lambda t} \sup_{s \in [0,T]} e^{-\lambda s}
\|\tilde \alpha^s-\bar{\tilde \alpha}_\gamma^s\|\,\d s+C\sqrt{\gamma}\]
Using $t-\floor t \leq \gamma$, we have $(II) \leq C\gamma$. So choosing
$\lambda>0$ large enough and rearranging shows
\[d(\tilde \alpha,\bar{\tilde \alpha}_\gamma)
=\sup_{t \in [0,T]} e^{-\lambda t}\|\tilde \alpha^t-\bar{\tilde \alpha}_\gamma^t\|
\leq C\sqrt{\gamma}.\]
This concludes the proof of (a). The proof of (b) is analogous, and we omit this
for brevity.
\end{proof}

\begin{lemma}\label{lem:dmft_discretization_error}
Let $\{\theta^t\}_{t \in [0,T]}$, $\{\eta^t\}_{t \in [0,T]}$, and
$\{\alpha^t\}_{t \in [0,T]}$ be the components of the solution to the DMFT
system in Theorem \ref{thm:dmftsolexists}, and let
$\{\bar \theta_\gamma^t\}_{t \in [0,T]}$,
$\{\bar \eta_\gamma^t\}_{t \in [0,T]}$,
$\{\bar \alpha_\gamma^t\}_{t \in [0,T]}$ be defined from the components of
the fixed point (\ref{eq:discretefixedpoint}) via
(\ref{def:dmft_piecelinear_theta}) and (\ref{def:dmft_piecelinear_eta}).
Then for any fixed $m \geq 0$ and $t_1,\ldots,t_m \in [0,T]$,
as $\gamma \to 0$,
\begin{align}
\sP(\theta^*,\bar \theta^{t_1}_\gamma,\ldots,\bar \theta^{t_m}_\gamma)
&\overset{W_2}{\to}
\sP(\theta^*,\theta^{t_1},\ldots,\theta^{t_m})\label{eq:discretedmftthetaconv}\\
\sP(\bar w_\gamma^*,\eps,\bar \eta_\gamma^{t_1},\ldots,\bar \eta_\gamma^{t_m}) &\overset{W_2}{\to}
\sP(w^*,\eps,\eta^{t_1},\ldots,\eta^{t_m})\label{eq:discretedmftetaconv}\\
\{\bar \alpha^t_\gamma\}_{t \in [0,T]} &\rightarrow \{\alpha^t\}_{t \in [0,T]}
\label{eq:discretedmftalphaconv}
\end{align}
where (\ref{eq:discretedmftalphaconv}) holds in the sense of uniform
convergence on $C([0,T],\R^K)$.
\end{lemma}
\begin{proof}
Let $X^\gamma=(\bar R_\eta^\gamma,\bar C_\eta^\gamma,\bar\alpha_\gamma)$
and $Y^\gamma=(\bar R_\theta^\gamma,\bar C_\theta^\gamma,\bar\alpha_\gamma)$
be the components of the fixed point (\ref{eq:discretefixedpoint}),
and let $X=(R_\eta,C_\eta,\alpha)$ and $Y=(R_\theta,C_\theta,\alpha)$
be those of the unique solution to the continuous DMFT system prescribed by
Theorem \ref{thm:dmftsolexists}.
Let $d(\cdot)$ denote the metrics introduced in (\ref{eq:dist_Seta}) and
(\ref{eq:dist_Stheta}), for a sufficiently large choice of $\lambda>0$.
By Lemma \ref{lem:discrete_dmft_rewrite}, 
$X^\gamma \in \cD_\eta^\gamma \cap \cS_\eta$
for all sufficiently small $\gamma>0$, so $\cT_{\eta \to \eta}(X^\gamma)$ is
well-defined. Then, applying the fixed point conditions
for $X^\gamma$ and $X$,
\begin{align*}
d(X, X^\gamma) = d\big(\cT_{\eta\rightarrow\eta}(X), \cT_{\eta\rightarrow\eta}^\gamma(X^\gamma)\big) \leq d\big(\cT_{\eta\rightarrow\eta}(X), \cT_{\eta\rightarrow\eta}(X^\gamma)\big) + d\big(\cT_{\eta\rightarrow\eta}(X^\gamma), \cT_{\eta\rightarrow\eta}^\gamma(X^\gamma)\big).
\end{align*}
By Lemmas \ref{lem:modulus_eta_2_theta} and \ref{lem:modulus_theta_2_eta},
$\cT_{\eta \to \eta}$ is a contraction on $\cS_\eta$
for large enough $\lambda>0$, for which
the first term satisfies $d\big(\cT_{\eta\rightarrow\eta}(X),
\cT_{\eta\rightarrow\eta}(X^\gamma)\big) \leq \frac{1}{2}d(X,X^\gamma)$.
Thus, rearranging shows
\[d(X,X^\gamma) \leq 2d\big(\cT_{\eta\rightarrow\eta}(X^\gamma),
\cT_{\eta\rightarrow\eta}^\gamma(X^\gamma)\big)
=2d\big(\cT_{\theta\rightarrow\eta} \circ \cT_{\eta \to \theta}(X^\gamma),
\cT_{\theta\rightarrow\eta}^\gamma \circ \cT_{\eta \to
\theta}^\gamma(X^\gamma)\big).\]
Letting $Y'=\cT_{\eta\rightarrow\theta}(X^\gamma) \in \cS_\theta$ and
$Y^\gamma=\cT_{\eta\rightarrow\theta}^\gamma(X^\gamma) \in \cD_\theta^\gamma \cap
\cS_\theta$, this shows
\[d(X,X^\gamma) \leq 2d\big(\cT_{\theta\rightarrow\eta}(Y'),
\cT_{\theta\rightarrow\eta}^\gamma(Y^\gamma)\big)
\leq 2d\big(\cT_{\theta\rightarrow\eta}(Y'),
\cT_{\theta\rightarrow\eta}(Y^\gamma)\big)
+2d\big(\cT_{\theta\rightarrow\eta}(Y^\gamma),
\cT_{\theta\rightarrow\eta}^\gamma(Y^\gamma)\big).\]
By Lemma \ref{lem:gamma_mod_map1}, $d(Y',Y^\gamma)
=d(\cT_{\eta\rightarrow\theta}(X^\gamma),\cT_{\eta\rightarrow\theta}^\gamma(X^\gamma))
\leq C\sqrt{\gamma}$. Then by Lemma
\ref{lem:modulus_theta_2_eta}, the first term is bounded as
$d(\cT_{\theta\rightarrow\eta}(Y'),\cT_{\theta\rightarrow\eta}(Y^\gamma))
\leq Cd(Y',Y^\gamma) \leq C'\sqrt{\gamma}$.
By Lemma \ref{lem:gamma_mod_map1}, the second term is also bounded as
$d(\cT_{\theta\rightarrow\eta}(Y^\gamma),
\cT_{\theta\rightarrow\eta}^\gamma(Y^\gamma)) \leq C\sqrt{\gamma}$. So combining
these statements and taking $\gamma \to 0$ shows
\begin{equation}\label{eq:XYdiscretization}
\lim_{\gamma \to 0} d(X,X^\gamma)=0,
\qquad \lim_{\gamma \to 0} d(Y,Y^\gamma)=0.
\end{equation}

By definition of the metrics $d(\cdot)$, the
convergence (\ref{eq:XYdiscretization}) implies the uniform convergence
statement (\ref{eq:discretedmftalphaconv}). It also implies
\[\lim_{\gamma \to 0}
\sup_{0 \leq s \leq t \leq T} |C_\eta(t,s)-\bar C_\eta^\gamma(t,s)|=0.\]
We recall that Theorem \ref{thm:dmftsolexists} shows
$(R_\eta,C_\eta,\alpha) \in \cS_\eta^\text{cont}$, for which
the continuity property (\ref{eq:Ceta_cond_2}) holds for all $0 \leq s \leq t
\leq T$. Then there exists a coupling of $\{u^t\}_{t \in [0,T]}$ and $\{\bar
u^t_\gamma\}_{t \in [0,T]}$ with covariance kernels $C_\eta(t,s)$ and $\bar
C_\eta^\gamma(t,s)$ for which
\[\lim_{\gamma \to 0} \sup_{t \in [0,T]} \E(u^t-\bar u^t_\gamma)^2=0,\]
see e.g.\ \cite[Lemma D.1]{paper2}. Defining $\{\theta^t\}$ and $\{\bar
\theta_\gamma^t\}$ by this coupling of $\{u^t\}$ and $\{\bar u_\gamma^t\}$
and a common Brownian motion, the same arguments as leading
to (\ref{eq:theta_L2_bound}) shows
\[\lim_{\gamma \to 0} \sup_{t \in [0,T]} e^{-\lambda t}
\E(\theta^t-\bar \theta^t_\gamma)^2=0,\]
hence in particular $\lim_{\gamma \to 0} \E(\theta^t-\bar \theta^t_\gamma)^2=0$
for each fixed $t \in [0,T]$ under this coupling, which implies
(\ref{eq:discretedmftthetaconv}). A similar argument shows
(\ref{eq:discretedmftetaconv}).
\end{proof}

\subsection{Step 3: Discretization of Langevin
dynamics}\label{subsec:discretizelangevin}

We now consider a piecewise constant embedding 
$\{\bar \btheta^t_\gamma,\bar{\widehat{\alpha}}^t_\gamma\}_{t \in [0,T]}$
of the discretized Langevin process
(\ref{eq:langevin_discrete_1}--\ref{eq:langevin_discrete_2}), defined as
\[\bar \btheta^t_\gamma=\btheta_\gamma^{[t]},
\qquad \bar{\widehat{\alpha}}^t_\gamma=\widehat{\alpha}_\gamma^{[t]},
\qquad \bar\bbeta_\gamma^t=\X\bar\btheta_\gamma^t\]
where $[t] \in \Z_+$ is as previously defined in (\ref{eq:t_floor_ceil}).
A simple induction shows that this is equivalently the solution to a modification of the
dynamics (\ref{eq:langevin_sde}--\ref{eq:gflow}),
\begin{equation}\label{eq:langevin_embedded}
\begin{aligned}
\bar\btheta_\gamma^t
&=\btheta^0+\int_0^{\floor{t}} \Big[{-}\beta\,\X^\top (\X\bar\btheta_\gamma^s - \y)
+\big(s(\bar \theta_{\gamma,j}^s,\bar{\widehat\alpha}_\gamma^s)\big)_{j=1}^d\Big]\d s +\sqrt{2}\,\b^{\floor{t}},\\
\bar{\widehat\alpha}_\gamma^t&=
\widehat\alpha^0+\int_0^{\floor{t}}
\cG\Big(\bar{\widehat\alpha}_\gamma^s,\frac{1}{d}\sum_{j=1}^d
\delta_{\bar \theta_{\gamma,j}^s}\Big)\d s.
\end{aligned}
\end{equation}
We compare this to the solution
$\{\btheta^t,\widehat\alpha^t\}_{t \geq 0}$ of the original
dynamics (\ref{eq:langevin_sde}--\ref{eq:gflow}), with
$\bbeta^t=\X\btheta^t$, to show the following lemma.

\begin{lemma}\label{lem:langevin_discretize_error}
Let $\{\btheta^t,\bbeta^t,\widehat\alpha^t\}$ be defined by
(\ref{eq:langevin_sde}--\ref{eq:gflow}), and let
$\{\bar\btheta_\gamma^t,\bar\bbeta_\gamma^t,\bar{\widehat\alpha}_\gamma^t\}$
be defined by the piecewise constant process (\ref{eq:langevin_embedded}).
Then for any fixed $m \geq 1$ and $t_1,\ldots, t_m \in [0,T]$,
there exists a function $\iota:\R_+ \to \R_+$
satisfying $\lim_{\gamma \to 0} \iota(\gamma)=0$ such that
almost surely
\begin{align*}
\limsup_{n,d\rightarrow\infty}W_2\Big( \frac{1}{d}\sum_{j=1}^d
\delta_{(\theta_j^*,\theta^{t_1}_j, \ldots, \theta^{t_m}_j)}, \frac{1}{d}\sum_{j=1}^d
\delta_{(\theta_j^*,\bar\theta^{t_1}_{\gamma,j}, \ldots,
\bar\theta_{\gamma,j}^{t_m})}\Big)<\iota(\gamma)\\
\limsup_{n,d\rightarrow\infty}W_2\Big( \frac{1}{n}\sum_{i=1}^n
\delta_{(\eta_i^*,\eps_i,\eta^{t_1}_i, \ldots, \eta^{t_m}_i)}, \frac{1}{n}\sum_{i=1}^n
\delta_{(\eta_i^*,\eps_i,\bar\eta^{t_1}_{\gamma,i}, \ldots,
\bar\eta_{\gamma,i}^{t_m})}\Big)<\iota(\gamma)\\
\limsup_{n,d\rightarrow\infty} \sup_{t \in [0,T]}
\bigpnorm{\widehat{\alpha}^t-\bar{\widehat\alpha}_\gamma^t}{}<\iota(\gamma).
\end{align*}
\end{lemma}

We proceed to prove Lemma \ref{lem:langevin_discretize_error}.

\begin{lemma}\label{lem:BM_maximal}
Let $\{\b^t\}_{t\geq 0}$ be a standard Brownian motion on $\R^d$.
For any fixed $T>0$, there exists a constant $C>0$ depending on $T$ such that
almost surely
\begin{align*}
\limsup_{d \to \infty}\sup_{t\in[0,T]} \frac{1}{\sqrt{d}}\pnorm{\b^t}{} \leq C,
\quad \limsup_{d \to \infty}\sup_{t\in[0,T]}
\frac{1}{\sqrt{d}}\big(\pnorm{\b^t - \b^{\floor{t}}}{} + \pnorm{\b^t -
\b^{\ceil{t}}}{}\big) \leq C\sqrt{\gamma\max(\log(1/\gamma),1)}.
\end{align*}
\end{lemma}
\begin{proof}
We first show that for any $a,b \in \R_+$ with $a\leq b$, we have
$\Prob(\sup_{t\in[a,b]} d^{-1/2}\pnorm{\b^t - \b^a}{} \geq u) \leq
\exp\big({-}\frac{cdu^2}{b-a}\big)$ for any $u \geq \sqrt{4(b-a)}$ and some constant $c>0$. To see
this, for any $\lambda \in (0,\frac{d}{2(b-a)})$, we have
\begin{align*}
\Prob(\sup_{t\in[a,b]} d^{-1/2}\pnorm{\b^t - \b^a}{} \geq u) &= \Prob(\sup_{t\in[a,b]} \exp(\lambda \pnorm{\b^t - \b^a}{}^2/d) \geq \exp(\lambda u^2))\\
&\stackrel{(*)}{\leq} e^{-\lambda u^2} \E[\exp(\lambda \pnorm{\b^b -
\b^a}{}^2/d)] \stackrel{(**)}{=} e^{-\lambda u^2} (1 - 2\lambda(b-a)/d)^{-d/2},
\end{align*}
where $(*)$ applies Doob's maximal inequality for the nonnegative submartingale $\{\exp(\lambda \pnorm{\b^t-\b^a}{}^2/d)\}_{t \in [a,b]}$, and $(**)$
applies the moment generating function of the $\chi^2$
distribution. Choosing $\lambda = cd/(b-a)$ for some small enough $c > 0$
and applying $(1-x) \geq e^{-2x}$ for small $x > 0$, we have 
\begin{align*}
\Prob(\sup_{t\in[a,b]} d^{-1/2}\pnorm{\b^t - \b^a}{} \geq u) \leq
\exp\Big({-}\frac{cdu^2}{b-a} + 2cd\Big) \leq \exp\Big({-}\frac{cdu^2}{2(b-a)}\Big)
\end{align*}
for $u \geq \sqrt{4(b-a)}$, proving the inequality. For the first claim, we
apply this with $a = 0$, $b = T$, $u = \sqrt{4T}$ to yield that
$\Prob(\sup_{t\in[0,T]} d^{-1/2}\pnorm{\b^t}{} \geq \sqrt{4T}) \leq \exp(-2cd)$,
so the claim follows by the Borel Cantelli lemma. For the second claim, let $N =
T/\gamma$ (assumed without loss of generality to be an integer greater than 1), and $I_i = [(i-1)\gamma,
i\gamma)$ for $i\in [N]$. Then applying the inequality over these intervals yields
\begin{align*}
\Prob(\sup_{t\in [0,T]} (d\gamma)^{-1/2}\pnorm{\b^t - \b^{\floor{t}}}{} \geq u) \leq N \max_{i\leq N}\Prob\Big(\sup_{t\in I_i} (d\gamma)^{-1/2}\pnorm{\b^t - \b^{i\gamma}}{} \geq u\Big) \leq Ne^{-cdu^2}
\end{align*} 
for any $u\geq 2$. Hence by choosing $u = C\sqrt{\log N}$ for large enough
$C > 0$, we have $\Prob(\sup_{t\in [0,T]} (d\gamma)^{-1/2}\pnorm{\b^t -
\b^{\floor{t}}}{} \geq C\sqrt{\log N}) \leq \exp(-c'd)$. A similar argument applies to $\sup_{t\in [0,T]} (d\gamma)^{-1/2}\pnorm{\b^t - \b^{\ceil{t}}}{}$, proving the second claim. 
\end{proof}

\begin{lemma}\label{lem:alpha_boundedness}
For any fixed $T>0$, there exists a constant
$C>0$ depending on $T$ such that almost surely
\[\limsup_{n,d \to \infty}
\sup_{t\in[0,T]} (\pnorm{\btheta^t}{}/\sqrt{d} + \pnorm{\widehat{\alpha}^t}{})
\leq C.\]
\end{lemma}
\begin{proof}
Let $C>0$ denote a constant depending on $T$ and changing from instance to
instance. Since
\begin{align*}
\btheta^t = \btheta^0 - \int_0^t \Big(\beta\X^\top(\X\btheta^s - \y) - s(\btheta^s,\widehat{\alpha}^s)\Big)\d s + \sqrt{2}\b^t, 
\end{align*}
and $\pnorm{s(\btheta^s, \widehat\alpha^s)}{} \leq C(\sqrt{d} + \pnorm{\btheta^s}{} + \sqrt{d}\pnorm{\widehat\alpha^s}{})$ by Assumption \ref{assump:prior}, we have for every $t\in[0,T]$ that
\begin{align*}
\pnorm{\btheta^t }{} \leq \pnorm{\btheta^0}{} + C(\pnorm{\X}{\op}^2+1)\int_0^t \Big(\pnorm{\btheta^s}{} + \sqrt{d}\pnorm{\widehat\alpha^s}{}\Big)\d s + C(\pnorm{\X^\top\y}{} + \sqrt{d} + \sup_{t\in[0,T]} \pnorm{\b^t}{}).
\end{align*}
Next by Assumption \ref{assump:gradient}, we have
\begin{align*}
\pnorm{\widehat\alpha^t}{}  \leq \pnorm{\widehat\alpha^0}{} + C\int_0^t \big(1 + \pnorm{\btheta^s}{}/\sqrt{d} + \pnorm{\widehat\alpha^s}{}\big)\d s.
\end{align*}
Combining the above two bounds yields
\begin{align}\label{eq:l2_gronwall_langevin}
\frac{\pnorm{\btheta^t }{}}{\sqrt{d}} + \pnorm{\widehat\alpha^t}{} \leq C(\pnorm{\X}{\op}^2+1)\int_0^t \Big(\frac{\pnorm{\btheta^s}{}}{\sqrt{d}} + \pnorm{\widehat\alpha^s}{}\Big)\d s + C\Big(\frac{\pnorm{\btheta^0}{}}{\sqrt{d}} + \pnorm{\widehat\alpha^0}{} + \frac{\pnorm{\X^\top\y}{}}{\sqrt{d}} + \sup_{t\in[0,T]} \frac{\pnorm{\b^t}{}}{\sqrt{d}} + 1\Big).
\end{align}
Hence by Gronwall's inequality, we have
\begin{align*}
\sup_{t \in [0,T]} \frac{\pnorm{\btheta^t }{}}{\sqrt{d}} + \pnorm{\widehat\alpha^t}{} \leq C\exp(C(\pnorm{\X}{\op}^2+1))\Big(\frac{\pnorm{\btheta^0}{}}{\sqrt{d}} + \pnorm{\widehat\alpha^0}{} + \frac{\pnorm{\X^\top\y}{}}{\sqrt{d}} + \sup_{t\in[0,T]} \frac{\pnorm{\b^t}{}}{\sqrt{d}} + 1\Big).
\end{align*}
Under Assumption \ref{assump:model}, we have almost surely that
\begin{align}\label{eq:almost_sure_event}
\limsup_{n,d \to \infty}
\max\Big(\pnorm{\widehat\alpha^0}{},\frac{1}{\sqrt{d}}\pnorm{\btheta^0}{},\frac{1}{\sqrt{d}}\pnorm{\X^\top
\y}{},\pnorm{\X}{\op}\Big) \leq C,
\end{align}
so the conclusion follows from the first claim of Lemma \ref{lem:BM_maximal}. 
\end{proof}

\begin{proof}[Proof of Lemma \ref{lem:langevin_discretize_error}]
Here and throughout, $C>0$ denotes a constant depending on $T$ but not on
$\gamma$, and changing from instance to instance.
We restrict to the almost-sure event where
Lemmas \ref{lem:BM_maximal} and \ref{lem:alpha_boundedness} hold,
and (\ref{eq:almost_sure_event}) holds for all large $n,d$.
Then, coupling (\ref{eq:langevin_sde}) and (\ref{eq:langevin_embedded})
by the same Brownian motion, for any $0 \leq t \leq T$,
\begin{align*}
\pnorm{\btheta^{t} - \bar\btheta^{t}_\gamma}{}
&\leq C\int_0^{\floor{t}}\Big(\pnorm{\X^\top\X(\btheta^s - \bar \btheta^s_\gamma)}{}
+ \pnorm{s(\btheta^s;\widehat{\alpha}^s) -
s(\btheta^s_\gamma;\bar{\widehat\alpha}^s_\gamma)}{}\Big)\d s
+C\int_{\floor{t}}^t \Big[\X^\top\X\btheta^s
+s(\btheta^s,\widehat\alpha^s)\Big]\d s\\
&\hspace{1in}+\sqrt{2}\pnorm{\b^t - \b^{\floor{t}}}{}.
\end{align*}
Applying Lipschitz continuity of $s(\cdot)$ in
Assumption \ref{assump:prior} and the bounds of
Lemma \ref{lem:alpha_boundedness} and (\ref{eq:almost_sure_event}), this shows
\begin{align}\label{eq:gronwall_theta}
\pnorm{\btheta^t - \bar \btheta^{t}_\gamma}{} &\leq C
\int_0^t \big(\pnorm{\btheta^s - \bar\btheta^s_\gamma}{} +
\sqrt{d}\,\pnorm{\widehat{\alpha}^s - \bar{\widehat\alpha}^s_\gamma}{}\big)\d s
+C\gamma \sqrt{d}+C\sup_{t\in[0,T]}\pnorm{\b^t - \b^{\floor{t}}}{}.
\end{align}
Similarly, using Assumption \ref{assump:gradient} and
$W_2^2(d^{-1}\sum_{j=1}^d \delta_{u_j}, d^{-1}\sum_{j=1}^d \delta_{v_j}) \leq
d^{-1}\pnorm{\u-\v}{}^2$ for $\u,\v\in\R^d$,
\begin{align*}
\pnorm{\widehat{\alpha}^t - \bar{\widehat\alpha}^t_\gamma}{} &\leq
\int_0^{\floor t}
\biggpnorm{\cG\Big(\widehat{\alpha}^s,\frac{1}{d}\sum_{j=1}^d
\delta_{\theta^s_j}\Big) - \cG\Big(\bar{\widehat\alpha}^s_\gamma,
\frac{1}{d}\sum_{j=1}^d \delta_{\bar \theta^s_{\gamma,j}}\Big)}{}\d s
+\int_{\floor t}^t \biggpnorm{\cG\Big(\widehat{\alpha}^s,\frac{1}{d}\sum_{j=1}^d
\delta_{\theta^s_j}\Big)}{}\d s\\
&\leq C\int_0^t \Big(\pnorm{\widehat{\alpha}^s -
\bar{\widehat\alpha}_\gamma^s}{} + \frac{1}{\sqrt{d}}\pnorm{\btheta^s -
\bar\btheta_\gamma^s}{}\Big) \d s+C\gamma.
\end{align*}
Combining the above display with (\ref{eq:gronwall_theta}), we have
by Gronwall's lemma
\begin{align}\label{eq:joint_alpha_theta}
\sup_{t \in [0,T]}
\pnorm{\widehat{\alpha}^t - \bar{\widehat\alpha}^t_\gamma}{} +
\frac{1}{\sqrt{d}}\pnorm{\btheta^t - \bar \btheta^{t}_\gamma}{} \leq C\gamma
+\frac{C}{\sqrt{d}}\sup_{t\in[0,T]}\pnorm{\b^t - \b^{\floor{t}}}{}.
\end{align}
By Lemma \ref{lem:BM_maximal}, there exists some $C>0$ such that
\begin{align*}
\limsup_{d \to \infty}\sup_{t\in[0,T]} \frac{1}{\sqrt{d}}\pnorm{\b^t
- \b^{\floor{t}}}{} \leq C\sqrt{\gamma \cdot \max(\log(1/\gamma),1)}.
\end{align*}
Substituting this bound in (\ref{eq:joint_alpha_theta})
proves the claim on $\alpha$. Noting that
\begin{align*}
W_2\Big( \frac{1}{d}\sum_{j=1}^d \delta_{(\theta_j^*,\theta^{t_1}_j, \ldots,
\theta^{t_m}_j)}, \frac{1}{d}\sum_{j=1}^d \delta_{(\theta_j^*,\bar\theta^{t_1}_{\gamma,j},
\ldots, \bar\theta_{\gamma,j}^{t_m})}\Big) \leq
\sqrt{\frac{1}{d}\sum_{\ell=1}^m\sum_{j=1}^d \big(\theta^{t_\ell}_j -
\bar\theta^{t_\ell}_{\gamma,j}\big)^2} \leq \sqrt{m} \cdot
\frac{1}{\sqrt{d}}\sup_{t\in[0,T]}\pnorm{\btheta^t - \bar\btheta^t_\gamma}{},
\end{align*}
this proves also the claim on $\theta$, and the claim on $\eta$ follows from 
$\|\boldeta^t-\bar \bbeta_\gamma^t\| \leq \|\X\|_\op
\|\btheta^t-\bar \btheta_\gamma^t\|$ and the same argument.
\end{proof}

\subsection{Completing the proof}

\begin{proof}[Proof of Theorem \ref{thm:dmft_approx}]
For part (a), by the triangle inequality,
\[\sup_{t \in [0,T]} \|\widehat \alpha^t-\alpha^t\|
\leq \sup_{t \in [0,T]} \|\widehat \alpha^t-\bar{\widehat\alpha}_\gamma^t\|
+\sup_{t \in [0,T]} \|\bar{\widehat\alpha}_\gamma^t-\bar\alpha_\gamma^t\|
+\sup_{t \in [0,T]} \|\bar\alpha_\gamma^t-\alpha^t\|.\]
Since $\{\bar{\widehat\alpha}_\gamma^t\}$ and $\{\bar\alpha_\gamma^t\}$ are
piecewise constant with values equal to those of the discrete processes of
Section \ref{subsec:discretedmft},
Lemma \ref{lem:finite_dim_converge} implies that the middle 
term converges to 0 a.s.\ as $n,d \to \infty$. Then,
taking $n,d\rightarrow\infty$ followed by $\gamma \rightarrow 0$ and
applying also Lemmas \ref{lem:dmft_discretization_error} 
and \ref{lem:langevin_discretize_error}
to bound the first and third terms in this limit, this shows (a).

For part (b), similarly combining Lemmas \ref{lem:finite_dim_converge},
\ref{lem:dmft_discretization_error}, and \ref{lem:langevin_discretize_error}
shows that almost surely, for any $m \geq 1$
and $t_0,t_1,\ldots,t_m \in [0,T]$,
\begin{align}\label{eq:pointwiseW2convergence}
\frac{1}{d}\sum_{j=1}^d \delta_{\big(\theta_j^*,\theta^{t_0}_j,\ldots,
\theta^{t_m}_j\big)} \overset{W_2}{\to} \sP(\theta^*,\theta^{t_0},\ldots,
\theta^{t_m}).
\end{align}
We now strengthen this to almost-sure convergence in the Wasserstein-2 sense
over $\R \times C([0,T])$, equipped with the product norm
\[\|\theta\|_\infty:=|\theta^*|+\sup_{t \in [0,T]}|\theta^t|.\]
By \cite[Definition 6.8 and Theorem 6.9]{villani2008optimal}, it suffices to show weak convergence together with convergence
of the squared norm $\|\theta\|_\infty^2$, which will be
implied by convergence for all pseudo-Lipschitz test functions $f:\R \times
C([0,T]) \to \R$ satisfying
\begin{equation}\label{eq:pseudolipschitz}
\big|f(\theta)-f(\theta')\big|
\leq
C\|\theta-\theta'\|_\infty(1+\|\theta\|_\infty+\|\theta'\|_\infty).
\end{equation}
Consider the event $\event$ where
(\ref{eq:pointwiseW2convergence}) holds for each $m \geq 2$ and
$\{t_0,t_1,t_2,\ldots,t_m\}=\{0,\gamma,2\gamma,\ldots,T\}$ with $\gamma=T/m$.
Let $\{\tilde\btheta^t\}_{t \in [0,T]}$ be a piecewise linear interpolation
of $\{\btheta^t\}_{t \in [0,T] \cap \gamma \Z_+}$, and similarly let
$\{\tilde\theta^t\}_{t \in [0,T]}$ be a piecewise linear interpolation of the
DMFT process $\{\theta^t\}_{t \in [0,T]}$. For
any pseudo-Lipschitz function $f:\R \times C([0,T]) \to \R$, we have
\begin{equation}\label{eq:PL2bound}
\Big|\frac{1}{d}\sum_{j=1}^d f(\theta_j^*,\{\theta_j^t\}_{t \in [0,T]})
-\E[f(\theta^*,\{\theta^t\}_{t \in [0,T]})]\Big| \leq (I)+(II)+(III)
\end{equation}
with
\begin{align*}
(I)&=\Big|\frac{1}{d}\sum_{j=1}^d f(\theta_j^*,\{\tilde\theta_j^t\}_{t \in [0,T]})
-\E[f(\theta^*,\{\tilde\theta^t\}_{t \in [0,T]})]\Big|\\
(II)&=\Big|\frac{1}{d}\sum_{j=1}^d f(\theta_j^*,\{\theta_j^t\}_{t \in [0,T]})
-f(\theta_j^*,\{\tilde\theta_j^t\}_{t \in [0,T]})\Big|\\
(III)&=\Big|\E[f(\theta^*,\{\theta^t\}_{t \in [0,T]})]
-\E[f(\theta^*,\{\tilde\theta^t\}_{t \in [0,T]})]\Big|.
\end{align*}
For a piecewise linear process $\tilde \theta$ with knots at $\gamma\Z_+$,
$f(\theta_j^*,\{\tilde\theta^t\}_{t \in [0,T]})$ may be understood as a function
of
$(\theta_j^*,\tilde\theta^0,\tilde\theta^\gamma,\tilde\theta^{2\gamma},\ldots,\tilde\theta^T)$,
where this function is pseudo-Lipschitz on $\R^{m+2}$
by the pseudo-Lipschitz property (\ref{eq:pseudolipschitz}) for $f$. Then 
on the above event $\event$, the
Wasserstein-2 convergence (\ref{eq:pointwiseW2convergence}) implies
\[\lim_{n,d \to \infty} (I)=0.\]

To bound $(II)$, let $C,C'>0$ be constants depending on $T$ (but not $\gamma$)
and changing from instance to instance.
Writing $\theta_j=(\theta_j^*,\{\theta_j^t\}_{t \in [0,T]}) \in \R \times C([0,T],\R)$
and applying (\ref{eq:pseudolipschitz}), we have
\begin{equation}\label{eq:thetamodulusbound}
(II) \leq \frac{C}{d}\sum_{j=1}^d
\|\theta_j-\tilde\theta_j\|_\infty\big(1+\|\theta_j\|_\infty\big)
\leq C'\Big(\frac{1}{d}\sum_{j=1}^d
\|\theta_j-\tilde\theta_j\|_\infty^2\Big)^{1/2}
\Big(1+\frac{1}{d}\sum_{j=1}^d \|\theta_j\|_\infty^2\Big)^{1/2}.
\end{equation}
Set
\[F(\btheta,\widehat\alpha)={-}\beta\X^\top(\X\btheta-\y)+s(\btheta,\widehat\alpha)\]
so that by definition,
$\theta_j^t=\theta_j^0+\int_0^t \e_j^\top F(\btheta^s,\widehat\alpha^s)\d
s+\sqrt{2}\,b_j^t$. Hence
\[\sup_{t \in [0,T]} (\theta_j^t)^2 \leq C\Big((\theta_j^0)^2+\int_0^T 
(\e_j^\top F(\btheta^s,\widehat\alpha^s))^2\d s+\|b_j\|_\infty^2\Big),\]
so
\[\frac{1}{d}\sum_{j=1}^d \|\theta_j\|_\infty^2
\leq C\Big(\frac{1}{d}\sum_{j=1}^d (\theta_j^*)^2+(\theta_j^0)^2
+ \frac{1}{d}\sup_{t \in [0,T]} \|F(\btheta^t,\widehat\alpha^t)\|^2
+\frac{1}{d}\sum_{j=1}^d \|b_j\|_\infty^2\Big).\]
On an almost-sure event $\event'$, for all large $n,d$, we have that
$d^{-1}\sum_j (\theta_j^*)^2+d^{-1}\sum_j (\theta_j^0)^2 \leq C$ by
Assumption \ref{assump:prior}, that $\sup_{t \in [0,T]} d^{-1}\|F(\btheta^t,\widehat\alpha^t)\|^2
\leq C$ by the
definition of $F(\cdot)$ together with Assumption \ref{assump:prior}
and Lemma \ref{lem:alpha_boundedness},
and that $d^{-1}\sum_j \|b_j\|_\infty^2 \leq C$ by Doob's maximal inequality
$\P[\|b_j\|_\infty>x] \leq 2e^{-x^2/(2T)}$ and Bernstein's inequality for a sum
of independent subexponential random variables \cite[Theorem 2.8.1]{vershynin2018high}.
Thus, on $\event'$,
\begin{equation}\label{eq:thetajinfty}
\frac{1}{d}\sum_{j=1}^d \|\theta_j\|_\infty^2 \leq C'.
\end{equation}
Now fixing any $\alpha \in (0,1/2)$, define the H\"older semi-norm
$\|\theta_j\|_\alpha=\sup_{s,t \in [0,T]} |\theta_j^t-\theta_j^s|/|t-s|^\alpha$.
Then, since $\tilde\theta_j$ linearly interpolates $\theta_j$ at the knots
$\gamma\Z_+$,
\[\|\theta_j-\tilde\theta_j\|_\infty \leq \gamma^\alpha\|\theta_j\|_\alpha.\]
We have by definition $\theta_j^t-\theta_j^s=
\int_s^t \e_j^\top F(\btheta^r,\widehat\alpha^r)\d r+\sqrt{2}(b_j^t-b_j^s)$, so
by H\"older's inequality,
\begin{align*}
|\theta_j^t-\theta_j^s|
&\leq |t-s|^\alpha \Big(\int_s^t \big|\e_j^\top
F(\btheta^r,\widehat\alpha^r)\big|^{\frac{1}{1-\alpha}}\d r\Big)^{1-\alpha}
+\sqrt{2}|b_j^t-b_j^s|\\
&\leq C|t-s|^\alpha \Big(\Big(\int_0^T (\e_j^\top
F(\btheta^r,\widehat\alpha^r))^2\Big)^{1/2}+\|b_j\|_\alpha\Big)
\end{align*}
and hence
\[\frac{1}{d}\sum_{j=1}^d \|\theta_j\|_\alpha^2
\leq C\Big(\frac{1}{d}\sup_{t \in [0,T]} \|F(\btheta^t,\widehat\alpha^t)\|^2
+\frac{1}{d}\sum_{j=1}^d \|b_j\|_\alpha^2\Big).\]
On an almost-sure event $\event''$, for all large $n,d$, we have
$\sup_{t \in [0,T]} d^{-1}\|F(\btheta^t,\widehat\alpha^t)\|^2 \leq C$ as
above, and $d^{-1}\sum_j \|b_j\|_\alpha^2 \leq C$ by the tail bound
$\P[\|b_j\|_\alpha>C+x] \leq e^{-cx^2}$ for some $C,c>0$ (see e.g.\
\cite[Theorem 5.32, Example 5.37]{van2014probability})
and Bernstein's inequality. Thus on $\event''$,
\begin{equation}\label{eq:thetajholder}
\frac{1}{d}\sum_{j=1}^d \|\theta_j-\tilde\theta_j\|_\infty^2
\leq C\gamma^{2\alpha}.
\end{equation}
Applying (\ref{eq:thetajinfty}) and (\ref{eq:thetajholder}) to
(\ref{eq:thetamodulusbound}), on $\event' \cap \event''$,
\[\limsup_{n,d \to \infty} (II) < C\gamma^\alpha.\]

To bound $(III)$, similarly we have
\begin{equation}\label{eq:DMFTmodulusbound}
(III) \leq C\Big(\E\|\theta-\tilde\theta\|_\infty^2\Big)^{1/2}
\Big(1+\E\|\theta\|_\infty^2\Big)^{1/2}.
\end{equation}
By definition
\[\theta^t=\theta^0+\int_0^t \Big[{-}\delta\beta(\theta^s-\theta^*)
+s(\theta^s,\alpha^s)+\int_0^s R_\eta(s,s')(\theta^{s'}-\theta^*)\d s'
+u^s\Big]\d s+\sqrt{2}b^t.\]
Hence, applying $|s(\theta^s,\alpha^s)| \leq C(1+|\theta^s|+\|\alpha^s\|)$ by
Assumption \ref{assump:prior} and uniform boundedness of the continuous
functions $\alpha^s$ and $R_\eta(s,s')$ over $[0,T]$,
\[(\theta^t)^2 \leq C\Big(1+(\theta^0)^2
+(\theta^*)^2+\|u\|_\infty^2+\|b\|_\infty^2
+\int_0^t \Big(\sup_{r \in [0,s]} (\theta^s)^2\Big) \d s\Big).\]
Then Gronwall's lemma gives
\[\sup_{t \in [0,T]} (\theta^t)^2
\leq C\Big(1+(\theta^0)^2+(\theta^*)^2+\|u\|_\infty^2+\|b\|_\infty^2\Big).\]
We have $\E(\theta^0)^2,\E(\theta^*)^2 \leq C$ by assumption.
Since $\{u^t\}_{t \in [0,T]}$ has covariance $C_\eta(t,s)$ satisfying
$|C_\eta(t,s)| \leq C|t-s|$ by the condition (\ref{eq:Ceta_cond_2}) defining
$\cS(T)^\text{cont}$, we have $\P[\|u\|_\infty \geq C+t] \leq e^{-ct^2}$
for some constants $C,c>0$ by a standard application of Dudley's inequality
\cite[Theorem 8.1.6]{vershynin2018high}, so $\E\|u\|_\infty^2 \leq C$.
Similarly $\E\|b\|_\infty^2 \leq C$, so this gives
\begin{equation}\label{eq:DMFTinftybound}
\E\|\theta\|_\infty^2 \leq C.
\end{equation}
By definition we have also
\[\theta^t-\theta^s=\int_s^t \Big[{-}\delta\beta(\theta^r-\theta^*)
+s(\theta^r,\alpha^r)+\int_0^r R_\eta(r,r')(\theta^{r'}-\theta^*)\d r'
+u^r\Big]\d r+\sqrt{2}(b^t-b^s),\]
so
\begin{align*}
|\theta^t-\theta^s| &\leq C|t-s|^\alpha\Big(\int_s^t\Big(
1+|\theta^*|+\sup_{r'\in [0,r]} |\theta^{r'}|+|u^r|\Big)^{\frac{1}{1-\alpha}}\d r \Big)^{1-\alpha}+\sqrt{2}|b^t-b^s|\\
&\leq C'|t-s|^\alpha\Big(1+\|\theta\|_\infty+\|u\|_\infty+\|b\|_\alpha\Big).
\end{align*}
Then
\[\E\|\theta\|_\alpha^2 \leq C(1+\E\|\theta\|_\infty^2+\E\|u\|_\infty^2
+\E\|b\|_\alpha^2) \leq C',\]
so
\begin{equation}\label{eq:DMFTholder}
\E\|\theta-\tilde\theta\|_\infty^2 \leq \gamma^{2\alpha}\E\|\theta\|_\alpha^2
\leq C'\gamma^{2\alpha}.
\end{equation}
Applying (\ref{eq:DMFTinftybound}) and (\ref{eq:DMFTholder}) to
(\ref{eq:DMFTmodulusbound}) shows
\[(III) \leq C\gamma^\alpha.\]

Applying these bounds for $(I)$, $(II)$, and $(III)$ to take the limit $n,d \to
\infty$ followed by $\gamma \to 0$ in (\ref{eq:PL2bound}), this shows that on
the almost-sure event $\event \cap \event' \cap \event''$ (which does not depend
on $f$), for every pseudo-Lipschitz function $f:\R \times C([0,T]) \to \R$,
\[\lim_{n,d \to \infty}
\frac{1}{d}\sum_{j=1}^d f(\theta_j^*,\{\theta_j^t\}_{t \in [0,T]})
=\E[f(\theta^*,\{\theta^t\}_{t \in [0,T]})].\]
This implies on $\event \cap \event' \cap \event''$ that
\[\frac{1}{d}\sum_{j=1}^d \delta_{\theta_j^*,\{\theta_j^t\}_{t \in [0,T]}}
\overset{W_2}{\to} \sP(\theta^*,\{\theta^t\}_{t \in [0,T]}).\]

For the convergence for $\bbeta$, we note that
$\eta_i^t=\e_i^\top \X\btheta^t=\e_i^\top\X\btheta^t+\int_0^t \e_i^\top \X
F(\btheta^s,\widehat\alpha^s)\d s+\sqrt{2}\,\e_i^\top \X\b^t$. Then applying
similar arguments as above, fixing any $\alpha \in (0,1/2)$,
on an almost-sure event we have for all large $n,d$ that
\[\frac{1}{n}\sum_{i=1}^n \|\eta_i\|_\infty^2 \leq C,
\qquad \frac{1}{n}\sum_{i=1}^n \|\eta_i\|_\alpha^2 \leq C\gamma^{2\alpha}.\]
For the DMFT process we have $\eta^t={-}\beta \int_0^t
R_\theta(t,s)(\eta^s+w^*-\eps)\d s-w^t$, hence
\[(\eta^t)^2 \leq C\Big((w^*)^2+(w^t)^2+\eps^2+\int_0^t (\eta^s)^2\d s\Big)\]
so Gronwall's lemma and a similar argument as above
gives $\E\|\eta\|_\infty^2 \leq C$. Also
\[|\eta^t-\eta^s|
\leq C|t-s|^\alpha\Big(\int_s^t
\Big(|w^*|+|\eps|+|\eta^r|\Big)^{\frac{1}{1-\alpha}}\d r\Big)^{1-\alpha}
+|w^t-w^s|
\leq C'|t-s|^\alpha\Big(|w^*|+|\eps|+\|\eta\|_\infty+\|w\|_\alpha\Big)\]
so
\[\E\|\eta\|_\alpha^2 \leq C\Big(\E(w^*)^2+\E \eps^2+\E\|\eta\|_\infty^2
+\E\|w\|_\alpha^2\Big).\]
We recall that $\{w^t\}_{t \in [0,T]}$ has covariance satisfying
$|C_\theta(t,s)| \leq C|t-s|$, so $\P[\|w\|_\alpha>C+x] \leq 2e^{-cx^2}$
for some $C,c>0$ (c.f.\ \cite[Theorem 5.32]{van2014probability}).
Thus $\E\|\eta\|_\alpha^2 \leq C$. Applying these bounds,
the same arguments as above show the almost-sure convergence
\[\frac{1}{n}\sum_{i=1}^n \delta_{\eta_i^*,\eps_i,\{\eta_i^t\}_{t \in [0,T]})}
\overset{W_2}{\to} \sP(\eta^*,\eps,\{\eta^t\}_{t \in [0,T]})\]
where we recall that $\eta^*$ on the right side is, by definition,
$\eta^*={-}w^*$.
\end{proof}

\section{Convergence of the linear response}\label{sec:dmftresponse}

In this section, we prove Theorem \ref{thm:dmftresponse}. We assume throughout Assumptions \ref{assump:model}, \ref{assump:prior}, \ref{assump:gradient} and the H\"older-continuity conditions of Theorem \ref{thm:dmftresponse}.
We first state and prove in Section \ref{sec:discreteresponse} an
analogue of Theorem \ref{thm:dmftresponse} for the discrete-time dynamics
introduced previously in Section \ref{subsec:discretedmft}, and
then analyze the discretization error and complete the proof of
Theorem \ref{thm:dmftresponse} in Section 
\ref{subsec:discretize_langevin_response}.

\subsection{Convergence of response functions for discrete dynamics}\label{sec:discreteresponse}

We recall the discrete, integer-indexed dynamics
(\ref{eq:langevin_discrete_1}--\ref{eq:langevin_discrete_2}), which we reproduce
here as
\begin{align}
\btheta_\gamma^{t+1} &= \btheta_\gamma^{t} - \gamma\Big[\beta\X^\top(\X\btheta_\gamma^{t} - \y) -
s(\btheta_\gamma^{t}, \widehat\alpha^t_\gamma)\Big] +
\sqrt{2}(\b_\gamma^{t+1} - \b_\gamma^t),\quad \bbeta_\gamma^t =
\X\btheta_\gamma^t\notag\\
\widehat{\alpha}_\gamma^{t+1} &= \widehat{\alpha}_\gamma^t + \gamma \cdot
\cG\Big(\widehat{\alpha}_\gamma^t, \widehat{\sP}(\btheta^t_\gamma)\Big),
\quad \widehat{\sP}(\btheta)=\frac{1}{d}\sum_{j=1}^d \delta_{\theta_j}.
 \label{eq:original_alpha}
\end{align}
We first show an analogue of Theorem \ref{thm:dmftresponse} for these discrete
dynamics.

For any $s\in \Z_+$ and any $j\in[d]$ or $i\in[n]$, letting
$\e_j$ denote the $j^\text{th}$ standard basis vector in either $\R^d$ or
$\R^n$, define two sets of perturbed dynamics
\begin{align}
\btheta_\gamma^{t+1, (s,j), \eps} &= \btheta_\gamma^{t,(s,j),\eps} -
\gamma\Big[\beta\X^\top(\X\btheta_\gamma^{t,(s,j),\eps} - \y) -
s(\btheta_\gamma^{t,(s,j),\eps}, \widehat{\alpha}_\gamma^{t,(s,j),\eps}) - \eps
\e_j \bm{1}_{t = s}\Big] + \sqrt{2}(\b_\gamma^{t+1} - \b_\gamma^t)\notag\\
\widehat{\alpha}_\gamma^{t+1, (s,j),\eps} &= \widehat{\alpha}_\gamma^{t,
(s,j),\eps} + \gamma \cdot \cG\Big(\widehat{\alpha}_\gamma^{t, (s,j),\eps},
\widehat{\sP}(\btheta^{t,(s,j),\eps}_\gamma)\Big),\label{eq:alpha_col_perturb}
\end{align}
and
\begin{align}
\btheta_\gamma^{t+1, [s,i], \eps} &= \btheta_\gamma^{t,[s,i],\eps} -
\gamma\Big[\beta\X^\top(\X\btheta_\gamma^{t,[s,i],\eps} - \y) -
s(\btheta_\gamma^{t,[s,i],\eps}, \widehat{\alpha}_\gamma^{t,[s,i],\eps}) - \eps
\X^\top  \e_i\bm{1}_{t = s}\Big] + \sqrt{2}(\b_\gamma^{t+1} - \b_\gamma^t)\notag\\
\widehat{\alpha}_\gamma^{t+1, [s,i], \eps} &= \widehat{\alpha}_\gamma^{t,
[s,i], \eps} + \gamma \cdot \cG\Big(\widehat{\alpha}_\gamma^{t, [s,i], \eps},
\widehat{\sP}(\btheta^{t,[s,i],\eps}_\gamma)\Big)\label{eq:alpha_row_perturb}
\end{align}
with the same initial conditions as (\ref{eq:original_alpha}). We set
\begin{align}\label{eq:eta_row_perturb}
\bbeta_\gamma^{t,[s,i],\eps} = \X\btheta_\gamma^{t, [s,i], \eps}.
\end{align}
Comparing with (\ref{eq:original_alpha}), these
dynamics have a perturbation to the drift in the direction of $\e_j$ or
$\X^\top \e_i$ at the single time $s \in \Z_+$.
Let $\bR_\theta^\gamma(t,s) = (\bR_\theta^\gamma(t,s))_{i,j=1}^d \in\R^{d\times
d}$ and $\bR_\eta^\gamma(t,s) = (\bR_\eta^\gamma(t,s))_{i,j=1}^n \in\R^{n\times
n}$ be matrices of response functions defined by
\begin{align*}
(\bR_\theta^\gamma(t,s))_{i,j} = \partial_\eps|_{\eps = 0}
\langle \theta_{\gamma,i}^{t,(s,j),\eps}\rangle,
\qquad (\bR_\eta^\gamma(t,s))_{i,j} = \delta\beta^2 
\cdot \partial_\eps|_{\eps = 0} \langle \eta_{\gamma,i}^{t,[s,j],\eps} \rangle,
\end{align*}
where $\langle \cdot \rangle$ denotes the expectation over only the randomness
of $\{\b_\gamma^t\}_{t \in \Z_+}$,
i.e.\ conditional on $(\X,\btheta^*,\beps)$ and on the initial conditions
$(\btheta_\gamma^{0,(s,j),\eps},\widehat{\alpha}_\gamma^{0,(s,j),\eps})=(\btheta_\gamma^{0,[s,i],\eps},\widehat{\alpha}_\gamma^{0,[s,i],\eps})=(\btheta^0,\widehat\alpha^0)
\in \R^{d+K}$.

Recall also the discrete-time DMFT response functions
$R_\theta^\gamma(t,s),R_\eta^\gamma(t,s)$ defined by
(\ref{eq:dmft_CRtheta_discrete}) and (\ref{eq:dmft_CReta_discrete}).
The goal of this section is to prove the following analogue of the
convergence statements for the response functions in Theorem \ref{thm:dmftresponse}.

\begin{lemma}\label{lem:discrete_time_response}
For any fixed $s,t \in \Z_+$ with $s<t$, almost surely
 \begin{align*}
\lim_{n,d \to \infty}
\frac{1}{d}\Tr \bR_\theta^\gamma(t,s)=R_\theta^{\gamma}(t,s), \qquad 
\lim_{n,d \to \infty} \frac{1}{n}\Tr \bR_\eta^\gamma(t,s)=R_\eta^{\gamma}(t,s).
 \end{align*}
 \end{lemma}
 
To ease notation, in the remainder of this section
we will drop all subscripts $\gamma$ and write simply
$\btheta^t=\btheta^t_\gamma$, $\widehat\alpha^t=\widehat\alpha^t_\gamma$,
$\b^t=\b^t_\gamma$ etc.\ to refer to the above discrete-time processes.
We first establish in Section \ref{subsec:loo} a set of dynamical cavity
estimates, which we will then use to prove Lemma
\ref{lem:discrete_time_response} in Section
\ref{subsec:discrete_response_converge}.

\subsubsection{Dynamical cavity estimates}\label{subsec:loo}

We introduce the following notations: For any $j\in[d]$ and $i \in
[n]$, denote
\[\btheta^t=(\theta^t_j,\btheta^t_{-j}) \in \R^d, \qquad
\theta^t_j \in \R, \qquad \btheta^t_{-j} \in \R^{d-1},\]
\[\bbeta^t=(\eta^t_i,\bbeta^t_{-i}) \in \R^n, \qquad
\eta^t_i \in \R, \qquad \bbeta^t_{-i} \in \R^{n-1},\]
where $\{\btheta^t,\bbeta^t\}_{t \in \Z_+}$ are the components of the 
discrete-time process (\ref{eq:original_alpha}), and
$\btheta_{-j}^t$ are the coordinates of $\btheta^t$ excluding the $j^\text{th}$
(and similarly for $\bbeta^t$).

We consider the following leave-one-out versions of
(\ref{eq:original_alpha}): For $j \in [d]$, let
\begin{equation}\label{eq:leaveoutX}
\X^{(j)}=(X_{ik}\1_{k \neq j})_{i,k} \in \R^{n\times d},\qquad
\y^{(j)}=\X^{(j)}\btheta^*+\beps \in \R^n
\end{equation}
where $\X^{(j)}$ denotes $\X$ with $j^\text{th}$ column set to 0. Define
\begin{align}\label{eq:leave_col_out}
\notag\btheta^{t+1,(j)} &= \btheta^{t,(j)} -
\gamma\Big[\beta(\X^{(j)})^\top(\X^{(j)}\btheta^{t,(j)}-\y^{(j)}) -
s(\btheta^{t,(j)}, \widehat{\alpha}^{t,(j)})\Big] + \sqrt{2}(\b^{t+1}-
\b^t)\in\R^d\\
\widehat{\alpha}^{t+1,(j)} &= \widehat{\alpha}^{t,(j)} + \gamma \cdot
\cG\Big(\widehat{\alpha}^{t,(j)},\widehat{\sP}(\btheta^{t,(j)})\Big)
\end{align}
with initialization $(\btheta^{0,(j)},\widehat{\alpha}^{0,(j)})=(\btheta^0, \widehat{\alpha}^0)$, and write as above
\[\btheta^{t,(j)}=(\theta^{t,(j)}_j,\btheta^{t,(j)}_{-j}) \in
\R^d, \qquad \theta^{t,(j)}_j \in \R,
\qquad \btheta^{t,(j)}_{-j} \in \R^{d-1}.\]
We note that for convenience of the proof, we define $\btheta^{t,(j)}$ to be of
the same dimension as $\btheta$, where one may check from (\ref{eq:leave_col_out}) that the dynamics of $\btheta^{t,(j)}_{-j}$ do not involve $\theta^{t,(j)}_j$.
Similarly, for $i\in[n]$, let
\[\X^{[i]}=(X_{kj}\1_{k \neq i})_{k,j} \in \R^{n\times d},
\qquad \y^{[i]} = \X^{[i]}\btheta^\ast + \beps,\]
where $\X^{[i]}$ sets the $i^\text{th}$ row of $\X$ to 0. Define
\begin{align}\label{eq:leave_row_out}
\notag\btheta^{t+1,[i]} &= \btheta^{t,[i]} -
\gamma\Big[\beta(\X^{[i]})^\top(\X^{[i]}\btheta^{t,[i]} - \y^{[i]}) -
s(\btheta^{t,[i]},\widehat{\alpha}^{t,[i]})\Big] + \sqrt{2}(\b^{t+1} -
\b^t)\in\R^d\\
 \widehat{\alpha}^{t+1,[i]} &= \widehat{\alpha}^{t,[i]} + \gamma \cdot
\cG\Big(\widehat{\alpha}^{t,[i]}, \widehat{\sP}(\btheta^{t,[i]})\Big)
\end{align}
also with initialization
$(\btheta^{0,[i]},\widehat{\alpha}^{0,[i]})=(\btheta^0,
\widehat{\alpha}^0)$, and write as above
\[\bbeta^{t,[i]}=(\eta^{t,[i]}_i,\bbeta^{t,[i]}_{-i}) \in
\R^n, \qquad \eta^{t,[i]}_i \in \R,
\qquad \bbeta^{t,[i]}_{-i} \in \R^{n-1}.\]
By construction, $\{\btheta^{t,(j)}, \widehat\alpha^{t,(j)}\}$
is independent of the $j^\text{th}$ column of $\X$, and $\{\btheta^{t,[i]},
\widehat\alpha^{t,[i]}\}$ is independent of the $i^\text{th}$ row of $\X$.

The following lemma gives $\ell_2$ estimates on the original
dynamics (\ref{eq:original_alpha}) as well as on its
difference with the cavity versions
(\ref{eq:leave_col_out}) and (\ref{eq:leave_row_out}).

\begin{lemma}\label{lem:loo_error}
Fix any $T>0$. Then there exists a constant $C>0$ (depending on $T$ but not
$\gamma$) such that for any $\gamma>0$, almost surely for all large $n,d$,
we have
for all $0\leq t\leq T/\gamma$ and all $j \in [d]$, $i \in [n]$ that
\begin{align}
&\frac{\pnorm{\btheta^t}{}}{\sqrt{d}}+\|\widehat\alpha^t\| \leq C,
\quad
\frac{\pnorm{\btheta^{t,(j)}}{}}{\sqrt{d}}+\pnorm{\widehat{\alpha}^{t,(j)}}{} \leq C,
\quad \frac{\pnorm{\btheta^{t,[i]}}{}}{\sqrt{d}}
+\pnorm{\widehat{\alpha}^{t,[i]}}{} \leq C, \label{eq:rough_l2}\\
&|\theta^{t,(j)}_j| \leq C(1+|\theta^0_j|
+\max_{t \in [0,T/\gamma]} |b^t_j|),\label{eq:loo_inactive_bound}\\
& \pnorm{\btheta^{t,(j)} - \btheta^t}{} +
\sqrt{d}\,\pnorm{\widehat{\alpha}^{t,(j)}-\widehat{\alpha}^t}{} \leq
C(|\theta^0_j| + |\theta^\ast_j| + \max_{t\in[0,T/\gamma]}
|b^t_j| + \sqrt{\log d}),\label{eq:rough_loo_l2_col}\\
&\pnorm{\btheta^{t,[i]} - \btheta^t}{} +
\sqrt{d}\,\pnorm{\widehat{\alpha}^{t,[i]}-\widehat{\alpha}^t}{} \leq
C(|\eps_i| + \sqrt{\log d}) \label{eq:rough_loo_l2_row}.
\end{align}
\end{lemma}
\begin{proof}
Fixing a constant $C_0>0$ large enough (depending on $T$)
and any $\gamma>0$, define the event
\begin{align*}
\event&=\Big\{\|\X\|_\op \leq C_0,\,\|\widehat\alpha^0\|\leq C_0,
\,\|\btheta^*\|_2,\|\btheta^0\|_2
\leq C_0\sqrt{d},\,\|\beps\|_2 \leq C_0\sqrt{d},\,\\
&\hspace{1in} \max_{t \in [0,T/\gamma]}
\|\b^t\|_2 \leq C_0\sqrt{d} \text{ for all large } n,d\Big\}.
\end{align*}
Note that we have $\b^t \sim \N(0,t\gamma \I)$,
so $\P[\|\b^t\|_2>C_0\sqrt{t\gamma d}] \leq e^{-cd}$
for some constants $C_0,c>0$ and all large $n,d$
by a chi-squared tail bound. Then, taking a union
bound over all $t \in [0,T/\gamma] \cap \Z_+$ and applying the
conditions of Assumption \ref{assump:model} together with the Borel-Cantelli
lemma, we see that this event $\event$ holds almost surely.

We restrict to the event $\event$.
Let $C,C'>0$ denote constants depending
on $C_0,T$ (but not on $\gamma$) and changing from instance to instance.
For (\ref{eq:rough_l2}), we have by definition of
$\{\btheta^t,\widehat\alpha^t\}$ in
(\ref{eq:original_alpha}) that
\begin{align*}
\btheta^t&=\btheta^0{-}\gamma \sum_{s=0}^{t-1}
\Big[\beta \X^\top(\X\btheta^s-\y)-s(\btheta^s,\widehat\alpha^s)\Big]
+\sqrt{2}\b^t\\
\widehat\alpha^t&=\widehat\alpha^0+\gamma \sum_{s=0}^{t-1}
\cG\Big(\widehat\alpha^s,\widehat{\sP}(\btheta^s)\Big).
\end{align*}
Applying the bounds for $s(\cdot)$ and $\cG(\cdot)$
in Assumptions \ref{assump:prior} and \ref{assump:gradient} and the conditions
of $\event$,
\begin{align*}
\|\btheta^t\|& \leq C\gamma \sum_{s=0}^{t-1}
\Big(\|\btheta^s\|+\sqrt{d}\,\|\widehat\alpha^s\|+\sqrt{d}\Big)
+\|\btheta^0\|+\sqrt{2}\,\|\b^t\|\\
\|\widehat\alpha^t\|&\leq C\gamma \sum_{s=0}^{t-1}
\Big(\|\widehat\alpha^s\|+\|\btheta^s\|/\sqrt{d}+1\Big)
+\|\widehat\alpha^0\|
\end{align*}
so
\[1+\frac{\|\btheta^t\|}{\sqrt{d}}
+\|\widehat\alpha^t\|
\leq C\gamma\sum_{s=0}^{t-1}
\Big(1+\frac{\|\btheta^s\|}{\sqrt{d}}
+\|\widehat\alpha^s\|\Big)
+1+\frac{\|\btheta^0\|}{\sqrt{d}}+\|\widehat\alpha^0\|
+\sqrt{\frac{2}{d}}\|\b^t\|.\]
Iterating this bound over $t$ shows
\begin{align*}
\frac{\|\btheta^t\|}{\sqrt{d}}
+\|\widehat\alpha^t\|
&\leq (1+C\gamma)^t \bigg[
1+\frac{\|\btheta^0\|}{\sqrt{d}}+\|\widehat\alpha^0\|
+\sqrt{\frac{2}{d}}\max_{s \in [0,t]} \|\b^s\|\bigg] \leq C',
\end{align*}
the last bound holding for $t \leq T/\gamma$ and on $\event$. This establishes
the first claim of (\ref{eq:rough_l2}). The other two claims of
(\ref{eq:rough_l2}) for the cavity dynamics hold by the same argument, noting
that on $\event$ we have also
$\|\X^{(j)}\|_\op,\|\X^{[i]}\|_\op \leq C_0$ for all
$j \in [d]$ and $i \in [n]$.

For (\ref{eq:loo_inactive_bound}), we have by definition of
(\ref{eq:leave_col_out}) that
\[\theta^{t+1,(j)}_j=\theta^{t,(j)}_j
+\gamma\,s(\theta^{t,(j)}_j,\widehat\alpha^{t,(j)})
+\sqrt{2}(b^{t+1}_j-b^t_j).\]
Then
\[\theta_j^{t,(j)}
=\theta_j^0+\gamma \sum_{s=0}^{t-1} s(\theta_j^{s,(j)},\widehat
\alpha^{t,(j)})+\sqrt{2}\,b_j^t\]
so applying the bound for $s(\cdot)$ in Assumption \ref{assump:prior}
and the bound $\|\widehat \alpha^{t,(j)}\| \leq C$ already shown in
(\ref{eq:rough_l2}),
\[1+|\theta_j^{t,(j)}|
\leq C\gamma \sum_{s=0}^{t-1} (1+|\theta_j^{s,(j)}|)
+1+|\theta_j^0|+\sqrt{2}\,|b^t_j|.\]
Iterating this bound gives, for all $t \leq T/\gamma$,
\[1+|\theta_j^{t,(j)}| \leq
(1+C\gamma)^t(1+|\theta_j^0|+\sqrt{2}\max_{s \in [0,t]} |b_j^t|)
\leq C'(1+|\theta_j^0|+\max_{t \leq T/\gamma} |b_j^t|)\]
which shows (\ref{eq:loo_inactive_bound}).

For (\ref{eq:rough_loo_l2_col}), by definition, 
\begin{align*}
\btheta^{t+1} - \btheta^{t+1,(j)}&=\Big(\I - \gamma\beta \X^\top
\X\Big)\btheta^t - \Big(\I - \gamma\beta \X^{(j)\top}
\X^{(j)}\Big)\btheta^{t,(j)}+\gamma\beta\Big(\X^\top\y - \X^{(j)\top}
\y^{(j)}\Big)\\
&\quad + \gamma
\Big(s(\btheta^t,\widehat{\alpha}^t)-s(\btheta^{t,(j)},\widehat\alpha^{t,(j)})\Big).
\end{align*}
Then, applying the Lipschitz bound for $s(\cdot)$ in Assumption
\ref{assump:prior} and the conditions defining $\event$,
\begin{align}
\notag\pnorm{\btheta^{t+1}-\btheta^{t+1,(j)}}{}
&\leq (1+C\gamma)\pnorm{\btheta^t-\btheta^{t,(j)}}{} +
C\gamma\sqrt{d}\pnorm{\widehat\alpha^{t-1} - \widehat\alpha^{t-1,(j)}}{}\\
&\hspace{1in} + C\gamma\underbrace{\big(\pnorm{(\X^{(j)\top}
\X^{(j)} - \X^\top \X)\btheta^{t,(j)}}{} +\pnorm{\X^\top\y - \X^{(j)\top}
\y^{(j)}}{}\big)}_{:=\Delta_{t,j}}.\label{eq:leave_col_out_bound}
\end{align}
Similarly, by the Lipschitz bound for $\cG(\cdot)$ in Assumption
\ref{assump:gradient}, 
\begin{align*}
\pnorm{\widehat\alpha^{t+1} - \widehat\alpha^{t+1,(j)}}{} \leq
(1+C\gamma)\pnorm{\widehat\alpha^t-\widehat\alpha^{t,(j)}}{} +
C\gamma\pnorm{\btheta^t - \btheta^{t,(j)}}{}/\sqrt{d}. 
\end{align*}
Combining the above two inequalities yields
\begin{align}\label{eq:one_step_bound}
\pnorm{\btheta^{t+1} - \btheta^{t+1,(j)}}{} +
\sqrt{d}\pnorm{\widehat\alpha^{t+1}-\widehat\alpha^{t+1,(j)}}{} \leq
(1+C\gamma)\big(\pnorm{\btheta^t - \btheta^{t,(j)}}{} +
\sqrt{d}\pnorm{\widehat\alpha^t - \widehat\alpha^{t,(j)}}{}\big) +
C\gamma\Delta_{t,j},
\end{align}
and hence iterating this bound and using
$(\btheta^{0,(j)},\widehat\alpha^{0,(j)})=(\btheta^0,\widehat\alpha^0)$,
for any $t \leq T/\gamma$,
\begin{align*}
\pnorm{\btheta^t - \btheta^{t,(j)}}{} + \sqrt{d}\pnorm{\widehat\alpha^t -
\widehat\alpha^{t,(j)}}{} \leq \sum_{s=0}^{t-1} (1+C\gamma)^s\max_{s=0}^{t-1}
C\gamma\Delta_{s,j} \leq C'\max_{s=0}^{t-1}\Delta_{s,j}.
\end{align*}

Let us now bound $\Delta_{t,j}$. Writing $\x_j \in \R^n$ for the
$j^\text{th}$ column of $\X$, we have
$\X^{(j)} = \X - \x_j \e_j^\top$, hence
$\X^\top \X - \X^{(j)\top} \X^{(j)} = \X^\top \x_j \e_j^\top + \e_j \x_j^\top \X
- \e_j \x_j^\top \x_j\e_j^\top$, and
\begin{align}\label{ineq:loo_1}
\notag\pnorm{(\X^{(j)\top} \X^{(j)} - \X^\top \X)\btheta^{t,(j)}}{} &\leq
\pnorm{\X^\top \x_j}{}|\theta^{t,(j)}_j| + |\x_j^\top \X
\btheta^{t,(j)}| + \pnorm{\X}{\op}^2 |\theta^{t,(j)}_j|\\
&\leq C\pnorm{\X}{\op}^2|\theta^{t,(j)}_j| + |\x_j^\top
\X^{(j)}\btheta^{t,(j)}|.
\end{align}
Similarly, we have
$\X^\top\y - \X^{(j)\top} \y^{(j)} = (\X^\top \X - \X^{(j)\top}
\X^{(j)})\btheta^\ast + (\X - \X^{(j)})^\top \beps$ so
\begin{align}\label{ineq:loo_2}
\pnorm{\X^\top\y - \X^{(j)\top} \y^{(j)}}{} \leq
C\pnorm{\X}{\op}^2|\theta^\ast_j| + |\x_j^\top \X^{(j)}\btheta^\ast| +
|\x_j^\top \beps|. 
\end{align}
By (\ref{eq:loo_inactive_bound}), we have 
$|\theta^{t,(j)}_j| \leq C(1 +
|\theta_j^0|+\sup_{t\in[0,T/\gamma]} |b^t_j|)$ for $t\leq
\floor{T/\gamma}$. Applying this in the above two bounds yields, on $\event$,
\begin{align*}
\Delta_{t,j} \leq C\Big[1+|\theta^0_j| + |\theta^\ast_j| +
\sup_{t \in [0,T/\gamma]} |b_j^t| + |\x_j^\top
\X^{(j)}\btheta^\ast| + |\x_j^\top \X^{(j)}\btheta^{t,(j)}|  + |\x_j^\top \beps|\Big].
\end{align*}
Define the additional event $\event'$ where
\[\sup_{j \in [d]} \max_{t \in [0,T/\gamma]}
|\x_j^\top \X^{(j)}\btheta^\ast| + |\x_j^\top \X^{(j)}\btheta^{t,(j)}|  +
|\x_j^\top \beps| \leq C_0\sqrt{\log d} \text{ for all large } n,d.\]
Then the desired bound (\ref{eq:rough_loo_l2_col}) holds for all large $n,d$
on the event $\event \cap \event'$,
so it remains to show that $\event'$ holds almost surely for sufficiently large $C_0>0$.
For each $j \in [d]$ and $t \in [0,T/\gamma]$,
by independence between $\x_j$ and
$\X^{(j)},\btheta^{t,(j)}$, we have that $\x_j^\top
\X^{(j)}\btheta^{t,(j)}$ is subgaussian conditional on
$\X^{(j)}\btheta^{t,(j)}$, so
\[\P\bigg[|\x_j^\top \X^{(j)}\btheta^{t,(j)}| \geq C\sqrt{\frac{
\|\X^{(j)}\btheta^{t,(j)}\|\log d}{d}}\bigg] \leq e^{-cd}\]
for some constants $C,c>0$ (conditional on $\X^{(j)}\btheta^{t,(j)}$, and
hence also unconditionally). Then, taking a union bound over
$j \in [d]$ and $t \in [0,T/\gamma] \cap \Z_+$ and applying
the Borel-Cantelli lemma, almost surely for all large $n,d$,
\[\sup_{j,t} |\x_j^\top \X^{(j)}\btheta^{t,(j)}|
\leq C\sup_{j,t} \sqrt{\frac{\|\X^{(j)}\btheta^{t,(j)}\|\log d}{d}}.\]
On the event $\event$ we have $\sup_{j,t} \|\X^{(j)}\btheta^{t,(j)}\| \leq C\sqrt{d}$ by (\ref{eq:rough_l2}) already shown, so 
$\sup_{j,t} |\x_j^\top \X^{(j)}\btheta^{t,(j)}| \leq C'\sqrt{\log d}$
a.s.\ for all large $n,d$. The terms
$|\x_j^\top \X^{(j)}\btheta^\ast|$ and $|\x_j^\top \beps|$
are bounded similarly, verifying that $\event'$ holds almost surely
as claimed, and concluding the proof of (\ref{eq:rough_loo_l2_col}).

For (\ref{eq:rough_loo_l2_row}), similar to above, we have
\begin{align*}
\pnorm{\btheta^{t+1} - \btheta^{t+1,[i]}}{} &\leq (1+C\gamma)\pnorm{\btheta^t
- \btheta^{t,[i]}}{} + C\gamma\sqrt{d}\pnorm{\widehat\alpha^t -
\widehat\alpha^{t,[i]}}{}\\
&\hspace{1in} +C\gamma \underbrace{\Big(\pnorm{\X^{[i]\top}
\X^{[i]} - \X^\top \X)\btheta^{t,[i]}}{} +\pnorm{\X^\top\y - \X^{[i]\top}
\y^{[i]}}{}\Big)}_{\Delta_{t,i}},\\
\pnorm{\widehat\alpha^{t+1} - \widehat\alpha^{t+1,[i]}}{} &\leq
\pnorm{\widehat\alpha^t - \widehat\alpha^{t,[i]}}{} +
C\gamma\big(\pnorm{\widehat\alpha^t - \widehat\alpha^{t,[i]}}{} +
\pnorm{\btheta^t - \btheta^{t,[i]}}{}/\sqrt{d}\big).
\end{align*}
which implies
\begin{align*}
\pnorm{\btheta^t - \btheta^{t,[i]}}{} + \sqrt{d}\,\pnorm{\widehat\alpha^t -
\widehat\alpha^{t,[i]}}{} \leq C\max_{s=0}^{t-1}\Delta_{s,i}.
\end{align*}
Using $\X=\X^{[i]} + \e_i\x_i^\top$, where 
$\x_i\in\R^d$ now denotes (the transpose of) the $i^\text{th}$ row of $\X$, we
have $\X^\top \X = \X^{[i]\top} \X^{[i]} + \x_i\x_i^\top$ and $\X^\top\y -
\X^{[i]\top} \y^{[i]} = \x_i(\x_i^\top \btheta^\ast + \eps_i)$, so
\begin{align*}
\Delta_{t,i} \leq \|\X\|_\op\Big(|\x_i^\top \btheta^{t,[i]}|
+|\x_i^\top \btheta^*|+|\eps_i|\Big).
\end{align*}
Using independence between $\x_i$ and $\btheta^{t,[i]},\btheta^*$, we obtain
as above that on an almost sure event $\event'$, for all large $n,d$
we have $\Delta_{t,i} \leq C(|\eps_i| + \sqrt{\log d})$ for all $t \in
[0,T/\gamma] \cap \Z_+$ and $i \in [n]$, showing (\ref{eq:rough_loo_l2_row}).
\end{proof}

\subsubsection{Proof of Lemma \ref{lem:discrete_time_response}}\label{subsec:discrete_response_converge}

\begin{lemma}\label{lem:response_l2_estimate}
For any $T>0$, on the event where $\|\X\|_\op \leq C_0$, there exists a constant
$C>0$ (depending on $T,C_0$ but not $\gamma$) such that
\begin{align*}
\max_{0\leq s \leq t\leq T/\gamma} \max_{j\in [d]} \Big[\pnorm{\deps \btheta^{t,(s,j),\eps} }{} + \sqrt{d}\pnorm{\deps \widehat\alpha^{t,(s,j),\eps}}{} \Big]\leq C\gamma,\\
\max_{0\leq s \leq t\leq T/\gamma} \max_{i\in [n]} \Big[\pnorm{\deps \btheta^{t,[s,i],\eps} }{} + \sqrt{d}\pnorm{\deps \widehat\alpha^{t,[s,i],\eps}}{} \Big]\leq C\gamma.
\end{align*}
\end{lemma}
\begin{proof}

For the first statement, we fix $s,j$, and shorthand
$\btheta^{t,(s,j),\eps},\widehat\alpha^{t,(s,j),\eps}$ as $\btheta^{t,\eps},
\widehat\alpha^{t,\eps}$. By definition of the process
(\ref{eq:alpha_col_perturb}), we have for $t\geq s+1$
\begin{align*}
\deps \btheta^{t+1,\eps} &= \Big(\I - \gamma\beta\X^\top \X + \gamma \diag(\partial_\theta s(\btheta^{t}, \widehat\alpha^{t}))\Big) \deps \btheta^{t,\eps} + \gamma \nabla_\alpha s(\btheta^{t}, \widehat\alpha^{t})\deps \widehat\alpha^{t,\eps},\\
\deps \widehat\alpha^{t+1,\eps} &= \Big(1 + \gamma \d_\alpha \cG(\widehat\alpha^{t}, \widehat\sP(\btheta^{t}))\Big) \deps \widehat\alpha^{t,\eps} + \gamma\d_{\btheta} \cG(\widehat\alpha^{t}, \widehat\sP(\btheta^{t})) \deps \btheta^{t,\eps}.
\end{align*}
Then applying the conditions for $s(\cdot)$ and $\cG(\cdot)$ in
Assumptions \ref{assump:prior} and \ref{assump:gradient}
and $\|\X\|_\op \leq C_0$,
\begin{align*}
\pnorm{\deps \btheta^{t+1,\eps}}{} &\leq (1+C\gamma)\pnorm{\deps \btheta^{t,\eps}}{} + C\sqrt{d}\pnorm{\deps \widehat\alpha^{t,\eps}}{},\\
\pnorm{\deps \widehat\alpha^{t+1,\eps}}{} &\leq (1+C\gamma)\pnorm{\deps \widehat\alpha^{t,\eps}}{} + C\gamma \pnorm{\deps \btheta^{t,\eps}}{}/\sqrt{d}, 
\end{align*}
where $C>0$ is a constant independent of $\gamma$. Combining and iterating these
inequalities yields, for all $t \in [s+1,T/\gamma]$,
\begin{align}\label{eq:alpha_theta_derivative_bound}
\pnorm{\deps \btheta^{t+1,\eps}}{} + \sqrt{d}\pnorm{\deps
\widehat\alpha^{t+1,\eps}}{} \leq (1+C\gamma)^{t-s}\big(\pnorm{\deps
\btheta^{s+1,(s,j),\eps}}{} + \sqrt{d}\pnorm{\deps
\widehat\alpha^{s+1,(s,j),\eps}}{}\big) \leq C'\gamma, 
\end{align}
using $\deps \btheta^{s+1,(s,j),\eps} = \gamma \e_j$ and $\deps
\widehat\alpha^{s+1,(s,j),\eps} = 0$. This holds for all $s \in [0,T/\gamma]$
and $j \in [d]$, showing the first claim. The proof of the second claim is
analogous, and omitted for brevity.
\end{proof}

\begin{lemma}\label{lem:A_trace}
Let $\{\btheta^t,\widehat\alpha^t\}$ be given by (\ref{eq:original_alpha}). For each $t\in\Z_+$ define the matrix
\begin{align}\label{def:A_matrix}
\bOmega^t = \I - \gamma\beta\X^\top \X + \gamma\diag(\partial_\theta s(\btheta^{t}, \widehat\alpha^{t})) \in \R^{d\times d}.
\end{align}
Then for any fixed $s,t \in [0,T/\gamma]$ with $t \geq s+1$,
almost surely
\begin{align*}
\lim_{n,d \to \infty}
\frac{1}{d}\sum_{j=1}^{d} \partial_\eps|_{\eps = 0} \theta_j^{t,(s,j), \eps} -
\frac{\gamma}{d}\Tr\Big(\bOmega^{t-1}\ldots \bOmega^{s+1}\Big) &=0\\
\lim_{n,d \to \infty}
\frac{1}{n}\sum_{i=1}^n \partial_\eps|_{\eps = 0}\eta_i^{t,[s,i], \eps} -
\frac{\gamma}{n}\Tr\Big(\X \bOmega^{t-1}\ldots \bOmega^{s+1} \X^\top\Big) &=0
\end{align*}
where by convention we set $\bOmega^{t-1}\ldots \bOmega^{s+1}=\I$ for $t=s+1$.
\end{lemma}
\begin{proof}
Let us denote
\[\nabla_\alpha s(\btheta^t,\widehat\alpha^t)
=\Big(\nabla_\alpha s(\theta_i^t,\widehat\alpha^t)^\top\Big)_{i=1}^d
\in \R^{d \times K},
\quad \r^{t,(s,j)} = \nabla_\alpha
s(\btheta^t,\widehat\alpha^t) \deps \widehat\alpha^{t,(s,j),\eps} \in
\R^d.\]
Then
\[\partial_\eps|_{\eps = 0} \btheta^{t+1,(s,j), \eps}=\underbrace{\Big(I
- \gamma\beta\X^\top \X + \gamma\diag(\partial_\theta s(\btheta^t,
\widehat\alpha^t))\Big)}_{\bOmega^t}\notag\partial_\eps|_{\eps = 0}
\btheta^{t,(s,j), \eps} + \gamma \r^{t,(s,j)}.\]
Iterating this identity with $\deps\btheta^{s+1,(s,j), \eps}=\gamma\e_j$ shows
\begin{align}\label{eq:expansion_derivative}
\partial_\eps|_{\eps = 0} \btheta^{t,(s,j), \eps}
&= \gamma \bOmega^{t-1}\ldots \bOmega^{s+1} \e_j + \gamma\sum_{\ell=s+1}^{t-1}
\bOmega^{t-1}\ldots \bOmega^{\ell+1} \r^{\ell,(s,j)}
\end{align}
(where $\bOmega^{t-1}\ldots \bOmega^{\ell+1}=\I$ for $\ell=t-1$).
This implies
\begin{align}\label{eq:avgderivtmp}
\frac{1}{d}\sum_{j=1}^d \partial_\eps|_{\eps = 0} \theta_j^{t,(s,j), \eps} =
\frac{\gamma}{d}\Tr \big(\bOmega^{t-1}\ldots \bOmega^{s+1}\big) +
\frac{\gamma}{d} \sum_{\ell=s+1}^{t-1} \sum_{j=1}^d \e_j^\top
\bOmega^{t-1}\ldots \bOmega^{\ell+1} \r^{\ell,(s,j)}.
\end{align}
On an event where $\|\X\|_\op \leq C_0$ for all large $n,d$ (which holds almost
surely), by the Lipschitz continuity of $s(\cdot)$ in
Assumption \ref{assump:model} and bound in Lemma \ref{lem:response_l2_estimate},
we have $\|\bOmega^t\|_\op \leq C$,
$\|\nabla_\alpha s(\btheta^t,\widehat\alpha^t)\|_F \leq C\sqrt{d}$, and
$\|\deps \widehat\alpha^{t,(s,j),\eps}\| \leq C/\sqrt{d}$ for all $t$,
where $C>0$ is a constant (possibly depending on $\gamma$) changing
from instance to instance. Then by Cauchy-Schwarz,
\begin{align*}
\Big|\sum_{j=1}^d \e_j^\top \bOmega^{t-1}\ldots \bOmega^{\ell+1}
\r^{\ell,(s,j)}\Big|
\leq \sqrt{\Big\|\bOmega^{t-1}\ldots \bOmega^{\ell+1}
\nabla_\alpha s(\btheta^\ell,\widehat\alpha^\ell)\Big\|_F^2}
\cdot \sqrt{\sum_{j=1}^d \pnorm{\deps\widehat\alpha^{\ell,(s,j),\eps}}{}^2} \leq
C\sqrt{d},
\end{align*}
which implies that the second term of (\ref{eq:avgderivtmp}) converges to 0
a.s.\ as $n,d \to \infty$. This proves the first claim.
The proof of the second claim is analogous, and omitted for brevity.
\end{proof}

Let us now introduce a notation for the discrete DMFT response process
(\ref{eq:dmft_theta_discrete1}) prior to taking an expectation.
Fixing a univariate process $\theta=\{\theta^t\}_{t\in \Z_+}$ and $\R^K$-valued
process $\alpha=\{\alpha^t\}_{t\in\Z_+}$ as inputs, define the following
auxiliary process $\{r_\theta^{(\theta,\alpha)}(t,s)\}_{s<t}$:
\begin{align}\label{eq:dmft_response_rho}
\notag r_\theta^{(\theta,\alpha)}(t+1,s) &= 
\begin{cases} \gamma & \text{ for } s=t,\\
\Big(1 - \gamma\delta\beta +
\gamma \partial_\theta s(\theta^t,\alpha^t)\Big) r_\theta^{(\theta,\alpha)}(t,s) + \gamma
\sum_{\ell=s+1}^{t-1} R_\eta^{\gamma}(t,\ell)
r_\theta^{(\theta,\alpha)}(\ell,s) & \text{ for } s<t.
\end{cases}\\
\end{align}
Note that if the inputs $\{\theta^t,\alpha^t\}$ are given by the
discrete-time DMFT processes defined in (\ref{eq:dmft_theta_discrete1}) and 
(\ref{eq:dmft_CRtheta_discrete}), then
\begin{align*}
r_\theta^{(\theta,\alpha)}(t,s)=\frac{\partial \theta_\gamma^t}{\partial
u^s_\gamma}, \qquad \E[r_\theta^{(\theta,\alpha)}(t,s)]=R_\theta^\gamma(t,s)
 \end{align*}
which are precisely the auxiliary process defined in (\ref{eq:response_1}) and DMFT
response function in (\ref{eq:dmft_CRtheta_discrete}). We will instead consider
(\ref{eq:dmft_response_rho}) with inputs $\{\theta_j^t,\widehat\alpha^t\}_{t \in
\Z_+}$ given by the coordinates of $\{\btheta^t,\widehat\alpha^t\}$ solving
(\ref{eq:original_alpha}).

\begin{lemma}\label{lem:dmft_emp_concentrate}
Let $\{\btheta^t,\widehat\alpha^t\}_{t \in \Z_+}$ be defined by
(\ref{eq:original_alpha}), and let
$R_\theta^\gamma(t,s)$ be the response function of its DMFT limit
defined in (\ref{eq:dmft_CRtheta_discrete}).
Then for any fixed $s,t \in \Z_+$ with $s<t$, almost surely
\begin{align*}
\lim_{n,d \to \infty}
\frac{1}{d}\sum_{j=1}^d r_{\theta}^{(\theta_j,\widehat\alpha)}(t,s)=
R_\theta^\gamma(t,s).
\end{align*}
\end{lemma}
\begin{proof}
First note that by the Lipschitz bound for $\partial_\theta s(\cdot)$
in Assumption \ref{assump:prior}, the a.s.\ convergence
$\{\widehat\alpha^t\} \to \{\alpha^t\}$ in Lemma \ref{lem:finite_dim_converge},
and a simple induction argument, we have almost surely
\[\lim_{n,d \to \infty}
\Big(\frac{1}{d}\sum_{j=1}^d r_{\theta}^{(\theta_j,\widehat\alpha)}(t,s)
-\frac{1}{d}\sum_{j=1}^d r_{\theta}^{(\theta_j,\alpha)}(t,s)\Big)=0.\]
By a similar induction argument using the boundedness and
Lipschitz-continuity of $\partial_\theta s(\cdot)$, for each fixed $s<t$,
the map $(\theta_j^0,\ldots,\theta_j^t) \mapsto
r_\theta^{(\theta_j,\alpha)}(t+1,s)$ is Lipschitz for each $j$.
Then by the empirical Wasserstein-2 convergence for $\btheta^0,\ldots,\btheta^T$
in (\ref{eq:finite_dim_converge_knot_theta}) of
Lemma \ref{lem:finite_dim_converge}, almost surely
\[\lim_{n,d \to \infty}
\frac{1}{d}\sum_{j=1}^d r_{\theta}^{(\theta_j,\alpha)}(t,s)
=\E[r_{\theta}^{(\theta,\alpha)}(t,s)]\]
where the inputs
$(\theta,\alpha)$ on the right side are the discrete-time DMFT processes
(\ref{eq:dmft_theta_discrete1}) and 
(\ref{eq:dmft_CRtheta_discrete}), and $\E[\cdot]$ is the expectation over their
law. The lemma follows from noting that, by
definition, $R_\theta^\gamma(t,s)=\E[r_{\theta}^{(\theta,\alpha)}(t,s)]$.
\end{proof}

We now proceed to prove Lemma \ref{lem:discrete_time_response}.

\begin{proof}[Proof of Lemma \ref{lem:discrete_time_response}]
For any fixed $s,t \in \Z_+$ with $s<t$,
set also
\[r_\eta(t,s)=\frac{\partial \eta^t}{\partial w^s}
=(\delta\beta)^{-1}R_\eta^\gamma(t,s)\]
and define the error terms
\begin{align}
E_\theta^{t,(s,j)}&=\partial_\eps|_{\eps = 0} \theta_j^{t,(s,j), \eps}-
r_\theta^{(\theta_j,\widehat{\alpha})}(t,s),\label{eq:induce_theta}\\
E_\eta^{t,[s,i]}&=\partial_\eps|_{\eps = 0} \eta_i^{t,[s,i],
\eps}-\beta^{-1}r_\eta(t,s).\label{eq:induce_eta}
\end{align}
We first prove by induction on $t$ that for any $p \geq 1$
and $s,t \in \Z_+$ with $s<t$, almost surely
\begin{align}\label{eq:induction_moment}
\lim_{n,d \to \infty} \frac{1}{d} \sum_{j=1}^d |E_\theta^{t,(s,j)}|^p=0,
\quad
\lim_{n,d \to \infty} \frac{1}{n} \sum_{i=1}^n |E_\theta^{t,[s,i]}|^p=0.
\end{align} 
Fixing any $s \in \Z_+$, the base case $t=s+1$ holds, as
direct calculation via (\ref{eq:alpha_col_perturb}) shows that
\begin{align*}
\partial_\eps|_{\eps = 0} \theta_j^{s+1,(s,j),\eps} = \gamma =
r_\theta^{(\theta_j,\widehat\alpha)}(s+1,s), \quad j\in[d],
\end{align*}
and similarly via (\ref{eq:alpha_row_perturb}),
\begin{align*}
\partial_\eps|_{\eps = 0} \eta_i^{s+1,[s,i],\eps} = \gamma \pnorm{\X^\top
\e_i}{}^2 = \gamma + E_\eta^{s+1,[s,i]}
\end{align*}
where $n^{-1}\sum_{i=1}^n |E_\eta^{s+1,[s,i]}|^p \to 0$ a.s.\ under Assumption \ref{assump:model} while $r_\eta(s+1,s) = \beta
R_\theta^{\gamma}(s+1,s)=\gamma\beta$.

Suppose by induction that (\ref{eq:induction_moment}) holds for this fixed $s
\in \Z_+$ up to time $t$.
Note that Lemmas \ref{lem:A_trace} and \ref{lem:dmft_emp_concentrate}
then imply, with the matrix
$\bOmega^t$ defined in (\ref{def:A_matrix}), for each $s<t$, almost surely
\begin{equation}\label{eq:induction_consequence}
\begin{aligned}
\lim_{n,d \to \infty}
\bigg\{\gamma \cdot \frac{1}{d}\Tr\Big(\bOmega^{t-1}\ldots \bOmega^{s+1}\Big),
\;\frac{1}{d}\sum_{j=1}^d \partial_\eps|_{\eps = 0} \theta_j^{t,(s,j), \eps}
\bigg\}&=R_\theta^{\gamma}(t,s),\\
\lim_{n,d \to \infty}
\bigg\{\gamma \cdot \frac{1}{n}\Tr\Big(\X \bOmega^{t-1}\ldots \bOmega^{s+1}
\X^\top \Big), \;\frac{1}{n}\sum_{i=1}^n \partial_\eps|_{\eps = 0}
\eta_i^{t,[s,i], \eps}\bigg\}&=\beta^{-1}r_\eta(t,s)=(\delta\beta^2)^{-1}R_\eta^{\gamma}(t,s).
\end{aligned}
\end{equation}

\textbf{Claim for $\partial_\eps|_{\eps = 0} \theta_j^{t,(s,j), \eps}$.} 
We establish the claim (\ref{eq:induction_moment}) for $E_\theta^{t+1,(s,j)}$. Fixing both $s \in
\Z_+$ and $j \in [d]$, let us shorthand
$\btheta^{t,(s,d),\eps}$ as $\btheta^{t,\eps}$
and recall the notations $\btheta^t=(\theta^t_j,\btheta^t_{-j})$ from Section
\ref{subsec:loo} where $\theta_j^t$ is the $j$th coordinate
of $\btheta^t$. Writing correspondingly $\X=(\x_j,\X_{-j})$, we have
\begin{align*}
\btheta_{-j}^{t+1,\eps} &= \Big(\I - \gamma\beta\X_{-j}^\top
\X_{-j}\Big)\btheta_{-j}^{t,\eps}-\gamma\beta \X_{-j}^\top (\x_j
\theta_j^{t,\eps} - \y) + \gamma
s(\btheta_{-j}^{t,\eps},\widehat\alpha^{t,\eps}) +
\sqrt{2}(\b_{-j}^{t+1}-\b_{-j}^t)\\
\theta_j^{t+1,\eps} &= \Big(1-\gamma\beta\pnorm{\x_j}{}^2\Big)\theta_j^{t,\eps}
- \gamma\beta\x_j^\top (\X_{-j}\btheta_{-j}^{t,\eps} - \y) + \gamma
s(\theta_j^{t,\eps}, \widehat\alpha^{t,\eps}) + \sqrt{2}(b_j^{t+1} - b_j^t)\\
\widehat\alpha^{t+1,\eps} &= \widehat\alpha^{t,\eps} + \gamma\cdot
\cG(\widehat\alpha^{t,\eps}, \widehat\sP(\btheta^{t,\eps}))
\end{align*}
Define
\[\nabla_\alpha s(\btheta^t,\widehat\alpha^t)
=\Big(\nabla_\alpha s(\theta_i^t,\widehat\alpha^t)^\top\Big)_{i=1}^d
\in \R^{d \times K},
\quad \r^t = \nabla_\alpha
s(\btheta^t,\widehat\alpha^t) \deps \widehat\alpha^{t,\eps} \in
\R^d.\]
Then, taking the derivative of $\theta_j^{t+1,\eps}$ in $\eps$,
\begin{align}\label{eq:mu_expand_1}
\partial_\eps|_{\eps = 0} \theta_j^{t+1,\eps} &= \Big(1 -
\gamma\beta\pnorm{\x_j}{}^2 + \gamma \partial_\theta s(\theta_j^t,
\widehat\alpha^{t})\Big)\partial_\eps|_{\eps = 0} \theta_j^{t,\eps} +
\gamma r_j^t
- \gamma\beta \x_j^\top \X_{-j}\partial_\eps|_{\eps=0}\btheta_{-j}^{t,\eps}.
\end{align}
Taking the derivative of $\btheta_{-j}^{t,\eps}$ in $\eps$,
\[\deps \btheta_{-j}^{t,\eps} = \underbrace{\Big(\I -
\gamma\beta\X_{-j}^\top \X_{-j} + \gamma \diag(\partial_\theta
s(\btheta_{-j}^{t-1}, \widehat\alpha^{t-1}))\Big)}_{:=\bOmega_{-j}^{t-1}} \deps
\btheta_{-j}^{t-1,\eps} - \gamma\beta \X_{-j}^\top \x_j \deps
\theta_j^{t-1,\eps} + \gamma \r_{-j}^{t-1}.\]
Then, iterating this equality and using $\deps \btheta_{-j}^{s+1,(s,j),\eps}=0$ gives
\begin{align*}
\deps \btheta_{-j}^{t,\eps}
&= - \gamma\beta \sum_{k=s+1}^{t-1} \bOmega_{-j}^{t-1}\ldots
\bOmega_{-j}^{k+1} \X_{-j}^\top \x_j \cdot \deps \theta_j^{k,\eps}
+ \gamma \sum_{k=s+1}^{t-1} \bOmega_{-j}^{t-1}\ldots \bOmega_{-j}^{k+1,\eps}
\r_{-j}^k.
\end{align*}
Plugging the above expression into (\ref{eq:mu_expand_1}), we have
\begin{align}
\deps \theta_j^{t+1,\eps} &= \underbrace{\Big(1 -  \gamma\beta\pnorm{\x_j}{}^2 +
\gamma \partial_\theta s(\theta_j^t,\widehat\alpha^t)\Big) \cdot \deps
\theta_j^{t,\eps}}_{(I_j)}\notag\\
&\hspace{1in}+(\gamma\beta)^2 \sum_{k=s+1}^{t-1} \underbrace{\x_j^\top \X_{-j}
\bOmega_{-j}^{t-1}\ldots \bOmega_{-j}^{k+1} \X_{-j}^\top \x_j \cdot \deps
\theta_j^{k,\eps}}_{(II_{j,k})}\notag\\
&\hspace{1in} - \gamma^2\beta \sum_{k=s+1}^{t-1} \underbrace{\x_j^\top
\X_{-j}\bOmega_{-j}^{t-1}\ldots \bOmega_{-j}^{k+1} \r_{-j}^k}_{(III_{j,k})} +
\underbrace{\gamma r_j^t}_{(IV_j)}.\label{eq:inductivedecomp}
\end{align}

\noindent \textbf{Analysis of $(I_j)$.} We have
\begin{equation}\label{eq:analysisIj}
(I_j) = (1-
\gamma\delta\beta + \gamma \partial_\theta s(\theta_j^t, \widehat\alpha^t))
r_\theta^{(\theta_j,\widehat\alpha)}(t,s) + \mathsf{r}_1^{(j)},
\end{equation}
where
\begin{align*}
\mathsf{r}_1^{(j)} = \gamma\beta(\delta -
\pnorm{\x_j}{}^2)(\deps \theta_j^{t,\eps})
+  \Big(1- \gamma\delta\beta + \gamma \partial_\theta
s(\theta_j^t,\widehat\alpha^t)\Big) E^{t,(s,j)}_\theta. 
\end{align*}
For any $p \geq 1$,
by the induction hypothesis we have
$d^{-1}\sum_j |E^{t,(s,j)}_\theta|^p \to 0$ a.s. By the conditions for $\X$ in Assumption \ref{assump:model}, $d^{-1}\sum_j |\delta-\|\x_j\|^2|^p \to 0$ a.s. By Lemma \ref{lem:response_l2_estimate},
$\max_j |\deps \theta_j^{t,\eps}| \leq C$ a.s.\ for all large $n,d$, while by Assumption \ref{assump:prior}, $\partial_\theta s(\cdot)$ is also bounded.
Combining these bounds gives
$d^{-1}\sum_j |\mathsf{r}_1^{(j)}|^p \to 0$
a.s.\ for any $p \geq 1$.\\

\noindent\textbf{Analysis of $(II_{j,k})$.} Let
\[\bOmega_{-j}^{t,(j)}
=\I-\gamma\beta\X_{-j}^\top \X_{-j}+\gamma\diag(\partial_\theta
s(\btheta_{-j}^{t,(j)},\widehat\alpha^{t,(j)}))\]
be the analogue of $\bOmega_{-j}^t$ defined by the cavity dynamics
$\{\btheta^{t,(j)},\widehat\alpha^{t,(j)}\}$.
We first show that a.s.\ for all large $n,d$, we have for every $j \in [d]$ that
\begin{align}\label{eq:loo_approximation}
|\mathsf{r}^{(j,k)}_{2,1}|&:= \Big|\x_j^\top \X_{-j} \bOmega_{-j}^{t-1}\ldots
\bOmega_{-j}^{k+1} \X_{-j}^\top \x_j - \x_j^\top \X_{-j}
\bOmega_{-j}^{t-1,(j)}\ldots \bOmega_{-j}^{k+1,(j)} \X_{-j}^\top
\x_j\Big|\notag\\
&\leq C\sqrt{\frac{\log d}{d}}\Big(|\theta^0_j| + |\theta^\ast_j| +
\max_{u\in[0,t]}|b^u_j| + \sqrt{\log d}\Big). 
\end{align}
To see this, note that
\begin{align*}
|\mathsf{r}^{(j,k)}_{2,1}|
&\leq \sum_{\ell = k+1}^{t-1} \underbrace{\Big|\x_j^\top
\X_{-j}\bOmega_{-j}^{t-1,(j)} \ldots \bOmega_{-j}^{\ell+1,(j)}
(\bOmega_{-j}^\ell - \bOmega_{-j}^{\ell,(j)}) \bOmega_{-j}^{\ell-1} \ldots
\bOmega_{-j}^{k+1}\X_{-j}^\top \x_j\Big|}_{:=T^\ell_{(j,k)}}.
\end{align*}
Here
$\bOmega_{-j}^{\ell}-\bOmega_{-j}^{\ell,(j)}=\gamma \diag(\partial_\theta
s(\btheta_{-j}^\ell,\widehat\alpha^\ell)-\partial_\theta
s(\btheta_{-j}^{\ell,(j)}, \widehat\alpha^{\ell,(j)}))$.
Then we may bound
\begin{align}
|T^\ell_{(j,k)}|
&\leq \gamma\pnorm{\bOmega_{-j}^{\ell+1,(j)} \ldots
\bOmega_{-j}^{t-1,(j)}\X_{-j}^\top \x_j}{\infty} \cdot
\pnorm{\bOmega_{-j}^{\ell-1} \ldots \bOmega_{-j}^{k+1}\X_{-j}^\top \x_j}{2}
\cdot \pnorm{\partial_\theta
s(\btheta_{-j}^\ell,\widehat\alpha^\ell)-\partial_\theta
s(\btheta_{-j}^{\ell,(j)}, \widehat\alpha^{\ell,(j)})}{2}\notag\\
&\leq C\gamma\pnorm{\bOmega_{-j}^{\ell+1,(j)} \ldots
\bOmega_{-j}^{t-1,(j)}\X_{-j}^\top \x_j}{\infty} \cdot
\pnorm{\bOmega_{-j}^{\ell-1}}{\op} \ldots \pnorm{\bOmega_{-j}^{k+1}}{\op}
\pnorm{\X_{-j}^\top \x_j}{2}\notag\\
&\hspace{1in} \cdot(\pnorm{\btheta_{-j}^{\ell,(j)} -
\btheta_{-j}^\ell}{2} + \sqrt{d}  \pnorm{\widehat\alpha^{\ell,(j)} - \widehat\alpha^\ell}{2}).\label{eq:Tjklbound}
\end{align}
Since $\x_j$ is independent of $\bOmega_{-j}^{t,(j)}$ and $\X_{-j}$, we have by a
subgaussian tail bound
\[\P\Big[\e_i^\top \bOmega_{-j}^{\ell+1,(j)} \ldots
\bOmega_{-j}^{t-1,(j)}\X_{-j}^\top \x_j \geq C\sqrt{\frac{\log d}{d}}\,
\|\e_i^\top \bOmega_{-j}^{\ell+1,(j)} \ldots
\bOmega_{-j}^{t-1,(j)}\X_{-j}^\top\|_2\Big] \leq e^{-cd}\]
for each $i=1,\ldots,d$ and some constants $C,c>0$. Then, taking a union bound
and applying the Borel-Cantelli lemma, almost surely for all large $n,d$,
\[\sup_{j,\ell} \|\bOmega_{-j}^{\ell+1,(j)} \ldots
\bOmega_{-j}^{t-1,(j)}\X_{-j}^\top \x_j\|_\infty
\leq \sup_{j,\ell} C\sqrt{\frac{\log d}{d}}\,\|\bOmega_{-j}^{\ell+1,(j)}\|_\op \ldots
\|\bOmega_{-j}^{t-1,(j)}\|_\op \|\X_{-j}\|_\op.\]
The right side is bounded by $C'\sqrt{(\log d)/d}$ on an almost-sure event
where $\|\X\|_\op \leq C_0$ holds for all large $n,d$. Then, applying this to (\ref{eq:Tjklbound}) and applying also
Lemma \ref{lem:loo_error} to bound the last term of (\ref{eq:Tjklbound}), this shows (\ref{eq:loo_approximation}).
By the conditions of Assumption \ref{assump:model} and
the tail estimates of the Brownian motion in Lemma
\ref{lem:BM_maximal}, this bound (\ref{eq:loo_approximation}) in turn implies
$d^{-1}\sum_j |\mathsf{r}_{2,1}^{(j,k)}|^p \to 0$ a.s.\
for any $p \geq 1$.

Now consider
\begin{align*}
\mathsf{r}^{(j,k)}_{2,2}:=\x_j^\top \X_{-j}\bOmega_{-j}^{t-1,(j)}\ldots
\bOmega_{-j}^{k+1,(j)} \X_{-j}^\top \x_j -
\frac{1}{d}\Tr\Big(\X_{-j}\bOmega_{-j}^{t-1,(j)}\ldots
\bOmega_{-j}^{k+1,(j)}\X_{-j}^\top \Big).
\end{align*}
Since $\x_j$ is independent of $\bOmega_{-j}^{t,(j)}$ and $\X_{-j}$, the
Hanson-Wright inequality yields
\[\P\Big[|\mathsf{r}^{(j,k)}_{2,2}|
 \geq \max\Big(\frac{C\sqrt{\log d}}{d}\,\|\W\|_F,
\frac{C\log d}{d}\,\|\W\|_\op\Big)\Big] \leq e^{-cd}\]
for some $C,c>0$, where
$\W=\X_{-j}\bOmega_{-j}^{t-1,(j)}\ldots \bOmega_{-j}^{k+1,(j)}\X_{-j}^\top$.
Again taking a union bound over $j \in [d]$ and applying $\|\W\|_\op \leq C_0$
and $\|\W\|_F \leq C_0\sqrt{d}$ a.s.\ for all large $n,d$, this implies
$d^{-1}\sum_j |\mathsf{r}^{(j,k)}_{2,2}|^p \to 0$ a.s.\
for any $p \geq 1$.

Finally, let $\X^{(j)} \in \R^{n \times d}$ be the embedding of $\X_{-j}$ with
$j^\text{th}$ column set to 0 as defined in (\ref{eq:leaveoutX}), and let
\[\bOmega^{t,(j)}=\I-\gamma \beta \X^{(j)\top}\X^{(j)}
+\gamma \diag(\partial_\theta s(\btheta^{t,(j)},\widehat\alpha^{t,(j)}))
 \in \R^{d \times d}.\]
Consider
\begin{align*}
\mathsf{r}_{2,3}^{(j,k)}
&:=\frac{1}{d}\Tr\Big(\X_{-j}\bOmega_{-j}^{t-1,(j)}\ldots
\bOmega_{-j}^{k+1,(j)}\X_{-j}^\top \Big)
-\frac{1}{d}\Tr\Big(\X\bOmega^{t-1}\ldots
\bOmega^{k+1}\X^\top \Big)\\
&=\frac{1}{d}\Tr\Big(\X^{(j)}\bOmega^{t-1,(j)}\ldots
\bOmega^{k+1,(j)}\X^{(j)\top} \Big)
-\frac{1}{d}\Tr\Big(\X\bOmega^{t-1}\ldots \bOmega^{k+1}\X^\top \Big)\\
&=\frac{1}{d}\Tr \Big(\X^{(j)}\bOmega^{t-1,(j)}\ldots
\bOmega^{k+1,(j)}(\X^{(j)}-\X)^\top \Big)\\
&\hspace{1in}
+\sum_{\ell={k+1}}^{t-1}
\frac{1}{d}\Tr \Big(\X^{(j)}\bOmega^{t-1,(j)}\ldots \bOmega^{\ell+1,(j)}
(\bOmega^{\ell}-\bOmega^{\ell,(j)})
\bOmega^{\ell-1}\ldots
\bOmega^{k+1}\X^{\top} \Big)\\
&\hspace{1in}
+\frac{1}{d}\Tr \Big((\X^{(j)}-\X)\bOmega^{t-1}\ldots
\bOmega^{k+1}\X^\top\Big)
\end{align*}
Almost surely for all large $n,d$, for every $j \in [d]$ we have
$\|\X^{(j)}-\X\|_F=\|\x_j\| \leq C$ and
\begin{align*}
\|\bOmega^{t,(j)}-\bOmega^t\|_F &\leq
\gamma\beta\|\X^{(j)\top}\X^{(j)}-\X^\top \X\|_F
+\gamma \|\partial_\theta s(\btheta^{t,(j)},\widehat\alpha^{t,(j)})
-\partial_\theta s(\btheta^t,\widehat\alpha^t)\|\\
&\leq C\Big(1+|\theta_j^0|+|\theta_j^*|+\max_{u \in [0,t]} |b_j^u|+\sqrt{\log
d}\Big),
\end{align*}
the second inequality applying the Lipschitz continuity of $s(\cdot)$ in
Assumption \ref{assump:prior} and Lemma \ref{lem:loo_error}. Then, applying
$\Tr (A-B)C \leq \|A-B\|_F\|C\|_F \leq \sqrt{d}\,\|A-B\|_F\|C\|_\op$, we obtain
a.s.\ for all large $n,d$ that for every $j \in [d]$,
\[|r_{2,3}^{(j,k)}|
\leq \frac{C}{\sqrt{d}}\Big(1+|\theta_j^0|+|\theta_j^*|+\max_{u \in [0,t]}
|b_j^u|+\sqrt{\log d}\Big),\]
which implies as above that
$d^{-1}\sum_j |\mathsf{r}^{(j,k)}_{2,3}|^p \to 0$ a.s.\
for any $p \geq 1$.

Combining these bounds for
$\mathsf{r}_{2,1}^{(j,k)},\mathsf{r}_{2,2}^{(j,k)},\mathsf{r}_{2,3}^{(j,k)}$,
the second statement of 
(\ref{eq:induction_consequence}) for
almost sure convergence of
$d^{-1}\Tr(\X\bOmega^{t-1}\ldots \bOmega^{k+1}\X^\top)$,
the induction hypothesis for
approximation of $\deps \theta_j^{k,\eps}$ by $r_\theta^{(\theta_j,\widehat\alpha)}$, and
the bound $|\deps \theta_j^{k,\eps}| \leq C$ a.s.\ for all large $n,d$ by Lemma \ref{lem:response_l2_estimate}, we get that
\begin{equation}\label{eq:analysisIIjk}
(II_{j,k})=\frac{1}{\gamma
\beta^2}R_\eta^\gamma(t,k)r_\theta^{(\theta_j,\widehat\alpha)}(k,s)
+\mathsf{r}_{2}^{(j,k)}
\end{equation}
where $d^{-1}\sum_j |\mathsf{r}_{2}^{(j,k)}|^p \to 0$ a.s.\ for any $p \geq 1$.\\

\noindent\textbf{Analysis of $(III_{j,k})$.} We apply a similar leave-one-out argument as above. Let
\[\r_{-j}^{t,(j)}=\nabla_\alpha s(\btheta_{-j}^{t,(j)},\widehat\alpha^{t,(j)})
\deps \widehat\alpha^{t,\eps} \in \R^{d-1}.\]
(Note that we replace only the first factor
$\nabla_\alpha s(\btheta_{-j}^{t,(j)},\widehat\alpha^{t,(j)})$ by the cavity
dynamics, leaving the second factor unchanged.) Then
\begin{align}\label{eq:theta_loo_arg}
&\Big|\x_j^\top \X_{-j} \bOmega_{-j}^{t-1}\ldots
\bOmega_{-j}^{k+1}\r_{-j}^k - \x_j^\top \X_{-j} \bOmega_{-j}^{t-1,(j)}\ldots
\bOmega_{-j}^{k+1,(j)}\r_{-j}^{k,(j)}\Big|\notag\\
&\leq \underbrace{\sum_{\ell = k+1}^{t-1} \Big|\x_j^\top
\X_{-j}\bOmega_{-j}^{t-1,(j)} \ldots \bOmega_{-j}^{\ell+1,(j)}
(\bOmega_{-j}^{\ell}-\bOmega_{-j}^{\ell,(j)}) \bOmega_{-j}^{\ell-1} \ldots
\bOmega_{-j}^{k+1}\r_{-j}^k\Big|}_{:=u^{(j,k)}}\notag\\
&\hspace{1in}+\underbrace{\Big|\x_j^\top \X_{-j}\bOmega_{-j}^{t-1,(j)} \ldots
\bOmega_{-j}^{k+1,(j)}(\r_{-j}^k - \r_{-j}^{k,(j)})\Big|}_{:=v^{(j,k)}}.
\end{align}
We note that a.s.\ for all large $n,d$, we have
$\|\r^k\| \leq C\sqrt{d} \cdot C/\sqrt{d} \leq C'$
by the Lipschitz bound for $s(\cdot)$ in Assumption \ref{assump:prior}
and Lemma \ref{lem:response_l2_estimate}. Then, using
similar arguments as in the analysis of $\mathsf{r}_{2,1}^{(j,k)}$ above,
we have $d^{-1} \sum_j |u^{(j,k)}|^p \to 0$ a.s.\ for any $p \geq 1$.
For the second term, we have
\begin{align*}
v^{(j,k)} &= \Big|\x_j^\top \X_{-j}\bOmega_{-j}^{t-1,(j)} \ldots
\bOmega_{-j}^{j+1,(j)} \Big(\nabla_\alpha s(\btheta_{-j}^k,\widehat\alpha^k) -
\nabla_\alpha s(\btheta_{-j}^{k,(j)},\widehat\alpha^{k,(j)})\Big) \deps
\widehat \alpha^{k,\eps}\Big|\\
&\leq
C\pnorm{\x_j}{}\pnorm{\X_{-j}}{\op}\pnorm{\bOmega_{-j}^{t-1,(j)}}{\op}\ldots
\pnorm{\bOmega_{-j}^{k+1,(j)}}{\op}(\pnorm{\btheta_{-j}^k-\btheta_{-j}^{k,(j)}}{}
+ \sqrt{d}\pnorm{\widehat\alpha^k - \widehat\alpha^{k,(j)}}{})
\pnorm{\deps \widehat\alpha^{k,\eps}}{},
\end{align*}
which satisfies $d^{-1}\sum_j |v^{(j,k)}|^p \to 0$ a.s.\
for all $p \geq 1$ by Lemmas \ref{lem:loo_error} and \ref{lem:response_l2_estimate}. Thus
\[\x_j^\top \X_{-j} \bOmega_{-j}^{t-1}\ldots
\bOmega_{-j}^{k+1}\r_{-j}^k=\x_j^\top \X_{-j} \bOmega_{-j}^{t-1,(j)}\ldots
\bOmega_{-j}^{k+1,(j)}\r_{-j}^{k,(j)}+\mathsf{r}^{(j,k)}\]
where $d^{-1}\sum_j |\mathsf{r}^{(j,k)}|^p \to 0$ a.s.
On the other hand, we have
\begin{align*}
\Big|\x_j^\top \X_{-j} \bOmega_{-j}^{t-1,(j)}\ldots
\bOmega_{-j}^{k+1,(j)}\r_{-j}^{k,(j)}\Big| &= \Big|\x_j^\top \X_{-j}
\bOmega_{-j}^{t-1,(j)}\ldots \bOmega_{-j}^{k+1,(j)} \nabla_\alpha
s(\btheta_{-j}^{k,(j)},\widehat\alpha^{k,(j)}) \cdot \deps \widehat\alpha^{k,\eps}\Big|\\
&\leq \pnorm{\x_j^\top \X_{-j} \bOmega_{-j}^{t-1,(j)}\ldots
\bOmega_{-j}^{k+1,(j)} \nabla_\alpha
s(\btheta_{-j}^{k,(j)},\widehat\alpha^{k,(j)})}{}\pnorm{\deps \widehat\alpha^{k,\eps}}{}.  
\end{align*}
Since $\X_{-j},\bOmega_{-j}^{t,(j)},\btheta_{-j}^{k,(j)},\widehat\alpha^{k,(j)}$
are all independent of $\x_j$, a subgaussian tail bound and union bound shows,
a.s.\ for all large $n,d$, that for every $j \in [d]$,
\[\pnorm{\x_j^\top \X_{-j} \bOmega_{-j}^{t-1,(j)}\ldots
\bOmega_{-j}^{k+1,(j)} \nabla_\alpha
s(\btheta_{-j}^{k,(j)},\widehat\alpha^{k,(j)})}{}
\leq C\sqrt{\log d}.\]
Since $\pnorm{\deps \widehat\alpha^{k,\eps}}{} \leq C/\sqrt{d}$
by Lemma \ref{lem:response_l2_estimate}, this shows
\begin{equation}\label{eq:analysisIIIjk}
(III_{j,k})=\mathsf{r}_3^{(j,k)}
\end{equation}
where $\lim_{n,d \to \infty}
d^{-1}\sum_{j=1}^d |\mathsf{r}_3^{(j,k)}|^p=0$ a.s.\ for any $p \geq 1$.\\

\noindent\textbf{Analysis of $(IV_j)$.} 
By Assumption \ref{assump:prior} and Lemma \ref{lem:response_l2_estimate}, 
$|(IV_j)| \leq \gamma
\pnorm{\nabla_\alpha s(\theta_j^t,\widehat\alpha^t)}{}\pnorm{\deps
\widehat\alpha^{t,\eps}}{} \leq C/\sqrt{d}$ a.s.\ for all large $n,d$,
hence 
\begin{equation}\label{eq:analysisIVj}
(IV_j)=\mathsf{r}_4^{(j)}
\end{equation}
where $\lim_{n,d \to \infty}
d^{-1}\sum_{j=1}^d |\mathsf{r}_4^{(j)}|^p=0$ a.s.\ for any $p \geq 1$.\\

Applying (\ref{eq:analysisIj}), (\ref{eq:analysisIIjk}),
(\ref{eq:analysisIIIjk}), and (\ref{eq:analysisIVj}) back to
(\ref{eq:inductivedecomp}),
\begin{align*}
\partial_\eps|_{\eps = 0} \theta_j^{t+1,\eps} &= \big(1- \gamma\delta\beta +
\gamma \partial_\theta s(\theta_j^t,\widehat\alpha^t)\big)
r_\theta^{(\theta_j,\widehat\alpha)}(t,s) + \gamma\sum_{k=s+1}^{t-1}
R_\eta^{\gamma}(t,k) r_\theta^{(\theta_j,\widehat\alpha)}(k,s)
+E_\theta^{t+1,(s,j)}\\
&= r_\theta^{(\theta_j,\widehat\alpha)}(t+1,s) + E_\theta^{t+1,(s,j)}
\end{align*}
where $\lim_{n,d \to \infty} d^{-1}\sum_{j=1}^d |E_\theta^{t+1,(s,j)}|^p=0$
a.s.\ for each $p \geq 1$, concluding the proof the inductive claim (\ref{eq:induction_moment}) for $E_\theta^{t+1,(s,j)}$.\\

\textbf{Claim for $\partial_\eps|_{\eps = 0} \bbeta_i^{t,[s,i], \eps}$.} 
We now show the claim (\ref{eq:induction_moment}) for $E_\eta^{t+1,[s,i]}$. Again fixing $s \in \Z_+$
and $i \in [n]$, let us
shorthand $\bbeta^{t,[s,i],\eps}$ and $\btheta^{t,[s,i],\eps}$ as
$\bbeta^{t,\eps}$ and $\btheta^{t,\eps}$. Let us write
$\bbeta^t=(\eta_i^t,\bbeta_{-i}^t)$ as in Section \ref{subsec:loo},
and write correspondingly
$\y=(y_i,\y_{-i})$, $\beps=(\eps_i,\beps_{-i})$, and
$\X=[\x_i,\X_{-i}^\top]^\top$ where $\x_i \in \R^d$ denotes now (the transpose of) the $i^\text{th}$ \emph{row}
of $\X$, and $\X_{-i} \in \R^{(n-1)\times d}$. Then
\begin{align*}
\btheta^{t+1,\eps} &= \btheta^{t,\eps} + \gamma\Big[{-}\beta
\big(\X_{-i}^\top (\X_{-i}\btheta^{t,\eps} - \y_{-i}) + \x_i
(\eta_i^{t,\eps}-y_i)\big) + s(\btheta^{t,\eps},\widehat\alpha^{t,\eps})
\Big] + \sqrt{2}(\b^{t+1} - \b^t)\\
\eta_i^{t+1,\eps} &= \eta_i^{t,\eps} + \gamma\Big[{-}\beta \x_i^\top
\big(\X_{-i}^\top (\X_{-i}\btheta^{t,\eps} - \y_{-i}) + \x_i(\eta_i^{t,\eps} -
y_i)\big) + \x_i^\top s(\btheta^{t,\eps},\widehat\alpha^{t,\eps}) \Big] +
\sqrt{2}\x_i^\top (\b^{t+1} - \b^t)\\
\widehat\alpha^{t+1,\eps} &= \widehat\alpha^{t,\eps} + \gamma \cdot
\cG(\widehat\alpha^{t,\eps} , \widehat\sP(\btheta^{t,\eps})). 
\end{align*}
Set
\[\r^t=\nabla_\alpha s(\btheta^t,\widehat\alpha^t) \deps \widehat\alpha^{t,\eps}
\in \R^d.\]
Then, taking the derivative of $\eta_i^{t+1,\eps}$ yields
\begin{align}\label{eq:kappa_expand}
\partial_\eps|_{\eps = 0}\eta_i^{t+1,\eps}
= \Big(1 - \gamma\beta\pnorm{\x_i}{}^2\Big)\deps \eta_i^{t,\eps} +
\x_i^\top\Big({-}\gamma\beta\X_{-i}^\top \X_{-i}+\gamma \diag(\partial_\theta
s(\btheta^t,\widehat\alpha^t))\Big)\deps \btheta^{t,\eps} + \gamma \x_i^\top
\r^t.
\end{align}
Taking derivative of $\btheta^{t,\eps}$ yields
\begin{align*}
\partial_\eps|_{\eps = 0} \btheta^{t,\eps} &= 
\underbrace{\Big(\I-\gamma\beta\X_{-i}^\top \X_{-i} +
\gamma\diag(\partial_\theta s(\btheta^{t-1},
\widehat\alpha^{t-1}))\Big)}_{:=\bOmega_{-i}^{t-1}}\partial_\eps|_{\eps = 0}
\btheta^{t-1,\eps} - \gamma\beta\x_i \deps \eta_i^{t-1,\eps}+\gamma \r^{t-1}.
\end{align*}
Iterating this equality and using
$\partial_\eps|_{\eps = 0} \btheta^{s+1,[s,i],\eps} = \gamma \x_i$ gives
\begin{align*}
\partial_\eps|_{\eps = 0} \btheta^{t,\eps} 
&=\gamma \bOmega_{-i}^{t-1}\ldots \bOmega_{-i}^{s+1}\x_i -
\gamma\beta\sum_{k=s+1}^{t-1} \bOmega_{-i}^{t-1}\ldots \bOmega_{-i}^{k+1} \x_i
\cdot  \deps \eta_i^{k,\eps} + \gamma \sum_{k=s+1}^{t-1}
\bOmega_{-i}^{t-1}\ldots \bOmega_{-i}^{k+1} \r^k.
\end{align*}
Plugging the above expression into (\ref{eq:kappa_expand}), we have
\begin{align}\label{eq:kappa_expand_2}
\notag&\partial_\eps|_{\eps = 0}\eta_i^{t+1,\eps} = \underbrace{\Big(1 -
\gamma\beta\pnorm{\x_i}{}^2\Big)\deps \eta_i^{t,\eps}}_{(I_i)} + \gamma
\underbrace{\x_i^\top(\bOmega_{-i}^t-\I)\bOmega_{-i}^{t-1}\ldots
\bOmega_{-i}^{s+1}\x_i}_{(II_i)}\\
&\quad-\gamma\beta\sum_{k=s+1}^{t-1} \underbrace{\x_i^\top(\bOmega_{-i}^t-\I)
\bOmega_{-i}^{t-1}\ldots \bOmega_{-i}^{k+1}\x_i \cdot  \deps
\eta_i^{k,\eps}}_{(III_{i,k})}
+\gamma \sum_{k=s+1}^{t-1} \underbrace{\x_i^\top (\bOmega_{-i}^t-\I)
\bOmega^{t-1}\ldots \bOmega^{k+1}\r^k}_{(IV)_{i,k}}+\underbrace{\gamma
\x_i^\top \r^t}_{(V_i)}.
\end{align}
The arguments to analyze these terms are similar to the above, and we will
omit some details.\\

\noindent\textbf{Analysis of $(I_i)$.} By the induction hypothesis and
concentration of $\|\x_i\|^2$ around 1,
\begin{equation}\label{eq:analysisIi}
(I_i)=(1-\gamma\beta)\beta^{-1}r_\eta(t,s)+\mathsf{r}_1^{[i]}
\end{equation}
where $n^{-1}\sum_{i=1}^n |\mathsf{r}_1^{[i]}|^p \to 0$ a.s.\ for any $p \geq
1$.\\

\noindent\textbf{Analysis of $(II_i)$.} Let $\bOmega_{-i}^{t,[i]} = \I -
\gamma\beta \X_{-i}^\top \X_{-i} + \gamma \diag(\partial_\theta
s(\btheta^{t,[i]},\widehat\alpha^{t,[i]}))$ with
$\btheta^{t,[i]},\widehat\alpha^{t,[i]}$ given by the cavity dynamics of
(\ref{eq:leave_row_out}). Set
\begin{align*}
\mathsf{r}^{[i]}_{2,1}&=\x_i^\top \bOmega_{-i}^{t-1}\ldots
\bOmega_{-i}^{s+1}\x_i - \x_i^\top\bOmega_{-i}^{t-1,[i]}\ldots
\bOmega_{-i}^{s+1,[i]}\x_i\\
\mathsf{r}^{[i]}_{2,2}&=\x_i^\top\bOmega_{-i}^{t-1,[i]}\ldots
\bOmega_{-i}^{s+1,[i]}\x_i-\frac{1}{d}\Tr \bOmega_{-i}^{t-1,[i]}\ldots
\bOmega_{-i}^{s+1,[i]}\\
\mathsf{r}^{[i]}_{2,3}&=\frac{1}{d}\Tr\bOmega_{-i}^{t-1,[i]}\ldots
\bOmega_{-i}^{s+1} - \frac{1}{d}\Tr\bOmega^{t-1}\ldots
\bOmega^{s+1,[i]}.
\end{align*}
Then the same arguments above yield
$n^{-1}\sum_{i=1}^n |\mathsf{r}^{[i]}_{2,j}|^p \to 0$ for each $j=1,2,3$.
Applying the same arguments for $t$ in place of $t-1$, and
the first statement of (\ref{eq:induction_consequence}) for both $t$
and $t-1$,
\begin{equation}\label{eq:analysisIIi}
(II_i)=\gamma^{-1}R_\theta^\gamma(t+1,s)-\gamma^{-1}
R_\theta^\gamma(t,s)+\mathsf{r}_2^{[i]}
\end{equation}
where $n^{-1}\sum_{i=1}^n |\mathsf{r}_2^{[i]}|^p \to 0$ a.s.\\

\noindent\textbf{Analysis of $(III_{i,k})$, $(IV_{i,k})$, and $(V_i)$.}
Similar arguments as above show
\begin{align}
(III_{i,k})&=(\gamma\beta)^{-1}R_\theta^\gamma(t+1,k)r_\eta(k,s)
-(\gamma\beta)^{-1}R_\theta^\gamma(t,k)r_\eta(k,s)+\mathsf{r}_3^{[i,k]}\label{eq:analysisIIIik}\\
(IV)_{i,k}&=\mathsf{r}_4^{[i,k]}\label{eq:analysisIVik}\\
(V)_i&=\mathsf{r}_5^{[i]}\label{eq:analysisVi}
\end{align}
where $n^{-1}\sum_{i=1}^n |\mathsf{r}_3^{[i,k]}|^p \to 0$,
$n^{-1}\sum_{i=1}^n |\mathsf{r}_4^{[i,k]}|^p \to 0$, and
$n^{-1}\sum_{i=1}^n |\mathsf{r}_5^{[i]}|^p \to 0$ a.s.\\

Applying (\ref{eq:analysisIi}), (\ref{eq:analysisIIi}),
(\ref{eq:analysisIIIik}), (\ref{eq:analysisIVik}), and (\ref{eq:analysisVi})
back to (\ref{eq:kappa_expand_2}), for an error term
$E_\eta^{t+1,[s,i]}$ satisfying
$n^{-1}\sum_{i=1}^n |E_\eta^{t+1,[s,i]}|^p \to 0$ a.s., we have
\begin{align*}
\deps \eta_i^{t+1,\eps}&=(1 - \gamma\beta)\beta^{-1}r_\eta(t,s) + R_\theta^{\gamma}(t+1,s) 
- R_\theta^{\gamma}(t,s)\\
&\hspace{1in}
-\sum_{k=s+1}^{t-1} (R_\theta^{\gamma}(t+1,k)
- R_\theta^{\gamma}(t,s)) r_\eta(k,s)
+E_\eta^{t+1,[s,i]}\\
&= \Big[{-}\gamma r_\eta(t,s) + R_\theta^{\gamma}(t+1,s) - \sum_{k=s+1}^{t-1}
R_\theta^{\gamma}(t+1,k) r_\eta(k,s)\Big]\\
&\hspace{1in} + \underbrace{\Big[\beta^{-1} r_\eta(t,s) - R_\theta^{\gamma}(t,s)  +
\sum_{k=s+1}^{t-1} R_\theta^{\gamma}(t,k) r_\eta(k,s)\Big]}_{=0}
 + E_\eta^{t+1,[s,i]}
\\
&= R_\theta^{\gamma}(t+1,s) - \sum_{k=s+1}^{t} R_\theta^{\gamma}(t+1,k)
r_\eta(k,s) + E_\eta^{t+1,[s,i]}\\
&= \beta^{-1} r_\eta(t+1,s) + E_\eta^{t+1,[s,i]}.
\end{align*}
This shows the inductive claim (\ref{eq:induction_moment}) for $E_\eta^{t+1,[s,i]}$, and hence concludes the induction.

To conclude the proof, by boundedness of $\partial_\theta
s(\cdot)$ and the definition (\ref{eq:dmft_response_rho}),
$d^{-1}\sum_{j=1}^d r_\theta^{(\theta_j,\widehat\alpha)}(t,s)$ is bounded
by a constant. Furthermore, by the expansion (\ref{eq:avgderivtmp}) and its
following arguments (which hold also at non-zero $\eps>0$),
on the event $\|\X\|_\op \leq C_0$, we have that
$d^{-1}\sum_{j=1}^d \partial_\eps \theta_j^{t,(s,j),\eps}$ is also bounded by a
constant for all sufficiently small $\eps \geq 0$.
Then, writing $\langle \cdot \rangle$ for the expectation over only
the discrete Brownian motion $\{\b^t\}_{t \in \Z_+}$, we may apply the dominated
convergence theorem to the first statement of (\ref{eq:induction_moment}) to get, almost surely,
\begin{align*}
\lim_{n,d \to \infty} d^{-1}\Tr \bR_\theta(t,s)
&=\lim_{n,d \to \infty} \frac{1}{d}\sum_{j=1}^d
\deps \langle \theta_j^{t,(s,j),\eps} \rangle
=\lim_{n,d \to \infty} \frac{1}{d}\sum_{j=1}^d
\langle \deps \theta_j^{t,(s,j),\eps}  \rangle\\
&=\lim_{n,d \to \infty} \frac{1}{d}\sum_{j=1}^d
\langle r_\theta^{(\theta_j,\alpha)}(t,s) \rangle
=R_\theta^\gamma(t,s),
\end{align*}
the last equality holding by Lemma
\ref{lem:dmft_emp_concentrate}. Similarly we may apply the dominated convergence
theorem to the second statement of (\ref{eq:induction_moment}) to get, almost surely,
\[\lim_{n,d \to \infty} n^{-1}\Tr \bR_\eta(t,s)
=\lim_{n,d \to \infty} \delta\beta^2 \cdot \frac{1}{n}\sum_{i=1}^n
\deps \langle \eta_i^{t,[s,i],\eps} \rangle
=\lim_{n,d \to \infty} \delta\beta \cdot \frac{1}{n}\sum_{i=1}^n
\langle r_\eta(t,s) \rangle=R_\eta^\gamma(t,s),\]
concluding the proof.
\end{proof}

\subsection{Discretization of Langevin response function}\label{subsec:discretize_langevin_response}

In the following, we denote $\x=(\btheta,\widehat\alpha) \in \R^{d+K}$ and consider (\ref{eq:langevin_sde}--\ref{eq:gflow}) as a joint diffusion in the variables $\x^t=(\btheta^t,\widehat\alpha^t)$.
Let $u:\R^{d+K} \to \R^{d+K}$ and $\M \in \R^{(d+K) \times (d+K)}$ be defined by
\begin{equation}\label{eq:def_M_u}
\begin{aligned}
u(\x)&=u(\btheta,\widehat\alpha)=\Big({-}\beta \X^\top (\X\btheta-\y)
+\big(s(\theta_j,\widehat\alpha)\big)_{j=1}^d,\;
\cG(\widehat\alpha,\widehat\sP(\btheta))\Big)\\
\M&=\diag(\I_{d \times d},0_{K \times K})
\end{aligned}
\end{equation}
Given an initial condition $\x^0 \in \R^{d+K}$, we consider the continuous-time
dynamics for $\x^t \in \R^{d+K}$ and $\V^t \in \R^{(d+K) \times (d+K)}$
defined by
\begin{align}\label{eq:theta_V}
\notag\x^t &= \x^0+\int_0^t u(\x^s)\d s + \sqrt{2}\,\M \b^t\\
\qquad \V^t &= \I_{d+K}+\int_0^t [\d u(\x^s) \V^s]\d s
\end{align}
where $\d u(\x) \in \R^{(d+K) \times (d+K)}$ is the derivative of $u(\cdot)$ at
$\x$. We consider also a piecewise-constant version of these dynamics
\begin{align}\label{eq:theta_V_embed}
\notag\bar{\x}_\gamma^t &= \x+\int_0^{\floor{t}} u(\bar{\x}_\gamma^s)\d s +
\sqrt{2}\,\M \b^{\floor{t}}\\
\bar{\V}_\gamma^t &= \I_{d+K}+\int_0^{\floor{t}} [\d u(\bar{\x}_\gamma^s)
\bar{\V}_\gamma^s]\d s
\end{align}
where $\floor{t} \in \gamma\Z_+$ is as previously defined in
(\ref{eq:t_floor_ceil}). We note that the process
$\x^t=(\btheta^t,\widehat\alpha^t)$ in
(\ref{eq:theta_V}) is precisely our adaptive Langevin process of interest
(\ref{eq:langevin_sde}--\ref{eq:gflow}). Similarly, the process $\bar
\x_\gamma^t$ in (\ref{eq:theta_V_embed}) is the piecewise-constant embedding
from Section \ref{subsec:discretizelangevin} of
the discrete dynamics for $\x_\gamma^t=(\btheta_\gamma^t,\widehat\alpha_\gamma^t)$
which we have rewritten in (\ref{eq:original_alpha}). Denoting
$[t]=\floor{t}/\gamma \in \Z_+$ as in (\ref{eq:t_floor_ceil}), we have
\begin{equation}\label{eq:xVembedding}
\bar \x_\gamma^t=\x_\gamma^{[t]}=(\btheta_\gamma^{[t]},\widehat\alpha_\gamma^{[t]})
\text{ for all } t \geq 0.
\end{equation}
Throughout, we will write $\langle \cdot \rangle_{\x^0}$ for expectations only over
the Brownian motion $\b^t$, i.e.\ conditional
on $\X,\btheta^*,\beps$ and the initial condition $\x^0$.

\begin{lemma}\label{lem:theta_V_rough}
Let us write the block forms
\[\V^t=\begin{pmatrix} \U^t & \ast \\ \W^t & \ast
\end{pmatrix}, \quad
\bar{\V}_\gamma^t=\begin{pmatrix}\bar{\U}_\gamma^t & \ast \\
\bar{\W}_\gamma^t & \ast \end{pmatrix},
\quad \d u(\x^t)=\begin{pmatrix} \J^t_1 & \J^t_2 \\ \J^t_3 & \J^t_4
\end{pmatrix}, \quad \d u(\bar{\x}_\gamma^t)=
\begin{pmatrix} \bar\J_{\gamma,1}^t & \bar\J_{\gamma,2}^t \\
\bar\J_{\gamma,3}^t & \bar \J_{\gamma,4}^t \end{pmatrix}\]
with blocks of sizes $d$ and $K$.
Fixing any $T>0$, on the event $\{\|\X\|_\op \leq
C_0,\|\y\| \leq C_0\sqrt{d}\}$,
there is a constant $C>0$ (depending on $T,C_0$ but not on $\gamma$)
such that for any $\gamma>0$, we have
\begin{equation}\label{eq:Jbound}
\begin{gathered}
\sup_{t\in[0,T]} \pnorm{\J_1^t}{\op},\pnorm{\bar\J_{\gamma,1}^t}{\op} \leq C,
\qquad \sup_{t\in[0,T]}\pnorm{\J^t_2}{F},\pnorm{\bar\J_{\gamma,2}^t}{F} \leq C\sqrt{d}, \\
\sup_{t\in[0,T]}\pnorm{\J^t_3}{F},\pnorm{\bar\J_{\gamma,3}^t}{F} \leq
C/\sqrt{d}, \qquad\sup_{t\in[0,T]}\pnorm{\J^t_4}{F},\pnorm{\bar
\J_{\gamma,4}^t}{F}\leq C,
\end{gathered}
\end{equation}
\begin{equation}\label{eq:UWbound}
\sup_{t\in[0,T]}
\{\pnorm{\U^t}{\op},\sqrt{d}\pnorm{\W^t}{F},\pnorm{\bar{\U}_\gamma^t}{\op},\sqrt{d}\pnorm{\bar{\W}_\gamma^t}{F}\}
\leq C,
\end{equation}
\begin{equation}\label{eq:UWchangebound}
\sup_{t \in [0,T]} \|\bar\U_\gamma^{t+\gamma}-\bar\U_\gamma^t\|_\op
\leq C\gamma,
\quad \sup_{t \in [0,T]} \|\bar\W_\gamma^{t+\gamma}-\bar\W_\gamma^t\|_F
\leq C\gamma/\sqrt{d}.
\end{equation}
Furthermore, for some $\iota:\R_+ \to \R_+$ satisfying $\lim_{\gamma \to 0}
\iota(\gamma)=0$ and for any initial condition
$\x^0=(\btheta^0,\widehat\alpha^0)$,
\begin{equation}\label{eq:Jdiscretizebound}
\sup_{t\in[0,T]} \frac{\langle\pnorm{\J_1^t -
\bar\J_{\gamma,1}^t}{F}\rangle_{\x^0}}{\sqrt{d}},\frac{\langle\pnorm{\J_2^t-\bar\J_{\gamma,2}^t}{F}\rangle_{\x^0}}{\sqrt{d}},\sqrt{d}\langle\pnorm{\J_3^t
- \bar\J_{\gamma,3}^t}{F}\rangle_{\x^0},\langle \pnorm{\J_4^t - \bar\J_{\gamma,4}^t}{F}\rangle_{\x^0}
\leq \iota(\gamma)\Big(\frac{\pnorm{\btheta^0}{}}{\sqrt{d}} +
\pnorm{\widehat\alpha^0}{} + 1\Big).
\end{equation}
\end{lemma}
\begin{proof}
For (\ref{eq:Jbound}), we have by definition that
\begin{align*}
\J_1^t &= -\beta\X^\top \X + \diag\Big[\Big(\partial_\theta
s(\theta_j^t,\widehat\alpha^t)\Big)_{j=1}^d\Big], \quad \J_2^t =
\Big(\nabla_\alpha s(\theta_j^t,\widehat\alpha^t)^\top\Big)_{j=1}^d,\\
\J_3^t &= \d_{\btheta} \cG(\widehat\alpha^t,\sP(\btheta^t)),
\quad \J_4^t = \d_{\alpha} \cG(\widehat\alpha^t, \sP(\btheta^t)),
\end{align*}
and similarly for
$\bar\J_{\gamma,1}^t,\bar\J_{\gamma,2}^t,\bar\J_{\gamma,3}^t,\bar\J_{\gamma,4}^t$.
Then the desired bounds (\ref{eq:Jbound}) hold on the event where $\|\X\|_\op \leq C_0$,
by Assumptions \ref{assump:prior} and \ref{assump:gradient} for the derivatives
of $s(\cdot)$ and $\cG(\cdot)$.

For (\ref{eq:UWbound}), let us first prove the bounds for the discrete
dynamics $\pnorm{\bar{\U}_\gamma}{\op}$ and $\pnorm{\bar{\W}_\gamma}{F}$. By
definition, for each $t \in \gamma\Z_+$,
\begin{equation}\label{eq:UgammaWgamma}
\bar{\U}_\gamma^{t+\gamma}=(\I+\gamma \bar{\J}_{\gamma,1}^t)\bar{\U}_\gamma^t +
\gamma \bar{\J}^t_{\gamma,2} \bar{\W}_\gamma^t, \quad \bar{\W}_\gamma^{t+1} =
\gamma\bar{\J}^t_{\gamma,3} \bar{\U}_\gamma^t +
(\I+\gamma\bar{\J}^t_{\gamma,4})\bar{\W}_\gamma^t.
\end{equation}
Then applying (\ref{eq:Jbound}),
\begin{align*}
\pnorm{\bar{\U}_\gamma^{t+\gamma}}{\op} \leq
(1+C\gamma)\pnorm{\bar{\U}_\gamma^t}{\op} + C\gamma
\sqrt{d}\pnorm{\bar{\W}_\gamma^t}{F}, \quad
\pnorm{\bar{\W}_\gamma^{t+\gamma}}{F} \leq \frac{C\gamma}{\sqrt{d}}
\pnorm{\bar{\U}_\gamma^t}{\op} + (1+C\gamma)\pnorm{\bar{\W}_\gamma^t}{F},
\end{align*}
which further implies that
\begin{align*}
\pnorm{\bar{\U}_\gamma^{t+\gamma}}{\op} +
\sqrt{d}\pnorm{\bar{\W}_\gamma^{t+\gamma}}{F} \leq
(1+2C\gamma)\big(\pnorm{\bar{\U}_\gamma^{t}}{\op} +
\sqrt{d}\pnorm{\bar{\W}_\gamma^{t}}{F}\big).
\end{align*}
Iterating this bound from the initial conditions
$\bar{\U}_\gamma^0 = \I_d$ and $\bar{\W}_\gamma^0=0_{K\times d}$ shows
(\ref{eq:UWbound}) for $\bar\U_\gamma^t,\bar\W_\gamma^t$ and all $t \leq T$.
For the continuous version $\pnorm{\U^t}{\op}$ and $\pnorm{\W^t}{F}$, note that analogously
\begin{align*}
\U^t = \U^0 + \int_0^t (\J^s_1 \U^s + \J^s_2 \W^s)\d s, \quad \W^t = \W^0 +
\int_0^t (\J^s_3 \U^s + \J^s_4 \W^s)\d s
\end{align*}
so $\frac{\d}{\d t} (\pnorm{\U^t}{\op} + \sqrt{d}\pnorm{\W^t}{F}) \leq
C(\pnorm{\U^t}{\op} + \sqrt{d}\pnorm{\W^t}{F})$. Then (\ref{eq:UWbound})
follows by Gronwall's lemma.

For (\ref{eq:UWchangebound}), we have by (\ref{eq:UgammaWgamma}) and
(\ref{eq:Jbound})
\begin{align*}
\pnorm{\bar\U_\gamma^{t+\gamma} - \bar\U_\gamma^t}{\op} &\leq \gamma
\Big(\pnorm{\bar\J_{\gamma,1}^t}{\op}\pnorm{\bar\U_\gamma^t}{\op} +
\pnorm{\bar\J_{\gamma,2}^t}{F}\pnorm{\bar\W_\gamma^t}{F}\Big) \leq
C\gamma,\\
\pnorm{\bar\W_\gamma^{t+\gamma} - \bar\W_\gamma^t}{F} &\leq
\gamma\Big(\pnorm{\bar\J_{\gamma,3}^t}{F}\pnorm{\bar\U_\gamma^t}{\op} +
\pnorm{\bar\J_{\gamma,4}^t}{F}\pnorm{\bar\W_\gamma^t}{F}\Big) \leq C\gamma/\sqrt{d}.
\end{align*}

For (\ref{eq:Jdiscretizebound}), we have by the Lipschitz continuity of
$s(\cdot)$ in Assumption \ref{assump:prior}
that $\J_1^t-\bar\J_{\gamma,1}^t$ is diagonal with
$\pnorm{\J_1^t-\bar\J_{\gamma,1}^t}{F} \leq
C(\pnorm{\btheta^t-\bar\btheta_\gamma^t}{}+\sqrt{d}\pnorm{\widehat\alpha^t -
\bar{\widehat\alpha}_\gamma^t}{})$. Next using the arguments that led to
(\ref{eq:joint_alpha_theta}), we have that on the event
$\{\|\X\|_\op \leq C_0,\|\y\| \leq C_0\sqrt{d}\}$,
with $\x^0 = (\btheta^0, \widehat\alpha^0)$,
\begin{align}\label{eq:theta_bound_expectation}
\langle \pnorm{\btheta^t}{}\rangle_{\x^0} + \langle\pnorm{\bar\btheta_\gamma^t}{}\rangle_{\x^0} +
\sqrt{d}\langle\pnorm{\widehat\alpha^t}{}\rangle_{\x^0} +  \sqrt{d}\langle\pnorm{\bar{\widehat\alpha}_\gamma^t}{}\rangle_{\x^0} &\leq C(\pnorm{\btheta^0}{} + \sqrt{d}\pnorm{\widehat\alpha^0}{} + \sqrt{d})
\end{align}
and
\begin{align}\label{eq:theta_diff_expectation}
\langle \pnorm{\btheta^t - \bar\btheta_\gamma^t}{} \rangle_{\x^0} + \sqrt{d}\langle\pnorm{\widehat\alpha^t -
\bar{\widehat\alpha}_\gamma^t}{}\rangle_{\x^0} &\leq \iota(\gamma)(\pnorm{\btheta^0}{} +
\sqrt{d}\pnorm{\widehat\alpha^0}{} + \sqrt{d}).
\end{align}
This implies the desired bound for $\langle \pnorm{\J_1^t -
\bar\J_{\gamma,1}^t}{F}\rangle_{\x}$, and a similar argument leads to the bound for
$\langle\pnorm{\J_2^t - \bar\J_{\gamma,2}^t}{F}\rangle_{\x}$. Next, by the
derivative bounds for $\cG(\cdot)$ in Assumption \ref{assump:gradient}, we have
$\pnorm{\J^3_t - \bar\J_{\gamma,3}^t}{F} \leq C(\pnorm{\btheta^t -
\bar\btheta_\gamma^t}{}/d + \pnorm{\widehat\alpha^t -
\bar{\widehat\alpha}_\gamma^t}{}/\sqrt{d})$ and $\pnorm{\J_4^t -
\bar\J_{\gamma,4}^t}{F} \leq C(\pnorm{\btheta^t -
\bar\btheta_\gamma^t}{}/\sqrt{d} + \pnorm{\widehat\alpha^t -
\bar{\widehat\alpha}_\gamma^t}{})$, hence the desired bounds also follow by (\ref{eq:theta_diff_expectation}).
\end{proof}

\begin{lemma}\label{lem:semigroup_discre_error}
Define
\begin{equation}\label{eq:goodevent}
\event=\{\|\X\|_\op \leq C_0,\|\y\| \leq C_0\sqrt{d},
\|\btheta^0\| \leq C_0\sqrt{d},\|\widehat\alpha^0\| \leq C_0 \text{ for all
large } n,d\}.
\end{equation}
Fixing any $T>0$, there exists a constant $C>0$ (depending on $T,C_0$ but not on
$\gamma$) and a function $\iota:\R_+ \to \R_+$ satisfying $\lim_{\gamma \to 0}
\iota(\gamma)=0$, such that on $\event$, for any $\gamma>0$
and all $0\leq s\leq t\leq T$,
\begin{align}
|d^{-1}\Tr\bR_\theta(t,s) - d^{-1}\gamma^{-1}\Tr\bR_\theta^\gamma([t]+1,[s])| &\leq
\iota(\gamma)\label{eq:Rthetaapprox}\\
|n^{-1}\Tr\bR_\eta(t,s) - n^{-1}\gamma^{-1}\Tr\bR_\eta^\gamma([t]+1,[s])| &\leq
\iota(\gamma).\label{eq:Retaapprox}
\end{align}
\end{lemma}
\begin{proof}
\noindent\textbf{Discretization of $\bR_\theta$.}
Let $\{P_t^\gamma\}_{t \in \Z_+}$ be the Markov semigroup for the discrete
dynamics (\ref{eq:original_alpha}),
i.e.\ $P_t^\gamma f(\x)=\langle f(\x_\gamma^t) \rangle_\x$. Then
applying Proposition \ref{prop:responsegeneraldiscrete}, for any $s,t \in \Z_+$
with $s<t$,
\begin{align*}
\partial_\eps|_{\eps = 0} \langle \theta_{\gamma,j}^{t,(s,j),\eps}\rangle_\x
=\gamma P_{s+1}^\gamma \partial_j P_{t-s-1}^\gamma e_j(\x).
\end{align*}
This implies, for the given initial condition of the dynamics
$\x^0=(\btheta^0,\widehat\alpha^0)$, that
\begin{align*}
\gamma^{-1}\Tr \bR_\theta^\gamma(t,s) = \sum_{j=1}^d
P_{s+1}^\gamma \partial_j P_{t-s-1}^\gamma e_j(\x^0).
\end{align*}
Let $\{P_t\}_{t \geq 0}$ analogously denote the Markov semigroup of the
continuous dynamics (\ref{eq:langevin_sde}--\ref{eq:gflow}), i.e.\ $P_t
f(\x)=\langle f(\x^t) \rangle_\x$. Then
applying Proposition \ref{prop:responsegeneral}, for any $s,t \in \R_+$ with $s
\leq t$,
\[\Tr \bR_\theta(t,s)=\sum_{j=1}^d  P_{s} \partial_j P_{t-s} e_j(\x^0).\]
Thus, for all $s,t \in \R_+$ with $s \leq t$,
\begin{align*}
&\Big|\Tr \bR_\theta(t,s) - \gamma^{-1}\Tr \bR^\gamma_\theta([t]+1,[s])\Big| =
\Big|\sum_{j=1}^d  P_{s} \partial_j P_{t-s} e_j(\x^0) - \sum_{j=1}^d
P_{[s]+1}^\gamma \partial_j P_{[t]-[s]}^\gamma e_j(\x^0)\Big|\\
&\hspace{1in}\leq \underbrace{\Big|P_{s}\Big(\sum_{j=1}^d\partial_j P_{t-s} e_j -
\sum_{j=1}^d \partial_j P_{[t]-[s]}^\gamma e_j\Big)(\x^0)\Big|}_{(I)} +
\underbrace{\Big|(P_s - P_{[s]+1}^\gamma)\Big(\sum_{j=1}^d \partial_j
P_{[t]-[s]}^\gamma e_j\Big)(\x^0)\Big|}_{(II)}.
\end{align*}\\

\noindent\textbf{Bound of $(I)$.} By Proposition
\ref{prop:langevinsolution}(c), $\sum_{j=1}^d\partial_j P_{t-s} e_j(\x) =
\sum_{j=1}^d \langle (\V^{t-s})_{jj}\rangle_{\x}$, where $\{\x^t,\V^t\}_{t \geq
0}$ are the solution to (\ref{eq:theta_V}) with initial condition $\x^0=\x$.
Similarly, by Lemma \ref{lem:discrete_bismut} and the identification
(\ref{eq:xVembedding}), $\sum_{j=1}^d \partial_j
P_{[t]-[s]}^\gamma e_j(\x)=\sum_{j=1}^d \langle  (\bar\V_\gamma^{\floor{t} -
\floor{s}})_{jj} \rangle_{\x}$, where
$\{\bar\x_\gamma^t,\bar\V_\gamma^t\}_{t \geq 0}$ are the solution to
(\ref{eq:theta_V_embed}).
Let us write
\[\V^t=\begin{pmatrix} \U^t & \ast \\ \W^t & \ast \end{pmatrix},
\quad \bar{\V}_\gamma^t = \begin{pmatrix}
\bar{\U}_\gamma^t & \ast \\ \bar{\W}_\gamma^t & \ast \end{pmatrix},
\quad \d u(\x^t)=\begin{pmatrix} \J_1^t & \J_2^t \\ \J_3^t & \J_4^t
\end{pmatrix}, \quad
\d u(\bar\x_\gamma^t)=\begin{pmatrix} \bar\J_{\gamma,1}^t & \bar\J_{\gamma,2}^t
\\ \bar\J_{\gamma,3}^t & \bar\J_{\gamma,4}^t \end{pmatrix}
\]
with blocks of sizes $d$ and $K$. Then
\begin{align}\label{eq:semigroup_discre_1}
\Big|\sum_{j=1}^d \partial_j P_{t-s} e_j(\x) - \partial_j
P_{[t]-[s]}^\gamma \e_j(\x)\Big|
&= \Big|\langle \Tr \U^{t-s} - \Tr \bar{\U}_\gamma^{\floor{t}-\floor{s}}\rangle_{\x}\Big|\notag\\
&\leq \sqrt{d}\langle \pnorm{\U^{t-s} - \bar{\U}_\gamma^{t-s}}{F}\rangle_{\x} +
d\langle \pnorm{\bar{\U}_\gamma^{t-s} - \bar{\U}_\gamma^{\floor{t}-\floor{s}}}{\op} \rangle_{\x}.
\end{align}
Since $|(t-s)-(\floor{t}-\floor{s})| \leq C\gamma$,
the second term satisfies $\pnorm{\bar{\U}^{t-s} -
\bar{\U}^{t-s}}{\op} \leq C\gamma$ by (\ref{eq:UWchangebound}).
For the first term, note that by definition
\[
\begin{gathered}
\U^t = \U^0 + \int_0^t (\J^s_1 \U^s + \J^s_2 \W^s) \d s,\quad \W^t = \W^0 +
\int_0^t (\J^s_3 \U^s + \J^s_4 \W^s) \d s,\\
\bar{\U}_\gamma^t = \U^0 + \int_0^{\floor{t}} (\bar\J_{\gamma,1}^s \bar{\U}_\gamma^s +
\bar \J_{\gamma,2}^s \bar{\W}_\gamma^s) \d s,\quad \bar{\W}_\gamma^t = \W^0 +
\int_0^{\floor{t}} (\bar \J_{\gamma,3}^s \bar{\U}_\gamma^s + \bar\J_{\gamma,4}^s
\bar{\W}_\gamma^s) \d s.
\end{gathered}
\]
Hence
\begin{align*}
\langle\pnorm{\U^t - \bar\U_\gamma^t}{F}\rangle_{\x} &\leq \int_0^{\floor{t}}
\Big[\langle\pnorm{\J^s_1 - \bar\J_{\gamma,1}^s}{F}\pnorm{\U^s}{\op}\rangle_{\x}
+ \langle\pnorm{\bar\J_{\gamma,1}^s}{\op}\pnorm{\U^s - \bar{\U}_\gamma^s}{F}\rangle_{\x}\\
&\hspace{1in} + \langle\pnorm{\J^s_2 -
\bar\J_{\gamma,2}^s}{F}\pnorm{\W^s}{F}\rangle_{\x} +
\langle\pnorm{\bar\J_{\gamma,2}^s}{F}\pnorm{\W^s -
\bar{\W}_\gamma^s}{F}\rangle_{\x}\Big] \d s\\
&\hspace{0.2in}+\int_{\floor{t}}^t
\Big[\langle\|\J^s_1\|_F \|\U^s\|_\op\rangle_\x +
\langle\|\J^s_2\|_F\|\W^s\|_F\rangle_\x\Big] \d s
\end{align*}
Let $C,C'>0$ be constants depending on $T$ but not $\gamma$, and let
$\iota(\gamma),\iota'(\gamma)$ be constants depending also on $\gamma$
and satisfying $\iota(\gamma),\iota'(\gamma) \to 0$ as $\gamma \to 0$, all
changing from instance to instance.
By Lemma \ref{lem:theta_V_rough}, with $\x=(\btheta,\widehat\alpha)$, we
have
\[\pnorm{\U^s}{\op}\leq C, \quad \sqrt{d}\pnorm{\W^s}{F} \leq C, \quad
\pnorm{\J_1^s}{\op},\pnorm{\bar\J_{\gamma,1}^s}{\op}\leq C, \quad
\pnorm{\J_2^s}{F},\pnorm{\bar\J_{\gamma,2}^s}{F} \leq C\sqrt{d},\]
\[\langle\pnorm{\J^s_1 - \bar\J_{\gamma,1}^s}{F}
\rangle_\x \leq
\iota(\gamma)(\|\btheta\|+\sqrt{d}\|\widehat\alpha\|+\sqrt{d}), \quad
\langle \pnorm{\J^s_2 - \bar\J_{\gamma,2}^s}{F}\rangle_{\x} \leq
\iota(\gamma)(\|\btheta\|+\sqrt{d}\|\widehat\alpha\|+\sqrt{d}),\]
hence
\begin{align}\label{eq:U_bound}
\langle\pnorm{\U^t - \bar{\U}_\gamma^t}{F}\rangle_{\x}\leq C\int_0^t
\big(\langle\pnorm{\U^s - \bar{\U}_\gamma^s}{F}\rangle_{\x} +
\sqrt{d}\langle\pnorm{\W^s - \bar{\W}_\gamma^s}{F}\rangle_{\x}\big) \d s +
\iota(\gamma)(\|\btheta\|+\sqrt{d}\|\widehat\alpha\|+\sqrt{d}).
\end{align}
Next we have
\begin{align*}
\langle\pnorm{\W^t - \bar{\W}_\gamma^t}{F}\rangle_{\x} &\leq \int_0^{\floor{t}}
\Big[\langle\pnorm{\J^s_3 - \bar\J_{\gamma,3}^s}{F}\pnorm{\U^s}{\op}\rangle_{\x}
+ \langle\pnorm{\bar\J_{\gamma,3}^s}{F}\pnorm{\U^s - \bar{\U}_\gamma^s}{F}\rangle_{\x}\\
&\hspace{1in} + \langle\pnorm{\J^s_4 -
\bar\J_{\gamma,4}^s}{F}\pnorm{\W^s}{F}\rangle_{\x} +
\langle\pnorm{\bar\J_{\gamma,4}^s}{F}\pnorm{\W^s -
\bar{\W}_\gamma^s}{F}\rangle_\x\Big]
\d s\\
&\hspace{0.2in}+\int_{\floor{t}}^t
\Big[\langle\|\J^s_3\|_F\|\U^s\|_\op \rangle_\x+\langle
\|\J^s_4\|_F\|\W^s\|_F \rangle_\x\Big] \d s.
\end{align*}
By Lemma \ref{lem:theta_V_rough}, we have also
\[\pnorm{\J_3^s}{F},\pnorm{\bar\J_{\gamma,3}^s}{F} \leq C/\sqrt{d}, \quad
\pnorm{\J_4^s}{F},\pnorm{\bar\J_{\gamma,4}^s}{F} \leq C,\]
\[d\langle\pnorm{\J^s_3 - \bar\J_{\gamma,3}^s}{F}\rangle_{\x} \leq
\iota(\gamma)(\|\btheta\|+\sqrt{d}\|\widehat\alpha\|+\sqrt{d}),
\quad \sqrt{d}\pnorm{\J^s_4 - \bar\J_{\gamma,4}^s}{F} \leq \iota(\gamma)
(\|\btheta\|+\sqrt{d}\|\widehat\alpha\|+\sqrt{d}),\]
which implies that
\begin{align}\label{eq:W_bound}
\sqrt{d}\langle\pnorm{\W^t - \bar{\W}_\gamma^t}{F}\rangle_{\x} \leq C\int_0^t
(\langle\pnorm{\U^s - \bar{\U}^s}{F}\rangle_{\x} + \sqrt{d}\langle\pnorm{\W^s -
\bar{\W}^s}{F}\rangle_{\x})\d s + \iota(\gamma)
\Big(\frac{\|\btheta^0\|}{\sqrt{d}}+\|\widehat\alpha^0\|+1\Big).
\end{align}
Combining (\ref{eq:U_bound}) and (\ref{eq:W_bound}) yields
\begin{align*}
&\langle\pnorm{\U^t - \bar{\U}_\gamma^t}{F} + \sqrt{d}\pnorm{\W^t -
\bar{\W}_\gamma^t}{F}\rangle_{\x}\\
& \leq C\int_0^t (\langle\pnorm{\U^s -
\bar{\U}_\gamma^s}{F} +
\sqrt{d}\pnorm{\W^s - \bar{\W}_\gamma^s}{F}\rangle_{\x})\d s + \iota(\gamma)
(\|\btheta\|+\sqrt{d}\|\widehat\alpha\|+\sqrt{d}),
\end{align*}
so Gronwall's lemma gives
$\sup_{t \in [0,T]} \langle\pnorm{\U^t - \bar{\U}_\gamma^t}{F}\rangle_{\x} +
\sqrt{d}\langle\pnorm{\W^t - \bar{\W}_\gamma^t}{F}\rangle_{\x} \leq
\iota(\gamma)(\|\btheta\|+\sqrt{d}\|\widehat\alpha\|+\sqrt{d})$.

Hence the bound (\ref{eq:semigroup_discre_1}) reads, for
$\x=(\btheta,\widehat\alpha)$,
\begin{align}\label{eq:V_frob_norm}
\Big|\sum_{j=1}^d \partial_j P_{t-s} e_j(\x) - \partial_j P_{[t]-[s]-1}^\gamma
e_j(\x)\Big| \leq
\iota(\gamma)(\sqrt{d}\|\btheta\|+d\|\widehat\alpha\|+d).
\end{align}
Applying this with $\x=\x^s=(\btheta^s,\widehat\alpha^s)$,
this implies that
\begin{align*}
(I) \leq \iota(\gamma)(\sqrt{d}\langle\pnorm{\btheta^s}{}\rangle_{\x^0} + d\langle
\pnorm{\widehat\alpha^s}{}\rangle_{\x^0}+d) \leq \iota(\gamma)d,
\end{align*}
the last step using the bound (\ref{eq:theta_bound_expectation}) and conditions
for $(\btheta^0,\widehat\alpha^0)$ on the event (\ref{eq:goodevent}).\\

\noindent\textbf{Bound of $(II)$.} Let $\mathsf{f}(\x)=\sum_{j=1}^d \partial_j
P_{[t]-[s]}^\gamma e_j(\x)$. We first establish a Lipschitz bound for $\mathsf{f}$: Let $\{\bar\x_\gamma^t,\bar\V_\gamma^t\}_{t\in \Z_+}$ and
$\{\tilde{\x}_\gamma^t, \tilde{\V}_\gamma^t\}_{t\in \Z_+}$ be defined by
(\ref{eq:theta_V_embed}) with initializations $\x=(\btheta,\widehat\alpha)$
and $\tilde\x=(\tilde\btheta,\tilde{\widehat\alpha})$
respectively, coupled by the same Brownian motion. We write $\langle \cdot
\rangle$ for the average over this Brownian motion, and denote by $\tilde{\U}_\gamma^{t},\tilde{\W}_\gamma^{t}$ and
$\tilde\J_{\gamma,1}^t,\tilde\J_{\gamma,2}^t,\tilde\J_{\gamma,3}^t,\tilde\J_{\gamma,4}^t$
the blocks of $\tilde \V_\gamma^t$ and $\d u(\tilde \x_\gamma^t)$.
Then, using $\mathsf{f}(\x) = \langle \Tr \bar{\U}_\gamma^{\tau} \rangle$
with $\tau = \floor{t} - \floor{s}$ as established above,
\begin{align}\label{eq:f_modulus}
|\mathsf{f}(\x) - \mathsf{f}(\tilde\x)| \leq |\langle\Tr \bar{\U}_\gamma^{\tau}
- \Tr \tilde{\U}_\gamma^{\tau} \rangle | \leq \sqrt{d} \langle
\pnorm{\bar{\U}_\gamma^{\tau} - \tilde{\U}_\gamma^{\tau}}{F} \rangle.
\end{align}
We apply a similar argument as in term $(I)$, noting that
\begin{align*}
\pnorm{\bar{\U}_\gamma^{t+\gamma} - \tilde{\U}_\gamma^{t+\gamma}}{F} &\leq
\gamma\pnorm{\bar\J_{\gamma,1}^t -
\tilde\J_{\gamma,1}^t}{F}\pnorm{\bar{\U}_\gamma^t}{\op} +
(1+\gamma\pnorm{\tilde\J_{\gamma,1}^t}{\op})\pnorm{\bar{\U}_\gamma^{t} -
\tilde{\U}_\gamma^{t}}{F}\\
&\quad +  \gamma\pnorm{\bar\J_{\gamma,2}^t -
\tilde\J_{\gamma,2}^t}{F}\pnorm{\bar{\W}_\gamma^t}{F} +
\gamma\pnorm{\tilde\J_{\gamma,2}^t}{F}\pnorm{\bar{\W}_\gamma^{t} -
\tilde{\W}_\gamma^{t}}{F},\\
\pnorm{\bar{\W}_\gamma^{t+\gamma} - \tilde{\W}_\gamma^{t+\gamma}}{F} &\leq
\gamma\pnorm{\bar\J_{\gamma,3}^t -
\tilde\J_{\gamma,3}^t}{F}\pnorm{\bar{\U}_\gamma^t}{\op} +
\gamma\pnorm{\tilde\J_{\gamma,3}^t}{F}\pnorm{\bar{\U}_\gamma^{t} -
\tilde{\U}_\gamma^{t}}{F}\\
&\quad +  \gamma\pnorm{\bar\J_{\gamma,4}^t -
\tilde\J_{\gamma,4}^t}{F}\pnorm{\bar{\W}_\gamma^t}{F} + (1 +
\gamma\pnorm{\tilde\J_{\gamma,4}^t}{F})\pnorm{\bar{\W}_\gamma^{t} -
\tilde{\W}_\gamma^{t}}{F}.
\end{align*}
By Lemma \ref{lem:theta_V_rough}, we have
$\pnorm{\bar{\U}_\gamma^t}{\op},\sqrt{d} \pnorm{\bar{\W}_\gamma^t}{F}\leq C$,
and
$\pnorm{\tilde\J_{\gamma,1}^t}{\op},\frac{\pnorm{\tilde\J_{\gamma,2}^t}{F}}{\sqrt{d}},\sqrt{d}\pnorm{\tilde\J_{\gamma,3}^t}{F},\pnorm{\tilde\J_{\gamma,4}^t}{F}
\leq C$. Furthermore similar arguments to (\ref{eq:Jdiscretizebound}) in
Lemma \ref{lem:theta_V_rough} show that
\begin{align*}
&\langle \pnorm{\bar\J_{\gamma,1}^t - \tilde\J_{\gamma,1}^t}{F} \rangle,\langle
\pnorm{\bar\J_{\gamma,2}^t - \tilde\J_{\gamma,2}^t}{F}  \rangle,d \langle
\pnorm{\bar\J_{\gamma,3}^t - \tilde\J_{\gamma,3}^t}{F}\rangle,\sqrt{d}\langle
\pnorm{\bar\J_{\gamma,4}^t - \tilde\J_{\gamma,4}^t}{F}\rangle\\
&\leq
C\Big\langle\|\bar\btheta^t-\tilde\btheta^t\|+\sqrt{d}\|\bar{\widehat\alpha}^t
-\tilde{\widehat\alpha}^t\|\Big \rangle
\leq
C'\Big(\|\btheta-\tilde\btheta\|+\sqrt{d}\|\widehat\alpha-\tilde{\widehat\alpha}\|\Big),
\end{align*}
the quantities in the last expression denoting the differences in initial
conditions. Hence
\begin{align*}
&\langle\pnorm{\tilde{\U}_\gamma^{t+\gamma}- \bar{\U}_\gamma^{t+\gamma}}{F} +
\sqrt{d} \pnorm{\bar{\W}_\gamma^{t+\gamma} -
\tilde{\W}_\gamma^{t+\gamma}}{F}\rangle \\
&\leq (1+C\gamma)\langle\pnorm{\bar{\U}_\gamma^{t} - \tilde{\U}_\gamma^{t}}{F} +
\sqrt{d} \pnorm{\bar{\W}_\gamma^{t} - \tilde{\W}_\gamma^{t}}{F}\rangle +
C\gamma\Big(\|\btheta-\tilde\btheta\|+\sqrt{d}\|\widehat\alpha-\tilde{\widehat\alpha}\|\Big).
\end{align*}
Iterating this bound gives
$\langle\pnorm{\bar{\U}_\gamma^{\tau} - \tilde{\U}_\gamma^{\tau}}{F}\rangle \leq
C(\|\btheta-\tilde\btheta\|+\sqrt{d}\|\widehat\alpha-\tilde{\widehat\alpha}\|)$,
which applied to (\ref{eq:f_modulus}) yields our desired Lipschitz bound
\[|\mathsf{f}(\x)-\mathsf{f}(\tilde\x)| \leq 
C\sqrt{d}(\|\btheta-\tilde\btheta\|+\sqrt{d}\|\widehat\alpha-\tilde{\widehat\alpha}\|).\]

Then, writing $\x^0=(\btheta^0,\widehat\alpha^0)$ for the original initial
conditions,
\begin{align*}
(II)=\Big|(P_s - P_{[s]+1}^\gamma)\mathsf{f}(\x^0)\Big| = \Big|\langle \mathsf{f}(\x^s)
\rangle_{\x^0}- \langle \mathsf{f}(\bar\x_\gamma^{\floor{s}+\gamma})
\rangle_{\x^0}\Big| \leq C\sqrt{d}\langle \pnorm{\btheta^s -
\bar{\btheta}_\gamma^{\floor{s} + \gamma}}{} +\sqrt{d}\|\widehat\alpha^s
-\bar{\widehat\alpha}_\gamma^{\floor{s}+\gamma}\|\rangle_{\x^0}
\end{align*}
where we couple $\{(\btheta^t,\widehat\alpha^t)\}_{t \geq 0}$ and
$\{\bar{\btheta}_\gamma^t,\bar{\widehat\alpha}_\gamma^t\}_{t \geq 0}$ by
the same Brownian motion. Bounding
\[\langle \pnorm{\btheta^s -
\bar{\btheta}_\gamma^{\floor{s} + \gamma}}{}\rangle_{\x^0}
\leq \langle \pnorm{\btheta^s -
\bar{\btheta}_\gamma^s}{}\rangle_{\x^0}
+\langle \pnorm{\bar\btheta_\gamma^s -
\bar{\btheta}_\gamma^{\floor{s} + \gamma}}{}\rangle_{\x^0}\]
and similarly for $\widehat\alpha$, and then applying
(\ref{eq:theta_bound_expectation}) and (\ref{eq:theta_diff_expectation}),
we obtain on the event (\ref{eq:goodevent}) that
\[\langle \pnorm{\btheta^s -
\bar{\btheta}_\gamma^{\floor{s} + \gamma}}{} +\sqrt{d}\|\widehat\alpha^s
-\bar{\widehat\alpha}_\gamma^{\floor{s}+\gamma}\|\rangle_{\x^0}
\leq \iota(\gamma)(\|\btheta^0\|+\sqrt{d}\|\widehat\alpha^0\|+\sqrt{d})
\leq C\iota(\gamma)\sqrt{d}.\]
Hence also
\[(II) \leq \iota(\gamma)d.\]
The proof of (\ref{eq:Rthetaapprox}) is completed
by combining the bounds for $(I)$ and $(II)$.\\

\noindent\textbf{Discretization of $\bR_\eta$.}
Let $P_t^\gamma$ and $P_t$ be the discrete and continuous Markov semigroups
defined above. Let $x_i(\btheta,\widehat\alpha)=\e_i^\top \X\btheta$.
Introduce the matrix $\bE \in \R^{d \times (d+K)}$ defined by
$\bE(\btheta,\widehat\alpha)=\btheta$, so that this reads $x_i(\x)=\e_i^\top \X\bE\x$ for $\x=(\btheta,\widehat\alpha)$.
Then for any $s,t \in \Z_+$ with $s<t$, Proposition
\ref{prop:responsegeneraldiscrete} gives
\[\deps \langle \eta_i^{t,[s,i],\eps} \rangle_{\x}
=\gamma P_{s+1}^\gamma \e_i^\top \X \bE\nabla P_{t-s-1}^\gamma x_i(\x).\]
Let us introduce the shorthand $P_t^\gamma(\x)=\langle \x_\gamma^t \rangle_\x$
as a map $P_t^\gamma:\R^{d+K} \to \R^{d+K}$, so
$P_t^\gamma x_i(\x)=\e_i^\top \X\bE P_t^\gamma(\x)$. Denote also
$\d P_t^\gamma(\cdot):\R^{d+K} \to \R^{(d+K) \times (d+K)}$ as the derivative of
this map $\x \mapsto P_t^\gamma(\x)$. Then the above may be written as
\begin{align*}
\deps \langle \eta_i^{t,[s,i],\eps} \rangle_{\x}
=\gamma P_{s+1}^\gamma\Big(\e_i^\top
\X \bE \d P_{t-s-1}^\gamma(\cdot)^\top \bE^\top\X^\top\e_i\Big)(\x),
\end{align*}
implying that
\begin{align*}
\gamma^{-1}\Tr \bR_\eta^\gamma(t,s) = \delta\beta^2\sum_{i=1}^n
P_{s+1}^\gamma\Big(\e_i^\top
\X \bE \d P_{t-s-1}^\gamma(\cdot)^\top \bE^\top\X^\top\e_i\Big)(\x^0)
= \delta\beta^2 P_{s+1}^\gamma \Tr\Big[\d P_{t-s-1}^\gamma(\cdot) \bE^\top
\X^\top \X \bE\Big](\x^0).
\end{align*}
By Proposition \ref{prop:responsegeneral}, we have analogously
for any $s,t \in \R_+$ with $s \leq t$ that
\begin{align*}
\Tr \bR_\eta(t,s) = \delta\beta^2 P_s \Tr \big[\d P_{t-s}(\cdot)\bE^\top \X^\top
\X \bE\big](\x^0).
\end{align*}
Hence for all $s,t \in \R_+$ with $s \leq t$,
\begin{align*}
&\Big|\Tr \bR_\eta(t,s) - \gamma^{-1}\Tr \bR_\eta^\gamma([t]+1,[s])\Big|\\
&= \delta\beta^2\bigg[\underbrace{\Big|P_s \Tr\Big[\Big(\d P_{t-s}(\cdot)-\d
P_{[t]-[s]}^\gamma(\cdot)\Big)\bE^\top \X^\top \X \bE\Big](\x^0)\Big|}_{(I)} +
\underbrace{\Big|(P_s-P_{[s]+1}^\gamma)\Tr \Big[\d P_{[t] - [s]}^\gamma(\cdot)\bE^\top \X^\top \X \bE\Big](\x^0)\Big|}_{(II)}\bigg].
\end{align*}\\

\noindent\textbf{Bound of $(I)$.} Note that by Proposition
\ref{prop:langevinsolution}(c), $\Tr \d P_{t-s}(\x)\bE^\top \X^\top \X \bE =
\langle \Tr \V^{t-s} \bE^\top \X^\top \X\bE \rangle_{\x} = \langle \Tr \U^{t-s}
\cdot \X^\top \X \rangle_{\x} $, where $\{\x^t, \V^t\}_{t \geq 0}$ follow the
dynamics (\ref{eq:theta_V}) and $\U^t$ as before is the upper-left block of
$\V^t$. Similarly, Lemma \ref{lem:discrete_bismut} yields that $\Tr \d
P_{[t]-[s]}^\gamma(\x)\bE^\top \X^\top \X\bE = \langle \Tr
\bar\V_\gamma^{\floor{t} -\floor{s}}\bE^\top \X^\top \X\bE\rangle_{\x} =
\langle \Tr \bar{\U}_\gamma^{\floor{t}-\floor{s}}\X^\top \X
\rangle_{\x}$, where $\{\bar\x_\gamma^t, \bar\V_\gamma^t\}_{t \geq 0}$ follow
(\ref{eq:theta_V_embed}). Hence, with $\x=(\btheta,\widehat\alpha)$,
\begin{align*}
\Big|\Tr \Big(\d P_{t-s}(\x) - \d P_{[t]-[s]}^\gamma(\x)\Big) \bE^\top
\X^\top \X \bE \Big| &= \langle \Tr (\U^{t-s}-\bar{\U}_\gamma^{\floor{t} - \floor{s}})\X^\top \X \rangle_{\x}\\
&\leq \sqrt{d}\pnorm{\X}{\op}^2 \langle\pnorm{\U^{t-s} - \bar{\U}_\gamma^{\floor{t} - \floor{s}}}{F} \rangle_{\x}\\
&\leq \iota(\gamma)(\sqrt{d}\pnorm{\btheta}{} + d\pnorm{\widehat\alpha}{} + d)
\end{align*}
using the preceding bounds leading to (\ref{eq:V_frob_norm}). Then applying this with $\x=\x^s$ shows $(I) \leq
\iota(\gamma)d$.\\

\noindent\textbf{Bound of $(II)$.} Let $\mathsf{f}(\x) = \Tr \d P_{[t] -
[s]}^\gamma(\x) \bE^\top \X^\top \X \bE = \Tr \langle \bar{\U}_\gamma^\tau
\rangle_{\x}\X^\top \X$, where $\tau = \floor{t} - \floor{s}$. By the
same arguments as above, 
\begin{align*}
|\mathsf{f}(\x) - \mathsf{f}(\tilde\x)| \leq
C\sqrt{d}(\|\btheta-\tilde\btheta\|+\sqrt{d}\|\widehat\alpha-\tilde{\widehat\alpha}\|),
\end{align*}
leading to $(II) \leq \iota(\gamma)d$. Combining these bounds for $(I)$ and
$(II)$ shows (\ref{eq:Retaapprox}).
\end{proof}

We now conclude the proof of Theorem \ref{thm:dmftresponse}.

\begin{proof}[Proof of Theorem \ref{thm:dmftresponse}]
The claims for $d^{-1}\Tr \bC_\theta(t,s)$, $d^{-1}\Tr \bC_\theta(t,*)$, and
$n^{-1}\Tr \bC_\eta(t,s)$ follow immediately from the definitions of these
quantities, Corollary \ref{cor:dmft_approx} applied with $f_\theta(\theta^s,\theta^t)=\theta^s\theta^t$, $f_\theta(\theta^*,\theta^t)=\theta^*\theta^t$, $f_\eta(\eta^*,\eps,\eta^s,\eta^t)=\delta\beta^2(\eta^s-\eta^*-\eps)(\eta^t-\eta^*-\eps)$,
and an application of the dominated convergence theorem to take expectations over $\{\b^t\}_{t \in [0,T]}$ in the almost-sure convergence statements of
Corollary \ref{cor:dmft_approx}.

For the claim for $d^{-1}\Tr \bR_\theta(t,s)$, for any $s,t \in [0,T]$ with
$s \leq t$, by Lemma \ref{lem:semigroup_discre_error}, almost surely
\begin{align*}
\limsup_{n,d \to \infty}
\Big|\frac{1}{d}\Tr \bR_\theta(t,s)
-\frac{1}{\gamma} \cdot \frac{1}{d}\Tr \bR_\theta^\gamma([t]+1,[s])\Big|
\leq \iota(\gamma).
\end{align*}
By Lemma \ref{lem:discrete_time_response} 
and the identification (\ref{eq:discretefixedpoint}) of
Lemma \ref{lem:discrete_dmft_rewrite}, almost surely
\begin{align*}
\lim_{n,d \to \infty}
\frac{1}{\gamma} \cdot
\frac{1}{d} \Tr \bR_\theta^\gamma([t]+1,[s])=\frac{1}{\gamma}\,
R_\theta^\gamma([t]+1,[s])=\bar R_\theta^\gamma(t+\gamma,s).
\end{align*}
The bound (\ref{eq:XYdiscretization}) implies uniform convergence of
$\bar R_\theta^\gamma(t,s)$ to $R_\theta^\gamma(t,s)$ as $\gamma \to 0$,
and $R_\theta^\gamma(t,s)$ is continuous in
$s,t$ by Theorem \ref{thm:dmftsolexists} and the definition of the space
$\cS^\text{cont}$. Thus
\[\lim_{\gamma \to 0} |\bar R_\theta^\gamma(t+\gamma,s)-R_\theta(t,s)|=0.\]
Then, taking the limit $n,d
\to \infty$ followed by $\gamma \to 0$ shows almost surely
\[\lim_{n,d \to \infty}
\Big|\frac{1}{d}\Tr \bR_\theta(t,s)
-R_\theta(t,s)\Big|=0.\]
The proof of the claim for $n^{-1}\Tr \bR_\eta(t,s)$ is the same.
\end{proof}

\appendix

\section{Existence of linear response functions}\label{sec:response}

\subsection{Continuous dynamics}
Fix any dimension $m \geq 1$, and consider the function classes
\begin{align*}
\cA&=\Big\{f:\R^m \to \R \text{ twice continuously-differentiable}:
\,\nabla f(\x),\nabla^2 f(\x) \text{ are globally bounded}\Big\},\\
\cB&=\Big\{f:\R^m \to \R^m \text{ twice continuously-differentiable}:\\
&\hspace{0.5in}
\nabla f_i(\x),\nabla^2 f_i(\x) \text{ are globally bounded
and H\"older-continuous for each } i=1,\ldots,m\Big\}.
\end{align*}
We consider a general stochastic diffusion over $\x^t \in \R^m$ given by
\begin{equation}\label{eq:joint_sde}
\d\x^t=u(\x^t)\d t+\sqrt{2}\,\M\,\d\b^t
\end{equation}
where $\b^t \in \R^m$ is a standard Brownian motion, $u(\cdot)$ a Lipschitz
drift function, and $\M \in \R^{m \times m}$ a deterministic diffusion
coefficient matrix. We note that
the joint evolution of $\x^t=(\btheta^t,\widehat\alpha^t)$ in
(\ref{eq:langevin_sde}--\ref{eq:gflow}) is of this form, with $m=d+K$ and with $u(\cdot)$ and $\M$ as defined in
(\ref{eq:def_M_u}).
The conditions of Theorem \ref{thm:dmftresponse} ensure that this drift
function $u(\cdot)$ satisfies $u \in \cB$.

We prove in this section the following result:

\begin{proposition}\label{prop:responsegeneral}
Suppose $u \in \cB$, and let $\{\x^t\}_{t \geq 0}$ be the solution of
(\ref{eq:joint_sde}) with initial condition $\x^0=\x$.
For any $a \in \cA$, $b \in \cB$, and $\x \in \R^m$, define
\begin{equation}\label{eq:explicitresponse}
R(t,s)=P_s\big(b^\top \nabla P_{t-s}a \big)(\x)
\end{equation}
where $P_tf(\x)=\E[f(\x^t) \mid \x^0=\x]$. Then $\{R(t,s)\}_{0 \leq s \leq t}$
is the unique continuous function for which the following holds:

Let $h:[0,\infty) \to \R$ be any continuous bounded
function, and for each $\eps>0$ let $\{\x^{t,\eps}\}_{t \geq 0}$ be the solution
of the perturbed dynamics
\begin{equation}\label{eq:joint_sde_perturbed}
\d \x^{t,\eps}=\Big(u(\x^{t,\eps})+\eps h(t)b(\x^{t,\eps})\Big)\d t
+\sqrt{2}\,\M\,\d\b^t
\end{equation}
with the same initial condition $\x^{0,\eps}=\x$. Then for any $t>0$,
\[\lim_{\eps \to 0} \frac{1}{\eps}\Big(\E[a(\x^{t,\eps}) \mid \x^{0,\eps}=\x]
-\E[a(\x^t) \mid \x^0=\x]\Big)=\int_0^t R(t,s)h(s)\d s.\]
\end{proposition}

Statements similar to Proposition \ref{prop:responsegeneral} have been
established in \cite{dembo2010markovian,chen2020mathematical}. Our setting here
is somewhat non-standard, in that $\M$ may be rank-degenerate, so 
the PDE describing the law of $\{\x^t\}_{t \geq 0}$ is not uniformly elliptic.
We show Proposition \ref{prop:responsegeneral} in two steps, first deriving
regularity estimates for the Markov semigroup $\{P_t\}_{t \geq 0}$
in such settings using the results of \cite{kunita1984stochastic}, and then
applying the proof idea of
\cite[Theorem 3.9]{chen2020mathematical} with these regularity estimates in
place of the Schauder estimates derived therein from uniform ellipticity.

We will write
\begin{equation}\label{eq:semigroup}
P_tf(\x)=\langle f(\x^t) \rangle_{\x^0=\x}=\E[f(\x^t) \mid \x^0=\x]
\end{equation}
for the Markov semigroup associated to (\ref{eq:joint_sde}). When the initial
condition $\x^0=\x$ is clear from context, we will abbreviate
$\langle f(\x^t) \rangle=\langle f(\x^t) \rangle_{\x^0=\x}$.
We denote the infinitesimal generator $\L$ of this semigroup by
\begin{equation}\label{eq:generator}
\L f(\x)=u(\x)^\top \nabla f(\x)+\Tr \M\M^\top \nabla^2 f(\x).
\end{equation}
Throughout this section, constants $C,C',c>0$ may depend on the dimension $m$
and the functions $u,a,b$.

\begin{proposition}\label{prop:langevinsolution}
Suppose the assumptions of Proposition \ref{prop:responsegeneral} hold.
Let $u_i:\R^m \to \R$ be the $i^\text{th}$ coordinate of $u$,
and let $\partial_j u_i$ and $\partial_j\partial_k u_i$
be its first-order and second-order partial derivatives.

\begin{enumerate}[(a)]
\item For each $\x \in \R^m$, the diffusion (\ref{eq:joint_sde}) has a 
unique solution $\{\x^t\}_{t \geq 0}$ with initial condition $\x^0=\x$.
Furthermore, there exists a modification $\x^t(\x)$ of this solution for each
initial condition $\x^0=\x$ such that $\x^t(\x)$ is jointly continuous in
$(t,\x)$ and twice continuously-differentiable in $\x$.
\item For every $i=1,\ldots,m$, let $x_i^t(\x)$ be the $i^\text{th}$
coordinate of $\x^t(\x)$, and let $\v_i^t(\x)=\nabla x_i^t(\x) \in \R^m$
and $\H_i^t(\x)=\nabla^2 x_i^t(\x) \in \R^{m \times m}$ be its gradient and
Hessian in $\x$.
Then $(\v_i^t(\x),\H_i^t(\x))$ are solutions to the first and second
variation processes
\begin{equation}\label{eq:firstsecondvar}
\begin{cases}
\d \v_i^t=\sum_{j=1}^m \partial_j u_i(\x^t(\x)) \cdot \v_j^t\,\d t\\
\d \H_i^t=\big(\sum_{j,k=1}^m \partial_j\partial_k u_i(\x^t(\x)) \cdot
\v_j^t\v_k^{t\top}+\sum_{j=1}^m \partial_j u_i(\x^t(\x))
\cdot \H_j^t\big)\d t
\end{cases}
\end{equation}
with initial conditions $\v_i^0(\x)=\e_i$ (the $i^\text{th}$ standard basis vector
in $\R^m$) and $\H_i^0(\x)=0$.

Furthermore
$\|\v_i^t(\x)\|_2,\|\H_i^t(\x)\|_\op \leq e^{Ct}$ for some $C>0$ and
all $\x \in \R^m$ and $t \geq 0$.

\item For any $f \in \cA$, the map $(t,\x) \mapsto P_t f(\x)$ is
continuously-differentiable in $t$ and twice continuously-differentiable in
$\x$, and furthermore $\nabla P_t f(\x),\nabla^2 P_t f(\x)$ are uniformly
bounded over $t \in [0,T]$ and $\x \in \R^m$ for any fixed $T>0$.
For any $t \geq 0$ and initial condition $\x^0=\x$,
letting $(\x^t,\v_i^t,\H_i^t) \equiv (\x^t(\x),\v_i^t(\x),\H_i^t(\x))$
be as defined in parts (a) and (b), we have
\begin{equation}\label{eq:Ptspaceder}
\begin{aligned}
\nabla P_t f(\x) &=\bigg \langle \sum_{j=1}^m \partial_j
f(\x^t)\v_j^t\bigg\rangle\\
\nabla^2 P_t f(\x)
&=\bigg\langle \sum_{j,k=1}^m \partial_j\partial_k f(\x^t)
\v_j^t\v_k^{t\top}+\sum_{j=1}^m \partial_j f(\x^t)\H_j^t \bigg\rangle
\end{aligned}
\end{equation}
and
\begin{equation}\label{eq:forwardbackward}
\partial_t P_t f(\x)=P_t \L f(\x)=\L P_t f(\x).
\end{equation}
\end{enumerate}
\end{proposition}
\begin{proof}
Since the coordinates of $u \in \cB$ are Lipschitz with bounded and
H\"older-continuous first and second derivatives, part (a) follows directly from
\cite[Theorems II.1.2, II.3.3]{kunita1984stochastic}.

For part (b), since $u \in \cB$ has bounded and H\"older-continuous first
derivative, \cite[Theorem II.3.1]{kunita1984stochastic} shows that
$\x^t(\x)$ has derivative $\V^t(\x)=\der_\x \x^t \in \R^{m \times m}$
solving the first-variation equation
\[\d\V^t=[\der u(\x^t)]\V^t\d t,\qquad \V^0=\I.\]
Noting that $\v_i^t=\nabla x_i^t(\x)$
is (the transpose of) the $i^\text{th}$ row of $\V^t$, this gives the first
equation of (\ref{eq:firstsecondvar}) with initial condition $\v_i^0=\e_i$.
Next, consider the joint diffusion
\begin{align*}
\d(\x^t,\V^t)=P(\x^t,\V^t) \d t
+\sqrt{2}(\M\,\d\b^t,0), \qquad P(\x,\V)=\Big(u(\x),[\der u(\x)]\V\Big).
\end{align*}
The condition $u \in \cB$ implies also that
$P(\x,\V)$ has bounded and H\"older-continuous first derivative
$\d P(\x,\V)$, which we identify as a square matrix of dimension $(m+m^2) \times (m+m^2)$
under the vectorization of $\V$. Then
\cite[Theorem II.3.1]{kunita1984stochastic} applied again
shows that $(\x^t(\x),\V^t(\x))$ has
derivative $\U^t=\der_{(\x,\V)}(\x^t,\V^t) \in \R^{(m+m^2) \times (m+m^2)}$
solving the second-variation equation
\begin{equation}\label{eq:Uevolution}
\d\U^t=[\der P(\x^t,\V^t)]\U^t\d t, \qquad \U^0=\I.
\end{equation}
Noting that $\H_i^t=\nabla^2 x_i^t(\x)$
is the block of $\U^t$ corresponding to $\d_\x \v_i^t$,
and that the block corresponding to $\d_\x \x^t$ is $\V^t$,
one may check that the restriction of (\ref{eq:Uevolution})
to the $\d_\x \v_i^t$ block gives exactly the second equation of
(\ref{eq:firstsecondvar}) with initialization $\H_i^0=0$.
If $C>0$ is an upper bound for $\sup_{\x \in \R^m} \|\der u(\x)\|_{\op}$
and $\sup_{\x \in \R^m} \|\der P(\x,\V)\|_\op$, then integrating these equations
gives $\|\V^t\|_\op \leq e^{Ct}\|\V^0\|_\op=e^{Ct}$
and $\|\U^t\|_\op \leq e^{Ct}\|\U^0\|_\op=e^{Ct}$, which implies
the bounds for $\v_i^t$ and $\H_i^t$.

For part (c), consider any $f \in \cA$. Applying (b) and the chain rule,
\begin{equation}\label{eq:fderivs}
\begin{aligned}
\nabla_\x f(\x^t(\x))&=\sum_{j=1}^m \partial_j f(\x^t)\v_i^t\\
\nabla_\x^2 f(\x^t(\x))&=\sum_{j,k=1}^m \partial_j\partial_k f(\x^t)
\v_j^t\v_k^{t\top}+\sum_{j=1}^m \partial_j f(\x^t)\H_j^t.
\end{aligned}
\end{equation}
By parts (a--b) and the condition $f \in \cA$, for any $T>0$,
the right sides of (\ref{eq:fderivs}) are uniformly bounded and
continuous in $(t,\x)$ over $t \in [0,T]$. Then dominated convergence
implies that $P_tf(\x)$ is twice continuously-differentiable in $\x$, that
$\nabla P_t f(\x)=\nabla_\x \langle f(\x^t) \rangle_{\x^0=\x}
=\langle \nabla_\x f(\x^t(\x)) \rangle$
and $\nabla^2 P_t f(\x)=\nabla_\x^2 \langle f(\x^t) \rangle_{\x^0=\x}
=\langle \nabla_\x^2 f(\x^t(\x)) \rangle$, and that
these are also uniformly bounded and continuous over $t \in [0,T]$ and $\x \in
\R^m$.

For the derivative in $t$, by It\^o's formula
\[f(\x^t)=f(\x)+\int_0^t \L f(\x^s)\d s+\int_0^t \nabla f(\x^s)^\top
\sqrt{2}\,\M\,\d\b^s\]
where $\L$ is the generator defined in (\ref{eq:generator}).
Since $\nabla f(\x^s)$ is bounded over $s \in [0,t]$ and $\x^s$ is adapted to
the filtration of $\{\b^s\}$, the last term is a
martingale, so taking expectations gives
\[P_t f(\x)=\langle f(\x^t) \rangle=f(\x)+\int_0^t \langle \L f(\x^s) \rangle\d s.\]
Hence, differentiating in $t$, for any $t>0$ we have
\begin{equation}\label{eq:generatorprop1}
\partial_t P_tf(\x)=\langle \L f(\x^t) \rangle=P_t\L f(\x).
\end{equation}
By Jensen's inequality, for any $s,t \geq 0$, we have
\[\langle (P_s \L f(\x^t))^2 \rangle
\leq \langle \L f(\x^{t+s})^2 \rangle
=\langle (u(\x^{t+s})^\top \nabla f(\x^{t+s})+\Tr \M\M^\top \nabla^2
f(\x^{t+s}))^2 \rangle \leq C(1+\langle \|\x^{t+s}\|_2^2 \rangle),\]
the last
inequality holding for some $C>0$ by boundedness of $\nabla f,\nabla^2 f$
and the Lipschitz continuity of $u$. Then
\cite[Theorem II.2.1]{kunita1984stochastic} implies that
$P_s \L f(\x^t(\x))$ is uniformly bounded in $L^2$ over compact domains of $s,t
\geq 0$ and of the initial condition $\x \in \R^m$,
and hence is also uniformly integrable over these domains.
This uniform integrability for $s=0$ and
dominated convergence shows that
$\langle \L f(\x^t) \rangle$ in (\ref{eq:generatorprop1})
is continuous in $(t,\x)$, and hence $P_tf$ is continuously-differentiable in
$t$. Taking the limit $t \to 0$ in (\ref{eq:generatorprop1}), also
$\L f(\x)=\lim_{t \to 0} \partial_t P_t f(\x)$. Then
applying this with $P_t f \in \cA$ in place of $f$,
\[\L P_t f(\x)=\lim_{s \to 0} \partial_s P_{t+s}f(\x)=\lim_{s \to 0}
\partial_s \langle P_sf(\x^t) \rangle
\overset{(*)}{=} \langle \L f(\x^t) \rangle=P_t \L f(\x).\]
Here, to justify $(*)$, we note that $\partial_s P_s f(\x^t)=P_s \L f(\x^t)$
by (\ref{eq:generatorprop1}), so $(*)$ follows from uniform integrability of
this quantity and dominated convergence to take the limit
$\lim_{s \to 0} \partial_s \langle P_s f(\x^t) \rangle
=\lim_{s \to 0} \langle P_s\L f(\x^t) \rangle=\langle \L f(\x^t) \rangle$.
Combining with (\ref{eq:generatorprop1}), this shows all claims about
$\partial_t P_tf$ in part (c).
\end{proof}

Now consider the perturbed dynamics (\ref{eq:joint_sde_perturbed}) for any
$\eps>0$. Let us denote the perturbed drift as
\[u^\eps(t,\x)=u(\x)+\eps h(t)b(\x).\]
For any $t \geq s \geq 0$, we define its (time inhomogeneous)
Markov semigroup and infinitesimal generator
\[P_{s,t}^\eps f(\x)=\langle f(\x^t) \rangle_{\x^s=\x}
=\E[f(\x^t) \mid \x^s=\x],
\qquad \L_t^\eps f(\x)=u^\eps(t,\x)^\top \nabla f(\x)+
\Tr \M\M^\top \nabla^2 f(\x).\]
The following extends the semigroup regularity estimates of
Proposition \ref{prop:langevinsolution} to this perturbed process.

\begin{proposition}\label{prop:perturbedsolution}
Suppose the assumptions of Proposition \ref{prop:responsegeneral} hold.
Then for any $f \in \cA$, the map $(s,t,\x) \mapsto P_{s,t}^\eps f(\x)$
is continuously-differentiable in $(s,t)$ and twice continuously-differentiable
in $\x$, and furthermore $\nabla P_{s,t}^\eps f(\x),\nabla^2 P_{s,t}^\eps f(\x)$
are uniformly bounded over $s,t \in [0,T]$ and $\x \in \R^m$ for any fixed $T>0$. We have
\begin{equation}\label{eq:backwardequation}
\partial_t P_{s,t}^\eps f(\x)=P_{s,t}^\eps \L_t^\eps f(\x),
\qquad \partial_s P_{s,t}^\eps f(\x)={-}\L_s^\eps P_{s,t}^\eps f(\x).
\end{equation}
\end{proposition}
\begin{proof}
We omit the superscript $\eps$ and write $\x^t \equiv \x^{t,\eps}$.
The same arguments as in Proposition \ref{prop:langevinsolution} using
\cite[Theorems II.1.2, II.3.1, II.3.3]{kunita1984stochastic} show,
for each $s \geq 0$ and $\x
\in \R^m$, there exists a modification $\{\x^t(s,\x)\}_{t \geq s}$ of the
solution to (\ref{eq:joint_sde_perturbed}) with initial condition $\x^s=\x$,
such that $\x^t(s,\x)$ is jointly continuous in $(s,t,\x)$ and twice
continuously-differentiable in $\x$. Each component
$x_i^t(s,\x)$ of this solution has gradient $\v_i^t=\nabla_\x x_i^t(s,\x)$
and Hessian $\H_i^t=\nabla_\x^2 x_i^t(s,\x)$ solving
\[\begin{cases}
\d \v_i^t=\sum_{j=1}^m \partial_j u_i^\eps(t,\x^t(s,\x)) \cdot \v_j^t\,\d t\\
\d \H_i^t=\big(\sum_{j,k=1}^m \partial_j\partial_k u_i^\eps(t,\x^t(s,\x)) \cdot
\v_j^t{\v_k^t}^\top+\sum_{j=1}^m \partial_j u_i^\eps(t,\x^t(s,\x)) \cdot
\H_j^t\big)\d t
\end{cases}\]
with initial conditions $\v_i^s(s,\x)=\e_i$ and $\H_i^s(s,\x)=0$. Furthermore,
$\|\v_i^t(s,\x)\|_2,\|\H_i^t(s,\x)\|_\op \leq e^{C(t-s)}$ for some $C>0$
and all $\x \in \R^m$ and $t \geq s \geq 0$.

Then for any $f \in \cA$, the same dominated convergence argument
as in Proposition \ref{prop:langevinsolution}
shows that $P_{s,t}^\eps f(\x)$ is twice continuously-differentiable in $\x$,
where its first and second derivatives are uniformly bounded and continuous in
$(s,t,\x)$ over $s,t \in [0,T]$ and may be computed by
differentiating in $\x$ under the integral. The same argument as in
Proposition \ref{prop:langevinsolution} using It\^o's formula shows also that
$P_{s,t}^\eps f(\x)$ is continuously-differentiable in $t$, with
\[\partial_t P_{s,t}^\eps f(\x)=P_{s,t}^\eps \L_t^\eps f(\x)
=\langle\L_t^\eps f(\x^t) \rangle_{\x^s=\x}.\]
For the derivative in $s$, we have by It\^o's formula for any $h>0$,
\[P_{s-h,s}^\eps f(\x)=\langle f(\x^s) \rangle_{\x^{s-h}=\x}
=f(\x)+\int_{s-h}^s \langle \L_r^\eps f(\x^r) \rangle_{\x^{s-h}=\x} \,\d r.\]
The same argument as in Proposition \ref{prop:langevinsolution} shows that
$\L_t^\eps f(\x^t(s,\x))$ is uniformly integrable over compact domains of $t
\geq s \geq 0$ and of $\x \in \R^m$, so by dominated convergence
we have $\lim_{h \downarrow 0,\,\,r \uparrow s}
\langle \L_r^\eps f(\x^r) \rangle_{\x^{s-h}=\x}=\L_s^\eps(\x)$.
So taking the limit $h \to 0$ above and rearranging shows
\begin{equation}\label{eq:lowerderiv}
\L_s^\eps f(\x)=\lim_{h \downarrow 0} \frac{P_{s-h,s}^\eps f(\x)-f(\x)}{h}.
\end{equation}
Then for any $s \leq t$, applying this to $P_{s,t}^\eps f \in \cA$
in place of $f$ gives
\[\lim_{h \downarrow 0} \frac{P_{s,t}^\eps f(\x)-P_{s-h,t}^\eps f(\x)}{h}
=\lim_{h \downarrow 0}
\frac{P_{s,t}^\eps f(\x)-P_{s-h,s}^\eps(P_{s,t}^\eps f)(\x)}{h}
={-}\L_s^\eps P_{s,t}^\eps f(\x),\]
i.e.\ $P_{s,t}^\eps f(\x)$ is left-differentiable in $s$. Here
${-}\L_s^\eps P_{s,t}^\eps f(\x)
={-}u^\eps(s,\x)^\top \nabla P_{s,t}^\eps f(\x)
-\Tr \M\M^\top \nabla^2 P_{s,t}^\eps f(\x)$ is continuous in $(s,t,\x)$
by the continuity of $\nabla P_{s,t}^\eps f$ and
$\nabla^2 P_{s,t}^\eps f$ argued above. Then
$P_{s,t}^\eps f(\x)$ is also continuously-differentiable in $s$ with
$\partial_s P_{s,t}^\eps f(\x)={-}\L_s^\eps P_{s,t}^\eps f(\x)$.
\end{proof}

\begin{proof}[Proof of Proposition \ref{prop:responsegeneral}]
Let $\{\x^t\}_{t \geq 0}$ and
$\{\x^{t,\eps}\}_{t \geq 0}$ be the solutions to the
unperturbed and perturbed diffusions. Let $\{P_t\}$ and $\L$ be the semigroup
and infinitesimal generator for $\{\x^t\}_{t \geq 0}$, and let
$\{P^\eps_{s,t}\}$ and $\L^\eps_t$ be those for $\{\x^{t,\eps}\}_{t \geq 0}$.
We write $\partial_s,\partial_t$ for the derivatives in $s,t$ and reserve
$\nabla f(t,\x)$ for the gradient of $f$ in its second argument $\x$.

For any $t>s$ and $r \in [s,t]$, define
$f^\eps(r,\x)=P_{r,t}^\eps a(\x)$. Then by It\^o's formula applied to
the unperturbed process $\{\x^t\}_{t \geq 0}$,
\[f^\eps(t,\x^t)=f^\eps(s,\x^s)+\int_s^t (\partial_r+\L)f^\eps(r,\x^r)
\d r+\int_s^t \nabla f^\eps(r,\x^r)^\top \sqrt{2}\,\M\d\b^r.\]
Proposition \ref{prop:perturbedsolution} shows
$P_{r,t}^\eps a \in \cA$, so $\nabla
f^\eps(r,\x^r)$ is uniformly bounded and the last term is a
martingale. Then, taking expectations under the initial condition $\x^s=\x$
and applying (\ref{eq:backwardequation}),
\begin{align*}
\langle a(\x^t) \rangle_{\x^s=\x}=\langle f^\eps(t,\x^t) \rangle_{\x^s=\x}
&=f^\eps(s,\x)+\int_s^t \langle (\partial_r+\L)f^\eps(r,\x^r) \rangle_{\x^s=\x}
\d r\\
&=P_{s,t}^\eps a(\x)+\int_s^t \langle ({-}\L_r^\eps+\L)P_{r,t}^\eps a(\x^r)
\rangle_{\x^s=\x}\d r\\
&=P_{s,t}^\eps a(\x)-\int_s^t \eps h(r)\langle(b^\top
\nabla  P_{r,t}^\eps a)(\x^r)\rangle_{\x^s=\x}\d r\\
&=P_{s,t}^\eps a(\x)-\eps \int_s^t h(r)P_{r-s}(b^\top
\nabla P_{r,t}^\eps a)(\x) \d r.
\end{align*}
Applying this also with $\eps=0$ and $P_{s,t}^0=P_{t-s}$ and
taking the difference, we obtain the identity
\begin{equation}\label{eq:perturbationidentity}
P_{s,t}^\eps a(\x)-P_{t-s} a(\x)
=\eps \int_s^t h(r)P_{r-s}(b^\top \nabla P_{r,t}^\eps a)(\x) \d r.
\end{equation}

From the definition of $P_t f(\x)$ and
form of $\nabla P_t f(\x)$ in (\ref{eq:Ptspaceder}), we have
\begin{align}
P_{r-s}(b^\top \nabla P_{r,t}^\eps a)(\x)
&=\Big\langle (b^\top \nabla P_{r,t}^\eps a)(\x^{r-s})
\Big\rangle_{\x_0=\x},\label{eq:UI1}\\
\nabla P_{r-s}(b^\top \nabla P_{r,t}^\eps a)(\x)
&=\Bigg\langle \sum_{i=1}^m \partial_{x_i}[b^\top \nabla
P_{r,t}^\eps a](\x^{r-s})\v_i^{r-s}\Bigg\rangle_{\x_0=\x}.\label{eq:UI2}
\end{align}
Since $b \in \cB$ is Lipschitz by assumption, and
$P_{r,t}^\eps a \in \cA$ by Proposition \ref{prop:perturbedsolution}, we have
\[|(b^\top \nabla P_{r,t}^\eps a)(\x^{r-s})|,\,
|\partial_{x_i}[b^\top \nabla P_{r,t}^\eps a](\x^{r-s})| \leq C(1+\|\x^{r-s}\|_2)\]
for some $C>0$. Then these quantities are uniformly integrable over bounded
domains of $s \leq r \leq t$ and $\x$,
by \cite[Theorem II.2.1]{kunita1984stochastic}.
Furthermore $\|\v_i^{r-s}\|_2$ is bounded by
Proposition \ref{prop:langevinsolution}(b), so the integrands on the right sides
of both (\ref{eq:UI1}--\ref{eq:UI2}) are also uniformly integrable over these
domains. Then applying dominated convergence, we may
differentiate (\ref{eq:perturbationidentity}) in $\x$ under the integral to
obtain
\[\nabla P_{s,t}^\eps a(\x)-\nabla P_{t-s} a(\x)
=\eps \int_s^t h(r)\nabla P_{r-s}(b^\top \nabla P_{r,t}^\eps a)(\x) \d r,\]
and take the limit $\eps \to 0$ to get
$\nabla P_{s,t}^\eps a(\x) \to \nabla P_{t-s} a(\x)$. Applying this with $s=r$
to the right side of (\ref{eq:perturbationidentity}), and
taking the limit $\eps \to 0$ in (\ref{eq:perturbationidentity}) using uniform
integrability of (\ref{eq:UI1}), we arrive at
\[\lim_{\eps \to 0} \frac{P_{s,t}^\eps a(\x)-P_{t-s} a(\x)}{\eps}
=\int_s^t h(r)P_{r-s}(b^\top \nabla P_{t-r}a)(\x) \d r.\]
For $s=0$, this means
\[\lim_{\eps \to 0} \frac{1}{\eps}\Big(\langle a(\x^{t,\eps}) \rangle-\langle
a(\x^t) \rangle\Big)
=\int_0^t h(r)P_r(b^\top \nabla P_{t-r}a)(\x) \d r,\]
verifying that (\ref{eq:responsedef})
holds with response function $R(t,s)$ given by (\ref{eq:explicitresponse}).
Continuity of this function $R(t,s)$ in $(s,t)$ follows from the above uniform
integrability statements, together with continuity of
$t \mapsto \nabla P_t(\x)$ in $t$ as shown in Proposition
\ref{prop:langevinsolution}.

For uniqueness, observe that if $\tilde R(t,s)$ is any continuous function
different from $R(t,s)$, then they must differ on a subset of $(s,t)$ of
positive Lebesgue measure. Then there exists a continuous
bounded function $h:[0,\infty) \to \R$ such that $\int_0^t R(t,s)h(s)\d s
\neq \int_0^t \tilde R(t,s)h(s)\d s$, implying that
$\tilde R$ cannot satisfy (\ref{eq:responsedef}).
Thus this response function $R(t,s)$ is unique.
\end{proof}

\subsection{Discrete dynamics}

We record (elementary) analogues of the preceding results for discrete dynamics
\begin{equation}\label{eq:discrete_dynamics}
\x^{t+1}=\x^{t} + u(\x^t) + \sqrt{2}\,\M(\b^{t+1} - \b^t)
\end{equation}
where $\{\b^t\}_{t \in \Z_+}$ is a Gaussian process with $\b^0=0$
and independent increments $\b^{t+1}-\b^t \sim \N(0,\gamma\I)$, for some
$\gamma>0$. The following is an analogue of Proposition
\ref{prop:responsegeneral}.

\begin{proposition}\label{prop:responsegeneraldiscrete}
Suppose $u:\R^m \to \R^m$ is Lipschitz, and let $\{\x^t\}_{t \in \Z_+}$
be the solution of (\ref{eq:discrete_dynamics}) with initial condition
$\x^0=\x$. For any Lipschitz functions $a:\R^m \rightarrow\R$ and $b:\R^m \to \R^m$, define
\[R(t,s)=P_s\big(b^\top P(\nabla P_{t-s-1} a)\big)(\x)\]
where $P_t f(\x)=\E[f(\x^t) \mid \x^0=\x]$. Then for any $s,t \in \Z_+$
with $s<t$,
\[R(t,s)=\lim_{\eps \to 0}
\frac{1}{\eps}\Big(\E[a(\x^{t,\eps}) \mid \x^{0,\eps}=\x]
-\E[a(\x^t) \mid \x^0=\x]\Big)\]
where $\{\x^{t,\eps}\}_{t \in \Z_+}$ is the solution of the perturbed dynamics
\[\x^{t+1,\eps}=\x^{t,\eps} + u(\x^{t,\eps})
+\eps b(\x^{t,\eps})\1_{s=t}
+\sqrt{2}\,\M(\b^{t+1} - \b^t)\]
with the same initial condition $\x^{0,\eps}=\x$.
\end{proposition}
\begin{proof}
Write as shorthand $P=P_1$. If $f$ is $L$-Lipschitz, then
(coupling the processes with initializations $\x,\y$ by the same $\{\b^t\}$)
\begin{align*}
|P f(\x)-P f(\y)|&=\Big|\E[f(\x+u(\x)-\sqrt{2}\M\,\b^1)]
-\E[f(\y+u(\y)-\sqrt{2}\M\,\b^1)]\Big|
\leq L(1+L_u)\|\x-\y\|
\end{align*}
where $L_u$ is the Lipschitz constant of $u$. Hence $Pf$ is Lipschitz, so $P_tf$
is Lipschitz for all $t \geq 0$.

Let $P_t^\eps$ be the Markov semigroup for the dynamics
\[\x^{t+1}=\x^t + u(\x^t)+\eps b(\x^t)+\sqrt{2}\,\M(\b^{t+1}-\b^t),\]
and write as shorthand $P^\eps=P_1^\eps$. Then by definition,
\[\E[a(\x^{t,\eps}) \mid \x^{0,\eps}=\x]
=P_s P^\eps P_{t-s-1}(\x),\]
so
\begin{equation}\label{eq:responsecalculation}
\lim_{\eps \to 0}
\frac{1}{\eps}\Big(\E[a(\x^{t,\eps}) \mid \x^{0,\eps}=\x]
-\E[a(\x^t) \mid \x^0=\x]\Big)
=\deps P_s P^\eps P_{t-s-1}a(\x).
\end{equation}
Note that for any $L$-Lipschitz function $f$, we have
\begin{equation}\label{eq:dPeps}
\partial_\eps P^\eps f(\x)=\partial_\eps \E[f(\x+u(\x)+\eps b(\x)+\b^1)]
=b(\x)^\top \E[\nabla f(\x+u(\x)+\eps b(\x)+\b^1)]
\end{equation}
where the derivative may be taken under the expectation by dominated
convergence. In particular,
\[\deps P^\eps f(\x)=b(\x)^\top \E[\nabla f(\x+u(\x)+\b^1)]
=b(\x)^\top P(\nabla f)(\x).\]
The derivative (\ref{eq:dPeps}) is also
bounded for all $\eps \geq 0$ by $L\|b(\x)\|$, which is
integrable under $P_s$ since $b$ is Lipschitz.
Then again by dominated convergence,
\[\deps P_s P^\eps P_{t-s-1}a(\x)=P_s \deps P^\eps P_{t-s-1}a(\x)
=P_s(b^\top P(\nabla P_{t-s-1}a))(\x),\]
and the result follows from applying this to (\ref{eq:responsecalculation}).
\end{proof}

The following is an analogue of the first statement of (\ref{eq:Ptspaceder}).

\begin{lemma}\label{lem:discrete_bismut}
Let $\{\x^t\}_{t \in \Z_+}$ be the solution to (\ref{eq:discrete_dynamics})
where $u(\cdot)$ is Lipschitz, and consider the first variation processes
\begin{align*}
\v_i^{t+1}&=\v_i^t+\sum_{j=1}^m \partial_j u_i(\x^t) \cdot \v_j^t
\end{align*}
with initializations $\v_i^0=\e_i$. Denote $P_t f(\x)=\E[f(\x^t) \mid
\x^0=\x]=\langle f(\x^t) \rangle$.
Then for any Lipschitz function $f:\R^{d+K}\rightarrow\R$, 
\begin{align*}
\nabla P_t f(\x) = \bigg\langle \sum_{j=1}^m \partial_j f(\x^t)\v_j^t
\bigg\rangle.
\end{align*}
\end{lemma}
\begin{proof}
Stacking $\V^t=[\v_1^t,\ldots,\v_m^t]^\top \in \R^{m \times m}$ with initial
condition $\V^0=\I_m$, the evolution of $\V^t$ is
\[\V^{t+1}=[\I+\d u(\x^t)]\V^t\]
where $\d u$ is the derivative of $u(\cdot)$. Writing $\x^t(\x)$ for the
dependence of $\x^t$ on the initial condition $\x^0=\x$, and writing
$\d\x^t(\x)$ for its derivative in $\x$, by the chain rule we
have $\d \x^{t+1}(\x)=[\I+\d u(\x^t)]\d \x^t(\x)$, with initial condition
$\d\x^0(\x)=\I$. Thus $(\V^t)^\top=\d\x^t(\x)$ for all $t \geq 0$, so
\[\nabla_\x f(\x^t(\x))=[\d\x^t(\x)]^\top \nabla f(\x^t)
=\sum_{j=1}^m \partial_j f(\x^t)\v_j^t.\]
By dominated convergence we have $\nabla P_tf(\x)
=\nabla_\x \langle f(\x^t(\x)) \rangle=\langle \nabla_\x f(\x^t(\x)) \rangle$,
and the result follows.
\end{proof}

\bibliographystyle{unsrt}
\bibliography{eb_dmft}
\end{document}